\documentclass[reqno,10pt,letterpaper]{amsart}

\usepackage{lipsum}
\usepackage{amsmath}
\usepackage{amssymb}
\usepackage{amsthm}
\usepackage{mathrsfs}
\usepackage{accents}
\usepackage{calc}
\usepackage{arydshln}
\usepackage{upgreek}
\usepackage{slashed}
\usepackage{xifthen}
\usepackage{graphicx}
\usepackage{subcaption}
\usepackage{longtable}
\usepackage[inline]{enumitem}

\usepackage{tikz}

\newcommand{\smallrighttriangle}{{%
\tikz[scale=0.12, line join=round, baseline={(0,0)}]{
  \draw[line width=0.4pt] (0,0) -- (1,0) -- (0,1) -- cycle;
}}}

\usepackage{xcolor}
\definecolor{winered}{rgb}{0.6,0,0}
\definecolor{lessblue}{rgb}{0,0,0.7}

\usepackage[pdftex,colorlinks=true,linkcolor=winered,citecolor=lessblue,urlcolor=lessblue,breaklinks=true,bookmarksopen=true]{hyperref}

\makeatletter
\newcommand{\myitem}[2]{\item[\rm(#2)]\def\@currentlabel{#2}\label{#1}}
\makeatother

\usepackage{titletoc}

\makeatletter

\def\@tocline#1#2#3#4#5#6#7{
\begingroup
  \par
    \parindent\z@ \leftskip#3 \relax \advance\leftskip\@tempdima\relax
                  \rightskip\@pnumwidth plus 4em \parfillskip-\@pnumwidth
    \ifcase #1 
       \vskip 0.6em \hskip 0em 
       \or
       \or \hskip 0em 
       \or \hskip 1em 
    \fi%
    #6
    \nobreak\relax{\leavevmode\leaders\hbox{\,.}\hfill}
    \hbox to\@pnumwidth {\@tocpagenum{#7}}
  \par
\endgroup
}

 \def\l@section{\@tocline{0}{0pt}{0pc}{}{}}

\renewcommand{\tocsection}[3]{%
  \indentlabel{\@ifnotempty{#2}{ 
    \ignorespaces\bfseries{#2. #3}}}
  \indentlabel{\@ifempty{#2}{\ignorespaces\bfseries{#3}}{}} 
    \vspace{1.5pt}}

\renewcommand{\tocsubsection}[3]{%
  \indentlabel{\@ifnotempty{#2}{
    \ignorespaces#2. #3}}
  \indentlabel{\@ifempty{#2}{\ignorespaces #3}{}}
    \vspace{1.5pt}}

\renewcommand{\tocsubsubsection}[3]{%
  \indentlabel{\@ifnotempty{#2}{
    \ignorespaces#2. #3}}
  \indentlabel{\@ifempty{#2}{\ignorespaces #3}{}}
    \vspace{1.5pt}}

\makeatother

\makeatletter
\def\@nomenstarted{0}
\newlength{\@nomenoldtabcolsep}

\newcommand{\nomenstart}
  {%
    \def\@nomenstarted{1}%
    \setlength{\@nomenoldtabcolsep}{\tabcolsep}%
    \setlength{\tabcolsep}{3.5pt}%
    \begin{longtable}{p{0.11\textwidth} p{0.86\textwidth}}
  }

\newcommand{\nomenitem}[2]{%
    \ifcase\@nomenstarted%
      \or 
      \or \\ 
    \fi%
    #1\,{\leavevmode\leaders\hbox{\,.}\hfill} & #2%
    \def\@nomenstarted{2}%
  }%
\newcommand{\nomenend}
  {\\%
      \end{longtable}%
      \setlength{\tabcolsep}{\@nomenoldtabcolsep}%
      \def\@nomenstarted{0}%
  }
\makeatother

\makeatletter
\newcommand{\BIG}{\bBigg@{3.5}}
\newcommand{\vast}{\bBigg@{4}}
\newcommand{\Vast}{\bBigg@{5}}
\newcommand{\VAST}[1]{\bBigg@{#1}}
\makeatother

\allowdisplaybreaks

\numberwithin{equation}{section}
\numberwithin{figure}{section}
\newtheorem{thm}{Theorem}[section]

\newtheorem{prop}[thm]{Proposition}
\newtheorem{lemma}[thm]{Lemma}
\newtheorem{cor}[thm]{Corollary}

\newtheorem*{thm*}{Theorem}
\newtheorem*{prop*}{Proposition}
\newtheorem*{cor*}{Corollary}
\newtheorem*{conj*}{Conjecture}

\theoremstyle{definition}
\newtheorem{definition}[thm]{Definition}
\newtheorem{notation}[thm]{Notation}

\theoremstyle{remark}
\newtheorem{rmk}[thm]{Remark}

\makeatletter
\newcommand{\fakephantomsection}{%
  \Hy@MakeCurrentHref{\@currenvir.\the\Hy@linkcounter}
  \Hy@raisedlink{\hyper@anchorstart{\@currentHref}\hyper@anchorend}%
  \Hy@GlobalStepCount\Hy@linkcounter%
}
\makeatother

\newcommand{\mc}{\mathcal}
\newcommand{\cA}{\mc A}

\newcommand{\cC}{\mc C}

\newcommand{\cE}{\mc E}
\newcommand{\cF}{\mc F}

\newcommand{\cH}{\mc H}

\newcommand{\cK}{\mc K}
\newcommand{\cL}{\mc L}

\newcommand{\cO}{\mc O}

\newcommand{\cR}{\mc R}

\newcommand{\cT}{\mc T}
\newcommand{\cU}{\mc U}
\newcommand{\cV}{\mc V}

\newcommand{\ms}{\mathscr}
\newcommand{\sA}{\ms A}

\newcommand{\sC}{\ms C}

\newcommand{\C}{\mathbb{C}}
\newcommand{\N}{\mathbb{N}}
\newcommand{\R}{\mathbb{R}}
\newcommand{\Z}{\mathbb{Z}}

\newcommand{\Sph}{\mathbb{S}}

\newcommand{\sfG}{\mathsf{G}}
\newcommand{\sfH}{\mathsf{H}}

\newcommand{\sfM}{\mathsf{M}}

\newcommand{\fc}{\mathfrak{c}}

\newcommand{\fk}{\mathfrak{k}}
\newcommand{\fm}{\mathfrak{m}}

\newcommand{\ft}{\mathfrak{t}}

\newcommand{\sld}{\slashed{\dd}{}}
\newcommand{\slg}{\slashed{g}{}}

\newcommand{\slsfG}{\slashed{\sfG}{}}

\newcommand{\slU}{\slashed{U}{}}

\newcommand{\slGamma}{\slashed{\Gamma}{}}
\newcommand{\sldelta}{\slashed{\delta}{}}
\newcommand{\slDelta}{\slashed{\Delta}{}}
\newcommand{\slvarphi}{\slashed{\varphi}{}}
\newcommand{\slnabla}{\slashed{\nabla}{}}

\newcommand{\slstar}{\slashed{\star}}

\newcommand{\sltr}{\operatorname{\slashed\tr}}

\newcommand{\scal}{\mathsf{S}}
\newcommand{\scalspace}{\mathbf{S}}
\newcommand{\vect}{\mathsf{V}}

\newcommand{\End}{\operatorname{End}}

\newcommand{\Hom}{\operatorname{Hom}}
\renewcommand{\Re}{\operatorname{Re}}
\renewcommand{\Im}{\operatorname{Im}}
\newcommand{\Id}{\operatorname{Id}}

\newcommand{\mathspan}{\operatorname{span}}
\newcommand{\supp}{\operatorname{supp}}

\newcommand{\tr}{\operatorname{tr}}

\newcommand{\diag}{\operatorname{diag}}

\newcommand{\halfopen}{[0,1)} 

\newcommand{\cd}{\fc}

\newcommand{\Ups}{\Upsilon}

\newcommand{\eps}{\epsilon}

\newcommand{\hra}{\hookrightarrow}
\newcommand{\la}{\langle}

\newcommand{\ol}{\overline}
\newcommand{\pa}{\partial}
\newcommand{\dd}{{\mathrm d}}
\newcommand{\ra}{\rangle}
\newcommand{\spec}{\operatorname{spec}}

\newcommand{\wh}{\widehat}

\newcommand{\xra}{\xrightarrow}

\newcommand{\ubar}[1]{\underaccent{\bar}#1}
\newcommand{\pfstep}[1]{$\bullet$\ \underline{\textit{#1}}}

\newcommand{\bop}{{\mathrm{b}}}

\newcommand{\scop}{{\mathrm{sc}}}

\newcommand{\cl}{{\mathrm{cl}}}

\newcommand{\ebop}{{\mathrm{eb}}}

\newcommand{\tbop}{{3\mathrm{b}}}
\newcommand{\tscop}{{3\mathrm{sc}}}

\newcommand{\semi}{\hbar}

\newcommand{\cface}{{\mathrm{cf}}}

\newcommand{\sface}{{\mathrm{sf}}}

\newcommand{\rms}{{\mathrm{s}}}
\newcommand{\rmv}{{\mathrm{v}}}

\newcommand{\cp}{{\mathrm{c}}}

\newcommand{\Diff}{\mathrm{Diff}}

\DeclareMathOperator{\Op}{Op}

\newcommand{\Vb}{\cV_\bop}

\newcommand{\Diffb}{\Diff_\bop}

\newcommand{\Psisc}{\Psi_\scop}

\newcommand{\Vtb}{\cV_\tbop}

\newcommand{\Difftsc}{\Diff_\tscop}
\newcommand{\Difftb}{\Diff_\tbop}
\newcommand{\Difftbh}{\Diff_{\tbop,\hbar}}
\newcommand{\Psitb}{\Psi_\tbop}
\newcommand{\Psitbh}{\Psi_{\tbop,\semi}}
\newcommand{\Veb}{\cV_\ebop}

\newcommand{\Diffeb}{\Diff_\ebop}

\newcommand{\Vsc}{\cV_\scop}
\newcommand{\Diffsc}{\Diff_\scop}

\newcommand{\WF}{\mathrm{WF}}
\newcommand{\Ell}{\mathrm{Ell}}
\newcommand{\Char}{\mathrm{Char}}

\newcommand{\WFtbh}{\WF_{\tbop,\hbar}}
\newcommand{\Elltb}{{\rm Ell}_\tbop}
\newcommand{\Elltbh}{{\rm Ell}_{\tbop,\hbar}}

\newcommand{\Tb}{{}^{\bop}T}

\newcommand{\Tsc}{{}^{\scop}T}

\newcommand{\Teb}{{}^{\ebop}T}
\newcommand{\Ttb}{{}^{\tbop}T}
\newcommand{\Ttbh}{{}^{\tbop,\semi}T}

\newcommand{\Stb}{{}^{\tbop}S}

\newcommand{\WFebh}{\WF_{\ebop,\semi}}

\newcommand{\sigmab}{\upsigma_\bop}

\newcommand{\sigmah}{\upsigma_\hbar}

\newcommand{\sigmasc}{\upsigma_\scop}

\newcommand{\sigmatb}{\upsigma_\tbop}
\newcommand{\sigmatbh}{\upsigma_{\tbop,\semi}}
\newcommand{\sigmaebh}{\upsigma_{\ebop,\semi}}

\newcommand{\CI}{\cC^\infty}

\newcommand{\CIc}{\cC^\infty_\cp}

\newcommand{\Hb}{H_{\bop}}

\newcommand{\Hbext}{\bar H_{\bop}}

\newcommand{\Hbsupp}{\dot H_{\bop}}

\newcommand{\Heb}{H_{\ebop}}

\newcommand{\Htb}{H_\tbop}
\newcommand{\Htbh}{H_{\tbop,h}}

\newcommand{\Ric}{\mathrm{Ric}}

\newcommand{\bhm}{\fm}

\newcommand{\openbigpmatrix}[1]
  {%
    \def\@bigpmatrixsize{#1}%
    \addtolength{\arraycolsep}{-#1}%
    \begin{pmatrix}%
  }
\newcommand{\closebigpmatrix}
  {%
    \end{pmatrix}%
    \addtolength{\arraycolsep}{\@bigpmatrixsize}%
  }

\newlength{\enummargin}

\newcommand{\usref}[1]{{\upshape\ref{#1}}}

\DeclareGraphicsExtensions{.mps}

\makeatletter
\newcommand*{\fwbw}[1]{\expandafter\@fwbw\csname c@#1\endcsname}
\newcommand*{\@fwbw}[1]{\ifcase #1 \or {\rm fw}\or {\rm bw}\fi}
\AddEnumerateCounter{\fwbw}{\@fwbw}
\makeatother

\begin{document}

\title[Constraint damping on Kerr]{Constraint damping on subextremal Kerr spacetimes}

\date{\today}


\author{Peter Hintz}
\address{Department of Mathematics, Pennsylvania State University, 54 McAllister St, State
College,\newline PA 16801, United States}
\email{phintz@psu.edu}

\begin{abstract}
  In the context of hyperbolic formulations of Einstein's field equations obtained via gauge fixing, constraint damping is a desirable feature that ensures that violations of the gauge condition and thus of the constraint equations are suppressed in evolution. Besides its utility in numerical relativity, it has played a key role in several (linear and nonlinear) stability proofs of spacetimes as solutions of the Einstein equation. In this paper, we show that an enhanced form of constraint damping can be implemented for the linearization of the Einstein equation around any subextremal Kerr black hole metric: the constraint propagation wave operator satisfies mode stability and has an arbitrarily large indicial gap at zero energy. The results proved here are a key ingredient in the author's proof \cite{HintzKerrStab} of the nonlinear stability of the subextremal Kerr family.
\end{abstract}

\maketitle

\setlength{\parskip}{0.00pt}
\tableofcontents
\setlength{\parskip}{0.05in}

\section{Introduction}
\label{SI}

We study a problem concerning a certain wave operator on 1-forms which arises in the study of the nonlinear stability problem for subextremal Kerr spacetimes (settled in \cite{HintzKerrStab}). To introduce this operator, consider a Lorentzian metric $g$ (signature $(-,+,\ldots,+)$) on a smooth manifold $M$, and denote by $(\delta_g^*\omega)_{\mu\nu}=\frac12(\omega_{\mu;\nu}+\omega_{\nu;\mu})$, $\sfG_g=\Id-\frac12 g\tr_g$, and $(\delta_g h)_\mu=-g^{\kappa\lambda}h_{\mu\kappa;\lambda}$ the symmetric gradient on 1-forms, the ``trace reversal'' operator, and the (negative) divergence on symmetric 2-tensors, respectively. For a smooth vector bundle map $\cC\colon T^*M\to S^2 T^*M$, consider the modified symmetric gradient
\[
  \delta_{g,\cC}^*\omega := \delta_g^*\omega + \cC\omega.
\]
The \emph{constraint propagation wave operator} associated with $\cC$ is then
\[
  \Box_g^\cC := 2\delta_g\sfG_g\delta_{g,\cC}^*.
\]
This is the tensor wave operator on 1-forms on $(M,g)$ plus a first order operator depending on $\cC$.

In the context of the Einstein vacuum equation $\Ric(g)=0$, this operator arises as follows. Consider the gauge condition $\Ups(g)=0$, where $\Ups(g)^\sharp=\tr_g(\nabla^g-\nabla^{g^0})$; here $g^0$ is some fixed Lorentzian metric, e.g., the metric one is interested in perturbing in the context of a stability problem for a given spacetime. For any $\cC$ as above, one can then consider the \emph{gauge-fixed Einstein equation}\footnote{There are many possible modifications of this that preserve the quasilinear hyperbolic character of the equation~\eqref{EqIEinstein} and are thus equally suitable at least for short-time existence for $\Ric(g)=0$. For example, one can replace $\Ups(g)$ by $\Ups(g)+A g+\theta$ where $A$ is a bundle map and $\theta$ is a 1-form, both independent of $g$; or one can replace $\delta_{g,\cC}^*$ by $\delta_{g^0,\cC}^*$. Since in this paper the Einstein equation only serves as motivation but is not studied in itself, we do not strive for any generality here.}
\begin{equation}
\label{EqIEinstein}
  \Ric(g) - \delta_{g,\cC}^*\Ups(g) = 0
\end{equation}
which is a quasilinear wave equation for the Lorentzian metric $g$. If $g$ is a solution, let us apply $\delta_g\sfG_g$ to this equation; in light of the second Bianchi identity which states that $\delta_g\sfG_g\Ric(g)=0$ for any $g$, we obtain the wave equation $\Box_g^\cC(\Ups(g))=0$ for the gauge 1-form. When studying~\eqref{EqIEinstein} for general (small) Cauchy data (which may violate the gauge condition, as could happen, e.g., in numerical simulations), it is then convenient to arrange for homogeneous solutions of $\Box_g^\cC$ to decay as fast as possible as some time function $t$ tends to $+\infty$. (See \S\ref{SsIM} below for further motivation.) This also ensures that $g$ becomes an increasingly accurate solution of $\Ric(g)=0$ as $t\to\infty$, and thus the first and second fundamental forms of level sets of $t$ become increasingly accurate solutions of the constraint equations; hence the terminology ``stable constraint propagation'' \cite{BrodbeckFrittelliHubnerReulaSCP,HintzVasyKdSStability} and ``constraint damping'' \cite{GundlachCalabreseHinderMartinConstraintDamping,PretoriusBinaryBlackHole} for formulations of the gauge-fixed Einstein equation with this feature.

In this paper, we study the case that $g$ is a subextremal Kerr metric; see~\eqref{EqKBL} for its explicit expression. Then 1-form solutions $\omega$ of $\Box_g\omega=0$, with $\CIc$ initial data, say, generically do not decay over compact spatial sets in time at all. This is due to the existence of a 1-dimensional space of stationary (Coulomb-type) solutions \cite[Theorem~5.1(2)]{AnderssonHaefnerWhitingMode}. In order to enforce decay of solutions of $\Box_g^\cC$, it is thus necessary to use nontrivial modifications $\cC\neq 0$ of the symmetric gradient.

\begin{thm}[Mode stability of $\Box_g^\cC$: rough form]
\label{ThmIRough}
  Let $(M,g)$ denote a subextremal Kerr spacetime, with mass $\bhm>0$ and specific angular momentum $a\in(-\bhm,\bhm)$. Let $C_0>0$, $v\in(0,1)$. Then for sufficiently small $e>0$, there exist $h_0\in(0,1)$ and a smooth, timelike, stationary 1-form $\cd$, equal to $r^{-1}(\dd t-v\,\dd r)$ for all sufficiently large $r$, such that for all $h\in(0,h_0)$, the following statements hold true for the operator
  \begin{equation}
  \label{EqIRoughC}
    \Box_g^{\cC_h}=2\delta_g\sfG_g(\delta_g^*+\cC_h),\quad
    \cC_h\omega := h^{-1}\bigl(2\cd\otimes_s\omega - (1-e)g\,\iota_{\cd^\sharp}\omega\bigr),\ \cd^\sharp:=g^{-1}(\cd,\cdot).
  \end{equation}
  \begin{enumerate}
  \item\label{ItIRoughMS}{\rm (Mode stability.)} Let $\sigma\in\C$, $\Im\sigma\geq 0$, $\sigma\neq 0$. Then there do not exist nontrivial mode solutions (i.e., outgoing at infinity and at the future event horizon) of $\Box_g^{\cC_h}$ with frequency $\sigma$.
  \item\label{ItIRough0}{\rm (Enhanced mode stability at zero energy.)} Suppose $\omega$ is a smooth stationary 1-form on $M$ solving $\Box_g^{\cC_h}\omega=0$ with the property that its coefficients $\omega_\mu$ with respect to Cartesian coordinate differentials $\dd z^\mu$, $\mu=0,1,2,3$, obey the pointwise bound $|\omega_\mu|\leq C r^{C_0}$ for some $C<\infty$. Then $\omega=0$.
  \end{enumerate}
\end{thm}

The ``outgoing'' property in part~\eqref{ItIRoughMS} is clarified in Definition~\ref{DefTOutgoing}. A detailed version of Theorem~\ref{ThmIRough} is stated as Theorem~\ref{ThmT}. In particular, Theorem~\ref{ThmT}\eqref{ItTZero} significantly strengthens Theorem~\ref{ThmIRough}\eqref{ItIRough0} to the invertibility of the \emph{zero energy operator} $\wh{\Box_g^{\cC_h}}(0)$ (obtained by dropping all $t$-derivatives of $\Box_g^{\cC_h}$) acting on weighted function spaces encoding decay/growth rates anywhere between $r^{-1}$ and $r^{C_0}$ (and mapping into similar spaces with additional $r^{-2}$ decay).

The concrete choice of the parameter $v$ is not important here or for our application to the Kerr stability problem in~\cite{HintzKerrStab}; the reader may take $v=\frac12$ throughout this paper. The stationarity of $\cd$ however, as well as the degree of homogeneity of $\cd$ for large $r$ with respect to scaling, are important: they ensure that $\Box_g^{\cC_h}$ is time-translation-invariant (necessary for the definition of mode solutions) and has the same approximate degree $-2$ homogeneity with respect to spacetime dilations outside of any cone $|\frac{r}{t}|<\delta$, $\delta>0$, as the standard (scalar or tensor) wave operator on Kerr. The particular choice of sign of the $\dd r$ term in $\cd$ is also crucial; it ensures that the dual vector field $-r\cd^\sharp\approx\pa_t+v\pa_r$ points away from the black hole for large $r$. (The relevance of this is of a somewhat technical nature, cf.\ Remarks~\ref{RmkM0WrongV} and~\ref{RmkEBOut}.)

If one is only interested in mode stability and the non-existence of non-trivial zero modes that decay (i.e., $|\omega_\mu|=o(1)$ as $r\to\infty$), then a small modification $\cC$ suffices. This was shown using perturbative techniques in \cite[Proposition~3.7]{HintzGlueLocIII}, where it was proved that one can even take $\cC$ to be compactly supported in space (and independent of time). However, for compactly supported $\cC$, the operator $\Box_g^\cC$ always admits non-trivial stationary solutions which are of size $\cO(1)$ as $r\to\infty$; in fact, there is an at least $4$-dimensional space of these (of the form $\dd z^\mu+o(1)$ where $z=(t,x^1,x^2,x^3)$), as follows from the normal operator arguments used in the proof of \cite[Lemma~3.11(1)]{HintzGlueLocIII}. This remains true, with $\cO(1)$ replaced by $\cO(r^\delta)$ for any fixed $\delta>0$, when $\cC$ is not compactly supported but scales like $r^{-1}$ as $r\to\infty$ (as does $\cC_h$ in~\eqref{EqIRoughC}) and is small; this follows from indicial root considerations (see~\S\ref{SssIA0} below). As a consequence, the enhanced mode stability in Theorem~\ref{ThmIRough}\eqref{ItIRough0} \emph{cannot} be arranged \emph{perturbatively} for any fixed positive value of $C_0$, say, $C_0=1$. (In our application in \cite{HintzKerrStab}, we use $C_0=100$.) Instead, we utilize a large parameter $h^{-1}$ in~\eqref{EqIRoughC} and prove (enhanced) mode stability when $h^{-1}$ is sufficiently large.

\subsection{Context and motivation}
\label{SsIM}

We first recall the importance of mode stability for the constraint damping wave operator in the setting of de~Sitter (dS) type spacetimes in~\S\ref{SssIMdS} (where it first entered the mathematical literature on nonlinear stability problems) and then in the asymptotically flat setting in~\S\ref{SssIMFlat}. Finally, we discuss the significance of enhanced mode stability---which has no analogue in the dS setting---in~\S\ref{SssIME}.

\subsubsection{Mode stability for \texorpdfstring{$\Box_g^\cC$}{the constraint propagation wave operator} on de~Sitter type spacetimes}
\label{SssIMdS}

Arranging mode stability (including at $\sigma=0$) for the constraint propagation wave operator was a crucial step in the works \cite{HintzVasyKdSStability} (nonlinear stability of KdS, small angular momenta), \cite{HintzPetersenVasyKdS} (KdS, full subextremal range), and \cite{HintzKNdSStability} (KNdS, subextremal charges and small angular momenta), and also in the static de~Sitter case as studied in \cite[Appendix~C]{HintzVasyKdSStability}. We recall the setup: we study perturbations of $(M,g_0)$ (say, slowly rotating KdS), with $\Ric(g_0)-\Lambda g_0=0$, by solving initial value problems for (modifications of)
\[
  \Ric(g)-\Lambda g - \delta_{g_0,\cC}^*\Ups(g) = 0.
\]
We attempt to solve this quasilinear equation globally in the future $\{t\geq 0\}$ of a Cauchy hypersurface $\{t=0\}$ by applying a Nash--Moser iteration scheme in which we solve a linearized version of this equation, with initial data and right-hand side given by the error left over after the previous iteration step. Consider the linearization around $g=g_0$, which is the linear wave-type operator
\begin{equation}
\label{EqILinEin}
  L_{g_0}h := D_{g_0}\Ric(h)-\Lambda h - \delta_{g_0,\cC}^*D_{g_0}\Ups(h)
\end{equation}
on $(M,g_0)$ acting on symmetric 2-tensors $h$.

Now, solutions of $L_{g_0}h=0$ (or more generally of $L_{g_0}h=f=\cO(e^{-\alpha t})$, $\alpha>0$) admit resonance expansions, i.e., they are sums of \emph{mode solutions} $e^{-i\sigma t}h_0(x)\in\ker L_{g_0}$ (ignoring, for brevity, the possibility of modes of higher multiplicity), $\Im\sigma>-\alpha$, and an exponentially decaying remainder. Resonance expansions were first established for scalar waves in the static region of Schwarzschild--de~Sitter spacetimes by Bony--H\"afner \cite{BonyHaefnerDecay} using work on the discreteness and approximate lattice structure of the set of quasinormal modes (QNMs) $\sigma$ by S\'a Barreto--Zworski \cite{SaBarretoZworskiResonances} (see also the very precise description of QNMs given by Hitrik--Zworski \cite{HitrikZworskiQNM}); Melrose--S\'a Barreto--Vasy \cite{MelroseSaBarretoVasySdS,MelroseSaBarretoVasyResolvent} proved exponential decay to constants up to the horizons, and Besset \cite{BessetRNdSDecay} studied the case of non-rotating charged black holes. Dyatlov \cite{DyatlovQNM,DyatlovQNMExtended,DyatlovAsymptoticDistribution} proved analogous results for the wave and Klein--Gordon equations on slowly rotating KdS (see also \cite{IantchenkoRNdSDirac,IantchenkoKNdSDirac}). Vasy \cite{VasyMicroKerrdS} introduced a robust microlocal framework, i.e., not relying on separation of variables techniques, for obtaining resonance expansions (see \cite[Chapter~11]{HintzMicro} for a detailed account). Petersen--Vasy \cite{PetersenVasySubextremal} used this framework to treat the full subextremal case, and the author provided a description of QNMs for black holes of small mass \cite{HintzXieSdS,HintzKdSMS}. The author proved resonance expansions for tensorial wave equations such as~\eqref{EqILinEin} in~\cite{HintzPsdoInner}. (Literature on QNMs for anti--de~Sitter spacetimes includes \cite{WarnickQNMs,GannotSAdS,GannotKerrAdS}.)

Given a non-decaying mode solution $h=e^{-i\sigma t}h_0(x)$, $\Im\sigma\geq 0$, of $L_{g_0}h=0$, one applies $\delta_{g_0}\sfG_{g_0}$ to~\eqref{EqILinEin} and obtains
\[
  \Box_{g_0}^\cC\eta = 0,\quad \eta := D_{g_0}\Ups(h).
\]
Here we use the linearization of the second Bianchi identity $\delta_g\sfG_g(\Ric(g)-\Lambda g)=0$ around $g=g_0$. But due to the stationarity of $g_0$, the 1-form $\eta$ is itself a mode, i.e., of the form $\eta=e^{-i\sigma t}\eta_0(x)$. \emph{If mode stability holds for $\Box_{g_0}^\cC$}, one thus concludes that $\eta=0$, and therefore $D_{g_0}\Ric(h)-\Lambda h=0$; thus, non-decaying mode solutions $h$ of the linearized gauge-fixed Einstein equation $L_{g_0}h=0$ automatically satisfy the linearized gauge condition $D_{g_0}\Ups(h)=0$ and the linearized \emph{ungauged} Einstein equation $D_{g_0}\Ric(h)-\Lambda h=0$. The study of mode solutions of the latter has a long history, with \cite{ReggeWheelerSchwarzschild,VishveshwaraSchwarzschild,ZerilliPotential,TeukolskySeparation,WhitingKerrModeStability,AnderssonMaPaganiniWhitingModeStab} treating Schwarzschild and Kerr spacetimes and \cite{KodamaIshibashiMaster,HintzVasyKdSStability} the dS case. A description of these non-decaying modes is then a key input in all stability proofs \cite{HintzVasyKdSStability,HintzKNdSStability,FangKdS,HintzPetersenVasyKdS}.

When $g_0$ is the metric on a neighborhood of the static patch of de~Sitter space, then the 1-form wave operator $\Box_{g_0}=\Box_{g_0}^0$ does \emph{not} satisfy mode stability \cite[\S{C.3}]{HintzVasyKdSStability}: it has, e.g., a QNM in the upper half plane at $\frac{i}{2}(-3+\sqrt{33})$. While the QNMs of $\Box_g$ on KdS have not been computed, there is no reason to expect that they should all satisfy $\Im\sigma<0$. It is for this reason that \cite{HintzVasyKdSStability} introduced large parameter constraint damping as in~\eqref{EqIRoughC} and proved mode stability for suitable choices of $\cd,e,h$. (Note that merely \emph{perturbative} constraint damping cannot push the aforementioned QNM all the way into the lower half plane.) The main novelty of \cite{HintzPetersenVasyKdS} is the implementation of constraint damping in the full subextremal range of KdS.

Further papers in which constraint damping has been used, albeit not in the context of QNMs, include the stability proofs for expanding spacetimes by Ringstr\"om \cite{RingstromEinsteinScalarStability} and Hintz--Vasy \cite{HintzVasyKdSCosm}.

\begin{rmk}[When constraint damping is not needed]
\label{RmkIMdSNoCD}
  In approaches to nonlinear stability problems based on bootstrap (continuous-in-time induction) arguments, such as \cite{FangKdS} in the KdS setting, \cite{FriedrichStability,ChristodoulouKlainermanStability,LindbladRodnianskiGlobalStability} on Minkowski space, or \cite{KlainermanSzeftelPolarized,DafermosHolzegelRodnianskiTaylorSchwarzschild,KlainermanSzeftelKerr,GiorgiKlainermanSzeftelStability} for Schwarzschild and slowly rotating Kerr, one can sidestep constraint damping since the validity of the constraint equations is assumed for the initial data (and this then propagates).
\end{rmk}

\subsubsection{Mode stability for \texorpdfstring{$\Box_g^\cC$}{the constraint propagation wave operator} on asymptotically flat spacetimes}
\label{SssIMFlat}

The proof of the linear stability of subextremal Kerr black holes $(M,g)$ by H\"afner--Hintz--Vasy \cite{HaefnerHintzVasyKerr,HaefnerHintzVasyKerrLarge} utilized constraint damping as well: mode stability of the constraint propagation wave operator $\Box_g^\cC$ at nonzero frequencies served the same purpose as in the dS setting, and on Kerr holds even for $\cC=0$, i.e., without any modifications to the 1-form wave operator (which in view of $\Ric(g)=0$ is equal to the Hodge d'Alembertian and thus closely related to Maxwell's equations). The situation at zero frequency is more subtle. The relevant zero modes $h$ of the linearized gauge-fixed Einstein operator
\begin{equation}
\label{EqIMFlatLinEin}
  L_{g_0}=D_{g_0}\Ric-\delta_{g_0,\cC}^*D_{g_0}\Ups,
\end{equation}
for spatially compactly supported $\cC$ for now, are those with $\cO(r^{-1})$ decay as $r\to\infty$ (cf.\ \cite[Proposition~9.1]{HaefnerHintzVasyKerr}); thus $\eta=D_{g_0}\Ups(h)=\cO(r^{-2})$ and $\Box_{g_0}^\cC\eta=0$. This still implies $\eta=0$ when $\cC=0$ (since the nullspace of $\Box_{g_0}$ is spanned by a Coulomb-type solution with sharp $r^{-1}$ decay).

The need for a non-trivial modification $\cC$ arises in \cite{HaefnerHintzVasyKerr,HaefnerHintzVasyKerrLarge} only in the analysis of \emph{generalized} zero modes (i.e., polynomial-in-time and $o(1)$-in-space elements of $\ker L_{g_0}$), as discussed in \cite[\S{9.3}]{HaefnerHintzVasyKerr}: eliminating the zero mode of $\Box_{g_0}^\cC$ for $\cC=0$ altogether by making a suitable small choice of $\cC$ (see \cite[Equation~(10.2) and Definition~4.6]{HaefnerHintzVasyKerr} and \cite[Proposition~3.7 and the definitions after equation~(3.1a)]{HintzGlueLocIII}) ensures that generalized zero modes of $L_{g_0}$ are necessarily generalized zero modes also of $D_{g_0}\Ric$. (This was then shown to imply that solutions of $L_{g_0}h=0$, with suitably decaying initial data, only feature non-decaying terms in their late-time expansions that also lie in the kernel of $D_{g_0}\Ric$, such as linearized Kerr metrics and deformation tensors of approximate Lorentz boosts.) As shown in \cite[\S{10}]{HaefnerHintzVasyKerr} (Kerr with small angular momentum) and \cite[\S{3.1}]{HintzGlueLocIII} (full subextremal range), constraint damping can be implemented perturbatively so as to eliminate all non-decaying mode solutions of $\Box_g^\cC$, including $o(1)$ mode solutions at zero frequency.

Constraint damping also played a role in the proofs of the stability of Minkowski space \cite{HintzVasyMink4,HintzMink4Gauge}, albeit for a different reason: there it ensures, roughly speaking, that certain components of metric perturbations, in generalized harmonic gauge, have improved decay (better than $\cO(r^{-1})$) at future null infinity. This in turn ensures that the geometry of null infinity remains unchanged throughout the Nash--Moser iteration scheme. (This partial decay improvement plays an important role also in the bootstrap-based stability proof of Lindblad--Rodnianski \cite{LindbladRodnianskiGlobalStability} in harmonic gauge and for initial data satisfying the constraints. Keir \cite{KeirWeak} showed, remarkably, that global existence holds even without assuming validity of the constraints.)

\subsubsection{Enhanced mode stability for \texorpdfstring{$\Box_g^\cC$}{the constraint propagation wave operator} at zero energy}
\label{SssIME}

As follows from the discussion in~\S\ref{SssIMFlat}, it is solely the requirement of \emph{enhanced} mode stability at zero energy in Theorem~\ref{ThmIRough}\eqref{ItIRough0} which forces us to do non-perturbative constraint damping in this paper. The full justification for this requirement can only be given in the course of the nonlinear stability proof; see \cite[\S{1.4.3}]{HintzKerrStab}. A partial justification is that it enables us to \textit{control lower-order terms in the late-time asymptotics of solutions of the linearized gauge-fixed Einstein equation}. To explain this, we first recall that \cite{HintzPrice} showed that the late-time behavior of solutions of linear wave-type equations $P u=0$ (e.g., $P$ is the wave operator on subextremal Kerr, or the wave operator on Minkowski coupled to an inverse cubic potential) on stationary asymptotically flat spacetimes in $3+1$ dimensions is, generically, dictated by a \emph{large zero energy state}, i.e., a solution $u_0=u_0(x)$ of $P u_0=0$ with $u_0\to 1$ as $|x|\to\infty$. Concretely, \cite{HintzPrice} establishes the asymptotics $u(t,x)=c t^{-3}u_0(x)+o(t^{-3})$ for solutions of $P u=0$ where $c$ is a constant computable from initial data. Furthermore, for the wave operator on the spherically symmetric Schwarzschild spacetime and upon restriction to waves supported in spherical harmonic degree $\geq l$, \cite[Theorem~5.1 and Remark~5.2]{HintzPrice} assert $c t^{-2 l-3}Y_0+o(t^{-2 l-3})$ asymptotics, where $Y_0(x)$ is a large zero energy state that now has $r^l Y_{l m}(\theta,\phi)$ asymptotics as $r=|x|\to\infty$. The following heuristic emerges: \textit{the coefficient of $t^{-k}$ in the late-time asymptotics of a linear wave is related to large zero energy states with $\cO(r^k)$ asymptotics (up to possible constant shifts in the exponent) as $|x|\to\infty$.} A general result corroborating this (only for the leading order term, and in a less degenerate setting than in \cite{HintzPrice} in that various cancellations are absent which led to an improved regularity of the resolvent at zero energy in \cite{HintzPrice} and, in a different language, in \cite{AngelopoulosAretakisGajicLate,LukOhTwoTails}) is stated as \cite[Theorem~4.17]{HintzNonstat}. See also \cite{AngelopoulosAretakisGajicKerr} for lower-order terms in the expansion of linear waves on Kerr.

Consider now the late-time asymptotics of a solution of $L_{g_0}h=0$ (recalling~\eqref{EqIMFlatLinEin}). A generalization of the algorithm introduced in \cite{HintzPrice}---and pushed to its limit by Sussman \cite{SussmanResolventPhg} in the absence of zero modes---yields an asymptotic expansion of $h$ as $t\to\infty$ into terms $t^{-k}h_k(x)$ where the $h_k$ are large zero energy states with, roughly, $\cO(r^k)$ asymptotics.\footnote{Here $k$ may be fractional; we are omitting the stationary and linearly growing zero mode contributions arising from linearized Kerr metrics and certain deformation tensors, see \cite[Theorem~8.1]{HaefnerHintzVasyKerrLarge}; and we ignore the fact that one term in the expansion leads to further terms later on, so that $h_k$ is not a pure large zero energy state but contains also contributions necessitated by the presence of the earlier terms $h_j$, $j<k$.} But applying the second Bianchi argument to $L_{g_0}h_k=0$ shows that $D_{g_0}\Ups(h_k)=0$ and $D_{g_0}\Ric(h_k)=0$ provided that $k-1<C_0$ (roughly) where $C_0$ is the growth rate which large zero modes \emph{must exceed} by Theorem~\ref{ThmIRough}\eqref{ItIRough0}.

Therefore, enhanced mode stability ensures that the late-time expansion of $h$ is physical/geometric in origin, i.e., it only contains terms which solve the ungauged linearized Einstein equation and are automatically in the desired gauge, up to roughly $\cO(t^{-C_0})$ remainders. Gauge modifications then allow one to remove most terms in the late-time expansion and thus obtain fast decay of metric perturbations, which is crucial for the nonlinear analysis. In our approach to nonlinear stability in \cite{HintzKerrStab}, we will need $\cO(t^{-4-\eps})$ remainders (in some sense) and thus need, roughly, $C_0>4$.

\subsection{Elements of the proof}
\label{SsIA}

The proof of the second part of Theorem~\ref{ThmIRough} is largely based on computations of \emph{indicial roots} and thus to some extent algebraic in nature; we discuss this in~\S\ref{SssIA0}. The proof that for suitable choices of $\cC_h$ in~\eqref{EqIRoughC}, mode stability holds at nonzero frequencies (and at zero frequency on decaying spaces), uses advanced microlocal machinery. We discuss this in~\S\ref{SssIAh}; and we will also explain the deeper reason behind choosing to work with a bundle map $\cC_h\colon T^*M^\circ\to S^2 T^*M^\circ$ of the particular form~\eqref{EqIRoughC} (beyond the fact that it combines two particularly natural operations---symmetric tensor product, and contraction with a given 1-form/vector field---with the parameter $e$ providing flexibility in the relative weighting of the two).

\subsubsection{Zero energy behavior and indicial roots}
\label{SssIA0}

The action of $\Box_g^{\cC_h}$ on stationary 1-forms is given by its \emph{zero energy operator} $\wh{\Box_g^{\cC_h}}(0)$, obtained by dropping all derivatives in $t$ (time). To orient ourselves, note first that the zero energy operator of the scalar wave operator on Kerr takes the form
\[
  \wh{\Box_g}(0) = r^{-2} \bigl( -(r\pa_r)^2 - r\pa_r + \slDelta \bigr) + \cO(r^{-3})
\]
in polar coordinates $x=r\omega$, with $\slDelta$ being the non-negative Laplacian on the standard 2-sphere $(\Sph^2,\slg)$. The first term (called the \emph{normal operator} of $\wh{\Box_g}(0)$ at $r^{-1}=0$) is the Euclidean Laplacian; the $\cO(r^{-3})$ term is a second order differential operator constructed from $r\pa_r$ and spherical derivatives (i.e., a \emph{b-differential operator} in the parlance of \cite{MelroseMendozaB,MelroseAPS}), with an overall factor of $r^{-3}$ in front. The asymptotic behavior of solutions of $\wh{\Box_g}(0)u_0=0$ is dictated by the \emph{indicial roots} of $\wh{\Box_g}(0)$: these are those complex numbers $\lambda\in\C$ for which there exists $0\neq u_\pa\in\CI(\Sph^2)$ with the property that
\[
  \wh{\Box_g}(0) ( r^{-\lambda}u_\pa ) = \cO(r^{-\lambda-3}),
\]
or equivalently,
\[
  (-\lambda^2+\lambda+\slDelta)u_\pa = 0.
\]
Separating into spherical harmonics and using that the eigenvalues of $\slDelta$ are $l(l+1)$, $l\in\N_0$, yields the indicial roots $\lambda=-l,l+1$. An analogous analysis can be performed for the tensor wave operator $\Box_g$ on 1-forms. In this case, the decomposition into spherical harmonic \emph{1-forms} on $\Sph^2$ yields two sub-cases, namely, scalar type and vector type 1-forms; see Definition~\ref{DefM0Types}.

Returning to $\wh{\Box_g^{\cC_h}}(0)$, observe that $\Box_g^{\cC_h}=2\delta_g\sfG_g\delta_g^*+2\delta_g\sfG_g\cC_h$ differs from $2\delta_g\sfG_g\delta_g^*=\Box_g$ by the extra term $2\delta_g\sfG_g\cC_h$; for the choice $\cd$ in Theorem~\ref{ThmIRough}, which equals $r^{-1}(\dd t-v\,\dd r)$ for large $r$, this extra term is a b-differential operator with weight $r^{-2}$, and thus does contribute to the normal operator of $\wh{\Box_g^{\cC_h}}(0)$, and hence affects the indicial roots. The full expression is given in Corollary~\ref{CorM00}. One step towards proving Theorem~\ref{ThmIRough}\eqref{ItIRough0} is thus to show that, for a suitable choice of $\cC_h$,
\begin{equation}
\label{EqIA0Roots}
  \mbox{\textit{all indicial roots $\lambda\in\C$ of $\wh{\Box_g^{\cC_h}}(0)$ satisfy $\Re\lambda<-C_0$ or $\Re\lambda\geq 1$.}}
\end{equation}
See Theorem~\ref{ThmM0} for a more comprehensive statement. This implies that every $\omega=\cO(r^{C_0})$ with $\wh{\Box_g^{\cC_h}}(0)\omega=0$ must, in fact, satisfy $\omega=\cO(r^{-1})$. To conclude $\omega=0$, it then remains to exclude the possibility of \emph{decaying} zero modes, which the reader may find less surprising to be doable.\footnote{We shall, in fact, proceed slightly differently, and exclude the possibility of $\cO(r^{C_0})$ large zero modes \emph{directly} (see~\S\ref{SssIAh} and the proof of Theorem~\ref{ThmT}\eqref{ItTZero}); nonetheless, the full description of the indicial roots in Theorem~\ref{ThmM0} is crucially used in the proof of the \emph{invertibility} of $\wh{\Box_g^{\cC_h}}(0)$ on certain weighted Sobolev spaces (see Theorem~\ref{ThmT}\eqref{ItTZero}), and for the existence of a nontrivial indicial gap for the linearized gauge-fixed Einstein operator in \cite[Lemma~8.4]{HintzKerrStab}.}

Arranging~\eqref{EqIA0Roots} requires a careful fine-tuning of the parameters $e$, $\fc$, and $h$: we show that for small enough $e>0$ (and corresponding choices of $\fc$, see Proposition~\ref{PropEC}), every small $0<h\leq h_0=h_0(e)$ works. (One \emph{cannot} take $e=0$, as one may check numerically. Moreover, it is crucial to take $v\in(0,1)$, as $v\leq 0$ turns out to be incompatible with~\eqref{EqIA0Roots}; see Remark~\ref{RmkM0WrongV}.) The indicial roots are the roots of polynomials of degrees up to $6$ (see, e.g., \eqref{EqMsPoly}), and the determination of their roots is only possible asymptotically in the double limit $e\to 0$, $h\to 0$. The detailed proofs are presented in Appendix~\ref{SM0}.

\subsubsection{Nonzero frequencies: semiclassical estimates on spacetime}
\label{SssIAh}

Typically, one studies mode solutions
\begin{equation}
\label{EqIAhMode}
  \omega(t,x)=e^{-i\sigma t}\omega_0(x)
\end{equation}
of $\Box_g^{\cC_h}$ by analyzing the spectral family $\wh{\Box_g^{\cC_h}}(\sigma)$ (an operator in the spatial variables only), obtained by replacing $\pa_t$ in the expression for $\Box_g^{\cC_h}$ by $-i\sigma$; and spectral information, such as mode stability, is then used to infer information about asymptotics and decay of solutions of $\Box_g^{\cC_h}$. In this paper, as already previously in \cite{HintzVasyKdSStability,HintzPetersenVasyKdS}, we argue the other way around. By taking advantage of the presence of a large parameter $h^{-1}$, we prove that the wave operator $\Box_g^{\cC_h}$ is invertible as a map between appropriate \emph{polynomially weighted} Sobolev spaces \emph{on spacetime}. Concretely, we will show that, for small $e>0$ and small $h>0$, one can solve initial value problems for $\Box_g^{\cC_h}$ in the domain $\{t\geq 0\}$ in the weighted spacetime $L^2$-space
\begin{equation}
\label{EqIAhL2}
  \rho_\sface^{\alpha_\sface}\rho_\cK^{\alpha_\cK}L^2,\quad
  \rho_\sface=\frac{1}{r},\quad\rho_\cK=\frac{r}{t+r},
\end{equation}
provided the initial data at $t=0$ have decay/growth rates compatible with the weight $\alpha_\sface$. (Here $t$ is a time function whose level sets are transversal to the future event horizon, and which is equal to the Boyer--Lindquist time coordinate for large $r$.) Roughly speaking, $\alpha_\sface$, resp.\ $\alpha_\cK$ measures the rate of decay as $r\to\infty$ with fixed ratio $\frac{r}{t}$, resp.\ as $t\to\infty$ with $r$ held fixed. The weights $\alpha_\sface,\alpha_\cK$ will be subject to
\begin{equation}
\label{EqIAhWeights}
  \alpha_\sface<-\tfrac12+\min(\alpha_\cK,0).
\end{equation}
Thus, if we fix $\alpha_\cK=1$ and $\alpha_\sface\ll -1$, then solutions $\omega$ of $\Box_g^{\cC_h}\omega=0$ with possibly large ($\sim r^{-\alpha_\sface}$) initial data still decay as $t\to\infty$. Applying this to the initial data $(\omega_0,-i\sigma\omega_0)$ of a mode solution~\eqref{EqIAhMode}, we conclude the non-existence of non-decaying mode solutions~\eqref{EqIAhMode} and thus the validity of Theorem~\ref{ThmIRough}.

{\bf Semiclassical setup.} The proof of the invertibility of $\Box_g^{\cC_h}$ as a map between Sobolev spaces related to~\eqref{EqIAhL2} is the heart of the paper (\S\ref{SE}). The key observation, first made in \cite[\S{8}]{HintzVasyKdSStability}, is that the rescaling
\[
  P_h := h^2\Box_g^{\cC_h} = h^2\Box_g - i L_h,\quad L_h := 2 i h\delta_g\sfG_g\cC_1
\]
is a \emph{semiclassical} differential operator (i.e., built out of semiclassical derivatives $h\pa$) that is amenable to a full phase space analysis; this amounts, roughly speaking, to controlling the amplitudes of oscillations $\sim e^{i(x-x_0)\cdot\xi/h}$ of solutions $\omega$ of $P_h\omega=f$ in terms of those of $f$, modulo lower order terms (as measured in powers of $h$). Concretely, we will prove the estimate
\begin{equation}
\label{EqIAhEst}
  \|\omega\|_{\rho_\sface^{\alpha_\sface}\rho_\cK^{\alpha_\cK}L^2} \leq C\Bigl(h^{-1} \| P_h\omega \|_{\rho_\sface^{\alpha_\sface+2}\rho_\cK^{\alpha_\cK}L^2} + h^N \| \omega \|_{\rho_\sface^{\alpha_\sface}\rho_\cK^{\alpha_\cK}L^2} \Bigr)
\end{equation}
(or refinements thereof involving other differential orders, see~\eqref{EqEPAprioriAlmost}) for 1-forms $\omega$ vanishing for $t<0$. For small $h>0$, the second term on the right can be absorbed into the left-hand side.\footnote{Without the presence of a semiclassical parameter, i.e., for $h=1$, a version of the estimate~\eqref{EqIAhEst} where the error term is estimated in a Sobolev space with differentiability order less than that of the space on the left (which is $0$) would only be the first step in the spacetime analysis of linear waves, cf.\ \cite[(1.20)]{HintzNonstat}. Obtaining Fredholm and invertibility properties of $P_h$ would require, in addition, the inversion of various model operators (cf.\ \cite[\S{1.4.2}]{HintzNonstat}) including a stationary model operator---which in the present case would be exactly $P_h$---via spectral theory; at this point the argument would become circular.} A variant of this estimate for the $L^2$-adjoint of $P_h$ yields the solvability of the forward problem for $P_h\omega=f$ (i.e., $f$ and $\omega$ vanish for $t<0$) with $\omega\in\rho_\sface^{\alpha_\sface}\rho_\cK^{\alpha_\cK}L^2$ by duality.

{\bf Local structure of $P_h$.} The semiclassical principal symbol of $P_h$ (obtained by replacing $h D$ by the momentum $\zeta\in T_z^*M^\circ$, $z\in M^\circ$, and dropping terms which are $o(1)$ as $h\to 0$) is given by
\begin{equation}
\label{EqIAhSymb}
  p_e(\zeta) = G(\zeta) - i\ell_e(\zeta),\quad \zeta\in T_z^*M^\circ,
\end{equation}
where $G(\zeta)=\la\zeta,\zeta\ra$, $\la\cdot,\cdot\ra=g^{-1}(\cdot,\cdot)$, is the dual metric function (quadratic in $\zeta$) and $\ell_e(\zeta)$ is an endomorphism of $T_z^*M^\circ$ (depending linearly on $\zeta$), corresponding to the fact that $P_h$ acts on sections of $T^*M^\circ$; see Lemma~\ref{LemmaCSymb}. We improve on \cite{HintzVasyKdSStability} by introducing a general-purpose theory of large parameter/semiclassical constraint damping in~\S\ref{SC} which, moreover, is slightly sharper than that developed in \cite{HintzPetersenVasyKdS}.

In the limit $e\to 0$, one ``almost'' has $\frac12\ell_0(\zeta)=\la\cd,\zeta\ra$ (see Remark~\ref{RmkCNoe0}), but there are several reasons why we must work with (small) $e>0$: recall that $e>0$ is necessary for arranging~\eqref{EqIA0Roots}; moreover, $\ell_0(\zeta)$ has a nontrivial Jordan block structure, whereas for $e\in(0,1)$ and timelike $\cd$ there exists a fiber inner product on $T^*M^\circ$ such that $\ell_e(\zeta)$ is self-adjoint for all $\zeta$, and in fact positive definite for future causal $\zeta$ (see Lemmas~\ref{LemmaCInner} and \ref{LemmaCTimelike}). For $e=0$ and thus in the absence of a good fiber inner product, it is not clear how to prove any estimates for $P_h$. Nonetheless, we continue our discussion for now with $\la\cd,\zeta\ra$ in place of $\ell_e(\zeta)$, and thus with
\[
  \tilde p=G-i\la\cd,\cdot\ra
\]
in place of $p_e$; this is the semiclassical principal symbol of $\tilde P_h:=h^2\Box_g-i h\frac{1}{i}\nabla_{\cd^\sharp}$, $\cd^\sharp:=g^{-1}(\cd,\cdot)$.\footnote{On Minkowski space $(M^\circ,g)=(\R^4,-\dd t^2+\dd x^2)$, with $\cd=\dd t$, this is the (strongly) damped wave operator $(h\pa_t)^2-(h\pa_x)^2+h\pa_t$.}

Note that for timelike $\cd$, the symbol $\tilde p$ is elliptic at nonzero frequencies $\zeta$; it is characteristic only at the zero section and at infinite frequencies (see \cite[Chapter~6.5, Definition~6.52]{HintzMicro}) on the dual light cone $G^{-1}(0)$.

{\bf Propagation at infinite lightlike momenta.} Consider first infinite frequencies in $G^{-1}(0)$: phase space control (i.e., $L^2$-bounds on microlocalizations) of solutions of $\tilde P_h\omega=f$ there is provided by real principal type propagation estimates in the spirit of \cite{HormanderEnseignement,DuistermaatHormanderFIO2}, but with the term $-i\la\cd,\cdot\ra$ of $\tilde p$ acting as a damping term (a weak form of complex absorption) when propagating in the future direction. Due to the strength of this damping, very little information is needed on the global null-geodesic dynamics of $(M^\circ,g)$ to obtain global (but microlocalized near $G^{-1}(0)$) such propagation estimates: trapping, the event horizon, and the ergoregion of Kerr do not play any role here. See Proposition~\ref{PropEEInfty}.

{\bf Propagation along the zero section.} Microlocal control near the zero section is considerably more delicate to obtain. Since $G(\zeta)$ vanishes quadratically there, the main term of $\tilde p$ there is $-i\la\cd,\cdot\ra$, and that of $\tilde P_h$ is $-h\nabla_{\cd^\sharp}$. That is, we must study a transport operator along the vector field $-\cd^\sharp$, which will be chosen to be future timelike. Over precompact subsets of $M^\circ$, this is akin to real principal type propagation.

Issues arise once one needs to prove \emph{uniform} (i.e., as $|(t,x)|\to\infty$) estimates. Following the tradition of geometric singular analysis, we shall phrase uniform estimates as \emph{local} estimates on a \emph{compactification} of the spacetime manifold. Note that if we take $\cd=r^{-1}(\dd t-v\,\dd r)$ with $v\in(0,1)$ (as needed for arranging~\eqref{EqIA0Roots}), then on Minkowski space we have $-r^2\cd^\sharp=r(\pa_t+v\pa_r)$. Radially compactifying the manifold $\R^4$ (see \cite[Chapter~6.5]{HintzMicro}) underlying Minkowski space, this has a sink at the endpoint $\{\frac{r}{t}=v,\ r^{-1}=0\}$ at infinity of the timelike curve $r=v t$.

However, on the radial compactification $\ol{\R^4}$, the coefficients of the Kerr metric, and thus of $P_h$, are singular near the ``north pole'' (where local coordinates are $\frac{1}{t}$, $\frac{x}{t}$); we must thus pass to a resolved space, on which the spatial coordinate $x$ itself is a local coordinate, by blowing up the north pole. The vector field $-r^2\cd^\sharp$ then has a saddle point at the corner between the regimes ``$r$ bounded, $t=\infty$'' and ``$\frac{r}{t}>0$, $r=\infty$,'' namely where $\frac{1}{r}=0$ and $\frac{r}{t}=0$. When propagating $L^2$-estimates through this saddle point, we must work with a weight $\alpha_\sface$ at $\sface$ (as indicated in Figure~\ref{FigIBlowup}) that is less than $\alpha_\cK$ (up to a constant shift), cf.\ \eqref{EqIAhWeights} and Proposition~\ref{PropEB}.\footnote{If we took $v<0$, so $-r^2\cd^\sharp$ would be inward pointing, one would need to assume $\alpha_\cK<\alpha_\sface+{\rm const.}$, which would be incompatible with our needs, as discussed after~\eqref{EqIAhWeights}.}

\begin{figure}[!ht]
\centering
\includegraphics{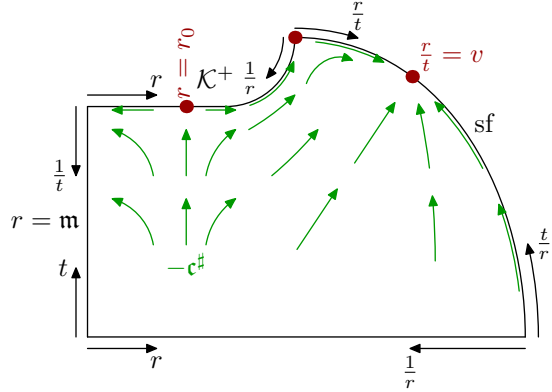}
\caption{The compactified Kerr spacetime manifold (here only the subset where $t\geq 0$, and projected to the $(t,r)$-plane) and some local coordinates. The boundary $r=\bhm$ is a spacelike hypersurface inside of the black hole. The flow of the vector field $-\cd^\sharp$ is indicated in green, and its critical sets are indicated as thick dots.}
\label{FigIBlowup}
\end{figure}

The requirement that $-\cd^\sharp$ be future timelike also for bounded $r$ means that it must be \emph{inward} pointing at the event horizon of the Kerr black hole. Since $\pa_t+v\pa_r$ is \emph{outward} pointing at large $r$, this forces $-\cd^\sharp$ to have yet another critical set; we design $\cd$ such that this critical set is given by $\{r=r_0,\ t=\infty\}$ for a large radius $r_0$, and $-\cd^\sharp$ is essentially given by $\pa_t+(\frac{r}{r_0}-1)\pa_r$ nearby. See Figure~\ref{FigIBlowup}.

An approximate normal form for $-r^2\cd^\sharp$ near the sink $\{\frac{r}{t}=v,\ r^{-1}=0\}$ is $-\rho\pa_\rho-y\pa_y$ where $\rho=r^{-1}$ and $y=\frac{r}{t}-v$. To determine the weight (power of $\rho$) of the weighted $L^2$-space in which one can uniformly bound solutions of $h\nabla_{\cd^\sharp}\omega=f$ (the proxy for $\tilde P_h\omega=f$ near the zero section), subprincipal terms matter: the behavior of solutions of $h(\nabla_{\cd^\sharp}+r^{-2}E)\omega=f$ depends on (the eigenvalues of) the subprincipal term $E$ of the operator $L_1$ (which thus far we had dropped). Taking into account this subprincipal term for the actual operator $P_h$ of interest yields the upper bound $\alpha_\sface<-\frac12$ in~\eqref{EqIAhWeights}, cf.\ Lemma~\ref{LemmaMISink} and Proposition~\ref{PropES}.

The most delicate place is the saddle point at $\{r=r_0,\ t=\infty\}$. A rough model for the operator $P_h$ nearby (and near the zero section) is the transport operator $\pa_t+(\frac{r}{r_0}-1)\pa_r+E$ for some $0$-th order term $E$. Since for present objectives we must work on spaces encoding decay in time, it is crucial that $E$ only have positive eigenvalues, for otherwise this transport operator would yield exponentially growing solutions. Performing the relevant subprincipal symbol calculation for the actual operator $P_h$, one finds that this positive definiteness holds only in a neighborhood $|r-r_0|\leq \frac{6}{5}r_0\sqrt{e}/v$ that \emph{shrinks} as $e\to 0$ (Lemma~\ref{LemmaMSSubpr}). But for any positive $e>0$, the principal symbol of $\ell_e$ in~\eqref{EqIAhSymb} is no longer cleanly related to $\cd^\sharp$, and thus the transport interpretation breaks down. We resolve this tension by showing that outside of a $1\cdot r_0\sqrt{e}/v$-neighborhood of $r=r_0$, the transport interpretation is valid (cf.\ Lemma~\ref{LemmaMSRad}); the fact that $\frac65>1$ is thus crucial in our proof.

\begin{rmk}[Comparison with constraint damping in the KdS setting]
\label{RmkICompKdS}
  The dynamics of the flow of $-\fc^\sharp$ here is more ornate than in~\cite{HintzVasyKdSStability,HintzPetersenVasyKdS}; this is also due to the presence of 2 asymptotic regimes\footnote{The number of asymptotic regimes one needs to distinguish depends on the problem at hand. Unlike in \cite{HintzKerrStab}, we do not need to separate spacelike or null infinity here.} in the Kerr case (indicated by the boundary hypersurfaces $\sface$ and $\cK^+$ in Figure~\ref{FigIBlowup}) compared to only 1 regime ($t\to\infty$) in the KdS case. Furthermore, only in the present paper is it necessary to keep track of the precise sizes of the regions of subprincipal symbol positivity and the region where the transport interpretation is valid.
\end{rmk}

\begin{rmk}[Function spaces]
\label{RmkIFn}
  Our analysis will take place on weighted semiclassical 3b-Sobolev spaces; these control distributions in the weighted $L^2$-space~\eqref{EqIAhL2} and also their derivatives along $h r D_t$, $h r D_r$, $h\slnabla$. These spaces correctly capture three features of the problem: the semiclassical parameter $h$; the approximate homogeneity with respect to spacetime scaling in $r\gtrsim t$; and the translation-invariance in $t$ when $r$ is bounded. Such Sobolev spaces, albeit without a semiclassical parameter, were introduced in \cite{Hintz3b} and subsequently used in the analysis of linear waves on (dynamical) asymptotically flat spacetimes in \cite{HintzNonstat}. Unlike in \cite{HintzNonstat,HintzVasyScrieb}, we do not need to pay any attention to null infinity here due to the strong damping at infinite momenta; see Remark~\ref{RmkERThres}.
\end{rmk}

\subsection{Outline}
\label{SsIO}

The plan of the paper is as follows.
\begin{itemize}
\item In~\S\ref{SK}, we define the compactified spacetime manifold $M$ on which our uniform analysis will take place, and explain how the Kerr metric behaves on it. We moreover show that geometric operators associated with the Kerr metric are weighted 3b-operators on $M$.
\item In~\S\ref{SC}, we prove general results on the principal and subprincipal symbol of the constraint propagation wave operator $h^2\Box_g^{\cC_h}$ which apply on any spacetime and are thus of independent interest.
\item Computations of subprincipal symbols at the critical sets of the constraint damping vector field $-\cd^\sharp$ are presented in~\S\ref{SM}. In the same section, we also compute the indicial roots of the zero energy operator of $\Box_g^{\cC_h}$ (Theorem~\ref{ThmM0}); this is a crucial ingredient in the proof of enhanced mode stability at zero energy.
\item The technical heart of the paper is~\S\ref{SE} where we give proofs for all microlocal propagation estimates (in particular, at infinite lightlike momenta and near zero momenta) mentioned in~\S\ref{SssIAh}, in addition to semiclassical energy estimates to deal with spacelike boundaries (namely, the initial hypersurface $\{t=0\}$ as well as $\{r=\bhm\}$ inside the black hole), and combine them to prove the solvability of $\Box_g^{\cC_h}$ on polynomially weighted spaces (Theorem~\ref{ThmET}).
\item The proof of Theorem~\ref{ThmIRough} is now a simple consequence, as shown in~\S\ref{ST}.
\item Appendix~\ref{SM0} contains the proof of Theorem~\ref{ThmM0}.
\item Appendix~\ref{SMi} proves a mode stability type result (Theorem~\ref{ThmMiInv}) for an operator arising in the low energy spectral analysis for the linearized gauge-fixed Einstein operator on Kerr. This is not needed for the proof of the main theorem of the present paper, but naturally fits here since the method of proof is closely related to that of Theorem~\ref{ThmIRough}. (It plays an important, albeit technical, role in \cite[\S\S{7--8}]{HintzKerrStab}.)
\end{itemize}

\subsection*{Acknowledgments}

I gratefully acknowledge support from the NSF grants DMS-1955614 and DMS-2554160. Part of this research was carried out during the periods I was a Clay Research Fellow and Sloan Research Fellow; I thank the Clay Mathematics Institute and the Sloan Foundation for their support. An AI tool (ChatGPT) was used in the proofreading of the current (second) version of the manuscript, but the work and writing are the author's alone.

\subsection*{Data availability statement}

Data sharing is not applicable to this article as no datasets were generated or analysed during the current study.

\section{Spacetime compactification and 3b-analysis}
\label{SK}

\subsection{Manifolds, scattering structures, and the Kerr metric}
\label{SsKMfd}

We write the standard coordinates on $\R^4=\R\times\R^3$ as
\[
  z=(t,x)=(t,x^1,x^2,x^3)=(z^0,z^1,z^2,z^3).
\]
We denote polar coordinates on the ``spatial manifold'' $\R^3$ by
\[
  r:=|x|,\ \omega:=\frac{x}{|x|};\quad \rho:=r^{-1}.
\]
We denote the radial compactification of $\R^4$ by
\[
  \ol{\R^4} := \Bigl( \R^4 \sqcup \bigl( [0,\infty)\times\Sph^3 \bigr) \Bigr) / \sim,\quad 0\neq z=R\zeta \sim (R^{-1},\zeta),
\]
where $R=\sqrt{t^2+|x|^2}\geq 0$ and $\zeta=\frac{(t,x)}{R}\in\Sph^3$ denote polar coordinates. Convenient local coordinates on $\ol{\R^4}$ near its ``north pole'' are the projective coordinates
\begin{equation}
\label{EqKCoord}
  T := \frac{1}{t}\geq 0,\quad
  \hat x := \frac{x}{t}\in\R^3.
\end{equation}
We write $\fk^+$ for the north pole, given in these coordinates by $(T,\hat x)=(0,0)$, and $\fk^-=-\fk^+$ for the south pole.

\begin{definition}[Spacetime manifold]
\label{DefKMfd}
  Fix $\bhm>0$. Define $M_0:=[\ol{\R^4};\fk^+,\fk^-]$ (real blow-up). Denote the blow-down map by $\upbeta\colon M_0\to\ol{\R^4}$. Then we define the manifold with corners
  \[
    M = \cl_{M_0}\biggl( M_0 \setminus \cl_{M_0}\Bigl( \{ r<\bhm\} \cup \Bigl\{ t < -\frac12 r \Bigr\} \Bigr)\biggr),
  \]
  where ``$\cl_{M_0}$'' denotes the closure of a subset of $M_0$ inside of $M_0$. We define the following submanifolds of $M_0$:
  \begin{enumerate}
  \item $\cK^+:=M\cap\upbeta^{-1}(\fk^+)$ is the \emph{Kerr face};
  \item $\sface:=\pa M\cap(\pa M_0 \setminus (\cK^+)^\circ)$ is the \emph{side face};
  \item $X:=\cl_M(\{t=0\})$ is the \emph{initial (Cauchy) hypersurface};
  \item $\Sigma^\sharp:=\cl_M(\{r=\bhm\})$ is the \emph{final interior hypersurface}.
  \end{enumerate}
  We denote by $\rho_\cK$ and $\rho_\sface\in\CI(M)$ defining functions of $\cK^+$ and $\sface$, respectively.\footnote{That is, $\rho_\cK\geq 0$, $\cK^+=\rho_\cK^{-1}(0)$, and $\dd\rho_\cK\neq 0$ on $\cK^+$; similarly for $\rho_\sface$ and $\sface$.}
\end{definition}

We can construct $M$ explicitly as the union of the following charts.
\begin{enumerate}
\item $[\bhm,\infty)_r\times[-\frac12,\infty)_\tau\times\Sph^2_\omega$, identified with a subset of $\R^4$ via $(t,r,\omega)=(r\tau,r,\omega)$. (Thus $\tau=\frac{t}{r}$.) This covers all of $M\setminus(\cK^+\cup\sface)$.
\item $[0,\frac{1}{\bhm}]_{\rho_\sface} \times [0,2]_{\rho_{\cK,1}} \times \Sph^2_\omega$, with $(t,r,\omega)=(\frac{1}{\rho_\sface\rho_{\cK,1}},\frac{1}{\rho_\sface},\omega)$ on $\{\rho_\sface,\rho_{\cK,1}\neq 0\}$. (Thus $\rho_\sface=\frac{1}{r}$ and $\rho_{\cK,1}=\frac{r}{t}$.) This covers a neighborhood of $\cK^+$.
\item $[0,\frac{1}{\bhm}]_{\rho_\sface} \times [-\frac12,\infty)_\tau\times\Sph^2_\omega$, with $(t,r,\omega)=(\frac{\tau}{\rho_\sface},\frac{1}{\rho_\sface},\omega)$. (Thus $\tau=\frac{1}{\rho_{\cK,1}}=\frac{t}{r}$ again.) This covers a neighborhood of $\sface^\circ\cap M$.
\end{enumerate}
(Globally, $M=[\bhm,\infty]_r\times[-\frac12,\infty]_\tau\times\Sph^2_\omega$.) See Figure~\ref{FigKMfd}. In particular, we have $\cK^+=[0,\frac{1}{\bhm}]_{\rho_\sface}\times\Sph^2_\omega=[\bhm,\infty]_r\times\Sph^2_\omega\cong X$; here $X$ is the closure of $\{x\in\R^3\colon|x|\geq \bhm\}$ inside of the radial compactification $\ol{\R^3}$. We only blow up $\fk^-$ in the definition of $M_0$ for reasons of aesthetics and later convenience.

\begin{figure}[!ht]
\centering
\includegraphics{FigKMfd}
\caption{A cross section of $M$ for a fixed value of the angular coordinate $\omega\in\Sph^2$. The starting points of the arrows labeled ``$\rho_\cK$'' lie in the hypersurface $\cK^+$ where $\rho_\cK$ vanishes; similarly for arrows labeled ``$\rho_\sface$.''}
\label{FigKMfd}
\end{figure}

The blow-down map $\upbeta\colon M\to\ol{\R^4}$ is the smooth extension of the map $M^\circ\to\R^4$ that assigns to each point its Cartesian coordinates $(t,x)=(t,r\omega)$. (For example, in the coordinates $\rho_\sface,\rho_{\cK,1},\omega$ and $T,\hat x$ from~\eqref{EqKCoord}, this map takes $(\rho_\sface,\rho_{\cK,1},\omega)\mapsto(T,\hat x)=(\rho_\sface\rho_{\cK,1},\rho_{\cK,1}\omega)$.) Concrete choices of boundary defining functions are $\rho_\cK=\frac{r}{t+r}$ and $\rho_\sface=\frac{1}{r}$. In the region $t>\frac12 r$, we shall work with the local defining functions
\begin{equation}
\label{EqKbdf}
  \rho_\cK = \frac{r}{t},\quad
  \rho_\sface = \frac{1}{r}.
\end{equation}
(Note that $\rho_\cK$ is a smooth positive multiple of $\frac{r}{t+r}$ in this region.)

We will effect the uniform description of tensors on $\R^4\cap\{r>\bhm,\ t>-\frac12 r\}$ using bundles defined over $M_0$ and $M$.

\begin{definition}[Scattering bundles over $\ol{\R^4}$]
\label{DefKTsc}
  We write $\Tsc\ol{\R^4}$ for the \emph{scattering tangent bundle} of $\ol{\R^4}$: this is the trivial bundle $\ol{\R^4}\times\R^4$ which over the interior $\R^4$ of $\ol{\R^4}$ we identify with $T\R^4$ by identifying $(p,v)\in\Tsc_{\R^4}\ol{\R^4}$, $v=(v^0,v^1,v^2,v^3)$, with the tangent vector $v^\mu\pa_{z^\mu}\in T_p\R^4$. The \emph{scattering cotangent bundle} $\Tsc^*\ol{\R^4}$ is its dual bundle. Sections of $\Tsc^*\ol{\R^4}$ are called \emph{scattering 1-forms}. The space of \emph{scattering vector fields} is defined as
  \[
    \Vsc(\ol{\R^4}) := \CI(\ol{\R^4};\Tsc\ol{\R^4}).
  \]
  We write
  \begin{equation}
  \label{EqKscT}
    \cT := \upbeta^*\Tsc\ol{\R^4},\quad
    \cT^* := \upbeta^*\Tsc^*\ol{\R^4}
  \end{equation}
  for the pullbacks of these scattering bundles to $M_0$ (and thus, by restriction, to $M$). Following \cite{VasyThreeBody}, they are called the \emph{three-body scattering} (or \emph{3sc-}) \emph{(co)tangent bundles} over $M_0$.
\end{definition}

For example, a smooth 3sc-vector field in the region $t>\frac12 r$ is a linear combination of the coordinate vector fields $\pa_t,\pa_{x^1},\pa_{x^2},\pa_{x^3}$ with coefficients which are smooth in $\rho_\cK=\frac{r}{t}$, $\rho_\sface=\frac{1}{r}$, and $\omega\in\Sph^2$, including down to $\rho_\cK=0$ and $\rho_\sface=0$.

The terminology in Definition~\usref{DefKMfd} arises from putting the Kerr metric $g_{\bhm,a}$ \cite{KerrKerr,BoyerLindquistKerr} with parameters $\bhm$ (black hole mass) and $a$ (specific angular momentum), in the subextremal range $a\in(-\bhm,\bhm)$, on $M^\circ$. Concretely, in Boyer--Lindquist coordinates $(\ft,r,\theta,\varphi)$ with $\ft\in\R$, $r>r_+:=\bhm+\sqrt{\bhm^2-a^2}$, and $(\theta,\varphi)\in(0,\pi)\times(0,2\pi)$, this is given by
\begin{equation}
\label{EqKBL}
\begin{split}
  g_{\bhm,a} &= -\frac{\Delta}{\varrho^2}(\dd\ft-a\,\sin^2\theta\,\dd\varphi)^2 + \varrho^2\Bigl(\frac{\dd r^2}{\Delta} + \dd\theta^2\Bigr) + \frac{\sin^2\theta}{\varrho^2}\bigl(a\,\dd\ft-(r^2+a^2)\,\dd\varphi\bigr)^2, \\
  g_{\bhm,a}^{-1} &= \varrho^{-2}\Bigl(-\frac{1}{\Delta}\bigl((r^2+a^2)\pa_\ft+a\pa_\varphi\bigr)^2 + \Delta\pa_r^2 + \pa_\theta^2 + (\sin\theta)^{-2}\bigl(\pa_\varphi+a\sin^2\theta\,\pa_\ft\bigr)^2\Bigr), \\
  &\qquad \varrho^2 := r^2+a^2\cos^2\theta,\quad
          \Delta := r^2-2\bhm r+a^2.
\end{split}
\end{equation}
The coordinate singularity at the root $r=r_+$ of $\Delta$ is resolved by passing to Kerr-star type coordinates. Thus, fixing a cutoff $\chi\in\CIc([0,4\bhm))$, identically $1$ on $[0,3\bhm]$, and fixing a radius $r_0>4\bhm$, we define $t_1$, $\phi_1$ by
\[
  \ft = t_1 - \int_{r_0}^r \frac{r'{}^2+a^2}{\Delta(r')}\chi(r')\,\dd r',\quad
  \varphi = \phi_1 - \int_{r_0}^r \frac{a}{\Delta(r')}\chi(r')\,\dd r'.
\]
In these coordinates, one computes
\begin{align}
\label{EqKMetStar}
\begin{split}
  g_{\bhm,a} &= -\frac{\Delta}{\varrho^2}\bigl(\dd t_1-a\sin^2\theta\,\dd\phi_1\bigr)^2 + 2\chi(\dd t_1-a\sin^2\theta\,\dd\phi_1)\otimes_s\dd r \\
    &\qquad + \frac{(1-\chi^2)\varrho^2}{\Delta}\dd r^2 + \varrho^2\,\dd\theta^2 + \frac{\sin^2\theta}{\varrho^2}\bigl(a\,\dd t_1-(r^2+a^2)\,\dd\phi_1\bigr)^2, \\
  g^{-1}_{\bhm,a} &= \varrho^{-2}\Bigl(-\frac{1-\chi^2}{\Delta}\bigl((r^2+a^2)\pa_{t_1}+a\pa_{\phi_1}\bigr)^2 + 2\chi\bigl((r^2+a^2)\pa_{t_1}+a\pa_{\phi_1}\bigr)\otimes_s\pa_r \\
    &\qquad\qquad + \Delta\pa_r^2 + \pa_\theta^2 + (\sin\theta)^{-2}\bigl(\pa_{\phi_1}+a\sin^2\theta\,\pa_{t_1}\bigr)^2\Bigr).
\end{split}
\end{align}
These tensors are smooth down to $r=r_+$ and can be analytically extended to the region $r>r_-=\bhm-\sqrt{\bhm^2-a^2}$, where $r_-<r_+$ is the other root of $\Delta$. For subextremal $a$, we have $\bhm\in(r_-,r_+)$, which is a partial motivation for the particular choices in Definition~\ref{DefKMfd}. Observe moreover that for large $\frac{r}{\bhm}$, the 1-form $\dd t_1=\dd t$ is (past) timelike for $g^{-1}_{\bhm,a}$; one can then easily construct a smooth function $F\in\CIc([\bhm,\infty))$ such that $t:=t_1-F(r)$ has timelike differential for all $r\geq\bhm$. (That is, the level sets of $t_1-F(r)$ are smooth spacelike hypersurfaces which are transversal to the event horizon.) Via pullback along the diffeomorphism
\begin{equation}
\label{EqKMetCoord}
  (t,r,\theta,\phi) \mapsto (t_1,r,\theta,\phi_1) = (t+F(r),r,\theta,\phi),
\end{equation}
we shall then regard $g_{\bhm,a}$ as a smooth metric on $\R_t\times[\bhm,\infty)_r\times\Sph^2_{\theta,\phi}$ and also, by restriction, on $M^\circ$. We call the null hypersurface $\cH^+=r^{-1}(r_+)$ the \emph{event horizon} of the black hole. We equip $(M,g_{\bhm,a})$ with the time orientation for which $\pa_t$ is future timelike near $\sface$, or equivalently $\dd t$ is past timelike everywhere.

\begin{lemma}[The Kerr metric as a 3sc-metric on $M$]
\label{LemmaK3sc}
  The Kerr metric $g_{\bhm,a}$ is a nondegenerate Lorentzian 3sc-metric on $M$, that is, $g_{\bhm,a}\in\CI(M;S^2\cT^*)$ has signature $(-,+,+,+)$, and $g_{\bhm,a}^{-1}\in\CI(M;S^2\cT)$.
\end{lemma}
\begin{proof}
  In the region of validity of polar coordinates on $\R^3$, a basis of smooth scattering 1-forms on $\ol{\R^4}$ is given by $\dd t$, $\dd r$, $r\,\dd\theta$, $r\,\dd\phi$. It thus suffices to observe from~\eqref{EqKMetStar} that the coefficients of $g_{\bhm,a}^{-1}$ in the frame $\dd t_1$, $\dd r$, $r\,\dd\theta$, $r\,\dd\phi_1$ are smooth functions of $(r^{-1},\theta,\phi_1)$ which are independent of $t_1$, and hence they are smooth functions on $M$. Similarly, the coefficients of $g_{\bhm,a}$ in the frame $\pa_{t_1}$, $\pa_r$, $r^{-1}\pa_\theta$, $r^{-1}\pa_{\phi_1}$ are smooth. Smoothness at $\theta=0,\pi$ can be verified by working with smooth coordinates on $\Sph^2$ there.
\end{proof}

The Minkowski metric, defined by
\[
  \ubar g:=-\dd t^2+\dd x^2=-\dd t^2+\dd r^2+r^2(\dd\theta^2+\sin^2\theta\,\dd\phi^2),
\]
is an important point of comparison. We note that $\ubar g$ is a nondegenerate Lorentzian scattering metric on $\ol{\R^4}$, i.e., $\ubar g\in\CI(\ol{\R^4};S^2\,\Tsc^*\ol{\R^4})$ and $\ubar g^{-1}\in\CI(\ol{\R^4};S^2\,\Tsc\ol{\R^4})$; its pullback to $M_0$ is thus a nondegenerate 3sc-metric on $M_0$ (and by restriction also on $M$), i.e., $\ubar g\in\CI(M;S^2\cT^*)$ has signature $(-,+,+,+)$, and $\ubar g^{-1}\in\CI(M;S^2\cT)$.

\begin{lemma}[Comparison of Kerr and Minkowski metrics]
\label{LemmaKMink}
  On $M$, we have
  \[
    g_{\bhm,a}-\ubar g \in \rho_\sface\CI(M;S^2\cT^*),\quad
    g_{\bhm,a}^{-1}-\ubar g^{-1} \in \rho_\sface\CI(M;S^2\cT).
  \]
\end{lemma}
\begin{proof}
  This is only nontrivial near $\sface$, where however it follows from the fact that $\ft=t$ there, and direct comparison of $\ubar g$, $\ubar g^{-1}$ with the expressions in~\eqref{EqKBL} using $\varrho^2,\Delta\in r^2(1+\rho_\sface\CI)$.
\end{proof}

Another choice of coordinates for large $\frac{r}{\bhm}$ is better adapted to the geometry of $g_{\bhm,a}$ near null infinity. Namely, in $r>2 r_+$ we introduce
\begin{equation}
\label{EqKtstar}
  t_* := \ft - \int_{2 r_+}^r \frac{r'{}^2+a^2}{\Delta(r')}\,\dd r' = t - r_* + c + \cO(r^{-1}),\qquad r_*:=r+2\bhm\log(r-2\bhm),
\end{equation}
where $c$ is an integration constant. Thus, when passing from $(\ft,r,\theta,\phi)$ coordinates to $(t_*,r,\theta,\phi)$, one needs to replace $\pa_r$ by $\pa_r-\frac{r^2+a^2}{\Delta}\pa_{t_*}$ in the expression~\eqref{EqKBL} for $g_{\bhm,a}^{-1}$, so
\begin{equation}
\label{EqKtstarMet}
\begin{split}
  g_{\bhm,a}^{-1} &= \varrho^{-2}\Bigl(-\frac{1}{\Delta}\bigl((r^2+a^2)\pa_{t_*}+a\,\pa_\varphi\bigr)^2 + \Delta\Bigl(\pa_r-\frac{r^2+a^2}{\Delta}\pa_{t_*}\Bigr)^2 \\
    &\quad\hspace{10em} + \pa_\theta^2 + (\sin\theta)^{-2}(\pa_\varphi+a\,\sin^2\theta\,\pa_{t_*})^2\Bigr).
\end{split}
\end{equation}
We compute the inner products of $\dd t_*$ with other scattering 1-forms:
\begin{equation}
\label{EqKtstarInner}
\begin{alignedat}{2}
  |\dd t_*|^2_{g_{\bhm,a}^{-1}}&=\frac{a^2\sin^2\theta}{\varrho^2}&&=\cO(r^{-2}), \\
  g_{\bhm,a}^{-1}(\dd t_*,\dd r)&=-\frac{r^2+a^2}{\varrho^2}&&=-1+\cO(r^{-2}), \\
  g_{\bhm,a}^{-1}(\dd t_*,r\,\dd\varphi)&=-\frac{2 a\bhm r^2}{\varrho^2\Delta}&&=\cO(r^{-2}).
\end{alignedat}
\end{equation}

Finally, for the purpose of solving initial value or forward problems for wave equations on $(M^\circ,g_{\bhm,a})$, we define the domain
\begin{equation}
\label{EqKDomain}
  \Omega := \cl_M\{t\geq 0\} = \{ \tau\geq 0 \},\quad \tau=\frac{t}{r},
\end{equation}
whose boundary hypersurfaces are thus $X$ and $\Sigma^\sharp\cap\pa\Omega$, in addition to the boundary hypersurfaces at infinity $\sface\cap\pa\Omega$ and $\cK^+$. The hypersurface $X$, resp.\ $\Sigma^\sharp$ is an initial, resp.\ final hypersurface in the sense that future timelike vector fields point into, resp.\ out of $\Omega$ at $X$, resp.\ $\Sigma^\sharp$. Initial value problems (possibly with nonzero forcing) in $\Omega$ with initial data at $X$ are well-posed. See Figure~\ref{FigKDomain}.

\begin{figure}[!ht]
\centering
\includegraphics{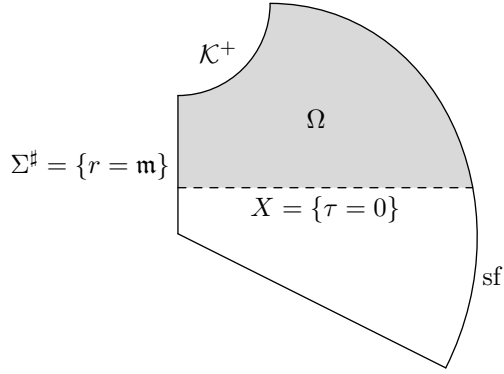}
\caption{The domain $\Omega$ inside of the compactified spacetime manifold $M$.}
\label{FigKDomain}
\end{figure}

\subsection{3b-structures and geometric differential operators}
\label{SsK3b}

We record some general observations regarding geometric operators related to 3sc-metrics on $M$ or $M_0$, such as the Kerr metric $g_{\bhm,a}$ or the Minkowski metric $\ubar g$. The first step is to understand the behavior of the basic (3-body-)scattering vector fields on $M$. Near $\cK^+$ and using the coordinates $\rho_\cK=\frac{r}{t}$, $\rho_\sface=\frac{1}{r}$, $\omega\in\Sph^2$, we compute
\[
  \pa_t=-\rho_\cK^2\rho_\sface\pa_{\rho_\cK},\quad
  \pa_r=\rho_\sface(\rho_\cK\pa_{\rho_\cK}-\rho_\sface\pa_{\rho_\sface}),\quad
  r^{-1}\pa_\omega=\rho_\sface\pa_\omega,
\]
where $\pa_\omega$ is our schematic notation for a vector field on $\Sph^2$. We leave the analogous computation on $M\setminus\cK^+$ in the coordinates $\tau=\frac{t}{r}$, $\rho_\sface=\frac{1}{r}$, $\omega$ to the reader. Passing back to coordinate derivatives in Cartesian coordinates, we conclude that the vector fields $\rho_\sface^{-1}\pa_{z^\mu}=r\pa_{z^\mu}$ are smooth on $M$ and tangent to its boundary hypersurfaces at infinity (i.e., $\pa M\cap\pa M_0$). Following \cite{Hintz3b}, we introduce:

\begin{definition}[3b-structures]
\label{DefK3b}
  The space of \emph{3b-vector fields} $\Vtb(M)$ is the $\CI(M)$-span of the vector fields $r\pa_{z^\mu}$, $\mu=0,\ldots,3$ (or equivalently of $r\pa_t$, $r\pa_r$, $\pa_\omega$). The \emph{3b-tangent bundle} $\Ttb M\to M$ is defined to be the trivial bundle $M\times\R^4$ which over $M^\circ$ is identified with $T M^\circ$ as follows: for $p\in M^\circ$ and $v=(v^0,\ldots,v^3)\in\R^4$, we identify $v\in\Ttb_p M$ with $v^\mu\,r\pa_{z^\mu}\in T_p M$. The \emph{3b-cotangent bundle} $\Ttb^*M\to M$ is the dual bundle. Finally, we write $\Difftb^k(M)$ for the space of up to $k$-fold compositions of 3b-vector fields (with $0$-fold compositions being defined as multiplication operators by elements of $\CI(M)$), and
  \[
    \rho_\sface^{-l}\Difftb^k(M)=\{\rho_\sface^{-l}A\colon A\in\Difftb^k(M)\}.
  \]
  For smooth vector bundles $\cE,\cF\to M$, we write $\Difftb^k(M;\cE,\cF)$ for the space of operators which, in local trivializations of $\cE,\cF$, are given by matrices of elements of $\Difftb^k$; and we write $\Difftb^k(M;\cE)=\Difftb^k(M;\cE,\cE)$; similarly for weighted versions.
\end{definition}

\begin{rmk}[3b-covectors]
\label{RmkK3b3sc}
  A frame of $\Ttb^*M$ is given by the 3b-1-forms $\frac{\dd z^\mu}{r}\in\CI(M;\Ttb^*M)$, or equivalently by $\frac{\dd t}{r}$, $\frac{\dd r}{r}$, $\dd\omega$ (1-forms on $\Sph^2$). Multiplying these 1-forms by $r=\rho_\sface^{-1}$ gives a frame of $\cT^*$. We also remark that $\Vtb(M)$ is a Lie subalgebra of $\cV(M)=\CI(M;T M)$.
\end{rmk}

The above computations show that
\begin{equation}
\label{EqK3bVF}
  \pa_{z^\mu} \in \rho_\sface\Vtb(M).
\end{equation}
Consider now a non-degenerate 3sc-metric $g\in\CI(M;S^2\cT^*)$, so $g^{-1}\in\CI(M;S^2\cT)$; the components with respect to $\pa_{z^\mu}$, where $z=(t,x)$, then satisfy $g_{\mu\nu}=g(\pa_{z^\mu},\pa_{z^\nu})$, $g^{\mu\nu}=g^{-1}(\dd z^\mu,\dd z^\nu)\in\CI(M)$. The Christoffel symbols of $g$ thus satisfy
\begin{equation}
\label{EqK3bChristoffel}
  \Gamma(g)_{\mu\nu}^\kappa = \frac12 g^{\kappa\lambda}(\pa_\mu g_{\nu\lambda} + \pa_\nu g_{\mu\lambda} - \pa_\lambda g_{\mu\nu}) \in \rho_\sface\CI(M),
\end{equation}
and therefore, for a vector field $V=V^\mu\pa_\mu$,
\[
  V^\nu{}_{;\mu}=(\nabla_\mu V)^\nu = \pa_\mu V^\nu + \Gamma_{\mu\lambda}^\nu V^\lambda.
\]
It then follows from~\eqref{EqK3bVF} and~\eqref{EqK3bChristoffel} that the covariant derivative on vector fields defines an element
\[
  \nabla \in \rho_\sface\Difftb^1(M;\cT,\cT^*\otimes\cT);
\]
more generally, for any tensor bundle $\cE=\cT^{p,q}:=\cT^{\otimes p}\otimes(\cT^*)^{\otimes q}$, we have
\begin{equation}
\label{EqK3bNabla}
  \nabla \in \rho_\sface\Difftb^1(M;\cE,\cT^*\otimes\cE).
\end{equation}
Note that this is stronger than the more natural seeming membership in $\Difftsc^1(M;\cE,\cT^*\otimes\cE)$ (3-body-scattering operators). As a consequence, we obtain the following memberships for geometric operators associated with Lorentzian 3sc-metrics $g$ such as $g_{\bhm,a}$ and $\ubar g$:
\begin{equation}
\label{EqK3bMem}
\begin{split}
  \Box_g &\in \rho_\sface^2\Difftb^2(M;\cT^{p,q}), \\
  \delta_g^* &\in \rho_\sface\Difftb^1(M;\cT^*,S^2\cT^*), \\
  \delta_g &\in \rho_\sface\Difftb^1(M;S^2\cT^*,\cT^*);
\end{split}
\end{equation}
here $\Box_g=-\tr_g\nabla^2$ is the tensor wave operator on the tensor bundle $\cT^{p,q}$, further $(\delta_g^*\omega)_{\mu\nu}=\frac12(\omega_{\mu;\nu}+\omega_{\nu;\mu})$, and finally $(\delta_g h)_\mu=-h_{\mu\nu}{}^{;\nu}$.

The metrics $g_{\bhm,a},\ubar g$ of interest in this paper are stationary, i.e., time translations $(t,x)\mapsto(t+c,x)$, $c\in\R$, are isometries. Therefore, all associated geometric operators $P$ such as~\eqref{EqK3bMem} commute with $t$-translations; here we use that time translations induce vector bundle isomorphisms $\tau_c\colon\cT_{(t,x)}\to\cT_{(t+c,x)}$ which simply map $\pa_{z^\mu}|_{(t,x)}\mapsto\pa_{z^\mu}|_{(t+c,x)}$; similarly for tensor bundles. Note now that the projection map $(t,x)\mapsto x$ extends to a smooth surjective map
\[
  \pi\colon M\to X,
\]
which is a left inverse to the inclusion map $X\cong\{0\}\times X\hra M$. This follows easily from $X=[0,\frac{1}{\bhm}]_\rho\times\Sph^2$, so the map $M\to X$ takes the form $(r,\tau,\omega)\mapsto(r^{-1},\omega)$, resp.\ $(\rho_\sface,\rho_{\cK,1},\omega)\mapsto(\rho_\sface,\omega)$, resp. $(\rho_\sface,\tau,\omega)\mapsto(\rho_\sface,\omega)$ in the three coordinate systems on $M$ described after Definition~\ref{DefKMfd}. (See \cite[Lemma~3.4 and Definition~3.26]{HintzGlueLocI} for an earlier instance of this construction.) Defining the bundle
\[
  \cT_X := \cT|_X \to X,
\]
the map assigning to $v\in\cT_{(t,x)}$ the element $\tau_{-t}(v)\in\cT_{(0,x)}=(\pi^*\cT_X)_{(t,x)}$ therefore gives an isomorphism
\begin{equation}
\label{EqK3bBundleIso}
  \cT|_M \cong \pi^*\cT_X
\end{equation}
of vector bundles over $M$. In concrete terms, we simply have $\cT=M\times\R^4$ and $\cT_X=X\times\R^4$, and the isomorphism~\eqref{EqK3bBundleIso} is the identity map on the $\R^4$ factor.

\begin{definition}[Spectral family]
\label{DefK3bSpec}
  Let $\cE_X\to X$ be a smooth vector bundle, and denote by $\pi\colon M\to X$ the projection. Suppose $P\in\rho_\sface^{-l}\Difftb^m(M;\pi^*\cE_X)$ commutes with $t$-translations. Then the \emph{spectral family} of $P$ is given by $\hat P(\sigma)\in\Diff^m(X^\circ;\cE_X)$, $\sigma\in\C$, where
  \[
    \hat P(\sigma)u := \bigl(e^{i\sigma t} P ( e^{-i\sigma t}u )\bigr)\big|_{t=0}\,,\quad u\in\CIc(X^\circ;\cE_X).
  \]
\end{definition}

For example, the spectral family of $\Box_{\ubar g}=-D_t^2+\Delta$, $\Delta:=\sum_{j=1}^3 D_{x^j}^2$, is $\wh{\Box_{\ubar g}}(\sigma)=\Delta-\sigma^2$; this also holds for the tensor wave operator $\Box_{\ubar g}\in\rho_\sface^2\Difftb^2(M;\cT^{p,q})$, which acts component-wise as the scalar wave operator in the tensor bundle splitting induced by the coordinate vector fields $\pa_{z^\mu}$.

Following \cite[\S{3.2}]{Hintz3b}, we record a structural result about the zero energy operator. We first recall:

\begin{definition}[b-structures]
\label{DefKb}
  The space of \emph{b-vector fields} $\Vb(\ol{\R^3})$ on $\ol{\R^3}$ is the $\CI(\ol{\R^3})$-span of the vector fields $\pa_{x^\mu}$ and $x^\nu\pa_{x^\mu}$, $1\leq\mu,\nu\leq 3$, or more efficiently of $\la x\ra\pa_{x^\mu}$, $\mu=1,\ldots,3$. We define $\Diffb^m(\ol{\R^3})$ and weighted versions $\la r\ra^l\Diffb^m(\ol{\R^3})$ as usual. For operators defined on $X=\ol{\R^3}\setminus\{r<\bhm\}$, we write spaces of weighted operators as $\rho^{-l}\Diffb^m(X)$, $\rho=r^{-1}$.
\end{definition}

A change of coordinates calculation shows that $\Vb(\ol{\R^3})$ is precisely the space of smooth vector fields on $\ol{\R^3}$ that are tangent to $\pa\ol{\R^3}$.

\begin{lemma}[Zero energy operator]
\label{LemmaK3bSpec0}
  Using the notation of Definition~\usref{DefK3bSpec}, so in particular $P\in\rho_\sface^{-l}\Difftb^m(M;\pi^*\cE_X)$, we have $\hat P(0) \in \rho^{-l}\Diffb^m(X;\cE_X)$.
\end{lemma}
\begin{proof}
  We can write
  \[
    P = r^l \sum_{j+k+|\alpha|\leq m} p_{j k\alpha} (r D_t)^j (r D_r)^k D_\omega^\alpha,\quad p_{j k\alpha}\in\CI(M;\End(\cE)),
  \]
  where $D_\omega^\alpha$ is schematic notation for an $|\alpha|$-th order differential operator on $\Sph^2$. The fact that $P$ commutes with $t$-translations is equivalent to $p_{j k\alpha}$ lying in (or more precisely, being the pullback along $\pi$ of an element of) $\CI(X;\End(\cE_X))$. Therefore,
  \[
    \hat P(0) = r^l\sum_{k+|\alpha|\leq m} p_{0 k\alpha} (r D_r)^k D_\omega^\alpha \in \rho^{-l}\Diffb^m(X;\cE_X).\qedhere
  \]
\end{proof}

The asymptotic behavior (as $\rho\to 0$) of solutions of b-differential equations such as $\hat P(0)u=0$ which are elliptic near $\rho=0$ is governed by the \emph{indicial roots} of $\hat P(0)$. For $\hat P(0)\in\rho^{-l}\Diffb^m(X;\cE_X)$, these are defined as follows. Fix a collar neighborhood $[0,\rho_0)_\rho\times\pa X$ of $\pa X$, and identify $\cE_X$ on this collar neighborhood with the pullback of $\cE_{\pa X}:=\cE_X|_{\pa X}$ along the projection $[0,\rho_0)\times\pa X\to\pa X$. Write then
\[
  \rho^l\hat P(0) = \sum_{j=0}^m P_j(\rho,\omega,\pa_\omega)(\rho\pa_\rho)^j,\quad P_j\in\CI\bigl([0,\rho_0);\Diff^{m-j}(\pa X;\cE_{\pa X})\bigr).
\]
Then the \emph{b-normal operator} of $\rho^l\hat P(0)$ is defined to be
\[
  N_{\pa X}(\rho^l\hat P(0)) := \sum_{j=0}^m P_j(0,\omega,\pa_\omega)(\rho\pa_\rho)^j \in \Diffb^m([0,\infty)_\rho\times\pa X;\cE_{\pa X}).
\]
An \emph{indicial root} $\lambda\in\C$ of $\hat P(0)$ is then a complex number for which there exists a nonzero $u_0\in\CI(\pa X;\cE_{\pa X})$ such that
\begin{equation}
\label{EqK3bIndRoots}
  N_{\pa X}(\rho^l\hat P(0)) ( \rho^\lambda u_0 ) = 0.
\end{equation}
Equivalently, define the elliptic operator $N_{\pa X}(\rho^l\hat P(0),\lambda):=\sum_{j=0}^m P_j(0,\omega,\pa_\omega)\lambda^j$; as a consequence of the ellipticity of $\hat P(0)$ near $\rho=0$, this is invertible for large $|\Im\lambda|$ when $\Re\lambda$ is fixed \cite[Proposition~5.3]{MelroseAPS}. Then $\lambda$ is an indicial root if and only if $N_{\pa X}(\rho^l\hat P(0),\zeta)^{-1}$ has a pole at $\zeta=\lambda$. We say that an indicial root has \emph{order $1$} if this pole has order $1$, or equivalently if there are no $u_0,u_1\in\CI(\pa X;\cE_{\pa X})$ with $u_1\neq 0$ such that $N_{\pa X}(\rho^l\hat P(0))(\rho^\lambda(u_1\log\rho+u_0))=0$. (Cf.\ the proof of \cite[Proposition~11.20]{HintzMicro} with the substitution $\rho=e^{-t}$.)

\begin{rmk}[Model at infinity]
\label{RmkK3bInfty}
  We stress that indicial roots (and their orders) depend only on the restrictions of the coefficients of $\rho^l\hat P(0)$ to $\rho=0$. In our application, this means that the indicial roots of the zero energy operator of a wave-type operator on Kerr can be computed by working with the corresponding Minkowskian operator.
\end{rmk}

\subsection{3b-analysis and its semiclassical version}
\label{SsK3bh}

The following is a natural scale of function spaces related to $\Vtb(M)$:

\begin{definition}[3b-Sobolev spaces]
\label{DefK3bSob}
  Let $\alpha_\sface,\alpha_\cK\in\R$ and $s\in\N_0$. We define $\Htb^0(M):=L^2(M;|\dd\ubar g|)$ (also written $\Htb^0(M;|\dd\ubar g|)$ if we need to make the volume density explicit); further
  \[
    \Htb^s(M) := \{ u\in\Htb^0(M) \colon (r\pa_z)^\alpha u\in\Htb^0(M)\ \forall\,\alpha\in\N_0^4,\ |\alpha|\leq s \},
  \]
  and finally
  \[
    \Htb^{s,(\alpha_\sface,\alpha_\cK)}(M) = \rho_\sface^{\alpha_\sface}\rho_\cK^{\alpha_\cK}\Htb^s(M) := \{ \rho_\sface^{\alpha_\sface}\rho_\cK^{\alpha_\cK}u \colon u\in\Htb^s(M) \}.
  \]
  If $\cE\to M$ is a vector bundle, then $\Htb^{s,(\alpha_\sface,\alpha_\cK)}(M;\cE)$ denotes functions which in local trivializations of $\cE$ are elements of $\Htb^{s,(\alpha_\sface,\alpha_\cK)}$.
\end{definition}

The natural squared norm on $\Htb^s(M)$ is the sum of squared $L^2$-norms of $(r\pa_z)^\alpha u$ for $|\alpha|\leq s$. One can then set $\|u\|_{\Htb^{s,(\alpha_\sface,\alpha_\cK)}}=\|\rho_\sface^{-\alpha_\sface}\rho_\cK^{-\alpha_\cK}u\|_{\Htb^s}$. Every $P\in\Difftb^m(M)$ maps $\Htb^s(M)\to\Htb^{s-m}(M)$ (for $s\geq m$, since so far we have only defined these spaces for nonnegative integer orders), similarly for weighted operators. We also note that every weighted 3b-Sobolev space is a subspace of the space of tempered distributions on $M$, defined as the dual of the space of smooth functions (or densities) on $M$ vanishing to infinite order at $\cK^+\cup\sface$.

The microlocal analysis of weighted 3b-operators acting between weighted 3b-Sobolev spaces is done on the 3b-cotangent bundle. The starting point is that the principal symbol of a 3b-differential operator $P\in\Difftb^m(M)\subset\Diff^m(M^\circ)$ extends to an element of the space $P^{[m]}(\Ttb^*M)$ of smooth functions on $\Ttb^*M$ which are homogeneous polynomials of degree $m$ on every fiber of $\Ttb^*M$; this follows from the fact that the principal symbol of $r D_\mu=-i r\pa_\mu$ is given by the function $\Ttb^*M\ni\frac{\dd z^\nu}{r}\mapsto \delta_\mu^\nu$ (Kronecker delta).

For the study of forward problems of wave equations on $(M^\circ,g_{\bhm,a})$, we need spaces of functions on the domain $\Omega$ from~\eqref{EqKDomain} which encode vanishing in the past of the initial hypersurface $X$, but which allow arbitrary behavior (consistent with membership in a 3b-Sobolev space) near the final interior hypersurface $\Sigma^\sharp=\cl_M\{r=\bhm\}$. We thus write
\[
  \Htb^s(\Omega)^{\bullet,-} := \{ u\in\Htb^s(M) \colon \supp u\subset\Omega \},
\]
equipped with the induced norm; similarly for weighted spaces. For a future generalization to arbitrary real orders $s$, it is useful to give a slightly different definition. First of all, we can define 3b-structures on $M_0$ (as done in \cite{Hintz3b}) using the 3b-vector fields $\la x\ra\pa_{z^\mu}$, the only difference to our previous discussion being that the weight $\la x\ra$ is smooth and positive also near $x=0$ (a region which was excised in $M$). (It is at this point that the blow-up of the south pole $\fk^-$ in the definition of $M_0$, which gives time-inversion symmetry, comes in handy.) One can then define 3b-Sobolev spaces and their weighted versions completely analogously to Definition~\ref{DefK3bSob}. The space $\Htb^s(M_0)$ for real $s$ can then be defined using duality and interpolation. For $s\in\N_0$, we may now set
\begin{equation}
\label{EqK3bSuppExt}
  \Htb^s(\Omega)^{\bullet,-} = \{ u|_{\{r>\bhm\}} \colon u\in\Htb^s(M_0),\ t\geq 0\ \text{on}\ \supp u \},
\end{equation}
equipped with the quotient norm. (Note that $\Htb^s(\Omega)^{\bullet,-}$ is the quotient of the space $\Htb^s(\cl_{M_0}\{t\geq 0\})^\bullet$ of elements of $\Htb^s(M_0)$ with support in $\cl_{M_0}\{t\geq 0\}$ by the subspace consisting of elements supported moreover in $\cl_{M_0}\{r\leq\bhm\}$.) This definition also makes immediate sense for $s\in\R$. The $L^2$-dual of $\Htb^s(\Omega)^{\bullet,-}$ is the space
\begin{equation}
\label{EqK3bExtSupp}
  \Htb^{-s}(\Omega)^{-,\bullet} := \{ u|_{\{t>0\}} \colon u\in\Htb^{-s}(M_0),\ r\geq\bhm\ \text{on}\ \supp u \},
\end{equation}
which we again equip with the quotient norm. For details regarding these spaces and equivalent definitions, we refer the reader to \cite[Chapter~9]{HintzMicro} (which can be straightforwardly adapted to the present 3b-setting). We omit spelling out the definitions of weighted versions of these spaces.

\begin{definition}[Conormality]
\label{DefK3bConormal}
  Let $\alpha,\alpha_\sface,\alpha_\cK\in\R$.
  \begin{enumerate}
  \item The space $\cA^\alpha(X)$ of \emph{weighted conormal} (or \emph{b-regular}) \emph{functions} consists of all functions of the form $\rho^\alpha u_0$ where $(\la x\ra\pa_x)^\gamma u_0\in\CI(X^\circ)\cap L^\infty(X^\circ)$ for all $\gamma\in\N_0^3$, i.e., $u_0$ and all its b-derivatives remain uniformly bounded on $X^\circ$.
  \item The space $\cA_\tbop^{\alpha_\sface,\alpha_\cK}(M)$ of \emph{weighted 3b-conormal} (or \emph{3b-regular}) \emph{functions} consists of all functions of the form $\rho_\sface^{\alpha_\sface}\rho_\cK^{\alpha_\cK}u_0$ where $u_0$ and all its 3b-derivatives remain uniformly bounded on $M^\circ$. We write $\cA_\tbop(M):=\cA_\tbop^{0,0}(M)$. We define $\cA_\tbop^{\alpha_\sface,\alpha_\cK}(M_0)$ analogously.
  \end{enumerate}
\end{definition}

\subsubsection{Bounded geometry perspective, 3b-ps.d.o.s}
\label{SssK3bBdd}

The spaces
\begin{equation}
\label{EqK3bDiffCoeff}
  \cA_\tbop\Difftb^m(M_0)
\end{equation}
(finite linear combinations of products $u P$ where\footnote{This is thus the space of 3b-differential operators with 3b-regular coefficients.} $u\in\cA_\tbop(M_0)$ and $P\in\Difftb^m(M_0)$) and $\Htb^s(M_0)$ (and their weighted versions) arise as special cases of \emph{bounded geometry operators} and Sobolev spaces (following \cite{ShubinBounded}), which are in turn special cases of a general construction introduced in \cite{HintzScaledBddGeo}; see \cite[Example~(11) in \S{1.4.4}]{HintzScaledBddGeo} for the case of 3b-operators with b-regular coefficients. We recall this using terminology from \cite{HintzScaledBddGeo} for 3b-operators with 3b-regular coefficients. We shall only state the results and refer the reader to \cite[\S{3}]{HintzScaledBddGeo} for proofs. Thus, we cover $M_0^\circ=\R^4$ with two families of sets (called \emph{unit cells}), each equipped with a map to the cube $(-2,2)^4$.
\begin{enumerate}
\item The first family covers the region $\R_t\times(-2,2)_x^3$:
  \begin{equation}
  \label{EqK3bCell1}
  \begin{split}
    U_{(1),j} &:= (j-2,j+2)_t \times (-2,2)_x^3, \\
    \phi_{(1),j}(t,x)&:=(t-j,x);
  \end{split}
  \end{equation}
  here $j\in\Z$.
\item The second family covers the complement of $\R_t\times[-\frac14,\frac14]_x^3$. In the region where $\pm x^1>0$, we use local coordinates $\omega_a=\frac{x^a}{x^1}$, $a=2,3$, on $\Sph^2$ and the sets
  \begin{equation}
  \label{EqK3bCell2}
  \begin{split}
    U_{(2),\pm 1,j,k,m_2,m_3} &:= \bigl(2^k(j-2),2^k(j+2)\bigr)_t \times \bigl(\pm(2^{k-2},2^{k+2})_{x^1}\bigr) \\
      &\quad\qquad \times (m_2-2,m_2+2)_{\omega_2} \times (m_3-2,m_3+2)_{\omega_3}, \\
    \phi_{(2),\pm 1,j,k,m_2,m_3}(t,x^1,\omega_2,\omega_3) &:= \bigl( 2^{-k}t-j,\,\log_2(\pm x^1)-k,\,\omega_2-m_2,\,\omega_3-m_3\bigr);
  \end{split}
  \end{equation}
  here $k\in\N_0$, $j\in\Z$, and $m_2,m_3=-2,-1,0,1,2$. We define the sets $U_{(2),\pm i,j,k,m_2,m_3}$ for $i=2,3$ analogously.
\end{enumerate}
See Figure~\ref{FigK3bCells}.

\begin{figure}[!ht]
\centering
\includegraphics{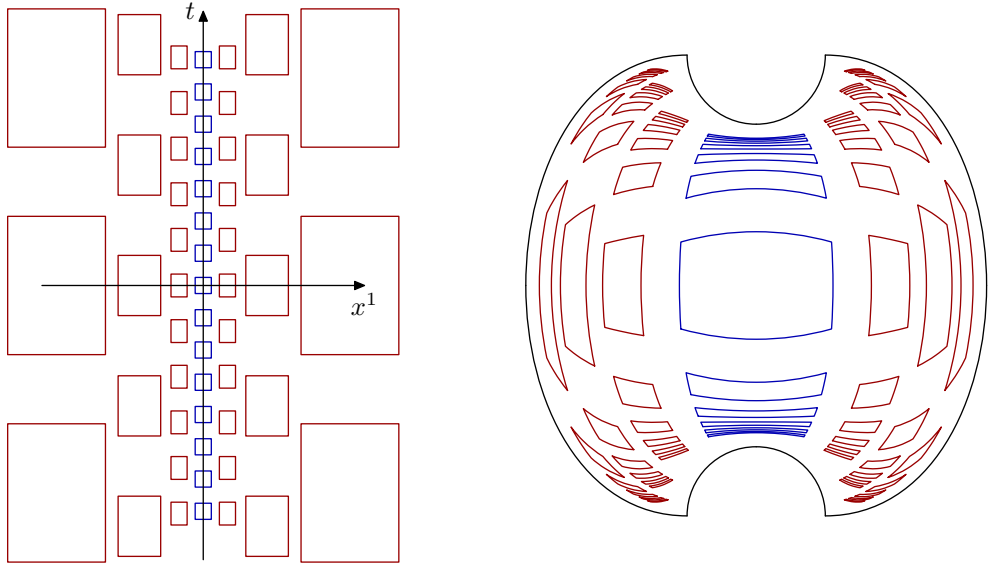}
\caption{Some unit cells $U_j^{(1)}$ (blue) and $U_{\pm 1,j k}^{(2)}$ (red) defined in~\eqref{EqK3bCell1}--\eqref{EqK3bCell2}, shown here in $2$ (instead of $4$) dimensions. \textit{On the left:} in $\R_{t,x^1}$ (dropping the $x^2,x^3$ coordinates). \textit{On the right:} in $M_0$, or more accurately the blow-up of $\ol{\R^2}$ at its north and south poles.}
\label{FigK3bCells}
\end{figure}

The index set for the sets and maps~\eqref{EqK3bCell1}--\eqref{EqK3bCell2} is thus
\[
  \sA := \bigl\{ ((1),j),\ ((2),\pm i,j,k,m_2,m_3) \colon j\in\Z,\ i=1,2,3,\ k\in\N_0,\ m_2,m_3=-2,-1,0,1,2 \bigr\},
\]
and the sets $\phi_\alpha^{-1}((-1,1)^4)$ cover $M_0^\circ$ with finite multiplicity. The key point is that the basic 3b-vector fields push forward under $\phi_\alpha$ to uniformly bounded (in $\CI((-2,2)^4)$) families of vector fields on $(-2,2)^4$; concretely, writing points in $(-2,2)^4$ as $(T,X)$ (for the sets~\eqref{EqK3bCell1}) and $(T,X^1,\Omega^2,\Omega^3)$ (for the sets~\eqref{EqK3bCell2}), we have
\[
  (\phi_{(1),j})_*\pa_t = \pa_T,\quad
  (\phi_{(1),j})_*\pa_{x^i} = \pa_{X^i},\ i=1,2,3,
\]
and, noting that $r=|x^1|\sqrt{1+\omega_2^2+\omega_3^2}$ is, on $U_{(2),\pm 1,j,k,m_2,m_3}$, uniformly equivalent to $|x^1|$ (in that the pushforwards under $\phi_{(2),\pm 1,j,k,m_2,m_3}$ of the quotients $r/|x^1|$, $|x^1|/r$ are uniformly bounded in $\CI((-2,2)^4)$),
\begin{align*}
  (\phi_{(2),\pm 1,j,k,m_2,m_3})_*(|x^1|\pa_t) &= 2^{X^1} \pa_T, \\
  (\phi_{(2),\pm 1,j,k,m_2,m_3})_*(|x^1|\pa_{x^1}) &= \pm\frac{1}{\log 2}\pa_{X^1} \mp (\omega_2\pa_{\omega_2}+\omega_3\pa_{\omega_3}), \\
  (\phi_{(2),\pm 1,j,k,m_2,m_3})_*(|x^1|\pa_{x^a}) &= \pm\pa_{\omega_a},\ a=2,3.
\end{align*}
These expressions show, conversely, that the coordinate vector fields on $(-2,2)^4$ are linear combinations of the pushforwards of the basic 3b-vector fields $\la x\ra\pa_{z^\mu}$, $\mu=0,1,2,3$, with uniformly bounded coefficients in $\CI((-2,2)^4)$. Therefore, $u\in\CI(\R^4)$ lies in $\cA_\tbop(M_0)$ if and only if $(\phi_\alpha)_*u$ is uniformly bounded in $\CI((-2,2)^4)$. Similarly, $V\in\cV(\R^4)=\CI(\R^4;T\R^4)$ lies in $\cA_\tbop\Vtb(M_0)$ if and only if $(\phi_\alpha)_*V$ is uniformly (in $\alpha\in\sA$) bounded in $\cV((-2,2)^4)$. We can then similarly characterize $\cA_\tbop\Difftb^m(M_0)$ as the space of differential operators on $\R^4$ that push forward under the maps $\phi_\alpha$ to a uniformly bounded family of differential operators on $(-2,2)^4$.

Fix now a cutoff function
\[
  \chi \in \CIc\bigl((-\tfrac32,\tfrac32)^4\bigr),\quad \chi=1\ \text{on}\ [-1,1]^4,
\]
and set
\[
  \chi_\alpha := \frac{\phi_\alpha^*\chi}{\sum_{\beta\in\sA} \phi_\beta^*\chi},
\]
then the functions $(\phi_\alpha)_*\chi_\alpha\in\CI((-2,2)^4)$ are uniformly bounded, and $\sum_\alpha \chi_\alpha=1$ on $\R^4$. We then have the equivalence of norms (i.e., the left-hand side is bounded by a multiple of the right-hand side, and vice versa)
\[
  \|u\|_{\Htb^s(M_0;|\dd\mu_\tbop|)}^2 \sim \sum_{\alpha\in\sA} \|(\phi_\alpha)_*(\chi_\alpha u)\|_{H^s(\R^4)}^2,
\]
where $|\dd\mu_\tbop|:=\la x\ra^{-4}|\dd t\,\dd x|$ is a smooth positive 3b-density. (This is easy to check for $s\in\N_0$, and can be used as a definition of $\Htb^s(M_0)$ for $s\in\R$ equivalent to the one using duality and interpolation.) Regarding weights, notice first that the pushforward of $\rho_\sface^{\alpha_\sface}\rho_\cK^{\alpha_\cK}$ under $\phi_{(1),j}$ and $\phi_{(2),\pm 1,j,k,m_2,m_3}$ is uniformly equivalent to $\la j\ra^{-\alpha_\cK}$ and $2^{-k\alpha_\sface}\la j\ra^{-\alpha_\cK}$. Denote these numbers by $w_\alpha$ on $U_\alpha$. Since the $w_\alpha$ change by uniformly bounded factors across adjacent unit cells, $\rho_\sface^{\alpha_\sface}\rho_\cK^{\alpha_\cK}$ is a weight (and $w_\alpha$ is a weight family) in the sense of \cite[Definition~3.4]{HintzScaledBddGeo}. We then have
\[
  \|u\|_{\rho_\sface^{\alpha_\sface}\rho_\cK^{\alpha_\cK}\Htb^s(M_0;|\dd\mu_\tbop|)}^2 \sim \sum_{\alpha\in\sA} \|w_\alpha^{-1}(\phi_\alpha)_*(\chi_\alpha u)\|_{H^s(\R^4)}^2.
\]
For $\alpha_\sface=2$ and $\alpha_\cK=0$, this is equivalent to the squared $H_\tbop^s(M_0;|\dd t\,\dd x|)$-norm.

We can moreover define spaces of \emph{3b-ps.d.o.s with 3b-regular coefficients} by patching together standard quantizations on $\R^4$ using the partition of unity $\chi_\alpha$. Concretely, write
\begin{equation}
\label{EqK3bSymb}
  S^m(\Ttb^*M_0)
\end{equation}
for the space of functions $a=a(z,\zeta)$ on $T^*\R^4$ such that $(\phi_\alpha)_*(\chi_\alpha a)$ is uniformly bounded in the space $S^m(\R^4;\R^4)$ of symbols on $\R^4$; the seminorms on the latter space are
\[
  |a|_{S^m;k} := \sup_{Z,\Xi\in\R^4}\max_{|\alpha|,|\beta|\leq k} |\la\Xi\ra^{-m+|\beta|}\pa_Z^\alpha\pa_\Xi^\beta a(Z,\Xi)|.
\]

\begin{rmk}[Regularity in the base]
\label{RmkK3bReg}
  The notation $S^m(\Ttb^*M_0)$ is an abuse of notation compared to \cite{Hintz3b} since we do not require the smoothness of elements of $S^m(\Ttb^*M_0)$ in the base variables down to $\pa M_0$. Rather, elements of $S^m(\Ttb^*M_0)$ as defined here are 3b-regular in the base variables; the reader may thus wish to write $\cA_\tbop S^m(\Ttb^*M_0)$ instead (which we will not do here). Nonetheless, if we define the symbol class $S^m(\cE)$ on the vector bundle $\cE\to M_0$ in the standard fashion (i.e., including smoothness down to $\pa M_0$) as, e.g., in \cite[Definition~5.35]{HintzMicro}, then this symbol class is contained in the space~\eqref{EqK3bSymb}.
\end{rmk}

Fixing $\psi\in\CIc((-\frac12,\frac12)^4)$ to be equal to $1$ near $0$, we define for $a_\alpha\in S^m(\R^4;\R^4)$ the Schwartz kernel of its quantization by
\[
  \Op(a_\alpha)(Z,Z') := (2\pi)^{-4}\int_{\R^4} e^{i(Z-Z')\cdot\Xi} \psi(Z-Z')a_\alpha(Z,\Xi)\,\dd\Xi \cdot|\dd Z'|
\]
(which is a right density, with $|\dd Z'|$ being the density factor) and then, for $a\in S^m(\Ttb^*M_0)$,
\[
  \Op(a) := \sum_{\alpha\in\sA} (\phi_\alpha)^* \Op\bigl((\phi_\alpha)_*(\chi_\alpha a)\bigr).
\]
Here, we identify a Schwartz kernel with the corresponding linear operator, and for the linear operator $A=\Op((\phi_\alpha)_*(\chi_\alpha a))\colon\CIc((-2,2)^4)\to\CIc((-2,2)^4)$, we write $(\phi_\alpha)^*A$ for the operator mapping $u\in\CI(M_0)$ to $\phi_\alpha^*(A((\phi_\alpha)_*u))$. Define, finally, the space of \emph{residual 3b-ps.d.o.s}
\[
  \Psitb^{-\infty}(M_0)
\]
to consist of all operators $R\colon\CIc(\R^4)\to\CI(\R^4)$ which are bounded as maps $w\Htb^{-N}(M_0)\to w\Htb^N(M_0)$ for all $N$ and all weights $w$ together with their $L^2$-adjoints $R^*$. Then
\begin{equation}
\label{EqK3bQuant}
  \Psitb^m(M_0) := \{ \Op(a) + R \colon a\in S^m(\Ttb^*M_0),\ R\in\Psitb^{-\infty}(M_0) \}.
\end{equation}

\begin{rmk}[Regularity of coefficients]
\label{RmkK3bPsdoReg}
  In analogy with the notation~\eqref{EqK3bDiffCoeff}, the reader may wish to record the mere 3b-regularity of the symbols $a$ we are quantizing in~\eqref{EqK3bQuant} by writing instead $\cA_\tbop\Psitb^m(M_0)$. We shall not do this in this paper for ease of notation. We caution the reader that what we call $\Psitb$ here is a considerably larger space than that defined in \cite{Hintz3b} (which features symbols which are smooth in the base); but for purposes of microlocalization in $\Ttb^*M_0$ and mapping properties on 3b-Sobolev spaces, the present space $\Psitb$ is perfectly adequate.
\end{rmk}

Spaces $\Psitb^{m,(\alpha_\sface,\alpha_\cK)}(M_0)=\rho_\sface^{-\alpha_\sface}\rho_\cK^{-\alpha_\cK}\Psitb^m(M_0)$ of weighted 3b-ps.d.o.s are defined similarly as sums of quantizations of weighted symbols $a\in\rho_\sface^{-\alpha_\sface}\rho_\cK^{-\alpha_\cK}S^m(\Ttb^*M_0)$ and weighted residual operators mapping $w\Htb^{-N}\to w\rho_\sface^{-\alpha_\sface}\rho_\cK^{-\alpha_\cK}\Htb^N(M_0)$ for all $N$ and all weights $w$ together with their adjoints. Every element of $\Psitb^m(M_0)$ defines a bounded linear map $\Htb^s(M_0)\to\Htb^{s-m}(M_0)$; similarly with weights.

It is shown in \cite[\S{3}]{HintzScaledBddGeo} that the equivalence class of the symbol $a$ in $S^m/S^{m-1}$ gives rise to a principal symbol map $\sigmatb^m$ which fits into the short exact sequence
\begin{subequations}
\begin{equation}
\label{EqK3bSES}
  0 \to \Psitb^{m-1}(M_0) \hra \Psitb^m(M_0) \xra{\sigmatb^m} S^m/S^{m-1}(\Ttb^*M_0) \to 0;
\end{equation}
moreover, the space of (weighted) 3b-ps.d.o.s is a *-algebra, and for $A\in\Psitb^m$, $B\in\Psitb^{m'}$, we have
\begin{equation}
\label{EqK3bAlg}
\begin{gathered}
  A\circ B\in\Psitb^{m+m'}(M_0),\quad \sigmatb^{m+m'}(A\circ B)=\sigmatb^m(A)\cdot\sigmatb^{m'}(B), \\
  \sigmatb^{m+m'-1}(i[A,B])=\{\sigmatb^m(A),\sigmatb^{m'}(B)\},
\end{gathered}
\end{equation}
\end{subequations}
where $\{\cdot,\cdot\}$ denotes the Poisson bracket on $T^*\R^4$.

Finally, for $A=\Op(a)+R\in\Psitb^m(M_0)$, we can define its elliptic, characteristic, and operator wave front sets
\[
  \Elltb(A),\ \Char_\tbop(A),\ \WF'_\tbop(A)
\]
as conic subsets of $\Ttb^*M_0\setminus o$, or equivalently as subsets of the 3b-cosphere bundle $\Stb^*M_0$ (the boundary at fiber infinity of $\ol{\Ttb^*}M_0$, identified with $(\Ttb^*M_0\setminus o)/\R^+$), in the usual fashion: a point $\alpha\in\Stb^*M_0$ lies in $\Elltb(A)$, resp.\ does not lie in $\WF'_\tbop(A)$, if and only if $\chi/a\in S^{-m}(\Ttb^*M_0)$, resp.\ $\chi a\in S^{-\infty}(\Ttb^*M_0)$ for some $\chi\in\CI(\ol{\Ttb^*}M_0)$ that does not vanish at $\alpha$. Note here that we use $\Ttb^*M_0$ as the phase space even though the symbols $a$ need not be smooth down to $\Ttb^*_{\pa M_0}M_0$. (In the parlance of \cite[\S{3.8}]{HintzScaledBddGeo}, $\ol{\Ttb^*}M_0$ is an \emph{admissible compactification}.)

\subsubsection{Semiclassical 3b-notions}

The constraint damping modification we introduce in this paper involves a large parameter $h^{-1}$ where $0<h\ll 1$; the resulting wave-type operators on Kerr that we need to study are then \emph{semiclassical} 3b-differential operators, which we will study on semiclassical 3b-Sobolev spaces using semiclassical 3b-pseudodifferential operators. We proceed to introduce these notions.

\begin{definition}[Semiclassical 3b-differential operators]
\label{DefK3bhDiff}
  By $\Difftbh^m(M)$, we denote the space of all families $P=(P_h)_{h\in(0,1)}$ of differential operators on $M^\circ$ such that
  \begin{equation}
  \label{EqK3bhDiffExpr}
    P_h = \sum_{j+k+|\alpha|\leq m} p_{j k\alpha}(h,t,r,\omega) (h r D_t)^j (h r D_r)^k (h D_\omega)^\alpha,
  \end{equation}
  where $p_{j k\alpha}\in\CI([0,1)_h\times M)$. Spaces of weighted such operators are denoted $\rho_\sface^{-l}\Difftbh^m(M)$, and spaces of such operators acting between sections of smooth vector bundles $\cE,\cF\to M$ are denoted $\rho_\sface^{-l}\Difftbh^m(M;\cE,\cF)$.
\end{definition}

We could equivalently have defined semiclassical 3b-vector fields to be families $V=(V_h)_{h\in(0,1)}$ of vector fields on $M^\circ$ such that $h^{-1}V_h\in\CI([0,1)_h;\Vtb(M))$, and then semiclassical 3b-differential operators are up to $m$-fold compositions of semiclassical 3b-vector fields. In the notation~\eqref{EqK3bhDiffExpr}, we define the principal symbol of an element $P=(P_h)_{h\in(0,1)}$ to be
\[
  \sigmatbh^m(P) \colon \Ttb^*M \ni \zeta\frac{\dd t}{r}+\xi\frac{\dd r}{r} + \eta\cdot\dd\omega \mapsto \sum_{j+k+|\alpha|\leq m} p_{j k\alpha}(0,t,r,\omega)\zeta^j\xi^k\eta^\alpha;
\]
thus $\sigmatbh^m(P)\in P^m(\Ttb^*M)$, where $P^m$ denotes the space of smooth functions which are fiber-wise polynomials of degree $m$.

Function spaces adapted to the study of such operators are the following 3b-Sobolev spaces with $h$-dependent norms:

\begin{definition}[Semiclassical 3b-Sobolev spaces]
\label{DefK3bhSob}
  Let $\alpha_\sface,\alpha_\cK\in\R$ and $s\in\N_0$. Fix a choice of defining functions $\rho_\cK$ and $\rho_\sface$ of $\cK^+$ and $\sface$, respectively. We define $H_{\tbop,h}^{s,(\alpha_\sface,\alpha_\cK)}(M)=\Htb^{s,(\alpha_\sface,\alpha_\cK)}(M)$ as a set, but with $h$-dependent norm
  \[
    \|u\|_{H_{\tbop,h}^{s,(\alpha_\sface,\alpha_\cK)}(M)}^2 = \sum_{j+k+|\alpha|\leq s} \| \rho_\sface^{-\alpha_\sface}\rho_\cK^{-\alpha_\cK}(h r D_t)^j (h r D_r)^k (h D_\omega)^\alpha u \|_{L^2(M;|\dd\ubar g|)}^2.
  \]
  If one wishes to emphasize the volume density, one writes $H_{\tbop,h}^s(M;|\dd\ubar g|)$, similarly for weighted versions.
\end{definition}

We can similarly define spaces $\Difftbh^m(M_0)$ and $\Htbh^s(M_0)$ on all of $M_0$ using $h\la x\ra D_t$, $h\la x\ra D_x$ instead of $h r D_t$, $h r D_r$, $h D_\omega$. We can then define
\begin{equation}
\label{EqK3bhSob}
  H_{\tbop,h}^s(\Omega)^{\bullet,-},\quad
  H_{\tbop,h}^s(\Omega)^{-,\bullet}
\end{equation}
as quotient spaces of (closed subspaces of) $H_{\tbop,h}^s(M_0)$ as in~\eqref{EqK3bSuppExt}--\eqref{EqK3bExtSupp} above. Defining the latter space for real $s\in\R$ using duality and interpolation, we thus also get the spaces~\eqref{EqK3bhSob} for real $s$; similarly for weighted versions.

The parameterized scaled bounded geometry perspective on semiclassical 3b-analysis, following the general recipe in~\cite{HintzScaledBddGeo}, is as follows. For every $h\in(0,1)$, we consider the same cover $(U_\alpha,\phi_\alpha)$ of $\R^4$ as in~\eqref{EqK3bCell1}--\eqref{EqK3bCell2}, but with $h$-dependent \emph{scalings} $\rho_{h,\alpha,\mu}=h$, $\mu=0,1,2,3$. The associated space of $m$-th order differential operators is then the space
\begin{equation}
\label{EqK3bhDiffCoeff}
  \cA_\tbop\Difftbh^m(M_0)
\end{equation}
of families $A=(A_h)_{h\in(0,1)}$ of operators $A_h$ with the property that, upon writing
\[
  (\phi_\alpha)_*A_h = \sum_{|\beta|\leq m} a_\alpha(h,Z)(h\pa_Z)^\beta,
\]
the functions $a(h,\cdot)\in\CI((-2,2)^4)$ are uniformly bounded in $\alpha\in\sA$ and $h\in(0,1)$. (Equivalently, we can write $A_h$ as a finite sum of operators of the form $u_h B$ where $B\in\Difftbh^m(M_0)$ has smooth coefficients and $(u_h)_{h\in(0,1)}\subset\cA_\tbop(M_0)$ is a bounded subset. This justifies the notation~\eqref{EqK3bhDiffCoeff}.) Similarly, we can define a norm equivalent to the $H_{\tbop,h}^s(M_0)$-norm defined above by
\[
  \|u\|_{H_{\tbop,h}^s(M_0;|\dd\mu_\tbop|)}^2 := \sum_{\alpha\in\sA} \| (\phi_\alpha)_*(\chi_\alpha u) \|_{H_h^s(\R^4)}^2,\quad
  \|v\|_{H_h^s(\R^4)} := \| \la h D\ra^s v \|_{L^2(\R^4)},
\]
where we recall $|\dd\mu_\tbop|=\la x\ra^{-4}|\dd t\,\dd x|$. Turning to pseudodifferential operators, we now quantize families $(a_h)_{h\in(0,1)}$ of elements $a_h\in S^m(\Ttb^*M_0)$ which are uniformly bounded in $h$ via
\begin{align*}
  &\Op_h(a_h) :=\sum_{\alpha\in\sA} (\phi_\alpha)^*\Op_h\bigl((\phi_\alpha)_*(\chi_\alpha a_h)\bigr), \\
  &\qquad \Op_h(a)(Z,Z') := (2\pi)^{-4}\int_{\R^4} e^{i(Z-Z')\cdot\Xi}\psi(Z-Z')a(Z,h\Xi)\,\dd\Xi\cdot|\dd Z'|.
\end{align*}
We denote the space of such families $(a_h)_{h\in(0,1)}$ by $S^m({}^{\tbop,\semi}T^*M_0)$. The residual space $h^\infty\Psitbh^{-\infty}(M_0)$ consists of all families of operators which for all $N$ are uniformly bounded as maps $h^{-N}H_{\tbop,h}^{-N}(M_0)\to h^N H_{\tbop,h}^N(M_0)$, similarly for all weighted versions, for their $L^2$-adjoints, and for any finite number of commutators with $h$-independent 3b-vector fields. (We use here the residual space from \cite[\S{4.4}]{HintzScaledBddGeo} since it is easier to describe, though the space in \cite[Definition~3.34]{HintzScaledBddGeo} would be just as good for present purposes.) Then
\[
  \Psitbh^m(M_0) := \{ (\Op_h(a_h))_{h\in(0,1)} + R \colon \{a_h\}\subset S^m(\Ttb^*M_0)\ \text{is bounded},\ R\in h^\infty\Psitbh^{-\infty}(M_0) \}.
\]
We refer readers objecting to the notation (which does not emphasize the notion of regularity in the base variables of the symbols underlying elements of $\Psitbh$) to Remark~\ref{RmkK3bPsdoReg}. The analogues of~\eqref{EqK3bSES}--\eqref{EqK3bAlg} in the semiclassical setting read
\begin{equation}
\label{EqK3bhSES}
  0 \to h\Psitbh^{m-1}(M_0) \hra \Psitbh^m(M_0) \xra{\sigmatbh^m} S^m/h S^{m-1}(\Ttbh^*M_0) \to 0
\end{equation}
and, for $A\in\Psitbh^m(M_0)$ and $B\in\Psitbh^{m'}(M_0)$,
\begin{gather*}
  A\circ B\in\Psitbh^{m+m'}(M_0),\quad \sigmatbh^{m+m'}(A\circ B)=\sigmatbh^m(A)\cdot\sigmatbh^{m'}(B), \\
  \sigmatbh^{m+m'-1}\Bigl(\frac{i}{h}[A,B]\Bigr)=\{\sigmatbh^m(A),\sigmatbh^{m'}(B)\}.
\end{gather*}
There are two differences between the semiclassical 3b-principal symbol thus defined for elements of $\Psitbh$ and the principal symbol of elements of $\Difftbh$ defined previously: first, since (following \cite{HintzScaledBddGeo}) we do not impose any smoothness in $h$ for the symbols underlying semiclassical 3b-ps.d.o.s, we cannot restrict them to $h=0$ but instead must work with a quotient by symbols featuring an extra power $h$. Second, for ps.d.o.s we elect to have $\sigmatbh$ capture operators modulo operators with an additional order of $h$ \emph{and} lower differential order (cf.\ the appearance of $m-1$ on the right in~\eqref{EqK3bhSES}). This will not lead to confusion in this paper since all semiclassical 3b-differential operators appearing here will have smooth coefficients on $[0,1)\times M$ and, in the notation~\eqref{EqK3bhDiffExpr}, the leading differential order coefficients $p_{j k\alpha}$, $j+k+|\alpha|=m$, will in fact be \emph{independent} of $h$, and so the symbol at $h=0$ does also capture $P$ modulo elements of $h\Difftbh^{m-1}(M_0)$.

Every $A\in\Psitbh^m(M_0)$ defines a \emph{uniformly} bounded family of maps $H_{\tbop,h}^s(M_0)\to H_{\tbop,h}^{s-m}(M_0)$, likewise for weighted versions (which we will leave to the reader to spell out). As a simple application of the principal symbol calculus, if $A\in\Difftbh^m(M_0)$ is elliptic, i.e., $|\sigmatbh(A)(z,\zeta)|\gtrsim\la\zeta\ra^m$ for all $z\in M_0$, $\zeta\in\Ttb^*_z M_0$, then we can construct an elliptic parametrix $B\in\Psitbh^{-m}(M_0)$ in the usual fashion (see, e.g., \cite[Theorems~4.29 and 4.53]{HintzMicro}) such that $B A=I+R$, with $R\in h^\infty\Psitbh^{-\infty}(M_0)$ residual. Therefore $\|R\|_{L^2\to L^2}\leq\frac12$ when $h$ is sufficiently small, and therefore $A$ is left-invertible then; arguing similarly for $A B=I+R'$ shows that $A$ is invertible. That is, for all sufficiently small $h>0$, $A_h\colon H_\tbop^s(M_0)\to H_\tbop^{s-m}(M_0)$ is invertible for all $s\in\R$, and the operator norm of its inverse, when acting between the semiclassical spaces $H_{\tbop,h}^{s-m}(M_0)$ and $H_{\tbop,h}^s(M_0)$, is uniformly bounded.

Given $A=(\Op_h(a_h))_{h\in(0,1)}+R\in\Psitbh^m(M_0)$, we define its semiclassical elliptic, characteristic, and operator wave front sets as subsets
\[
  \Elltbh(A),\ \Char_{\tbop,\semi}(A),\ \WF'_{\tbop,\semi}(A) \subset \ol{\Ttb^*}M_0,
\]
analogously to the non-semiclassical setting. These sets are not necessarily conic anymore.

\section{General theory of large parameter constraint damping}
\label{SC}

Let us work on an open Lorentzian manifold $(M,g)$, $\dim M=4$; we recall that our signature convention is $(-,+,+,+)$.

\begin{definition}[Modified symmetric gradient]
\label{DefCMod}
  For a timelike 1-form $\cd$ on $(M,g)$, a parameter $e\in\R$, and a semiclassical parameter $h>0$, we set $\cd^\sharp=g^{-1}(\cd,\cdot)$ and\footnote{This has the same structure as the modification used in \cite[\S{8}]{HintzVasyKdSStability}, except there we called $e$ what is $1-e$ in present notation. Our conventions here match those of \cite[\S{4}]{HintzPetersenVasyKdS} except for some factors of $2$; the considerations in \cite[\S{4}]{HintzPetersenVasyKdS} were, in turn, based on an earlier version of the present manuscript.}
  \begin{equation}
  \label{EqCMod}
    \delta_{g,\cd,e,h}^*\omega := \delta_g^*\omega + \cC_{g,\cd,e,h}\omega,\quad
    \cC_{g,\cd,e,h}\omega := h^{-1}\bigl(2\cd\otimes_s\omega - (1-e)g\,\iota_{\cd^\sharp}\omega\bigr).
  \end{equation}
\end{definition}

For a general zeroth order modification $\cC$ of $\delta_g^*$, we recall the notation
\[
  \Box_g^\cC=2\delta_g\sfG_g(\delta_g^*+\cC),\quad \sfG_g h=h-\frac12 g\tr_g h,
\]
for the resulting \emph{constraint propagation wave operator}. We then introduce
\begin{equation}
\label{EqCOp}
\begin{split}
  &\Box_{g,h}^{\cC_{g,\cd,e,h}} := h^2\Box_g^{\cC_{g,\cd,e,h}} = 2 h^2\delta_g\sfG_g\delta_g^* - i L_{g,\cd,e,h}, \\
  &\qquad L_{g,\cd,e,h} = 2 i h\delta_g\sfG_g\bigl(2\cd\otimes_s(\cdot)-(1-e)g\,\iota_{\cd^\sharp}\bigr) = 4 i h\delta_g(\cd\otimes_s(\cdot)) + 2 i h e\,\dd\,\iota_{\cd^\sharp}.
\end{split}
\end{equation}

\begin{lemma}[Principal symbol]
\label{LemmaCSymb}
  The semiclassical principal symbol $p_{g,\cd,e}$ of $\Box_{g,h}^{\cC_{g,\cd,e,h}}$ at a point $\zeta$ in phase space $T^*_z M$ over a point $z$ is the endomorphism of $T^*_z M$ given by
  \begin{equation}
  \label{EqCSymb}
    p_{g,\cd,e}(\zeta) = G(\zeta) - i\ell_{g,\cd,e}(\zeta),\quad
    G(\zeta)=\la\zeta,\zeta\ra,\ 
    \ell_{g,\cd,e}(\zeta)\omega=2\bigl(\la\cd,\zeta\ra\omega + \la\zeta,\omega\ra\cd - e\la\omega,\cd\ra\zeta\bigr),
  \end{equation}
  where $\la\cdot,\cdot\ra=g_z^{-1}(\cdot,\cdot)$.
\end{lemma}
\begin{proof}
  This follows from a simple calculation using
  \[
    \sigmah^1(h\dd)(\zeta)=i\zeta\cdot,\quad
    \sigmah^1(h\delta_g^*)(\zeta)=i\zeta\otimes_s(\cdot),\quad
    \sigmah^1(h\delta_g)(\zeta)=-i\iota_{\zeta^\sharp}.\qedhere
  \]
\end{proof}

Later on, we will work on a compactified spacetime and a suitable uniform (degenerate) version of the cotangent bundle (concretely, the 3b-cotangent bundle introduced in Definition~\ref{DefK3b}); the symbolic properties studied in this section will however be completely analogous.

Let us first work over a single fiber of $T^*M$, denoted $T^*:=T^*_z M$, and its dual $T:=(T^*)^*=T_z M$. We introduce the $g^{-1}$-orthogonal splitting
\begin{equation}
\label{EqCSplit}
  T^* = \R\cd \oplus \cd^\perp \cong \R\oplus\cd^\perp,
\end{equation}
where we identify $\lambda\cd$ with $\lambda\in\R$. Thus, $\cd=(1,0)$, and the inner product $g^{-1}$ is given by $g^{-1}=\diag(G(\cd),g^{-1}|_{\cd^\perp\times\cd^\perp})$. We moreover use the dual splitting
\[
  T \cong \R \oplus (\cd^\perp)^*,
\]
where we identify $(1,0)$ with the map $T^*\to\R$, $\cd=(1,0)\mapsto 1$, $\cd^\perp\ni\zeta'\mapsto 0$. Thus, $\cd^\sharp=(G(\cd),0)$. Writing $\zeta=(\zeta_0,\zeta')$ in the splitting~\eqref{EqCSplit}, thus $\zeta^\sharp=(G(\cd)\zeta_0,\zeta'{}^\sharp)$, we have
\begin{equation}
\label{EqCSplitMtx}
  G(\zeta) = G(\cd)\zeta_0^2 + G(\zeta'),\qquad
  \ell_{g,\cd,e}(\zeta) = 2\begin{pmatrix} (2-e)G(\cd)\zeta_0 & \zeta'{}^\sharp \\ -e G(\cd)\zeta' & G(\cd)\zeta_0 \end{pmatrix}.
\end{equation}

We first prove that there exists a fiber inner product on $T^*$ with respect to which $\ell_{g,\cd,e}(\zeta)$ is symmetric for all $\zeta$. We shall denote the adjoint of an operator $A$ with respect to an inner product $B$ by $A^{*B}$. Let us fix, as a reference, the positive definite fiber inner product
\[
  R := -2 G(\cd)^{-1}\cd^\sharp\otimes\cd^\sharp + g^{-1},
\]
which in the splitting~\eqref{EqCSplit} takes the form
\begin{equation}
\label{EqCInnerRef}
  R = \begin{pmatrix} -G(\cd) & 0 \\ 0 & g^{-1}|_{\cd^\perp\times\cd^\perp} \end{pmatrix}.
\end{equation}

\begin{lemma}[Fiber inner product]
\label{LemmaCInner}
  (Cf.\ \cite[Lemma~8.13]{HintzVasyKdSStability} and \cite[Proposition~4.5]{HintzPetersenVasyKdS}.) Let $e\neq 0$. Then the inner product
  \[
    B_{g,\cd,e}(\omega,\eta) := R(\omega,b_{g,\cd,e}\eta),\quad b_{g,\cd,e}=\Id-\frac{1-e}{G(\cd)}\cd\otimes\cd^\sharp,
  \]
  is the unique (up to scaling) inner product on $T^*$ relative to which $\ell_{g,\cd,e}(\zeta)$ is self-adjoint for all $\zeta\in T^*$. For $e>0$, it is positive definite.
\end{lemma}
\begin{proof}
  We want to find $b_{g,\cd,e}=b_{g,\cd,e}^{*R}\in\End(T^*)$ such that
  \[
    R(\omega,b_{g,\cd,e}\ell_{g,\cd,e}(\zeta)\eta)=R(\ell_{g,\cd,e}(\zeta)\omega,b_{g,\cd,e}\eta)
  \]
  for all $\zeta,\omega,\eta\in T^*$, or equivalently
  \[
    b_{g,\cd,e}\ell_{g,\cd,e}(\zeta) = (b_{g,\cd,e}\ell_{g,\cd,e}(\zeta))^{*R}.
  \]
  The general form of a self-adjoint (with respect to $R$) endomorphism $b_{g,\cd,e}$ is
  \[
    b_{g,\cd,e} = \begin{pmatrix} x & -y^\sharp \\ G(\cd)y & Z \end{pmatrix}
  \]
  where $y\in\cd^\perp$, and $Z\in\End(\cd^\perp)$ is self-adjoint with respect to $g^{-1}|_{\cd^\perp\times\cd^\perp}$. We then compute
  \[
    b_{g,\cd,e}\ell_{g,\cd,e}(\zeta)=
      2\begin{pmatrix}
        (2-e)G(\cd)\zeta_0 x+e G(\cd)\la\zeta',y\ra & x\zeta'{}^\sharp-G(\cd)\zeta_0 y^\sharp \\
        (2-e)G(\cd)^2\zeta_0 y-e G(\cd)Z(\zeta') & G(\cd)(y\otimes\zeta'{}^\sharp+\zeta_0 Z)
      \end{pmatrix}.
  \]
  The symmetry of the $(2,2)$ component with respect to $R$ requires $y\otimes\zeta'{}^\sharp=\zeta'\otimes y^\sharp$ for all $\zeta'$, which forces $y=0$. It then remains to analyze the conditions on $x,Z$ such that $-e G(\cd)Z(\zeta')=-x\zeta'G(\cd)$, i.e.\ $e Z(\zeta')=x\zeta'$. Up to scaling, we get $Z=\Id_{\cd^\perp}$ and $x=e$.
\end{proof}

The inner product $B_{g,\cd,e}$ depends smoothly on $G$ and $\cd$. In terms of $g^{-1}$, it is given by
\begin{equation}
\label{EqCInnerProd}
\begin{split}
  &B_{g,\cd,e}(\omega,\eta) = g^{-1}(\omega,\tilde b_{g,\cd,e}\eta),\quad
    g^{-1}(\omega,\eta) = B_{g,\cd,e}(\omega,\tilde b_{g,\cd,e}^{-1}\eta), \\
  &\qquad \tilde b_{g,\cd,e}=\Id-\frac{1+e}{G(\cd)}\cd\otimes\cd^\sharp,\quad
    \tilde b_{g,\cd,e}^{-1}=\Id-\frac{1+e}{e G(\cd)}\cd\otimes\cd^\sharp.
\end{split}
\end{equation}

\begin{rmk}[Why $e=0$ is excluded]
\label{RmkCNoe0}
  The operator $\ell_{g,\cd,0}(\zeta)$ in~\eqref{EqCSplitMtx} has eigenvalues $4 G(\cd)\zeta_0=4\la\zeta,\cd\ra$ and $2 G(\cd)\zeta_0=2\la\zeta,\cd\ra$, suggesting that $\ell_0$ is essentially the symbol of a first order transport operator along $2\cd^\sharp$. But for $\zeta=(0,\zeta')\neq 0$, $\ell_0(\zeta)$ has a $2\times 2$ Jordan block structure, and hence cannot be symmetric with respect to any positive definite inner product. Without a convenient fiber inner product as provided by Lemma~\ref{LemmaCInner}, it is not clear how to prove coercive estimates for the zero section propagation in~\S\ref{SE}. Nonetheless, we shall eventually choose $e>0$ to be very small, and thus one can still roughly think of $\ell_0$ as the symbol of a transport operator along $\cd^\sharp$.
\end{rmk}

\begin{lemma}[Timelike character of $\ell_{g,\cd,e}$]
\label{LemmaCTimelike}
  (Cf.\ \cite[Lemma~8.15]{HintzVasyKdSStability}.) Let $e\in(0,1)$, and let $\zeta\neq 0$. If $\zeta$ is in the same, resp.\ opposite causal cone as $\cd$, then $\ell_{g,\cd,e}(\zeta)$ is negative, resp.\ positive definite with respect to $B_{g,\cd,e}$.
\end{lemma}
\begin{proof}
  We use the expression~\eqref{EqCSplitMtx} of $\ell_{g,\cd,e}(\zeta)$, $\zeta=\zeta_0\cd+\zeta'$, $\zeta'\perp\cd$. It suffices to consider causal $\zeta\neq 0$ with $\zeta,\cd$ in different causal cones, so $\la\zeta,\cd\ra=\zeta_0 G(\cd)>0$, i.e., $\zeta_0<0$. For $\zeta'=0$, the matrix of $\ell_{g,\cd,e}(\zeta)$ in~\eqref{EqCSplitMtx} is block diagonal with positive (scalar) diagonal entries. For $\zeta'\neq 0$, we split $\cd^\perp=\R\zeta'\oplus\zeta'{}^\perp$, so
  \begin{equation}
  \label{EqCTimelikeEll}
    \frac12\ell_{g,\cd,e}(\zeta)=
      \begin{pmatrix}
        (2-e)G(\cd)\zeta_0 & G(\zeta') & 0 \\
        -e G(\cd) & G(\cd)\zeta_0 & 0 \\
        0 & 0 & G(\cd)\zeta_0
      \end{pmatrix}.
  \end{equation}
  Since by Lemma~\ref{LemmaCInner} all eigenvalues of $\frac12\ell_{g,\cd,e}(\zeta)$ are real, we only need to check the positivity of the trace and the determinant of the $2\times 2$ minor of the right-hand side of~\eqref{EqCTimelikeEll}. The trace is $(3-e)G(\cd)\zeta_0>0$. Using $G(\zeta)=\zeta_0^2 G(\cd)+G(\zeta')\leq 0$, so $-G(\zeta')\geq\zeta_0^2 G(\cd)$, its determinant is
  \[
    \bigl(-(2-e)G(\cd)\zeta_0^2-e G(\zeta')\bigr)\cdot(-G(\cd))\geq 2(1-e)G(\cd)^2\zeta_0^2 > 0.
  \]
  The proof is complete.
\end{proof}

\begin{cor}[A positive definite operator for nonzero $\zeta$]
\label{CorCPos}
  (Cf.\ \cite[Corollary~8.16]{HintzVasyKdSStability}.) Let $e\in(0,1)$. Then for all sufficiently large $C>0$, the operator $C\ell_{g,\cd,e}(\zeta)^2+G(\zeta)\Id$ is positive definite (with respect to $B_{g,\cd,e}$) for all $0\neq\zeta\in T^*$.
\end{cor}
\begin{proof}
  By homogeneity, it suffices to consider the compact set $S^*$ of $\zeta\in T^*$ with $R(\zeta)=1$. For spacelike such $\zeta$, we have $G(\zeta)>0$, so $C\ell_{g,\cd,e}(\zeta)^2+G(\zeta)\Id$ is positive definite for all $C\geq 0$. For the closed subset $K\subset S^*$ consisting of causal vectors on the other hand, Lemma~\ref{LemmaCTimelike} implies that $\inf_{\zeta\in K}\ell_{g,\cd,e}(\zeta)^2=:c>0$, while $\sup_{\zeta\in K}(-G(\zeta))=:c'<\infty$. Hence, any $C>c'/c$ works.
\end{proof}

\begin{lemma}[Ellipticity of $p_{g,\cd,e}(\zeta)$ at nonzero $\zeta$]
\label{LemmaCInv}
  (Cf.\ \cite[Lemma~8.5]{HintzVasyKdSStability}.) Let $e\in(0,1)$. Then the linear map $p_{g,\cd,e}(\zeta)\in\End(T^*)$ is invertible for all $\zeta\neq 0$.
\end{lemma}

This sharpens \cite[Proposition~4.8]{HintzPetersenVasyKdS}.

\begin{proof}[Proof of Lemma~\usref{LemmaCInv}]
  If $G(\zeta)\neq 0$, then $G(\zeta)^{-1}p_{g,\cd,e}(\zeta)=I-i G(\zeta)^{-1}\ell_{g,\cd,e}(\zeta)$ is the sum of the identity map $I$ and a skew-adjoint endomorphism, and hence invertible. If $G(\zeta)=0$, i.e., $\zeta$ is causal, then $i p_{g,\cd,e}(\zeta)=\ell_{g,\cd,e}(\zeta)$ is positive or negative definite by Lemma~\ref{LemmaCTimelike}, and therefore invertible.
\end{proof}

We end our fiberwise considerations with the following generalization of Lemma~\ref{LemmaCTimelike}:

\begin{lemma}[Timelike character of $\ell_{g,\cd,e}(\zeta)$ for certain spacelike $\zeta$]
\label{LemmaCTimelikeSp}
  Let $e\in(0,1)$, and let $\zeta\neq 0$. Then $\ell_{g,\cd,e}(\zeta)$ is positive, resp.\ negative definite if and only if
  \begin{equation}
  \label{EqCTimelikeSp}
    {\pm}\la\cd,\zeta\ra>0,\qquad
    G(\zeta) < -\frac{2(1-e)}{e G(\cd)}\la\cd,\zeta\ra^2.
  \end{equation}
\end{lemma}

For fixed $\zeta$ with $\la\cd,\zeta\ra\neq 0$, the second condition in~\eqref{EqCTimelikeSp} is automatically satisfied when $e>0$ is sufficiently small. This re-affirms the interpretation of $\ell_{g,\cd,e}$ in the limit $e\searrow 0$ as the symbol of transport along $\cd^\sharp$.

\begin{proof}[Proof of Lemma~\usref{LemmaCTimelikeSp}]
  For causal $\zeta$, we have $G(\zeta)\leq 0$, while $-2/G(\cd)>0$, so the conditions~\eqref{EqCTimelikeSp} are satisfied, and the positive/negative definiteness of $\ell_{g,\cd,e}(\zeta)$ was shown in Lemma~\ref{LemmaCTimelike}. For spacelike $\zeta=(\zeta_0,\zeta')$ on the other hand, split $\cd^\perp=\R\zeta'\oplus\zeta'{}^\perp$ and consider again~\eqref{EqCTimelikeEll}. The positivity of the $(3,3)$ entry is equivalent to $\la\cd,\zeta\ra=G(\cd)\zeta_0>0$. Under this condition, the trace of the $2\times 2$ minor is positive, too, and its determinant is a positive multiple of $-(2-e)G(\cd)\zeta_0^2-e G(\zeta')$ which we wish to be positive. Upon inserting $G(\cd)\zeta_0^2=G(\cd)^{-1}\la\cd,\zeta\ra^2$ and $G(\zeta')=G(\zeta)-G(\cd)^{-1}\la\cd,\zeta\ra^2$, one obtains the condition~\eqref{EqCTimelikeSp}.
\end{proof}

Near critical points of $\cd^\sharp$, where the transport along $\cd^\sharp$ degenerates, we need information about the skew-adjoint part of the operator $L_{g,\cd,e,h}$ in~\eqref{EqCOp}. This involves first derivatives of $\cd$:

\begin{lemma}[Subprincipal symbol]
\label{LemmaCSubpr}
  Let $e\neq 0$. Define adjoints using the metric volume density $|\dd g|$ and the fiber inner product $B_{g,\cd,e}$. Then the skew-adjoint part $S_{g,\cd,e}=\frac{1}{2 i h}(L_{g,\cd,e,h}-L_{g,\cd,e,h}^*)$ is the bundle map with matrix elements
  \begin{align*}
    (S_{g,\cd,e})_i{}^k &= -\bigl(\cd^k{}_{;i}+\cd_i{}^{;k}+\cd_j{}^{;j}\delta_i^k\bigr) \\
      &\quad + \frac{1+e}{G(\cd)}\Bigl[-\Bigl(\cd_j{}^{;j}+\frac{2(1-e)}{e G(\cd)}\cd_j\cd_m\cd^{j;m}\Bigr)\cd_i\cd^k + 2 e^{-1}\cd_i\cd_j\cd^{k;j} + \cd^j(\cd_{j;i}-\cd_{i;j})\cd^k \Bigr].
  \end{align*}
\end{lemma}
\begin{proof}
  Since $\ell_{g,\cd,e}=\sigmah^1(L_{g,\cd,e,h})$ is self-adjoint, the operator $L_{g,\cd,e,h}-L_{g,\cd,e,h}^*$, which a priori is of first order, is in fact of order zero. In its calculation, we can thus take $\omega$ to be a smooth 1-form which is covariantly constant at a point $p\in M$ and evaluate $(L_{g,\cd,e,h}-L_{g,\cd,e,h}^*)\omega$ at $p$, which means dropping all terms involving differentiation of $\omega$. Making the fiber inner products explicit, we have, in the notation~\eqref{EqCInnerProd},
  \[
    L_{g,\cd,e,h}^{*B_{g,\cd,e}} = \tilde b_{g,\cd,e}^{-1}L_{g,\cd,e,h}^{*g^{-1}}\tilde b_{g,\cd,e},\quad
    (2 i h)^{-1}L_{g,\cd,e,h}^{*g^{-1}} = -e\cd\delta_g-2\iota_{\cd^\sharp}\delta_g^*,
  \]
  as follows easily from~\eqref{EqCOp}. At $p$, we have
  \begin{equation}
  \label{EqCSubpr1}
    (2 i h)^{-1}(L_{g,\cd,e,h}\omega)_i = (e\cd^k{}_{;i}-\cd_i{}^{;k})\omega_k - \cd_j{}^{;j}\omega_i.
  \end{equation}
  The action of $(2 i h)^{-1}L_{g,\cd,e,h}^{*g^{-1}}$ on $\tilde b_{g,\cd,e}\omega$ with $(\tilde b_{g,\cd,e}\omega)_j=\omega_j-\frac{1+e}{G(\cd)}\cd_j\cd^k\omega_k$ has $i$-th component given at $p$ by
  \[
    -\frac{1+e}{G(\cd)}\omega_k\Bigl[\Bigl(e\cd_j{}^{;j}+\frac{2(1-e)}{G(\cd)}\cd_j\cd_m\cd^{j;m}\Bigr)\cd_i\cd^k - (1-e)\cd_i\cd_j\cd^{k;j} + \cd^j(\cd_{j;i}-\cd_{i;j})\cd^k - G(\cd)\cd^k{}_{;i}\Bigr].
  \]
  Applying $\tilde b_{g,\cd,e}^{-1}$ to this gives
  \[
    -\frac{1+e}{G(\cd)}\omega_k\Bigl[ -\Bigl(\cd_j{}^{;j}+\frac{2(1-e)}{e G(\cd)}\cd_j\cd_m\cd^{j;m}\Bigr)\cd_i\cd^k + 2 e^{-1}\cd_i\cd_j\cd^{k;j} + \cd^j(\cd_{j;i}-\cd_{i;j})\cd^k-G(\cd)\cd^k{}_{;i}\Bigr].
  \]
  Subtracting this from~\eqref{EqCSubpr1} yields the claim.
\end{proof}

\section{Symbol calculations on Minkowski spacetime}
\label{SM}

As already briefly discussed in~\S\ref{SI}, the choice of constraint damping 1-form $\cd$ on the Kerr spacetime $(M,g_{\bhm,a})$ is rather delicate. First of all, in order for the operator $\Box_{g_{\bhm,a}}^{\cC_{g_{\bhm,a},\cd,e,h}}$ defined in~\eqref{EqCOp} to have similar asymptotic symmetries as $\Box_{g_{\bhm,a}}$, we need $\cd$ to be stationary and moreover of size $\rho_\sface\sim r^{-1}$ as a scattering 1-form. (We will later on choose $\cd\in\rho_\sface\CI(M;\cT^*)$, where we recall the notation~\eqref{EqKscT}.) Indeed, this ensures that $\Box_g^\cC$ is approximately homogeneous of degree $-2$ with respect to scaling in any exterior cone $r>\delta t$, $\delta>0$. (If the coefficients of $\cd$ with respect to $\dd z^\mu$ are $t$-independent and exactly homogeneous of degree $-1$ with respect to scaling in the spatial variables, then the Minkowskian operator $\Box_{\ubar g}^{\cC_{\ubar g,\cd,e,h}}$ is homogeneous of degree $-2$ with respect to spacetime scaling.)

Next, working with the rescaled dual vector field $V=-r^2\cd^\sharp$ for clarity (which is a 3b-vector field on $M$), we need $V$ to be everywhere future\footnote{The \emph{future} timelike nature ensures that $-i\ell_{g,\cd,e}$ in~\eqref{EqCSymb} acts as complex absorption when propagating microlocal estimates in the causal \emph{future} direction, see Proposition~\ref{PropEEInfty}.}  timelike so that the results of~\S\ref{SC} are applicable; this in particular requires $V$ to be pointing into the black hole at the event horizon, so $V r<0$ for $r\leq r_+$. So far, one might be tempted to take $V\approx r(\pa_t-v\pa_r)$ with $v\in(0,1)$. However, it turns out that for $V$ which are inward pointing also near $\pa\cK^+=\cK^+\cap\sface$ (i.e., for large $r$ but with $\frac{r}{t}$ small), like $r(\pa_t-v\pa_r)$, the behavior of $\Box_g^\cC$ at zero energy (i.e., acting on stationary 1-forms) is inconsistent with the desideratum expressed in Theorem~\ref{ThmIRough}\eqref{ItIRough0}; we give a numerical reason in Remark~\ref{RmkM0WrongV}, and a conceptual one in Remark~\ref{RmkEBOut}.

Instead, $V$ needs to be \emph{outward} pointing near $\pa\cK^+$, and we shall take $V=r(\pa_t+v\pa_r)$ with $v\in(0,1)$ there. Viewed as a 3b-vector field on $M$, this has a saddle point at $\pa\cK^+$; we discuss this, and compute the relevant principal and subprincipal symbols of the zero section transport operator $L_{g,\cd,e,h}$ in~\S\ref{SsMI}. The requirement that $\cd$, thus $V$, be stationary, forces $V=r(\pa_t+v\pa_r)$ globally along $\sface$, which in particular means that $V$ has a critical point at $\sface$ where $\frac{r}{t}=v$; the relevant (sub)principal symbol calculation at this critical set (which is a sink) is also given in~\S\ref{SsMI}.

Returning to the behavior of $V$ near $\cK^+$, the incoming nature near the event horizon (i.e., $r\lesssim\bhm$) and the outgoing nature near $\pa\cK^+$ (i.e., $r\gg\bhm$) force $V$ to have saddle points in the interior of $\cK^+$. We compute the relevant (sub)principal symbols governing propagation at the zero section near such a saddle point in~\S\ref{SsMS}. All calculations here will be for the Minkowski metric, which near $\sface$ (i.e., for large $r$) is a perturbation of the Kerr metric by Lemma~\ref{LemmaKMink}. (We glue the various pieces together on a subextremal Kerr spacetime in~\S\ref{SsEC} by perturbing off the Minkowski case.)

Finally, in~\S\ref{SsM0}, we perform the key analysis for the proof of Theorem~\ref{ThmIRough}\eqref{ItIRough0}, namely the calculations/estimates for the indicial roots at infinity of the zero energy operator of $\Box_{g_{\bhm,a}}^{\cC_{g_{\bhm,a}},\cd,e,h}$ (cf.\ Definition~\ref{DefK3bSpec} and Lemma~\ref{LemmaK3bSpec0}).

Away from $r=0$ inside of $M_0$, we use the bundle splitting of $\cT^*=\upbeta^*\Tsc^*\ol{\R^4}$ given by smooth extension of the splitting
\begin{equation}
\label{EqMSplit}
  T^*\R^4 = \la\dd t\ra \oplus \la\dd r\ra \oplus r T^*\Sph^2
\end{equation}
from $\R^4\setminus\{r=0\}$ to $M_0\setminus\cl_{M_0}\{r=0\}$. This means that we identify $(a,b,\eta)$, where $a,b\in\R$ and $\eta\in T^*\Sph^2$, with the 1-form $a\,\dd t+b\,\dd r+r\eta$. \emph{Throughout this section, we work relative to the Minkowski metric $\ubar g$}
\[
  \ubar g = -\dd t^2 + \dd r^2 + r^2\slg,\quad\slg:=\dd\theta^2+\sin^2\theta\,\dd\phi^2,
\]
which we henceforth drop from the notation.

\subsection{Subprincipal symbol at critical sets at \texorpdfstring{$\sface$}{the side face}}
\label{SsMI}

We consider the past timelike 1-form
\begin{equation}
\label{EqMI1form}
  \cd_v = r^{-1}(\dd t-v\,\dd r),\quad v\in(0,1).
\end{equation}
Its rescaled dual vector field $V_v:=-r^2\cd_v^\sharp=r(\pa_t+v\pa_r)$ is a 3b-vector field on $M$ which on $\sface\subset\pa M$ has two critical sets,
\begin{equation}
\label{EqMICrit}
  \pa\cK^+ = \cK^+\cap\sface,\quad
  \cR_v := \Bigl\{ \frac{r}{t}=v \Bigr\} \cap \sface.
\end{equation}
Indeed, in the coordinates $\rho_\cK:=\frac{r}{t}$, $\rho_\sface:=\frac{1}{r}$ in $t,r>0$ from~\eqref{EqKbdf}, we have $V_v=-v\rho_\sface\pa_{\rho_\sface}+(v-\rho_\cK)\rho_\cK\pa_{\rho_\cK}$, and hence $\pa\cK^+=\{\rho_\cK=\rho_\sface=0\}$ is a saddle point (with stable manifold $\cK^+$ and unstable manifold $\sface$), while $\cR_v=\{\rho_\cK-v=0,\ \rho_\sface=0\}$ is a sink. See Figure~\ref{FigMIVv}.

\begin{figure}[!ht]
\centering
\includegraphics{FigMIVv}
\caption{The vector field $V_v$, shown here in a coordinate patch $[0,1)_{\rho_\sface}\times[0,1)_{\rho_\cK}\times\Sph^2$ without the $\Sph^2$-factor. The two critical sets $\pa\cK^+$ and $\cR_v$ are marked as thick dots.}
\label{FigMIVv}
\end{figure}

As a rough heuristic, the operator $L_{\cd_v,e,h}:=L_{\ubar g,\cd_v,e,h}$, for $0<e\ll 1$, can be regarded as the operator $\ell_{\cd_v,e}(h D)+i h S_{\cd_v,e}=i h(-\ell_{\cd_v,e}(\nabla)+S_{\cd_v,e})$ where
\[
  \ell_{\cd_v,e}:=\ell_{\ubar g,\cd_v,e},\quad
  S_{\cd_v,e}:=S_{\ubar g,\cd_v,e}
\]
are defined by Lemmas~\ref{LemmaCSymb} and \ref{LemmaCSubpr}, and $-\ell_{\cd_v,e}(\nabla)\approx F\nabla_{-\cd_v^\sharp}$, with $F=2,4$ depending on the part of the bundle it acts on, cf.\ Remark~\ref{RmkCNoe0}. Thus, when attempting to control solutions of $L_{\cd_v,e,h}\omega=0$ by following the flow of $-\cd_v^\sharp$ (thus, propagating in the future direction) near $\pa\cK^+$ on a 3b-Sobolev space with weights $w:=\rho_\sface^{\beta_\sface}\rho_\cK^{\beta_\cK}$, the crucial quantity to compute is
\[
  -\ell_{\cd_v,e}(\dd w)+S_{\cd_v,e}w = w\Bigl[-\ell_{\cd_v,e}\Bigl(\frac{\dd w}{w}\Bigr) + S_{\cd_v,e}\Bigr],
\]
or rather a suitable rescaling thereof in order to balance the weight $r^{-1}$ in the definition of $\cd_v$ in~\eqref{EqMI1form}. Indeed, if this is positive definite, one expects to be able to propagate in the future direction. (Cf.\ the simple-minded example $i h(t\pa_t+S)$ which is of this form if we used $\ell(\nabla)=-t\pa_t$: solutions can be controlled as $t\to\infty$ by a power $w=t^\beta$ if and only if $-\ell(\frac{\dd w}{w})+S=\beta+S>0$.) (Similar considerations apply near $\cR_v$ where, however, the only available weight is $\rho_\sface^{\beta_\sface}$; this will be discussed in Lemma~\ref{LemmaMISink} below.)

\begin{lemma}[Total symbol at $\pa\cK^+$]
\label{LemmaMIbdy}
  Let $\beta_\cK,\beta_\sface\in\R$, and define $w=\rho_\sface^{\beta_\sface}\rho_\cK^{\beta_\cK}$ where $\rho_\cK=\frac{r}{t}$, $\rho_\sface=\frac{1}{r}$. Then
  \begin{equation}
  \label{EqMIbdyTot}
    T_{\cd_v,e} := r^2\Bigl(-\ell_{\cd_v,e}\Bigl(\frac{\dd w}{w}\Bigr) + S_{\cd_v,e}\Bigr)
  \end{equation}
  depends smoothly on $e\in\R$. The limiting bundle endomorphism $T_{\cd_v,0}$ has positive eigenvalues at $\pa\cK^+$ if and only if $\beta_\cK>\beta_\sface-\frac12$.
\end{lemma}
\begin{proof}
  Write $\cd:=\cd_v$. Since
  \[
    \frac{\dd\rho_\sface}{\rho_\sface} = -\frac{\dd r}{r},\quad
    \frac{\dd\rho_\cK}{\rho_\cK} = \frac{\dd r}{r} - \frac{\dd t}{t}.
  \]
  we have
  \[
    \frac{\dd w}{w} = -(\beta_\sface-\beta_\cK)\frac{\dd r}{r} - \rho_\cK\beta_\cK\frac{\dd t}{r}.
  \]
  Direct calculation using the expression~\eqref{EqCSymb} in the bundle splitting~\eqref{EqMSplit} thus gives
  \begin{equation}
  \label{EqMIbdyEll}
    \frac12 r^2\ell_{\cd,e}\Bigl(\frac{\dd w}{w}\Bigr) = (\beta_\sface-\beta_\cK)\begin{pmatrix} v & -1 & 0 \\ -e & (2-e)v & 0 \\ 0 & 0 & v \end{pmatrix} + \beta_\cK\rho_\cK\begin{pmatrix} 2-e & -e v & 0 \\ -v & 1 & 0 \\ 0 & 0 & 1 \end{pmatrix}.
  \end{equation}
  We base our calculation of $r^2 S_{\cd,e}$ on the calculation of $\nabla\cd$ in the coordinates $(t,r,\theta,\phi)$: we have $\cd_{\mu;\nu}=0$ for all $\mu,\nu$ except for $\cd_{0;1}=-r^{-2}$, $\cd_{1;1}=v r^{-2}$, and $\cd_{2;2}=-v$, $\cd_{3;3}=-v\sin^2\theta$. One finds
  \begin{equation}
  \label{EqMIbdyS}
    \frac12 r^2 S_{\cd,e}=
      \begin{pmatrix}
        \frac{v(3-v^2+2 e)}{2(1-v^2)} & \frac{1+v^2+2 e v^2}{2(1-v^2)} & 0 \\
        \frac{-2 v^2+e(1-3 v^2)}{2(1-v^2)} & \frac{-2 v^3+e v(1-3 v^2)}{2(1-v^2)} & 0 \\
        0 & 0 & \frac{3 v}{2}
      \end{pmatrix}.
  \end{equation}
  This is smooth in $e$. By direct calculation,\footnote{This calculation is straightforward upon switching to the frame of $T^*\R^4$ adapted to the choice of $\cd_v$, namely $\dd t-v\,\dd r$, $-v\,\dd t+\dd r$, and $r T^*\Sph^2$. In this basis, $\frac12 r^2(-\ell_{\cd,0}(\frac{\dd w}{w})+S_{\cd,0})$ is upper triangular.} the eigenvalues of $\frac12 r^2(-\ell_{\cd,0}(\frac{\dd w}{w})+S_{\cd,0})$ at $\rho_\cK=0$ are
  \begin{equation}
  \label{EqMIbdyEval}
    v\Bigl(\frac12+\beta_\cK-\beta_\sface\Bigr),\quad
    2 v\Bigl(\frac12+\beta_\cK-\beta_\sface\Bigr),\quad
    v\Bigl(\frac32+\beta_\cK-\beta_\sface\Bigr)\ \ (2\times).
  \end{equation}
  These are all positive provided $\beta_\cK>\beta_\sface-\frac12$.
\end{proof}

\begin{lemma}[Total symbol at $\cR_v$]
\label{LemmaMISink}
  Let $\beta_\sface\in\R$, and define $w=\rho_\sface^{\beta_\sface}$. Define $T_{\cd_v,e}$ as in~\eqref{EqMIbdyTot}, which thus smoothly depends on $e$ down to $e=0$. The eigenvalues of $T_{\cd_v,0}|_{\cR_v}$ are positive if and only if $\beta_\sface<\frac12$.
\end{lemma}
\begin{proof}
  At $\cR_v$ we have $\rho_\cK=v$. The claim now follows by direct calculation using the expressions~\eqref{EqMIbdyEll}--\eqref{EqMIbdyS}. Indeed, one finds in the bundle splitting $\la \dd t-v\,\dd r\ra\oplus\la -v\,\dd t+\dd r\ra\oplus r T^*\Sph^2$,
  \[
    \frac12 T_{\cd_v,0}|_{\cR_v} =
    \begin{pmatrix}
      2 v(\frac12-\beta_\sface) & \frac12+\beta_\sface & 0 \\
      0 & v(\frac12-\beta_\sface) & 0 \\
      0 & 0 & v(\frac32-\beta_\sface)
    \end{pmatrix}.
  \]
  All eigenvalues are positive if and only if $\beta_\sface<\frac12$.
\end{proof}

\subsection{Subprincipal symbol near saddle points at coordinate spheres}
\label{SsMS}

We shall now consider the 1-form
\begin{equation}
\label{EqMS1form}
  \cd_{r_0,v} = r_0^{-2}\bigl(r_0\,\dd t - v(r-r_0)\,\dd r\bigr),\quad r_0>0,\ v\in(0,1),
\end{equation}
which will be the model for our final choice of $\cd$ in Theorem~\ref{ThmIRough} in an $\cO(r_0)$-neighborhood of $r=r_0$. Note that $\cd$ is past timelike for $r\in(0,r_0(1+\frac{1}{v}))$. We consider $\cd_{r_0,v}$ near the interior of $\cK^+$ (or more generally near the interior of the front face of $M_0$ arising from the blow-up of the north pole $\fk^+$ in Definition~\ref{DefKMfd}). In this region, we can use the coordinates $T=t^{-1}\geq 0$, $r>0$, $\omega\in\Sph^2$. The vector field
\[
  -\cd_{r_0,v}^\sharp=r_0^{-2}\bigl(r_0\pa_t+v(r-r_0)\pa_r\bigr)=r_0^{-2}\bigl(-r_0 T^2\pa_T+v(r-r_0)\pa_r\bigr)
\]
vanishes at $T=0$, $r=r_0$, and indeed has a (degenerate) saddle point there. Due to the degeneracy, polynomial weights $T^{\alpha_\cK}$ cannot produce positivity of expressions of the type~\eqref{EqMIbdyTot}. (Note that $\frac{\dd T}{T}=-\frac{\dd t}{t}$ vanishes, as a scattering 1-form, at $T=t^{-1}=0$, due to the factor of $t^{-1}$. By contrast, the subprincipal symbol does not come with such a vanishing factor.) Thus, it is only the subprincipal symbol $S_{\cd_{r_0,v},e}=S_{\ubar g,\cd_{r_0,v},e}$ that can give positivity near $T=0$, $r=r_0$; we compute this in Lemma~\ref{LemmaMSSubpr} below, and we also quantify the size of the neighborhood in which this positivity holds: this will be of size $\frac{6}{5}r_0\sqrt{e}/v$. As soon as one is at a definite distance away from $r=r_0$ (the required minimal distance being $r_0\sqrt{e}/v$, as shown below), we will be able to exploit again the approximate transport character of $\ell_{\cd_{r_0,v},e}$; see Lemma~\ref{LemmaMSRad} below for the quantitative details. Thus, there is a very delicate balance between the region of positivity of $S_{\cd_{r_0,v},e}$ ($\frac{6}{5}r_0\sqrt{e}/v$-close to $r=r_0$) and the region where $\ell_{\cd_{r_0,v},e}$ behaves like non-degenerate transport ($r_0\sqrt{e}/v$-far from $r=r_0$), and \emph{it is crucial for our purposes that the two regions overlap, i.e., that $\frac{6}{5}>1$}.

We provide further motivation using the following heuristic. In $1+1$ dimensions with coordinates $(t,r)$ and metric $-\dd t^2+\dd r^2$, and for $\cd=\dd t$, $\zeta=\zeta_0\,\dd t+\zeta_1\,\dd r$, one computes
\[
  \ell_{\cd,e}(\zeta) = 2\begin{pmatrix} -(2-e)\zeta_0 & \zeta_1 \\ e\zeta_1 & -\zeta_0 \end{pmatrix}
\]
using~\eqref{EqCSplitMtx}. The eigenvalues of $\ell_{\cd,e}(\zeta)$ are $\lambda_\pm(\zeta)=-(3-e)\zeta_0\pm\sqrt{(1-e)^2\zeta_0^2+4 e\zeta_1^2}$. We interpret this, for $0<e\ll 1$, as the symbol of an ``approximate (collection of) vector field(s)'' (note that for $e=0$ the eigenvalues are $-4\zeta_0$ and $-2\zeta_0$, cf.\ Remark~\ref{RmkCNoe0}) and compute
\[
  -H_{\lambda_\pm} = \Bigl(3-e\mp\frac{(1-e)^2\zeta_0}{\sqrt{(1-e)^2\zeta_0^2+4 e\zeta_1^2}}\Bigr)\pa_t \mp \frac{4 e\zeta_1}{\sqrt{(1-e)^2\zeta_0^2+4 e\zeta_1^2}}\pa_r.
\]
The coefficient of $\pa_t$ is bounded from below by $(3-e)-(1-e)=2$, and the coefficient of $\pa_r$ is bounded from above in absolute value by $\sqrt{4 e}=2\sqrt{e}$. Thus, $-\frac12 H_{\lambda_\pm}$ is ``between'' $\pa_t-\sqrt{e}\pa_r$ and $\pa_t+\sqrt{e}\pa_r$, thus spreading from $-\cd^\sharp=\pa_t$ by at most $\sqrt{e}$ in the direction orthogonal to $-\cd^\sharp$.

Applying this heuristic to $r_0^2\cd_{r_0,v}$---with $v(r-r_0)\pa_t+r_0\pa_r$ being orthogonal to $r_0^2\cd_{r_0,v}^\sharp$ and of the same squared length up to a sign---suggests that the operator $\ell_{\cd_{r_0,v},e}$ will be degenerate over $(\cK^+)^\circ$ at most at the zeros of $v(r-r_0)+s\sqrt{e}r_0$ for $s\in[-1,1]$, i.e., in $r\in[r_0(1-\sqrt{e}/v),r_0(1+\sqrt{e}/v)]$. The following computation can be viewed as a confirmation of this heuristic, or indeed a slight strengthening (since already at $s=\pm 1$ we show the nondegeneracy of $\ell_{\cd_{r_0},v,e}$):

\begin{lemma}[$\ell_{\cd_{r_0,v},e}(\dd r)$ is definite away from $r=r_0$]
\label{LemmaMSRad}
  Let $e\in(0,1)$. Then $\ell_{\cd_{r_0,v},e}(\pm\dd r)$ is negative definite when $\pm(r-r_0)\geq r_0\sqrt{e}/v$ and $r\in[\frac12 r_0,2 r_0]$. Moreover, for such $r$, the eigenvalues $\lambda$ of $\ell_{\cd_{r_0,v},e}(\pm\dd r)$ satisfy
  \begin{equation}
  \label{EqMSRadQuant}
    \lambda\leq-\lambda_0 r_0^{-1}
  \end{equation}
  for some constant $\lambda_0>0$ which only depends on $e,v$, but not on $r_0,r$.
\end{lemma}
\begin{proof}
  For the stated range of $r$, the 1-form $\cd_{r_0,v}$ is timelike. We use Lemma~\ref{LemmaCTimelikeSp} with $\zeta=\pm\dd r$. We compute $\la\cd_{r_0,v},\zeta\ra=\mp r_0^{-2}v(r-r_0)<0$ for $\pm(r-r_0)>0$; moreover, we have $G(\cd_{r_0,v})=r_0^{-4}(-r_0^2+v^2(r-r_0)^2)$. If $|r-r_0|\geq r_0\sqrt{e}/v$, then $w:=v^2(r-r_0)^2/r_0^2\geq e$, and thus
  \[
    -\frac{2(1-e)}{e G(\cd_{r_0,v})}\la\cd_{r_0,v},\zeta\ra^2=\frac{2(1-e)v^2(r-r_0)^2}{e(r_0^2-v^2(r-r_0)^2)} \geq \frac{2(1-e)w}{e(1-w)} \geq 2 > 1 = G(\zeta).
  \]
  
  The quantitative bound~\eqref{EqMSRadQuant} for $r\in[\frac12 r_0,2 r_0]\setminus[r_0(1-\sqrt{e}/v),r_0(1+\sqrt{e}/v)]$ follows by scaling from the case $r_0=1$, in which case the bound is automatic by compactness. Indeed, under pullback by the map $(t,r,r_0)\mapsto(q t,q r,q r_0)$ for $q\geq 1$, the Minkowski dual metric $\ubar g^{-1}$ (which is independent of $r_0$ of course) is homogeneous of degree $-2$, whereas $\cd_{r_0,v}$ is homogeneous of degree $0$, and $\pm\dd r$ is homogeneous of degree $1$. Therefore, overall, the eigenvalues of $\ell_{\cd_{r_0,v},e}$ are homogeneous of degree $-1$, which is the content of~\eqref{EqMSRadQuant}.
\end{proof}

For $r$ with $|r-r_0|\leq r_0\sqrt{e}/v$, one should regard $\ell_{\cd_{r_0,v}}$ as degenerate as far as positive commutators are concerned (see~\S\ref{SE}); thus, in this region we compute the subprincipal symbol $S_{\cd_{r_0,v},e}$ using Lemma~\ref{LemmaCSubpr}:

\begin{lemma}[Subprincipal symbol]
\label{LemmaMSSubpr}
  In the bundle splitting~\eqref{EqMSplit}, we have, at $r=r_0(1+x\sqrt{e}/v)$ with $x\in\R$,
  \[
    S_{\cd_{r_0,v},e} = r_0^{-2}\begin{pmatrix} a_{0 0} & a_{0 1} & 0 \\ a_{1 0} & a_{1 1} & 0 \\ 0 & 0 & a_{2 2} \end{pmatrix},
  \]
  where, for $\Delta:=(1+x\sqrt{e}/v)(1-e x^2)^2$, we have
  \begin{align*}
    a_{0 0} &= \Delta^{-1}\Bigl( v\bigl[ 2(1-x^2)+e(1-3 x^2)+e^2 x^2(x^2+1) \bigr] \\
      &\quad\hspace{2em} + x\sqrt{e}\bigl[2(3-x^2)+e(1-3 x^2)(3-e x^2)\bigr] \Bigr), \\
    a_{0 1} &= \Delta^{-1}\frac{(1+e)x}{\sqrt e}\Bigl( v\bigl[ -2+e+e^2 x^2 \bigr] + x\sqrt{e}\bigl[-2+3 e-e^2 x^2\bigr] \Bigr), \\
    a_{1 0} &= \Delta^{-1}(1+e)x\sqrt{e}\Bigl( v\bigl[-1+2 x^2-e x^2\bigr] + x\sqrt{e}\bigl[-3+2 x^2+e x^2\bigr] \Bigr), \\
    a_{1 1} &= \Delta^{-1}\Bigl( v \bigl[ 3+2 x^2 - 5 e x^2 - e^2 x^2(1-(2-e)x^2)\bigr] \\
      &\quad\hspace{2em} + x\sqrt{e}\bigl[5+2 x^2 - 11 e x^2 + e^2(-3 x^2+(6+e)x^4)\bigr] \Bigr), \\
    a_{2 2} &= v\Bigl(5-\frac{4 v}{v+x\sqrt{e}}\Bigr).
  \end{align*}
  There exists $e_0>0$ depending only on $v$ such that for all $e\in(0,e_0)$, the operator $S_{\cd_{r_0,v},e}$ is positive definite (with respect to the inner product $B_{\cd_{r_0,v},e}$ from Lemma~\usref{LemmaCInner}) for all $x\in[-\frac65,\frac65]$, i.e., for all $r$ with $|r-r_0|\leq\frac{6}{5}r_0\sqrt{e}/v$.
\end{lemma}

Numerical experimentation shows that the maximal possible value of $e_0$ is very small (namely, $e_0<c v^2$ for $c\approx 0.000152$).

\begin{proof}[Proof of Lemma~\usref{LemmaMSSubpr}]
  Write $\cd:=r_0^2\cd_{r_0,v}$; we shall compute $S_{\cd,e}=r_0^2 S_{\cd_{r_0,v},e}$. In the coordinates $(t,r,\theta,\phi)$, the Christoffel symbols $\Gamma_{\mu\nu}^\lambda$ of Minkowski space with $\mu$ equal to $0$ or $1$ vanish except for $\Gamma^2_{1 2}=r^{-1}$, $\Gamma^3_{1 3}=r^{-1}$. Therefore, the only nonzero components of $\cd_{\mu;\nu}$ are
  \begin{align*}
    \cd_{1;1}&=-v, \\
    \cd_{2;2}&=-v r(r-r_0)=-r_0^2 x\sqrt{e}\Bigl(1+\frac{x\sqrt{e}}{v}\Bigr), \\
    \cd_{3;3}&=-v r(r-r_0)\sin^2\theta = -r_0^2 x\sqrt{e}\Bigl(1+\frac{x\sqrt{e}}{v}\Bigr)\sin^2\theta.
  \end{align*}
  This implies that $\cd_{j;i}-\cd_{i;j}=0$ for all $i,j$. Moreover, using $\cd^0=-\cd_0=-r_0$ and $\cd^1=\cd_1=-v(r-r_0)$, we compute
  \begin{align*}
    \cd_j{}^{;j} &= \cd_{1;1}+r^{-2}\cd_{2;2}+r^{-2}\sin^{-2}\theta\,\cd_{3;3}=-v\Bigl(3-\frac{2 r_0}{r}\Bigr) = -v\Bigl(3-\frac{2 v}{v+x\sqrt{e}}\Bigr), \\
    \cd^j\cd^m\cd_{j;m} &= (\cd^1)^2\cd_{1;1} = -v^3(r-r_0)^2 = -v r_0^2 x^2 e,
  \end{align*}
  and finally $\cd_j\cd^{k;j}$ vanishes for $k=0,2,3$, while $\cd_j\cd^{1;j}=v^2(r-r_0)=v r_0 x\sqrt{e}$.
  
  Furthermore, we record $G(\cd)=-r_0^2+v^2(r-r_0)^2=-r_0^2(1-e x^2)$ and $g_{0 0}=-1$, $g_{1 1}=1$, $g_{2 2}=r_0^2(1+x\sqrt{e}/v)^2$, and $g_{3 3}=r_0^2(1+x\sqrt{e}/v)^2\sin^2\theta$. The explicit expression for $S_{\cd_{r_0,v},e}$ now follows by direct calculation from Lemma~\ref{LemmaCSubpr}.

  In order to prove the positive definiteness, we note that for $e=0$, we have $a_{2 2}=v>0$, further $a_{0 0}+a_{1 1}=5 v>0$ as well as
  \[
    a_{0 0}a_{1 1}-a_{0 1}a_{1 0} = 2 v^2(3-2 x^2),
  \]
  which is positive for $|x|<\sqrt{\frac32}$, which is implied by $|x|\leq\frac65$. Moreover, the functions $a_{0 0}+a_{1 1}$, $a_{0 0}a_{1 1}-a_{0 1}a_{1 0}$, and $a_{2 2}$ on $[-\frac65,\frac65]$ depend smoothly on $\sqrt{e}$, and hence these positivity statements hold true for all sufficiently small $e>0$. The proof is complete.
\end{proof}

\subsection{Indicial roots of the zero energy operator}
\label{SsM0}

We now use a different bundle splitting, defined in terms of the functions $x^0=t+r$, $x^1=t-r$, namely,
\begin{equation}
\label{EqM0Split}
\begin{split}
  T^*\R^4 &= \la\dd x^0\ra \oplus \la\dd x^1\ra \oplus r T^*\Sph^2, \\
  S^2 T^*\R^4 &= \la(\dd x^0)^2\ra \oplus \la 2\dd x^0\,\dd x^1\ra \oplus (2\dd x^0\otimes_s r T^*\Sph^2) \oplus \la(\dd x^1)^2\ra \\
    &\qquad \oplus (2\dd x^1\otimes_s r T^*\Sph^2) \oplus r^2 S^2 T^*\Sph^2.
\end{split}
\end{equation}
Thus, $\dd t=(\frac12,\frac12,0)$ and $\dd r=(\frac12,-\frac12,0)$; in the corresponding dual splitting, $\pa_t=\pa_0+\pa_1=(1,1,0)$ (where we write $\pa_0=\pa_{x^0}$, $\pa_1=\pa_{x^1}$) and $\pa_r=(1,-1,0)$. Moreover, $\ubar g=-\dd x^0\,\dd x^1+r^2\slg$, so
\begin{equation}
\label{EqM0Metric}
  \ubar g = (0, -\tfrac12, 0, 0, 0, \slg),\quad
  \tr_{\ubar g} = (0, -4, 0, 0, 0, \sltr),\quad
  \sfG_{\ubar g} =
    \begin{pmatrix}
      1 & 0 & 0 & 0 & 0 & 0 \\
      0 & 0 & 0 & 0 & 0 & \frac14\sltr \\
      0 & 0 & 1 & 0 & 0 & 0 \\
      0 & 0 & 0 & 1 & 0 & 0 \\
      0 & 0 & 0 & 0 & 1 & 0 \\
      0 & 2\slg & 0 & 0 & 0 & \slsfG
    \end{pmatrix},
\end{equation}
where $\slsfG:=I-\frac12\slg\sltr$ with $\sltr:=\tr_\slg$.

Our present aim is to compute the indicial roots of the zero energy operator of the constraint propagation wave operator on Kerr, defined using~\eqref{EqCOp} where the constraint damping 1-form $\cd$ is given by~\eqref{EqMI1form} for large $r$. Recalling Remark~\ref{RmkK3bInfty}, it thus suffices to work on Minkowski space with the homogeneous (with respect to dilations) 1-form
\begin{equation}
\label{EqM01form}
\begin{alignedat}{2}
  \cd_v &:= r^{-1}(\dd t-v\,\dd r) &&= r^{-1} \bigl( \tfrac12(1-v),\ \tfrac12(1+v),\ 0\bigr), \\
  \cd_v^\sharp &= -r^{-1}(\pa_t+v\pa_r) &&= r^{-1}\bigl( -(1+v),\ -(1-v), 0 \bigr),
\end{alignedat}
\end{equation}
and to consider as the modification of $\delta_{\ubar g}^*$ the bundle map
\[
  \cC_{v,e,h}=h^{-1}\bigl(2\cd_v\otimes_s(\cdot)-(1-e)\ubar g\,\iota_{\cd_v^\sharp}\bigr).
\]

\begin{notation}[Operators on the sphere]
  We use slashes to denote operators on $\Sph^2$ with the standard metric $\slg$; thus, we write $\sld$, $\sldelta$, $\sldelta^*$, $\sltr$, $\slsfG$, and $\slstar$ for the exterior derivative, (negative) divergence, symmetric gradient, trace, the map $I-\frac12\slg\sltr$, and the Hodge star operator on $(\Sph^2,\slg)$.
\end{notation}

\begin{lemma}[Differential operators]
\label{LemmaM0Ops}
  We use the coordinates $t_*=x^1$, $\rho=r^{-1}$, and spherical coordinates.
  \begin{enumerate}
  \item On 1-forms, we have\footnote{In this paper, we only use the explicit forms of the zero-energy operators. We record full operator expressions here for consistency with \cite[\S{2.4.1}]{HintzKerrStab}.}
    \begin{align*}
      \Box_{\ubar g} &= -2\pa_{t_*}\rho(\rho\pa_\rho-1) + \rho^2\left( -(\rho\pa_\rho)^2+\rho\pa_\rho+\slDelta+\begin{pmatrix} 1 & -1 & -\sldelta \\ -1 & 1 & \sldelta \\ -2\sld & 2\sld & 1 \end{pmatrix} \right), \\
      \cC_{v,e,h} &= h^{-1}\rho
        \begin{pmatrix}
          1-v & 0 & 0 \\
          \frac12 e(1+v) & \frac12 e(1-v) & 0 \\
          0 & 0 & \frac12(1-v) \\
          0 & 1+v & 0 \\
          0 & 0 & \frac12(1+v) \\
          (1-e)(1+v)\slg & (1-e)(1-v)\slg & 0
        \end{pmatrix},
    \end{align*}
    where $\slDelta$ is the block diagonal operator with diagonal entries given by the tensor Laplacian $-\sltr\slnabla^2$ on functions, functions, and 1-forms, in this order.
  \item On symmetric 2-tensors, we have
    \begin{equation}
    \label{EqM0OpsDel}
    \begin{split}
      \delta_{\ubar g} &= \begin{pmatrix} 2 & 0 & 0 & 0 & 0 & 0 \\ 0 & 2 & 0 & 0 & 0 & 0 \\ 0 & 0 & 2 & 0 & 0 & 0 \end{pmatrix}\pa_{t_*} \\
        &\qquad + \rho\begin{pmatrix} \rho\pa_\rho-2 & -\rho\pa_\rho+2 & \sldelta & 0 & 0 & \frac12\sltr \\ 0 & \rho\pa_\rho-2 & 0 & -\rho\pa_\rho+2 & \sldelta & -\frac12\sltr \\ 0 & 0 & \rho\pa_\rho-3 & 0 & -\rho\pa_\rho+3 & \sldelta \end{pmatrix}.
    \end{split}
    \end{equation}
  \end{enumerate}
\end{lemma}
\begin{proof}
  Introduce local coordinates $x^2,x^3$ on $\Sph^2$, and let us write $a,b,c$ for spherical indices $2,3$. In the coordinates $x^\mu$, $\mu=0,\ldots,3$, all Christoffel symbols vanish except for
  \[
    \Gamma_{0 b}^c=\Gamma_{b 0}^c=\frac12 r^{-1}\delta_b^c,\quad
    \Gamma_{1 b}^c=\Gamma_{b 1}^c=-\frac12 r^{-1}\delta_b^c,\quad
    \Gamma_{a b}^0=-r\slg_{a b},\quad
    \Gamma_{a b}^1=r\slg_{a b},\quad
    \Gamma_{a b}^c=\slGamma_{a b}^c.
  \]
  For this calculation, one uses that $\pa_0 r=-\pa_1 r=\frac12$, as follows from $t=\frac12(x^0+x^1)$, $r=\frac12(x^0-x^1)$. For future reference, we note that in the coordinates $t_*,r$ and $t_*,\rho$, we have
  \begin{equation}
  \label{EqM0OpsPa01}
    \pa_0 = \frac12\pa_r = -\frac12\rho^2\pa_\rho,\quad
    \pa_1 = \pa_{t_*}-\frac12\pa_r = \pa_{t_*}+\frac12\rho^2\pa_\rho.
  \end{equation}
  We compute $(\delta_{\ubar g}^*\omega)_{\mu\nu}=\frac12(\pa_\mu\omega_\nu+\pa_\nu\omega_\mu)-\Gamma_{\mu\nu}^\kappa\omega_\kappa$ in the coordinates $x^\mu$ and in the splittings~\eqref{EqM0Split} (in particular, taking care of the factors of $r$) to be given by
  \begin{align}
    \delta_{\ubar g}^*&=
      \begin{pmatrix}
        \pa_0 & 0 & 0 \\
        \frac12\pa_1 & \frac12\pa_0 & 0 \\
        \frac12 r^{-1}\sld & 0 & \frac12(r^{-1}\pa_0 r-r^{-1}) \\
        0 & \pa_1 & 0 \\
        0 & \frac12 r^{-1}\sld & \frac12(r^{-1}\pa_1 r+r^{-1}) \\
        r^{-1}\slg & -r^{-1}\slg & r^{-1}\sldelta^*
      \end{pmatrix} \nonumber\\
  \label{EqM0OpsDelStar}
    &= \begin{pmatrix} 0 & 0 & 0 \\ \frac12 & 0 & 0 \\ 0 & 0 & 0 \\ 0 & 1 & 0 \\ 0 & 0 & \frac12 \\ 0 & 0 & 0 \end{pmatrix}\pa_{t_*} + \rho\begin{pmatrix} -\frac12\rho\pa_\rho & 0 & 0 \\ \frac14 \rho\pa_\rho & -\frac14\rho\pa_\rho & 0 \\ \frac12\sld & 0 & -\frac14(\rho\pa_\rho+1) \\ 0 & \frac12\rho\pa_\rho & 0 \\ 0 & \frac12\sld & \frac14(\rho\pa_\rho+1) \\ \slg & -\slg & \sldelta^* \end{pmatrix}.
  \end{align}
  The divergence on symmetric 2-tensors can either be computed directly, or by computing the adjoint of $\delta_{\ubar g}^*$, noting that the metric volume density is $\frac12 r^2|\dd x^0\dd x^1\dd\slg|$, and the fiber inner products on $T^*\R^4$ resp.\ $S^2 T^*\R^4$ induced by $\ubar g$ are given by
  \[
    \begin{pmatrix}
      0 & -2 & 0 \\
      -2 & 0 & 0 \\
      0 & 0 & \slg^{-1}
    \end{pmatrix},\quad\text{resp.}\quad
    \begin{pmatrix}
      0 & 0 & 0 & 4 & 0 & 0 \\
      0 & 8 & 0 & 0 & 0 & 0 \\
      0 & 0 & 0 & 0 & -4\slg^{-1} & 0 \\
      4 & 0 & 0 & 0 & 0 & 0 \\
      0 & 0 & -4\slg^{-1} & 0 & 0 & 0 \\
      0 & 0 & 0 & 0 & 0 & \slg^{-1}_2
    \end{pmatrix},
  \]
  where $\slg^{-1}_2$ is the fiber inner product on $S^2 T^*\Sph^2$ induced by $\slg$. One finds
  \[
    \delta_{\ubar g}=
      \begin{pmatrix}
        2 r^{-2}\pa_1 r^2 & 2 r^{-2}\pa_0 r^2 & r^{-1}\sldelta & 0 & 0 & \frac12 r^{-1}\sltr \\
        0 & 2 r^{-2}\pa_1 r^2 & 0 & 2 r^{-2}\pa_0 r^2 & r^{-1}\sldelta & -\frac12 r^{-1}\sltr \\
        0 & 0 & 2(r^{-1}\pa_1 r-r^{-1}) & 0 & 2(r^{-1}\pa_0 r+r^{-1}) & r^{-1}\sldelta
      \end{pmatrix}.
  \]
  Upon inserting~\eqref{EqM0OpsPa01}, this gives~\eqref{EqM0OpsDel}. Since $2\delta_{\ubar g}\sfG_{\ubar g}\delta_{\ubar g}^*\omega=\Box_{\ubar g}\omega+\iota_{\omega^\sharp}\Ric(\ubar g)$ and $\Ric(\ubar g)=0$, we can compute $\Box_{\ubar g}=2\delta_{\ubar g}\sfG_{\ubar g}\delta_{\ubar g}^*$ using~\eqref{EqM0OpsDel}, \eqref{EqM0Metric}, and \eqref{EqM0OpsDelStar}, together with $2\sldelta\slsfG\sldelta^*=\slDelta-1$.

  Lastly, the expression for $\cC_{v,e,h}$ follows from the expression
  \[
    2\cd_v\otimes_s(\cdot)
     =\rho\begin{pmatrix}
        1-v & 0 & 0 \\
        \frac12(1+v) & \frac12(1-v) & 0 \\
        0 & 0 & \frac12(1-v) \\
        0 & 1+v & 0 \\
        0 & 0 & \frac12(1+v) \\
        0 & 0 & 0
      \end{pmatrix}
  \]
  and~\eqref{EqM0Metric}--\eqref{EqM01form}.
\end{proof}

\begin{cor}[Constraint propagation wave operator]
\label{CorM00}
  In the bundle splitting~\eqref{EqM0Split}, we have
  \begin{equation}
  \label{EqM00Box}
    \Box_{\ubar g}^{\cC_{v,e,h}} = 2\delta_{\ubar g}\sfG_{\ubar g}(\delta_{\ubar g}^*+\cC_{v,e,h}) = -2\pa_{t_*}\rho(\rho\pa_\rho-1-\ubar S^{\cC_{v,e,h}}) + \wh{\Box_{\ubar g}^{\cC_{v,e,h}}}(0),
  \end{equation}
  where
  \begin{equation}
  \label{EqM00S}
    \ubar S^{\cC_{v,e,h}} = h^{-1}
      \begin{pmatrix}
        2(1-v) & 0 & 0 \\
        (1-e)(1+v) & (1-e)(1-v) & 0 \\
        0 & 0 & 1-v
      \end{pmatrix}
  \end{equation}
  and
  \begin{equation}
  \label{EqM00}
    \rho^{-2}\wh{\Box_{\ubar g}^{\cC_{v,e,h}}}(0) = -(\rho\pa_\rho)^2 + \rho\pa_\rho + \slDelta + \begin{pmatrix} 1 & -1 & -\sldelta \\ -1 & 1 & \sldelta \\ -2\sld & 2\sld & 1 \end{pmatrix} + h^{-1}\begin{pmatrix} q_{0 0} & q_{0 1} & q_{0\slash} \\ q_{1 0} & q_{1 1} & q_{1\slash} \\ q_{\slash 0} & q_{\slash 1} & q_{\slash\slash} \end{pmatrix},
  \end{equation}
  and where the $q_{\bullet\bullet}$ are the first order operators
  \begin{align*}
    q_{0 0}&=(1-3 v+e+e v)\rho\pa_\rho + (-1+3 v+e+e v), \\
    q_{1 0}&=(1-e)(1+v)\rho\pa_\rho-(1+e)(1+v), \\
    q_{\slash 0}&=-2 e(1+v)\sld, \\
    q_{0 1}&=-(1-e)(1-v)\rho\pa_\rho+(1+e)(1-v), \\
    q_{1 1}&=(-1-3 v-e+e v)\rho\pa_\rho+(1+3 v-e+e v), \\
    q_{\slash 1}&=-2 e(1-v)\sld, \\
    q_{0\slash}&=(1-v)\sldelta, \\
    q_{1\slash}&=(1+v)\sldelta, \\
    q_{\slash\slash}&=-2 v\rho\pa_\rho+4 v.
  \end{align*}
\end{cor}
\begin{proof}
  The term involving the $q_{\bullet\bullet}$ is the operator $2\rho^{-2}\wh{\delta_{\ubar g}}(0)\sfG_{\ubar g}\cC_{v,e,h}$ where $\wh{\delta_{\ubar g}}(0)$ is the zero energy operator of $\delta_{\ubar g}$, obtained from~\eqref{EqM0OpsDel} by dropping the first term on the right. The endomorphism $\ubar S^{\cC_{v,e,h}}$ on the other hand is given by $\rho^{-1}[\delta_{\ubar g},t_*]\sfG_{\ubar g}\cC_{v,e,h}$.
\end{proof}

Recall now that an indicial root of $\wh{\Box_{{\ubar g}}^{\cC_{v,e,h}}}(0)$ (or equivalently of $\wh{\Box_{{\ubar g},h}^{\cC_{v,e,h}}}(0)=h^2\wh{\Box_{\ubar g}^{\cC_{v,e,h}}}(0)$) is a complex number $\lambda\in\C$ such that $\wh{\Box_{\ubar g}^{\cC_{v,e,h}}}(0)$ has a nontrivial kernel on 1-forms whose coefficients in the splitting~\eqref{EqM0Split} are homogeneous of degree $\lambda$ with respect to dilations in $\rho$. Equivalently, the operator
\begin{equation}
\label{EqM0NormOp}
  N_{v,e,h}(\lambda) := -\lambda^2+\lambda+\slDelta+\begin{pmatrix} 1 & -1 & -\sldelta \\ -1 & 1 & \sldelta \\ -2\sld & 2\sld & 1 \end{pmatrix} + Q_{v,e,h}(\lambda,\slnabla) \in \Diff^2(\Sph^2;\ubar\C\oplus\ubar\C\oplus T^*\Sph^2),
\end{equation}
has nontrivial kernel, where $Q_{v,e,h}(\lambda,\slnabla)$ is obtained from the last term in~\eqref{EqM00} (including the prefactor $h^{-1}$) upon replacing $\rho\pa_\rho$ by $\lambda$. (We write $\ubar\C$ for the trivial bundle $\Sph^2\times\C\to\Sph^2$.) Since $N_{v,e,h}(\lambda)$ commutes with the rotation action on $\Sph^2$ (acting via pullback on $T^*\Sph^2$), one can classify indicial roots into types as follows:
\begin{definition}[Scalar and vector type indicial roots]
\label{DefM0Types}
  Denote the space of degree $l$ spherical harmonics, $l\in\N_0$, by $\scalspace_l=\mathspan\{Y_{l m}\colon m=-l,\ldots,l\}$. Let $\lambda\in\C$ be an indicial root of $\wh{\Box_{{\ubar g}}^{\cC_{v,e,h}}}(0)$.
  \begin{enumerate}
  \item Let $l\geq 0$. We say that $\lambda$ is a \emph{scalar type $l$ indicial root} (for short: an \emph{$\rms l$ indicial root}) if, in the case $l\geq 1$, there exist $0\neq\scal\in\scalspace_l$ and constants $(a,b,c)\neq(0,0,0)$ such that $(a\scal,b\scal,c\,\sld\scal)=a\scal\,\dd x^0+b\scal\,\dd x^1+c r\,\sld\scal\in\ker N_{v,e,h}(\lambda)$. In the case $l=0$, we require the existence of constants $(a,b)\neq(0,0)$ such that $(a,b,0)=a\,\dd x^0+b\,\dd x^1\in\ker N_{v,e,h}(\lambda)$.
  \item Let $l\geq 1$. Then $\lambda$ is a \emph{vector type $l$ indicial root} (for short: a \emph{$\rmv l$ indicial root}) if there exists $0\neq\scal\in\scalspace_l$ such that $(0,0,\slstar\sld\scal)=r\slstar\sld\scal\in\ker N_{v,e,h}(\lambda)$.
  \end{enumerate}
  We call these special types of solutions scalar type $l$ and vector type $l$ solutions.
\end{definition}

The point is that $L^2(\Sph^2;T^*\Sph^2)$ has an orthogonal basis given by the 1-forms $\sld\scal$ and $\slstar\sld\scal$, where $\scal\in\scalspace_l$ runs over an orthogonal basis of $\scalspace_l$ for each $l\geq 1$. (This is a classical result and can be proved by noting that every $\omega\in\CI(\Sph^2;T^*\Sph^2)$ can be written as $\omega=\sld u+\eta$ where $u\in\CI(\Sph^2)$ and $\eta\in\CI(\Sph^2;T^*\Sph^2)$, $\sldelta\eta=0$, so $\sld\slstar\eta=0$; but since the first cohomology group of $\Sph^2$ is trivial, we can then further write $\eta=\slstar\sld v$ for some $v\in\CI(\Sph^2)$. Expanding $u$ and $v$ into spherical harmonics yields the claim.) Thus, any $\omega\in\CI(\Sph^2)\oplus\CI(\Sph^2)\oplus\CI(\Sph^2;T^*\Sph^2)$ with $N_{v,e,h}(\lambda)\omega=0$ can be decomposed into a sum of scalar type $l\in\N_0$ and vector type $l\in\N$ 1-forms, and all of the individual terms lie in $\ker N_{v,e,h}(\lambda)$ since $N_{v,e,h}(\lambda)$ preserves the spaces of scalar type $l$ and vector type $l$ sections of $\ubar\C\oplus\ubar\C\oplus T^*\Sph^2$ for all $l$. (Since $N_{v,e,h}(\lambda)$ is elliptic, only finitely many pure type parts of such $\omega$ can be non-zero, as each of these parts is an indicial solution by itself.)

The main result of this section is:

\begin{thm}[Indicial roots at zero energy]
\label{ThmM0}
  Let $v\in(0,1)$ and $C_0>0$. Then there exists $e_0=e_0(v)\in(0,1)$ such that for all $e\in(0,e_0)$, there exists $h_0>0$ such that for all $h\in(0,h_0)$, all indicial roots $\lambda\in\C$ of $\wh{\Box_{{\ubar g}}^{\cC_{v,e,h}}}(0)$ satisfy:
  \begin{enumerate}
  \item $\Re\lambda\geq 1$ or $\Re\lambda<-C_0$;
  \item the only root with $\Re\lambda=1$ is $\lambda=1$, which occurs only in scalar types $0$ and $1$, and has order $1$ in each;
  \item\label{ItM0Number} there are precisely $1$, resp.\ $2$, resp.\ $3$ vector type $l$ ($l\geq 1$), resp.\ scalar type $0$, resp.\ scalar type $l$ ($l\geq 1$) roots with real parts $\geq 1$, and an equal number of roots with real parts $<-C_0$;
  \item the real parts of the indicial roots $\lambda$ accumulate only at $\pm\infty$. In particular, there exists $\eps_{\rm ind}>0$ such that no root has real part in $(1,1+\eps_{\rm ind})$.
  \end{enumerate}
\end{thm}

The proof, which is given in Appendix~\ref{SM0}, follows from computations and estimates for roots of characteristic polynomials. We study the various types separately: vector type roots in~\S\ref{SsM0v}, scalar type $0$ roots in~\S\ref{SsM0s}, and scalar type $l\geq 1$ roots in~\S\ref{SsMls}. For scalar type $l\geq 1$ roots, the characteristic polynomial has degree $6$ (see~\eqref{EqMsPoly}) and has no explicit full set of zeros when $e,h>0$, and hence we must resort to careful asymptotic analysis to prove Theorem~\ref{ThmM0} in this case. While the proof is thus quite delicate (albeit an instructive example of the clarifying power of geometric singular analysis), we only use the \emph{statement} of Theorem~\ref{ThmM0} in later parts of the paper, and hence the reader may safely skip its proof at first reading.

\begin{rmk}[Conceptual proof of a weaker statement]
\label{RmkM0Weaker}
  A weaker statement, with the lower bound $\Re\lambda\geq 1$ relaxed to $\Re\lambda\geq 1-\eps$ for any fixed $\eps>0$, can be proved in a very conceptual manner using (semiclassical) microlocal techniques as in~\S\ref{SE} (see Remark~\ref{RmkTInv0}), with $h_0>0$ depending on the choice of $\eps>0$. In order to get the sharp result of Theorem~\ref{ThmM0}---which is crucial for our application to the Kerr black hole stability problem---it appears unavoidable to proceed directly, as we do in Appendix~\ref{SM0}.
\end{rmk}

\begin{rmk}[Inward pointing constraint damping vector field]
\label{RmkM0WrongV}
  If we were to choose the constraint damping 1-form $\cd_v$ such that the vector field $-r^2\cd_v^\sharp=r(\pa_t+v\pa_r)$ is \emph{inward} pointing for large $r$, i.e., $v<0$ (and $v>-1$ to ensure its timelike nature), then the indicial roots would behave in a manner incompatible with our needs for the stability problem: for example, a simple perturbative analysis of $\lambda_+(h,l)$ in~\eqref{EqM0vTaylor} shows that the number of vector type indicial roots in any fixed interval $[2-\delta,2]$, $\delta>0$, grows without bound as $h\searrow 0$.
\end{rmk}

\begin{rmk}[Complex roots]
\label{RmkM0Complex}
  For fixed small $e$ and all sufficiently small $h$, there is a finite number of indicial roots that are non-real (but then appear in complex conjugate pairs); see Figure~\ref{FigMsNumeric}.
\end{rmk}

\section{Semiclassical constraint damping on subextremal Kerr spacetimes}
\label{SE}

We now fix a subextremal Kerr metric
\[
  g = g_{\bhm,a},
\]
see~\eqref{EqKBL}--\eqref{EqKMetStar}. We recall from Definition~\ref{DefKMfd} the notation $X=\cl_M\{t=0\}$ (with $t$ as in~\eqref{EqKMetCoord}) for (the closure of) its Cauchy hypersurface. Recall the bundles $\cT$, $\cT^*$ from~\eqref{EqKscT}, and put $\cT_X:=\cT|_X$, $\cT^*_X:=\cT^*|_X$.

\subsection{Choice of the constraint damping 1-form and the parameter \texorpdfstring{$e$}{e}}
\label{SsEC}

Our first task is to construct a stationary 1-form $\cd$ out of the pieces constructed in~\S\ref{SM}. More precisely:

\begin{prop}[Existence of suitable $\cd$ and $e$]
\label{PropEC}
  Fix any $v\in(0,1)$ and $C_0>0$. Fix moreover $\alpha_\sface,\alpha_\cK\in\R$ subject to the conditions $\alpha_\sface<-\frac12$ and $\alpha_\cK > \alpha_\sface+\frac12$. Set $\rho_\sface=r^{-1}$ and $\rho_\cK=\frac{r}{t}$; put $w:=\rho_\sface^{\alpha_\sface+1}\rho_\cK^{\alpha_\cK}$ and $w':=\rho_\sface^{\alpha_\sface+1}$. Then for all sufficiently small $e>0$, there exists a 1-form $\cd\in\rho_\sface\CI(M;\cT^*)$ satisfying the following properties for some $r_0>100\bhm$, $0<\delta_r<(4 r_0)^{-1}$, and $\delta_\sface>0$:
  \begin{enumerate}
  \item\label{ItECCausal}{\rm (Basic properties.)} $\cd$ is stationary, i.e., $\cL_{\pa_t}\cd=0$ (equivalently, $\cd\in r^{-1}\CI(X;\cT^*_X)$). Moreover, $r\cd\in\CI(M;\cT^*)$ is past timelike and nowhere vanishing.
  \item\label{ItECExpl}{\rm (Explicit regions.)} We have
    \begin{alignat}{3}
    \label{EqECExpl}
      \cd&=\frac12 r_0^{-2}(r_0\,\dd t-v(r-r_0)\,\dd r),&\quad \frac12 r_0\leq r&\leq 2 r_0, \\
    \label{EqECExpl2}
      \cd&=r^{-1}(\dd t-v\,\dd r),&\quad r&\geq 3 r_0.
    \end{alignat}
  \item\label{ItECPosInt}{\rm (Interior saddle point, cf.\ Lemma~\ref{LemmaMSSubpr}.)} The endomorphism $S_{g,\cd,e}$ from Lemma~\usref{LemmaCSubpr} is positive definite (with respect to the inner product $B_{g,\cd,e}$ on $\cT^*$ from Lemma~\usref{LemmaCInner}) at all points with $|r-r_0|\leq\frac65 r_0\sqrt{e}/v$.
  \item\label{ItECPosCorner}{\rm (Saddle point at $\pa\cK^+$, cf.\ Lemma~\ref{LemmaMIbdy}.)} For $\rho_\sface\leq 2\delta_r$ and $\rho_\cK\leq 2\delta_r$, the endomorphism $r^2(-\ell_{g,\cd,e}(\frac{\dd w}{w})+S_{g,\cd,e})$ is positive definite.
  \item\label{ItECPosSink}{\rm (Sink in $\sface^\circ$, cf.\ Lemma~\ref{LemmaMISink}.)} At points where $\frac{r}{t}/v\in[1-2\delta_\sface,1+2\delta_\sface]$ and $r^{-1}\leq 2\delta_r$, the endomorphism $r^2(-\ell_{g,\cd,e}(\frac{\dd w'}{w'})+S_{g,\cd,e})$ is positive definite.
  \item\label{ItECPosEsc}{\rm (Positivity along escape functions, cf.\ Lemma~\ref{LemmaMSRad}.)} For $\bhm\leq r\leq r_0-r_0\sqrt{e}/v$, $-\ell_{g,\cd,e}(-\dd r)$ is positive definite. For $r_0+r_0\sqrt{e}/v\leq r<\infty$, $-\ell_{g,\cd,e}(\dd r)$ is positive definite. On $\sface\cap\{\frac{r}{t}\in[\delta_r,v(1-\delta_\sface)]\}$, $-r^2\ell_{g,\cd,e}(\dd(\frac{r}{t}))$ is positive definite, and on $\sface\cap\{\frac{r}{t}\in[v(1+\delta_\sface),10]\}$, $-r^2\ell_{g,\cd,e}(-\dd(\frac{r}{t}))$ is positive definite.
  \end{enumerate}
\end{prop}

The factor $\frac12$ in~\eqref{EqECExpl} (as compared to~\eqref{EqMS1form}) is inconsequential since it can be absorbed in the semiclassical parameter $h$ when studying $\Box_{g,h}^{\cC_{g,\cd,e,h}}$. It is inserted here so that~\eqref{EqECExpl} equals $r^{-1}(\dd t-v\,\dd r)$ at $r=2 r_0$.

\begin{proof}[Proof of Proposition~\usref{PropEC}]
  Let $e_1$ be smaller than the number $e_0(v)$ provided by Theorem~\ref{ThmM0} and smaller than the number $e_0$ from Lemma~\ref{LemmaMSSubpr}.
  
  \pfstep{Positivity properties for $\frac12 r_0\leq r\leq 2 r_0$.} If we use the Minkowski metric $\ubar g$ instead of the Kerr metric $g$, and if we define $\cd$ as in condition~\eqref{ItECExpl}, then Lemma~\ref{LemmaMSRad} ensures the positivity condition~\eqref{ItECPosEsc} for $r\in[\frac12 r_0,2 r_0]\setminus[r_0-r_0\sqrt{e}/v,r_0+r_0\sqrt{e}/v]$. More precisely, Lemma~\ref{LemmaMSRad} gives an $r_0$-independent lower bound
  \begin{equation}
  \label{EqECMinkRad}
    -r_0\ell_{\ubar g,\cd,e}(\pm \dd r) \geq \lambda_0 > 0,\quad \pm(r-r_0)\geq r_0\sqrt{e}/v,\ r\in[\tfrac12 r_0,2 r_0].
  \end{equation}
  We aim to show that if $r_0$ is sufficiently large, this positivity holds also for the Kerr metric. Note first that
  \begin{equation}
  \label{EqECUnifBdd}
    r\cd|_{\{\frac12 r_0\leq r\leq 2 r_0\}} = \frac{r}{2 r_0}\Bigl(\dd t-v\Bigl(\frac{r}{r_0}-1\Bigr)\dd r\Bigr)
  \end{equation}
  is uniformly bounded in $L^\infty$ when $r_0$ ranges from $100\bhm$ to $+\infty$ (meaning: its components with respect to the standard vector fields $\pa_t,\pa_{x^1},\pa_{x^2},\pa_{x^3}$ on $\R^4$ are uniformly bounded). Moreover, in view of Lemma~\ref{LemmaKMink} and the explicit formulas~\eqref{EqCInnerProd}, the fiber inner products from Lemma~\ref{LemmaCInner} satisfy $B_{g,\cd,e}-B_{\ubar g,\cd,e}\in\rho_\sface\CI(M;S^2\cT)$. In particular, denoting by $|\cdot|$ the $L^\infty$-norm of the coefficients with respect to the standard coordinate differentials $\dd t,\dd x^1,\dd x^2,\dd x^3$, we have
  \begin{equation}
  \label{EqECDiffB}
    |B_{g,\cd,e}-B_{\ubar g,\cd,e}|\leq C r_0^{-1},\quad \frac12 r_0\leq r\leq 2 r_0,
  \end{equation}
  where the constant $C$ is independent of the choice of $r_0$ (but $C$ does depend on $e$). Using that for fixed $e$, the norm of endomorphisms of $\cT^*$ with respect to the fiber inner product $B_{\ubar g,\cd,e}$ is uniformly equivalent to the $L^\infty$-norm of the matrix coefficients in the standard coordinate basis, we similarly have
  \begin{equation}
  \label{EqECDiffEll}
    |r\ell_{g,\cd,e}(\dd r) - r\ell_{\ubar g,\cd,e}(\dd r)|_{B_{\ubar g,\cd,e}} \leq C r_0^{-1},\quad \frac12 r_0\leq r\leq 2 r_0.
  \end{equation}
  Altogether then, choosing $r_1$ sufficiently large (depending on $e$) so that $C(\frac 12 r_1)^{-1}\leq\frac12\lambda_0$, we conclude from~\eqref{EqECMinkRad}--\eqref{EqECDiffEll} that for all $r_0\geq r_1$, and for $\cd$ as in~\eqref{EqECExpl}, we have
  \begin{equation}
  \label{EqECfrac122Pos}
    {-}\ell_{g,\cd,e}(\pm\dd r) \geq \frac12\lambda_0 r_0^{-1} > 0,\quad \pm(r-r_0)\geq r_0\sqrt{e}/v,\ r\in[\tfrac12 r_0,2 r_0].
  \end{equation}

  Consider next the positivity condition~\eqref{ItECPosInt}, which for the Minkowski metric holds for all $e\in(0,e_1)$ by Lemma~\ref{LemmaMSSubpr}; in fact, $r_0^2 S_{\ubar g,\cd,e}$ has an $r_0$-independent lower bound for $|r-r_0|\leq\frac65 r_0\sqrt{e}/v$. In order to perturb to the Kerr metric, we note that
  \[
    |r_0^2 S_{g,\cd,e}-r_0^2 S_{\ubar g,\cd,e}| \leq C r_0^{-1},\qquad \frac12 r_0\leq r\leq 2 r_0,
  \]
  as follows directly from the expression for $S_{g,\cd,e}$ given in Lemma~\ref{LemmaCSubpr}; here $C$ is independent of $r_0$ (but does depend on $e$). Thus, for sufficiently large $r_2\geq r_1$ (depending on $e$), we conclude that condition~\eqref{ItECPosInt} and the positivity~\eqref{EqECfrac122Pos} hold for all $r_0\geq r_2=r_2(e)$.\footnote{Numerically, on the Schwarzschild spacetime, one finds that for fixed $v$, the lower bound $r_2$ scales with $\bhm$ and $e^{-\frac12}$. Concretely, for $v=\frac12$, the requirement is roughly $r_2\geq 81 \bhm/\sqrt{e}$. For $e=0.0001 v^2$ (cf.\ the comment after Lemma~\ref{LemmaMSSubpr}), this reads $r_2\geq 16200\bhm$.}

  \pfstep{Control near the critical sets over $\sface$; positivity in $\sface^\circ$ away from the critical sets.} Using the explicit form~\eqref{EqECExpl2} of $\cd$ for $r\geq 3 r_0$, consider condition~\eqref{ItECPosCorner}. We again argue perturbatively. For the Minkowski metric, all eigenvalues of $T_{\ubar g,\cd,e}:=r^2(-\ell_{\ubar g,\cd,e}(\frac{\dd w}{w})+S_{\ubar g,\cd,e})$ are positive for $e=0$ in a compact neighborhood of $\pa\cK^+$ by Lemma~\ref{LemmaMIbdy} (with $\beta_\cK=\alpha_\cK$ and $\beta_\sface=\alpha_\sface+1$); in this neighborhood, this positivity persists for all $e\in(0,e_\pa)$ when $e_\pa>0$ is sufficiently small. For any fixed such $e$ then, note that
  \[
    T_{g,\cd,e}-T_{\ubar g,\cd,e} \in \rho_\sface\CI(M;\End(\cT^*)).
  \]
  Hence, for some $\delta_r>0$ (depending on $e$), $T_{g,\cd,e}$ is positive definite when $\rho_\cK\leq 2\delta_r$ and $\rho_\sface\leq 2\delta_r$. In fact, since $g$ and $\ubar g$ are \emph{equal at $\sface$} (as 3sc-metrics), there exists $\delta'_r>0$ which is \emph{independent of the choice of $e\in(0,e_\pa)$} such that $T_{g,\cd,e}|_\sface=T_{\ubar g,\cd,e}|_\sface$ is positive definite for $\rho_\cK\leq 2\delta'_r$.
  
  Analogous considerations apply near the sink $\cR_v$ defined in~\eqref{EqMICrit}, now using Lemma~\ref{LemmaMISink} (with $\beta_\sface=\alpha_\sface+1$). To summarize, we have now shown that there exist $e_2>0$, which we take to be equal to $\min(e_1,e_\pa)$, and $\delta'_r>0$ such that for all $e\in(0,e_2)$, there exist $r_2\geq 100\bhm$ and $\delta_r>0$ such that for all $r_0\geq r_2$, conditions~\eqref{ItECExpl}--\eqref{ItECPosSink} are satisfied, as is condition~\eqref{ItECPosEsc} for $r\in[\frac12 r_0,2 r_0]$; moreover, $T_{g,\cd,e}|_\sface$ is positive definite for $\rho_\cK\leq 2\delta'_r$, and $T'_{g,\cd,e}:=r^2(-\ell_{g,\cd,e}(\frac{\dd w'}{w'})+S_{g,\cd,e})$ is positive definite on $\sface\cap\{\frac{r}{t}/v\in[1-2\delta'_r,1+2\delta'_r]\}$.

  Consider next condition~\eqref{ItECPosEsc} on $\sface^\circ$, where we can replace $g$ by the Minkowski metric $\ubar g$. Writing $\tilde\rho_\cK=\frac{r}{t}$ for a defining function of $\cK^+$ in $\frac{r}{t}\leq 10$, we compute
  \[
    -r^2\ell_{g,\cd,e}\Bigl(\dd\Bigl(\frac{r}{t}\Bigr)\Bigr)=\ell_{g,r\cd,e}(\tilde\rho_\cK^2\,\dd t-\tilde\rho_\cK\,\dd r).
  \]
  We use Lemma~\ref{LemmaCTimelikeSp} with $\zeta=\tilde\rho_\cK^2\,\dd t-\tilde\rho_\cK\,\dd r$ and $r\cd=\dd t-v\,\dd r$. The first condition in~\eqref{EqCTimelikeSp} (with the ``$+$'' sign) reads $\la r\cd,\zeta\ra=\tilde\rho_\cK(v-\tilde\rho_\cK)>0$ and is thus satisfied for $0<\tilde\rho_\cK<v$. The second condition in~\eqref{EqCTimelikeSp} is equivalent to
  \begin{equation}
  \label{EqECPossf}
    -\tilde\rho_\cK^4 + \tilde\rho_\cK^2 < \frac{2(1-e)}{e(1-v^2)}\tilde\rho_\cK^2(v-\tilde\rho_\cK)^2.
  \end{equation}
  Upon division by $\tilde\rho_\cK^2(v-\tilde\rho_\cK)^2$ for $\delta'_r\leq\tilde\rho_\cK\leq v(1-\delta'_r)$, this would follow from
  \begin{equation}
  \label{EqECPossf2}
    v^{-2}(\delta'_r)^{-2} < \frac{2(1-e)}{e(1-v^2)}.
  \end{equation}
  But this indeed holds for all sufficiently small $e$. (This is where the $e$-independence of $\delta'_r$ is crucial.)
  
  Similarly, for $v(1+\delta'_r)\leq\tilde\rho_\cK\leq 10$, we use Lemma~\ref{LemmaCTimelikeSp} with $\zeta=-(\tilde\rho_\cK^2\,\dd t-\tilde\rho_\cK\,\dd r)$, so now $\la r\cd,\zeta\ra=\tilde\rho_\cK(\tilde\rho_\cK-v)>0$ verifies the first condition in~\eqref{EqCTimelikeSp}, while the second condition in~\eqref{EqCTimelikeSp} is equivalent to~\eqref{EqECPossf}; this is automatic for $\tilde\rho_\cK\geq 1$ (where $\zeta$ is future causal), and for $\tilde\rho_\cK\in[v(1+\delta'_r),1]$ is implied by~\eqref{EqECPossf2}. We shrink $e_2>0$ so that these conclusions are valid for all $e\in(0,e_2)$.

  \pfstep{Positivity for $3 r_0\leq r<\infty$.} In this region, we have the explicit form~\eqref{EqECExpl2} of $\cd$. We use Lemma~\ref{LemmaCTimelikeSp} with $\zeta=-r\,\dd r$. The first condition in~\eqref{EqCTimelikeSp} (with the ``$+$'' sign) is equivalent to $\la\dd t-v\,\dd r,\dd r\ra=-v+\cO(r^{-1})<0$ where the $\cO(r^{-1})$ term depends only on $\bhm$ and $v$; this holds for all sufficiently large $r$. The second condition in~\eqref{EqCTimelikeSp} is equivalent to
  \[
    r^2+\cO(r) < \frac{2(1-e)}{e(r^{-2}(1-v^2)+\cO(r^{-3}))} (v^2+\cO(r^{-1})),
  \]
  that is, to
  \begin{equation}
  \label{EqECPos3r0}
    \frac{2(1-e)v^2}{e(1-v^2)}-C e^{-1}r^{-1}>1,
  \end{equation}
  where $C$ only depends on $\bhm,v$. This holds for all small $e$ and large enough $r$, so for $e<e_2$ and $r\geq 3 r_0$ upon decreasing $e_2$ and increasing $r_0$, if necessary.

  \pfstep{Constructing $\cd$ in the remaining regions.} We first define a continuous 1-form $\cd'$ satisfying all conditions except for smoothness at $r=\frac12 r_0$ and $r=2 r_0$: namely, we let $\cd'$ be given by~\eqref{EqECExpl}--\eqref{EqECExpl2}, and further set
  \begin{equation}
  \label{EqECExpl3}
    \cd'=r^{-1}(\dd t-v\,\dd r),\quad r\geq 2 r_0
  \end{equation}
  (which at $r=2 r_0$ matches~\eqref{EqECExpl}). Then the arguments following~\eqref{EqECPos3r0} apply in the larger region $r\geq 2 r_0$ still when $r_0\geq r_2$ with $r_2$ sufficiently large (depending on $e$ and $\bhm$).

  At $r=\frac12 r_0$, we have $\cd'=(2 r_0)^{-1}(\dd t+\frac12 v\,\dd r)$, which we need to extend to $\bhm\leq r\leq\frac12 r_0$ as a past timelike 1-form and such that $-\ell_{g,\cd',e}(-\dd r)$ is positive definite; this requires using the form of the Kerr metric. Firstly, for $r\in[\bhm,r_+]$, we use the form~\eqref{EqKMetStar} of $g=g_{\bhm,a}$ with $\chi\equiv 1$ and set
  \[
    \cd'_1 := \dd t_1-q\,\dd r.
  \]
  Then, writing $\la\cdot,\cdot\ra=g^{-1}(\cdot,\cdot)$, the inner product $\varrho^2\la\cd'_1,\dd t_1\ra=-q(r^2+a^2)+a^2\sin^2\theta$ is negative provided we fix $q>a^2/(\bhm^2+a^2)$, and then $\varrho^2|\cd'_1|_{g^{-1}}^2=-2 q(r^2+a^2)+\Delta q^2+a^2\sin^2\theta\leq\varrho^2 g^{-1}(\cd'_1,\dd t_1)<0$ for $r\in[\bhm,r_+]$ since $\Delta\leq 0$ there. Thus, $\cd'_1$ is past timelike. Moreover, $\varrho^2\la\cd'_1,\dd r\ra=(r^2+a^2)-q \Delta>0$. Fixing $q$, there exists $\delta>0$ such that the past timelike character and the sign condition $\la\cd'_1,\dd r\ra>0$ continue to hold in the slightly larger interval $r\in[\bhm,r_++2\delta]$.

  Next, for sufficiently large $r'_1>r_+$ (depending on $\bhm$), which, by requiring $r_0>2 r'_1$, we may assume to be less than $\frac12 r_0$, the 1-form
  \begin{equation}
  \label{EqECcd3}
    \cd'_3=\dd t+\frac12 v\,\dd r
  \end{equation}
  (where we recall that $\frac12 v\in(0,\frac12)$) is past timelike for the Kerr metric for $r\geq r'_1$ since this is true for the Minkowski metric; and $\varrho^2\la\cd'_3,\dd r\ra=\frac12 v\Delta>0$.

  Finally, for $r>r_+$, the 1-form $\dd t$ is past timelike since, using~\eqref{EqKBL},
  \[
    \varrho^2\Delta|\dd t|_{g^{-1}}^2=-(r^2+a^2)^2+a^2\Delta\sin^2\theta<-(r^2+a^2)^2+a^2(r^2+a^2)<0.
  \]
  Thus, by continuity,
  \[
    \cd'_2 := \dd t+\eta\,\dd r
  \]
  is past timelike for $r_++\delta\leq r\leq 2 r'_1$ for sufficiently small $\eta>0$; and then also $\varrho^2\la\cd'_2,\dd r\ra=\eta\Delta>0$.

  Using a partition of unity $\chi_1+\chi_2+\chi_3=1$ on $[\bhm,\infty)$ subordinate to the cover $[\bhm,r_++2\delta)\cup(r_++\delta,2 r'_1)\cup(r'_1,\infty)$ (and with $0\leq\chi_j\leq 1$, $j=1,2,3$), the 1-form
  \begin{equation}
  \label{EqECGlue}
    \cd'_0 := \sum_{j=1}^3 \chi_j\cd'_j
  \end{equation}
  on $[\bhm,\frac12 r_0]$ is past timelike and satisfies $\la\cd'_0,\dd r\ra>0$.

  We want to apply Lemma~\ref{LemmaCTimelikeSp} to conclude that $-\ell_{g,\cd'_0,e}(-\dd r)$ is positive definite for $r\in[\bhm,\frac12 r_0]$ when $e>0$ is small; since however the lower bound on $r_0$ required in the previous steps of the proof depends on $e$, this requires a bit of care. Consider again the 1-form $\cd'_3$ in~\eqref{EqECcd3}; the second condition in~\eqref{EqCTimelikeSp} for $\zeta=\dd r$ and $\cd=\cd'_3$ reads
  \[
    1+\cO(r^{-1}) < \frac{2(1-e)}{e(1-v^2/4+\cO(r^{-1}))} (v^2/4+\cO(r^{-1})).
  \]
  Thus, there exist $e'>0$ and $r'_0>r_+$ (only depending on $\bhm$ and $v$) such that this inequality is verified for all $e\in(0,e')$ and $r\geq r'_0$. We then first fix this $r'_0$, and also fix $e<\min(e',e_2)$ so small that Lemma~\ref{LemmaCTimelikeSp} applies to $\cd'_0$ on $r^{-1}([\bhm,r'_0])$ (the second condition in~\eqref{EqCTimelikeSp} being automatic for sufficiently small $e>0$); thus $-\ell_{g,\cd'_0,e}(-\dd r)$ is positive definite in $r\leq r'_0$, and also for \emph{all} $r\geq r'_0$.

  With this fixed choice of $e$, we then select $r_0\geq r'_0$ sufficiently large such that all desired positivity conditions hold for $r\geq\frac12 r_0$; and then the piecewise smooth 1-form $\cd'$, given in $r\geq\frac12 r_0$ by~\eqref{EqECExpl}--\eqref{EqECExpl2} and~\eqref{EqECExpl3}, and for $\bhm\leq r\leq\frac12 r_0$ by $(2 r_0)^{-1}\cd'_0$ (see~\eqref{EqECGlue}), satisfies the positivity condition~\eqref{ItECPosEsc} everywhere. The only remaining issue is that $\cd'$ is not smooth, but merely continuous, at $r=\frac12 r_0$ and $r=2 r_0$. This is easily rectified by smoothing in $r\leq\frac12 r_0$ near $\frac12 r_0$ and in $r\geq 2 r_0$ near $2 r_0$. Note that for these $r$, the only condition on $\cd$ is condition~\eqref{ItECPosEsc}; but $\ell_{g,\cd,e}(\dd r)$, at any point $p\in M$, depends continuously on $\cd|_p$ (and, importantly, not on any derivatives of $\cd$), and hence a sufficiently small amount of smoothing does not destroy the positivity of $-\ell_{g,\cd,e}(\mp\dd r)$ already arranged there. This finishes the choice of $e$, $\delta_r$, $r_0$, and the construction of $\cd$.
\end{proof}

\subsection{Solvability of the constraint propagation wave operator on weighted 3b-spaces}
\label{SsET}

We can now state the global solvability statement for the constraint propagation wave operator on the Kerr spacetime $(M,g)=(M,g_{\bhm,a})$. We recall the scale of 3b-Sobolev spaces from~\S\ref{SsK3bh}, and the domain $\Omega$ from~\eqref{EqKDomain}.

\begin{thm}[Solvability]
\label{ThmET}
  Fix $v\in(0,1)$, $C_0>0$, and $\alpha_\sface,\alpha_\cK\in\R$ such that
  \begin{equation}
  \label{EqETThr}
    \alpha_\sface<-\frac12,\quad
    \alpha_\cK>\alpha_\sface+\frac12.
  \end{equation}
  Let $e\in(0,1)$ and $\cd\in\rho_\sface\CI(M;\cT^*)$ be given by Proposition~\usref{PropEC}. Define $\cC_h:=\cC_{g_{\bhm,a},\cd,e,h}$ and $\Box_g^{\cC_h}$ by~\eqref{EqCMod}--\eqref{EqCOp}. Then there exists $h_0>0$ such that for all $h\in(0,h_0)$ and $s\in[-C_0,C_0]$, the unique forward solution of the forcing problem
  \begin{equation}
  \label{EqET}
    \Box_g^{\cC_h}\omega = f \in \rho_\sface^{\alpha_\sface+2}\rho_\cK^{\alpha_\cK}\Htb^{s-1}(\Omega;\cT^*)^{\bullet,-}
  \end{equation}
  on $\Omega$ satisfies $\omega\in\rho_\sface^{\alpha_\sface}\rho_\cK^{\alpha_\cK}\Htb^s(\Omega;\cT^*)^{\bullet,-}$. Similarly, making the volume densities explicit in the notation, the solution of the initial value problem\footnote{The volume densities for the initial data in~\eqref{EqETIVP} are the Euclidean densities. (One could equally well have used the Minkowskian density instead of $|\dd g|$ in~\eqref{EqETIVP} and, implicitly, in~\eqref{EqET}.)}
  \begin{equation}
  \label{EqETIVP}
    \left\{
    \begin{alignedat}{2}
      \Box_g^{\cC_h}\omega&=f\in\rho_\sface^{\alpha_\sface+2}\rho_\cK^{\alpha_\cK}\bar H_\tbop^0(\Omega,|\dd g|;\cT^*), \\
      (\omega,r\cL_{\pa_t}\omega)|_X &\in \rho_\sface^{\alpha_\sface+\frac12}\Hbext^1(X,r^2\,|\dd r\,\dd\slg|;\cT^*_X) \oplus \rho_\sface^{\alpha_\sface+\frac12}\Hbext^0(X,r^2\,|\dd r\,\dd\slg|;\cT^*_X)
    \end{alignedat}
    \right.
  \end{equation}
  in $\Omega$ satisfies $\omega\in\rho_\sface^{\alpha_\sface}\rho_\cK^{\alpha_\cK}\bar H_\tbop^1(\Omega,|\dd g|;\cT^*)$.
\end{thm}

The solvability of $\Box_g^{\cC_h}\omega=f$ is automatic since $\Box_g^{\cC_h}$ is a differential operator and principally a wave operator. The point is that the solution lies in a space with the desired weights.

The mode stability statement, Theorem~\ref{ThmIRough}, is deduced from Theorem~\ref{ThmET} in~\S\ref{ST}. The rest of the present section is concerned with the proof of Theorem~\ref{ThmET}. We shall work with the operator $\Box_{g,h}^{\cC_h}=h^2\Box_g^{\cC_h}$. Recall also the operator $L_{g,\cd,e,h}$ from~\eqref{EqCOp}, so $\Box_{g,h}^{\cC_h}=h^2\Box_g-i L_{g,\cd,e,h}$. First, we justify working in 3b-phase space. 

\begin{lemma}[$\Box_{g,h}^{\cC_h}$ as a weighted semiclassical 3b-operator]
\label{LemmaET3b}
  We have
  \begin{equation}
  \label{EqET3bSeparate}
    h^2\Box_g\in h^2\rho_\sface^2\Difftb^2(M;\cT^*),\quad
    L_{g,\cd,e,h}\in h\rho_\sface^2\Difftb^1(M;\cT^*).
  \end{equation}
  In particular, in the notation of Lemma~\usref{LemmaCSymb}, the semiclassical 3b-principal symbol of $\Box_{g,h}^{\cC_h}$ is given by $\sigma_{\tbop,\semi}^2(\Box_{g,h}^{\cC_h}) = p_{g,\cd,e} = G - i\ell_{g,\cd,e}$.
\end{lemma}
\begin{proof}
  This follows from~\eqref{EqK3bMem} and $\cd\in\rho_\sface\CI(M;\cT^*)$.
\end{proof}

This implies $h^2\Box_g\in\rho_\sface^2\Diff_{\tbop,\semi}^2$ and $L_{g,\cd,e,h}\in\rho_\sface^2\Diff_{\tbop,\semi}^1$; but the more precise information~\eqref{EqET3bSeparate} will be important for the analysis near the zero section of $\Ttb^*M$ later on.

The proof of Theorem~\ref{ThmET} proceeds by proving (mostly microlocal) semiclassical estimates for $\Box_{g,h}^{\cC_h}$. In particular, for the forward solution of~\eqref{EqET}---thus for $\Box_{g,h}^{\cC_h}\omega=h^2 f$---we will in fact obtain uniform estimates (as $h\to 0$) in semiclassical 3b-Sobolev spaces,
\begin{equation}
\label{EqETQuant}
  \|\omega\|_{\rho_\sface^{\alpha_\sface}\rho_\cK^{\alpha_\cK}H_{\tbop,h}^s(\Omega)^{\bullet,-}} \leq C h^{-1}\|h^2 f\|_{\rho_\sface^{\alpha_\sface+2}\rho_\cK^{\alpha_\cK}H_{\tbop,h}^{s-1}(\Omega)^{\bullet,-}},
\end{equation}
similarly for the initial value problem. The steps in the proof are the following.

\begin{itemize}
\item In~\S\ref{SssEE}, we prove elliptic estimates at finite nonzero points of the semiclassical 3b-phase space, as well as complex absorption type estimates at infinite frequencies; note here that $\Box_{g,h}^{\cC_h}$ is a wave operator with a \emph{damping} term ($-i L_{g,\cd,e,h}$), cf.\ Lemma~\ref{LemmaCTimelike}, which is principal in the semiclassical sense. It remains to prove estimates near zero frequency, where $i\Box_{g,h}^{\cC_h}\approx L_{g,\cd,e,h}$ is to leading order a non-scalar transport operator.
\item In~\S\S\ref{SssEK}--\ref{SssES}, we prove propagation estimates near the three critical sets of the vector field $-r\cd^\sharp$ (i.e., near the saddle point in $(\cK^+)^\circ$, near the saddle point at $\pa\cK^+$, and near the sink over $\sface^\circ$); these have the character of radial point estimates, though some care is required since the leading order part of $\Box_{g,h}^{\cC_h}$ near the zero section, $L_{g,\cd,e,h}$, is not scalar.
\item Zero section propagation away from the critical sets, which has the character of real principal type propagation, is analyzed in~\S\ref{SssER}.
\item In order to obtain closed estimates near the finite boundary hypersurfaces $X$ and $\Sigma^\sharp$ of $\Omega$ (see Definition~\ref{DefKMfd}), we prove energy estimates for $\Box_{g,h}^{\cC_h}$ in~\S\ref{SssEI}. We again exploit the character of $\Box_{g,h}^{\cC_h}$ as a wave operator with semiclassical principal (but non-scalar) damping.
\item In~\S\ref{SssEP}, we combine all ingredients and prove Theorem~\ref{ThmET}.
\end{itemize}

The analysis is related to that of \cite[\S8]{HintzVasyKdSStability} and the more recent \cite{HintzPetersenVasyKdS}. The main differences are that we work in 3b-phase space over a manifold with corners, rather than in b-phase space over a manifold with boundary; and the dynamical properties of the flow of the constraint damping 1-form $\cd$ are more involved than in \cite{HintzVasyKdSStability,HintzPetersenVasyKdS} (e.g., $\cd$ has more critical sets). Furthermore, the positivity properties captured in Proposition~\ref{PropEC} are considerably more delicate than in \cite{HintzVasyKdSStability,HintzPetersenVasyKdS}.

We point out here already that we will ultimately prove solvability for $\Box_{g,h}^{\cC_h}\omega=h^2 f$ using a duality argument that is based on a priori estimates for the adjoint operator $(\Box_{g,h}^{\cC_h})^*$; these estimates are obtained by microlocally propagating in the opposite (i.e., past) direction. However, the forward propagation estimates are somewhat more intuitive and easier to parse, and hence we shall state these in~\S\S\ref{SssEE}--\ref{SssEI} in addition to the (ultimately more important) adjoint estimates.

For the remainder of this section, we shall fix $\alpha_\cK,\alpha_\sface,\cd,e$ according to the statement of Theorem~\ref{ThmET}. We abbreviate
\begin{equation}
\label{EqENotation}
\begin{gathered}
  P_h := \Box_{g,h}^{\cC_h} = \Box_h - i L_h,\quad
  \Box_h := h^2\Box_g,\quad
  L_h := L_{g,\cd,e,h}, \\
  p := p_{g,\cd,e},\quad
  \ell := \ell_{g,\cd,e},\quad
  B := B_{g,\cd,e}, \quad
  S := S_{g,\cd,e} = h^{-1}\Im L_h, \\
  \Htbh^{m,(\beta_\sface,\beta_\cK)} := \rho_\sface^{\beta_\sface}\rho_\cK^{\beta_\cK}\Htbh^m,
\end{gathered}
\end{equation}
where we recall that $p_{g,\cd,e}$ and $\ell_{g,\cd,e}$ are the semiclassical 3b-principal symbols of $P_h$ and $L_h$, respectively, and $B_{g,\cd,e}$ and $S_{g,\cd,e}$ are defined in~\eqref{EqCInnerProd} and Lemma~\ref{LemmaCSubpr}. We require the Schwartz kernels of all pseudodifferential operators to have supports whose closures in $M\times M$ are contained in $\Omega'\times\Omega'$ where $\Omega'=\Omega\setminus(X\cup\Sigma^\sharp)$. We shall work with the $L^2$-inner product on $M$, written $\la\cdot,\cdot\ra$, with respect to the metric volume density $|\dd g|$ and the fiber inner product $B$ on $\cT^*$. Finally, we shall drop the bundle $\cT^*$ from the notation.

\subsubsection{Elliptic estimates at nonzero and infinite semiclassical frequencies}
\label{SssEE}

Denote by
\[
  \Sigma := \{ \zeta\in\Ttb^*M \colon (\rho_\sface^{-2}G)(\zeta) = 0 \}
\]
the semiclassical 3b-characteristic set of the unweighted semiclassical 3b-operator $h^2\rho_\sface^{-2}\Box_g$; this is the closure in $\Ttb^*M$ of the double (dual) light cone $G^{-1}(0)\subset T^*M^\circ$. We denote the boundary of $\Sigma$ at fiber infinity by
\[
  \pa\Sigma \subset \Stb^*M \subset \pa(\ol{\Ttb^*}M).
\]
Since $(M,g)$ is time-orientable, $\pa\Sigma$ has two connected components $\pa\Sigma^\pm$ which are equal to the boundary at fiber infinity of the future (``$+$''), resp.\ past (``$-$'') light cones. The other submanifold of $\ol{\Ttb^*}M$ of main interest is the zero section,
\[
  o \subset \ol{\Ttb^*}M.
\]

\begin{lemma}[Characteristic set]
\label{LemmaEE}
  The semiclassical 3b-characteristic set of $\rho_\sface^{-2}P_h$ is equal to $\pa\Sigma\cup o$.
\end{lemma}
\begin{proof}
  The fact that $\rho_\sface^{-2}P_h$ is characteristic at the zero section is clear, since $\rho_\sface^{-2}G$ and $\rho_\sface^{-2}\ell$ are homogeneous (of degree $2$ and $1$, respectively) in the fibers of $\Ttb^*M$. For finite nonzero frequencies, the conclusion follows from Lemma~\ref{LemmaCInv}. At fiber infinity, $G$ is the principal term and $\ell$ is subprincipal, and hence the conclusion is immediate from the definition of $\Sigma$.
\end{proof}

Recall $\tau=\frac{t}{r}$; note that at $t=0$, the differential $\dd\tau=r^{-1}(\dd t-\tau\,\dd r)$ is past timelike, and this persists for $0\leq\tau\leq T$ when $T>0$ is sufficiently small. Denote by $\chi_{\tau,j}\in\CI(M)$, $\chi_{r,j}\in\CI(M)$, $j=0,1,2$, cutoff functions such that
\begin{equation}
\label{EqEECutoffs}
\begin{alignedat}{4}
  \chi_{\tau,j} &= 0\ \text{for}&\ \tau&\leq\frac{j+\frac14}{10}T, &\quad
  \chi_{\tau,j} &= 1\ \text{for}&\ \tau&\geq\frac{j+\frac34}{10}T, \\
  \chi_{r,j} &= 0\ \text{for}&\ r&\leq\bhm+\frac{j+\frac14}{10}(r_+-\bhm), &\quad
  \chi_{r,j} &= 1\ \text{for}&\ r&\geq\bhm+\frac{j+\frac34}{10}(r_+-\bhm).
\end{alignedat}
\end{equation}
Thus, $\chi_{\bullet,j}\equiv 1$ near $\supp\chi_{\bullet,j+1}$ for $\bullet=\tau,r$. We choose $\chi_{\tau,j}$ to only depend on $\tau$ and $\chi_{r,j}$ to only depend on $r$, and may moreover arrange that $\chi'_{\bullet,j}\geq 0$ and $(\chi_{\bullet,j}\chi'_{\bullet,j})^{\frac12}\in\CI$.

\begin{prop}[Elliptic estimates]
\label{PropEE}
  Let $B_0,B_1\in\Psitbh^0(M)$ be such that $\WFtbh'(B_j)\cap(o\cup\pa\Sigma)=\emptyset$, $j=0,1$, and $\WFtbh'(B_0)\subset\Elltbh(B_1)$. Then for all $s,N\in\R$ there exists a constant $C$ such that for all $\omega\in\rho_\sface^{\alpha_\sface}\rho_\cK^{\alpha_\cK}\Htbh^{-N}(M)$, we have
  \[
    \| B_0 \chi_{\tau,2}\chi_{r,2} \omega \|_{\Htbh^{s,(\alpha_\sface,\alpha_\cK)}} \leq C\Bigl( \|B_1\chi_{\tau,1}\chi_{r,1}P_h \omega \|_{\Htbh^{s-2,(\alpha_\sface+2,\alpha_\cK)}} + h^N\| \chi_{\tau,0}\chi_{r,0} \omega \|_{\Htbh^{-N,(\alpha_\sface,\alpha_\cK)}}\Bigr).
  \]
  The same holds true for $P_h^*$ in place of $P_h$.
\end{prop}
\begin{proof}
  Let $\omega'=\chi_{\tau,0}\chi_{r,0}\omega$, then $\chi_{\tau,1}\chi_{r,1}P_h \omega'=\chi_{\tau,1}\chi_{r,1}P_h \omega$ and $\chi_{\tau,2}\chi_{r,2}\omega=\chi_{\tau,2}\chi_{r,2}\omega'$. Thus, the estimate follows from the usual (symbolic) elliptic parametrix construction (cf.\ \cite[Proposition~6.61]{HintzMicro}).
\end{proof}

At $\pa\Sigma$, we do not have elliptic estimates. However, the skew-adjoint part of $P_h$, coming from $L_h$, has a sign by Lemma~\ref{LemmaCTimelike}, and even though it is subprincipal in the differential order sense, it is principal in the semiclassical sense. Therefore, we have semiclassical estimates with a loss of one (semiclassical 3b-)derivative but without a loss of a power of $h$:

\begin{prop}[Estimates at the characteristic set at fiber infinity]
\label{PropEEInfty}
  Let $B_0,B_1\in\Psitbh^0(M)$ be elliptic at $\pa\Sigma\cap\ol{\Ttb^*_{\supp(\chi_{\tau,0}\chi_{r,0})}}M$, and suppose that $\WFtbh'(B_0)\cap o=\emptyset$ and $\WFtbh'(B_0)\subset\Elltbh(B_1)$. Then for all $s,N\in\R$ there exist $h_0>0$ and $C<\infty$ such that
  \begin{equation}
  \label{EqEEInfty}
  \begin{split}
    &\|B_0\chi_{\tau,2}\chi_{r,2}\omega\|_{\Htbh^{s,(\alpha_\sface,\alpha_\cK)}} \\
    &\qquad \leq C\Bigl( \|B_1\chi_{\tau,1}\chi_{r,1} P_h \omega\|_{\Htbh^{s-1,(\alpha_\sface+2,\alpha_\cK)}} + h^{\frac12}\|B_1\chi_{\tau,0}(1-\chi_{\tau,2})\chi_{r,0}\omega\|_{\Htbh^{s,(\alpha_\sface,\alpha_\cK)}} \\
    &\qquad \hspace{4em} + h^N\|\chi_{\tau,0}\chi_{r,0}\omega\|_{\Htbh^{-N,(\alpha_\sface,\alpha_\cK)}}\Bigr)
  \end{split}
  \end{equation}
  for all $h\in(0,h_0)$. This estimate holds in the strong sense that if the norms on the right-hand side are finite, then so is the norm on the left-hand side, and the estimate holds. We have a similar estimate for the adjoint, which is valid in this strong sense as well:
  \begin{equation}
  \label{EqEEInftyAdj}
  \begin{split}
    &\|B_0\chi_{\tau,2}\chi_{r,2}\omega\|_{\Htbh^{-s+1,(-\alpha_\sface-2,-\alpha_\cK)}} \\
    &\qquad \leq C\Bigl( \|B_1\chi_{\tau,1}\chi_{r,1} P_h^*\omega\|_{\Htbh^{-s,(-\alpha_\sface,-\alpha_\cK)}} + h^{\frac12}\|B_1\chi_{\tau,0}\chi_{r,0}(1-\chi_{r,2})\omega\|_{\Htbh^{-s+1,(-\alpha_\sface-2,-\alpha_\cK)}} \\
    &\qquad \hspace{4em} + h^N\|\chi_{\tau,0}\chi_{r,0}\omega\|_{\Htbh^{-N,(-\alpha_\sface-2,-\alpha_\cK)}}\Bigr).
  \end{split}
  \end{equation}
\end{prop}

The weights $(-\alpha_\sface-2,-\alpha_\cK)$ on $\omega$ in~\eqref{EqEEInftyAdj} are dual to those on $P_h\omega$ in~\eqref{EqEEInfty}. Note also that the estimates~\eqref{EqEEInfty} and \eqref{EqEEInftyAdj} entail the elliptic estimates from Proposition~\ref{PropEE} for $P_h$ and $P_h^*$. The a priori control term in~\eqref{EqEEInfty}, resp.\ \eqref{EqEEInftyAdj}, is such that past-directed, resp.\ future-directed null-geodesics from $\supp(\chi_{\tau,2}\chi_{r,2})$ reach its support. Thus, \eqref{EqEEInfty}, resp.\ \eqref{EqEEInftyAdj} is a forward, resp.\ backward propagation estimate.

\begin{rmk}[Threshold regularity at the conormal bundle of the event horizon]
\label{RmkEEThr}
  The estimate~\eqref{EqEEInfty} controls $\omega$ microlocally also near the conormal bundle of the event horizon, where by comparison with \cite[\S{2}]{VasyMicroKerrdS} one might expect a radial point estimate to play a role and lower bounds on the regularity orders $s,-N$ to be required for the validity of~\eqref{EqEEInfty}. From this perspective, it is the smallness of $h$ that pushes the threshold regularity below any fixed number (such as $\min(s,-N)$).
\end{rmk}

\begin{proof}[Proof of Proposition~\usref{PropEEInfty}]
  The proof does not rely on any fine properties of the null-geodesic flow on $(M,g)$. Instead, it is based on a simple pairing argument which has the character of a positive commutator argument in which the ellipticity of the skew-adjoint part dominates, and only simple causal information is used at those places (near $\tau=0$ and $r=\bhm$) where the Hamiltonian vector field differentiates the localizers $\chi_{\tau,1}$, $\chi_{r,1}$.

  The details are as follows. We may replace $\omega$ by $\chi_{\tau,0}\chi_{r,0}\omega$. We only prove the result \emph{assuming} $\omega\in\Htbh^{s,(\alpha_\sface,\alpha_\cK)}$; the stronger statement follows from a standard regularization argument, see, e.g., \cite[\S4.4]{VasyMinicourse} or \cite[\S{8.5}]{HintzMicro}. Moreover, in view of Proposition~\ref{PropEE}, it suffices to prove an estimate in an arbitrarily small (but $h$-independent) neighborhood of $\pa\Sigma$.
  
  Instead of~\eqref{EqEEInfty}, we first prove the estimate
  \begin{equation}
  \label{EqEEInftyWeak}
    \|B_0\chi_{\tau,2}\chi_{r,2}\omega\|_{\Htbh^{s,(\alpha_\sface,\alpha_\cK)}} \leq C\Bigl( \|B_1\chi_{\tau,1}\chi_{r,1}P_h \omega\|_{\Htbh^{s-1,(\alpha_\sface+2,\alpha_\cK)}} + h^N\|\omega\|_{\Htbh^{s,(\alpha_\sface,\alpha_\cK)}}\Bigr)
  \end{equation}
  which does not require the use of \emph{any} properties of the null-geodesic flow. Fix a defining function $\hat\rho\in\CI(\ol{\Ttb^*}M)$ of fiber infinity $\Stb^*M$, and consider the rescaled dual metric function $\hat G=\hat\rho^2\rho_\sface^{-2}G\in\CI(\ol{\Ttb^*}M)$. Let $\chi\in\CIc(\R)$ denote a cutoff function which is identically $1$ near $0$. For simpler bookkeeping of signs, we shall also localize to the component $\pa\Sigma^\pm$ using a cutoff $\chi^\pm\in\CI(\ol{\Ttb^*}M)$ which is identically $1$, resp. $0$ in a fixed neighborhood of $\pa\Sigma^\pm$, resp.\ $\pa\Sigma^\mp$. With $\digamma>1$ to be specified, we define
  \begin{equation}
  \label{EqEEInftyA}
    a := \rho_\sface^{-\alpha_\sface-1}\rho_\cK^{-\alpha_\cK} \hat\rho^{-s+\frac12} \chi^\pm \chi(\digamma\hat\rho) \chi(\digamma\hat G) \chi_{\tau,1}\chi_{r,1}.
  \end{equation}
  We let $A=A^*\in\rho_\sface^{-\alpha_\sface-1}\rho_\cK^{-\alpha_\cK}\Psitbh^{s-\frac12}(M)$ denote a semiclassical quantization of $a$, with operator wave front set equal to $\supp a$ and with Schwartz kernel supported in a fixed small neighborhood of $(\pi(\supp a))^2$ where $\pi\colon\Ttb^*M\to M$ is the projection. Then
  \begin{equation}
  \label{EqEEInftyPair}
    {\mp}\Im\la P_h \omega,A^2 \omega\ra = \Big\la \Bigl(\pm(\Re L_h-\Im\Box_h)A^2 \mp \frac{i}{2}[P_h,A^2]\Bigr)\omega, \omega \Big\ra.
  \end{equation}
  For sufficiently large $\digamma$, Lemma~\ref{LemmaCTimelike} implies that the principal symbol $\pm\rho_\sface^{-2}\ell$ of $\pm\rho_\sface^{-2}\Re L_h\in\Difftbh^1$ is positive definite on $\supp a$. On the other hand, since $\Box_h$ is principally scalar with real principal symbol, the skew-adjoint part $\Im\Box_h\in h\Difftbh^1$ is semiclassically subprincipal. Therefore, fixing $\Lambda\in\rho_\sface\Psitbh^{\frac12}$ to be elliptic on $\supp a$, there exists $c_0>0$ such that we can write
  \begin{equation}
  \label{EqEEReL}
    {\pm}(\Re L_h-\Im\Box_h)A^2 = A(\tilde B_0^*\tilde B_0 + c_0\Lambda^*\Lambda) A + h E,\quad E\in\rho_\sface^{-2\alpha_\sface}\rho_\cK^{-2\alpha_\cK}\Psitbh^{2 s-1},
  \end{equation}
  where also $\tilde B_0\in\rho_\sface\Psitbh^{\frac12}$ is elliptic on $\supp a$. The commutator term in~\eqref{EqEEInftyPair} is
  \begin{equation}
  \label{EqEEInftyComm}
    {\mp}\frac{i}{2}[P_h,A^2] =: h F,\quad F\in\rho_\sface^{-2\alpha_\sface}\rho_\cK^{-2\alpha_\cK}\Psitbh^{2 s},
  \end{equation}
  i.e., of lower order in the semiclassical sense (powers of $h$).
  
  Applying Cauchy--Schwarz to the left-hand side of~\eqref{EqEEInftyPair}, we obtain the $L^2$-estimate
  \[
    \|\tilde B_0 A \omega\|^2 + c_0\|\Lambda A \omega\|^2 \leq \frac{1}{4 c_0}\|\Lambda^-A P_h \omega\|^2 + c_0\|\Lambda A \omega\|^2 + h|\la(E+F)\omega,\omega\ra| + C h^{2 N}\|\omega\|_{\Htbh^{-N,(\alpha_\sface,\alpha_\cK)}}^2
  \]
  for any fixed $N$, where $\Lambda^-\in\rho_\sface^{-1}\Psitbh^{-\frac12}$ is a microlocal parametrix of $\Lambda$ near $\supp a$. Canceling the second terms on both sides and taking square roots, we obtain
  \[
    \|\tilde B_0 A \omega\| \leq C\Bigl( \|\Lambda^-A P_h \omega\| + h^{\frac12}\|\tilde B_0'\omega\| + h^N\|\omega\|_{\Htbh^{-N,(\alpha_\sface,\alpha_\cK)}}\Bigr),
  \]
  where the term involving $\tilde B'_0\in\rho_\sface^{-\alpha_\sface}\rho_\cK^{-\alpha_\cK}\Psitbh^s$ (bounding the pairing $h|\la(E+F)\omega,\omega\ra|$, with $\tilde B'_0$ elliptic on $\supp a$) can be made to have Schwartz kernel and operator wave front set arbitrarily close to those of $A$ (with the constant $C$ depending on the particular choice of $\tilde B'_0$). An application of microlocal elliptic regularity yields the estimate~\eqref{EqEEInftyWeak} except for the presence of an additional error term $C h^{\frac12}\|\tilde B'_0 \omega\|$ on the right-hand side. Note then that $\tilde B'_0$ is microlocalized near $\supp a$, and hence one can proceed inductively, estimating $\|\tilde B'_0 \omega\|$ using the same argument (with suitably adapted cutoffs), and improving the power of $h$ of the error term by $\frac12$ at each step. This establishes~\eqref{EqEEInftyWeak} in the stated form after finitely many iterations.

  One can similarly establish~\eqref{EqEEInftyWeak} with $P_h^*$ in place of $P_h$ and $-\alpha_\sface-2,-\alpha_\cK$ in place of $\alpha_\sface,\alpha_\cK$, and with the same replacement in the definition of the commutant $a$ in~\eqref{EqEEInftyA}. Using that $P_h^*=\Box_h^*+i L_h^*$ (with $\Box_h^*-\Box_h\in h\Difftbh^1$ and $L_h^*-L_h\in h\Difftbh^0$), we have, in place of~\eqref{EqEEInftyPair},
  \begin{equation}
  \label{EqEEInftyAdjPair}
    {\mp}\Im\la P_h^*\omega,A^2\omega\ra = \Big\la \Bigl({\mp}(\Re L_h-\Im\Box_h)A^2 \mp \frac{i}{2}[P_h^*,A^2]\Bigr)\omega, \omega\Big\ra.
  \end{equation}
  The only (at this point inconsequential) change is that the principal symbol $\mp\rho_\sface^{-2}\ell$ of the main term $\mp\rho_\sface^{-2}\Re L_h$ is now \emph{negative} definite on $\supp a$.

  In order to prove the original estimate~\eqref{EqEEInfty}, we improve~\eqref{EqEEInftyComm} to the statement
  \begin{equation}
  \label{EqEEInftyComm2}
    {\mp}\frac{i}{2}[P_h,A^2] = -h E_\tau^*E_\tau + h E_r^*E_r + h A\tilde B_2 A + h A Q P_h + h F',
  \end{equation}
  where the operators
  \begin{alignat*}{2}
    \qquad E_\tau, E_r &\in \rho_\sface^{-\alpha_\sface}\rho_\cK^{-\alpha_\cK}\Psitbh^s, &\quad
    \tilde B_2 &\in \rho_\sface^2\Psitbh^1, \\
    Q &\in \rho_\sface^{-\alpha_\sface-1}\rho_\cK^{-\alpha_\cK}\Psitbh^{s-\frac32}, &\quad
    F' &\in \rho_\sface^{-2\alpha_\sface}\rho_\cK^{-2\alpha_\cK}\Psitbh^{2 s-1}.
  \end{alignat*}
  have operator wave front sets contained in $\supp a$, and with $\WFtbh'(E_\tau)$ in addition lying over $\supp(\chi_{r,1}\dd\chi_{\tau,1})$. The point is that the error term $h F'$ has lower differential order than $h F$ in~\eqref{EqEEInftyComm}, and thus is of the same class as $h E$ in~\eqref{EqEEReL}; all other terms on the other hand (apart from the term $E_\tau$ which will necessitate the second, a priori control, term on the right in~\eqref{EqEEInfty}) either have a sign matching that of the main term~\eqref{EqEEReL} or can be controlled by $\tilde B_0 A \omega$ or $P_h \omega$. Indeed, assuming~\eqref{EqEEInftyComm2}, we obtain from~\eqref{EqEEInftyPair} and~\eqref{EqEEReL} the estimate
  \begin{equation}
  \label{EqEEInftyEst}
    \|\tilde B_0 A \omega\|^2 \leq C\Bigl( \|\Lambda^- A P_h \omega\|^2 + h\|\tilde B_0' \omega\|^2 + h\|E_\tau \omega\|^2 + h^{2 N}\|\omega\|_{\Htbh^{-N,(\alpha_\sface,\alpha_\cK)}}^2 \Bigr)
  \end{equation}
  (where now $\tilde B_0'\in\rho_\sface^{-\alpha_\sface}\rho_\cK^{-\alpha_\cK}\Psitbh^{s-\frac12}$ is elliptic on $\supp a$, similarly to before, but now with reduced differential order) upon dropping the term involving $E_r$ (which has the same sign as the main term $\|\tilde B_0 A \omega\|^2$), and upon absorbing the first term on the right-hand side of $h|\la\tilde B_2 A \omega,A \omega\ra|\leq C h\|\tilde B_0 A \omega\|^2+h^N\|\omega\|_{\Htbh^{-N,(\alpha_\sface,\alpha_\cK)}}^2$ into the left-hand side of~\eqref{EqEEInftyEst} (for sufficiently small $h$). Taking the square root of~\eqref{EqEEInftyEst} and bounding $h^{\frac12}\|\tilde B_0'\omega\|$ in an iterative fashion yields~\eqref{EqEEInfty}.

  Now,~\eqref{EqEEInftyComm2} follows from a simple symbolic calculation. The term $-i L_h$ of $P_h$ contributes only to the error term $h F'$. It remains to analyze the principal symbol of $\mp\frac{i}{2 h}[\Box_h,A^2]$, which is given by $\mp a H_G a$. Evaluating this using the form~\eqref{EqEEInftyA} of $a$, the terms coming from differentiation of $\chi^\pm$ and $\chi(\digamma\hat G)$ are supported away from the characteristic set of $P_h$ and can thus be written as a product $a q p$ (and then $Q$ is a quantization of $q$). Differentiation of $\chi(\digamma\hat\rho)$ gives a symbol of differential order $-\infty$ which we put into $F'$. Differentiation of the weights $\rho_\sface^{-\alpha_\sface-1}$, $\rho_\cK^{-\alpha_\cK}$, and $\hat\rho^{-s+\frac12}$ of $a$ yields smooth bounded multiples of $a$, giving rise to the term $A\tilde B_2 A$ in~\eqref{EqEEInftyComm2}. Finally then, the term coming from $\mp\chi_{\tau,1}H_G\chi_{\tau,1}=\mp\chi_{\tau,1}\chi'_{\tau,1} H_G \tau=-e_T^2$ is a negative square since $\pm H_G \tau>0$ at future (top sign), resp.\ past (bottom sign) lightlike covectors. Similarly, $\mp\chi_{r,1}H_G\chi_{r,1}=\mp\chi_{r,1}\chi'_{r,1} H_G r=e_R^2$ since $\pm H_G r<0$ at future, resp.\ past lightlike covectors over $\supp\dd\chi_{r,1}\subset r^{-1}([\bhm,r_+))$; this is due to the future timelike nature of $\dd r$ in $r<r_+$ (where $\Delta<0$ in the expression~\eqref{EqKMetStar} for the dual metric). This finishes the proof of~\eqref{EqEEInftyComm2}.

  The proof of~\eqref{EqEEInftyAdj} is based on the fact that~\eqref{EqEEInftyComm2} holds for $P_h^*$ in place of $P_h$ as well, with the operators $E_\tau,E_r,\tilde B_2,Q,F'$ having the same properties; indeed, only the second order term $\Box_h^*$ of $P_h^*$ mattered in arranging~\eqref{EqEEInftyComm2}, and this has the same semiclassical principal symbol $G$ as $\Box_h$. Recall however that the main, conclusion, term (involving $\Re L_h$) in~\eqref{EqEEInftyAdjPair} now has the opposite sign (as compared to the main term in~\eqref{EqEEInftyPair}). Hence, now it is the term $E_\tau$ that has the same sign as this main term, whereas the term $E_r$ requires an a priori control term---the second term on the right in~\eqref{EqEEInftyAdj}. The proof is complete.
\end{proof}

\subsubsection{Zero section propagation near the saddle point at \texorpdfstring{$r=r_0$}{r=r0}}
\label{SssEK}

Near the zero section $o\subset\Ttb^*M$, the term $-i L_h$ of $P_h$ dominates, as its semiclassical principal symbol vanishes simply at $o$, whereas the semiclassical principal symbol of $\Box_h$ vanishes quadratically at $o$. By Corollary~\ref{CorCPos}, applied at each point of the (compact) domain $\Omega$, there exists $C>0$ such that $C(\rho_\sface^{-2}\ell(\zeta))^2+\rho_\sface^{-2}G(\zeta)$ is positive definite (with respect to $B$) for all $0\neq\zeta\in\Ttb^*_\Omega M$. Therefore, the elliptic multiple
\[
  -(1+i C\rho_\sface^{-2}\ell) i(G-i\ell) = -\ell(1-C\rho_\sface^{-2}G) - i(C\rho_\sface^{-2}\ell^2+G)
\]
is, near $o$, given roughly by the non-scalar transport (cf.\ Remark~\ref{RmkCNoe0}) along $-\ell$ (roughly a collection of \emph{future} timelike vector fields in view of the minus sign) with complex absorption $C\rho_\sface^{-2}\ell^2+G$ (which is weak in the sense that its principal symbol vanishes quadratically at $o$). On the level of operators, we consider
\begin{equation}
\label{EqEKcL}
  \cL_h := -i(1+i C L_h^*\rho_\sface^{-2})P_h = -L_h + J - i Q \in \rho_\sface^2\Difftbh^3,
\end{equation}
where $J=J^*$ and $Q=Q^*$ are given by
\begin{equation}
\label{EqEKDecomp}
\begin{alignedat}{2}
  J &= h^2 J_0 + h^3 J_1, \\
  &\qquad J_0=J_0^*=h^{-2}\Im\Box_h=\Im\Box_g&&\in \rho_\sface^2\Difftb^1, \\
  &\qquad J_1 =J_1^*=C h^{-3} \Re\bigl( L_h^*\rho_\sface^{-2}\Box_h \bigr)&&\in\rho_\sface^2\Difftb^3, \\
  Q &= h^2 Q_0 + h^3 Q_1, \\
  &\qquad Q_0=Q_0^*=h^{-2}\bigl(\Re\Box_h + C L_h^*\rho_\sface^{-2}L_h\bigr) &&\in \rho_\sface^2\Difftb^2, \\
  &\qquad Q_1=Q_1^*=-C h^{-3}\Im\bigl(L_h^*\rho_\sface^{-2}\Box_h\bigr) &&\in \rho_\sface^2\Difftb^2.
\end{alignedat}
\end{equation}
We used Lemma~\ref{LemmaET3b} and the self-adjointness of the principal symbol of $L_h$ for the stated memberships.

Since the weight at $\sface$ and the 3b-differential order are irrelevant (i.e., arbitrary) in the analysis near $(\cK^+)^\circ$ and the zero section $o\subset\Ttb^*M$, we omit them from the notation. Furthermore, we shall, only in this section, use $\rho_\cK=t^{-1}$ as the local defining function of $\cK^+$. Denote by $|\cdot|$ a positive definite norm on $\Ttb^*M$.

\begin{prop}[Propagation near the saddle point in $(\cK^+)^\circ$]
\label{PropEK}
  There exists $\delta>0$ such that the following holds for $K:=\{|r-r_0|\leq\frac32 r_0\sqrt{e}/v,\ \rho_\cK\leq 4\delta\}$, for any cutoff function $\chi\in\CI(M)$ which is identically $1$ on $K$ and supported in any small neighborhood of $K$, for any operators $B_0,B_1,E\in\Psitbh(M)$ with operator wave front sets which are compact subsets of $\Ttb^*M$ and with Schwartz kernels which are supported in both factors in $K$, and for all sufficiently small $h>0$.
  \begin{enumerate}
  \item\label{ItEKFw}{\rm (Estimate for $P_h$.)} Suppose that
  \begin{alignat*}{3}
    \WFtbh'(B_0) &\subset \bigl\{ |r-r_0|\leq\tfrac{55}{50}r_0\sqrt{e}/v,\ &&\rho_\cK\leq 2\delta,\ &&|\zeta|<1 \bigr\}, \\
    \Elltbh(E) &\supset \bigl\{ |r-r_0|\leq\tfrac{60}{50}r_0\sqrt{e}/v,\ &&\delta\leq\rho_\cK\leq 3\delta,\ &&|\zeta|<1 \bigr\}, \\
    \Elltbh(B_1) &\supset \bigl\{ |r-r_0|\leq\tfrac{60}{50}r_0\sqrt{e}/v,\ &&\rho_\cK\leq 3\delta,\ &&|\zeta|<2 \bigr\}.
  \end{alignat*}
  Then for any $N\in\R$ there exists a constant $C$ such that
  \begin{equation}
  \label{EqEK}
    \|B_0 \omega\|_{\rho_\cK^{\alpha_\cK}L^2} \leq C\Bigl( h^{-1}\|B_1 P_h \omega\|_{\rho_\cK^{\alpha_\cK}L^2} + \|E \omega\|_{\rho_\cK^{\alpha_\cK}L^2} + h^N\|\chi \omega\|_{\rho_\cK^{\alpha_\cK}\Htbh^{-N}} \Bigr).
  \end{equation}
  \item\label{ItEKBw}{\rm (Estimate for $P_h^*$.)} Suppose that $B_0,B_1$ are as in part~\eqref{ItEKFw}, while now
    \[
      \Elltbh(E) \supset \{ |r-r_0| \in [\tfrac{55}{50}r_0\sqrt{e}/v,\tfrac{60}{50}r_0\sqrt{e}/v],\ \rho_\cK\leq 2\delta,\ |\zeta|<1 \}.
    \]
    Then for any $N\in\R$ there exists a constant $C$ such that
    \begin{equation}
    \label{EqEKAdj}
      \|B_0\omega\|_{\rho_\cK^{-\alpha_\cK}L^2} \leq C\Bigl( h^{-1}\|B_1 P_h^* \omega\|_{\rho_\cK^{-\alpha_\cK}L^2} + \|E \omega\|_{\rho_\cK^{-\alpha_\cK}L^2} + h^N\|\chi \omega\|_{\rho_\cK^{-\alpha_\cK}\Htbh^{-N}} \Bigr).
    \end{equation}
  \end{enumerate}
\end{prop}

In practice, we take $B_0$ to satisfy
\[
  \Elltbh(B_0) \supset \{ |r-r_0|\leq\tfrac{54}{50}r_0\sqrt{e}/v,\ \rho_\cK\leq\delta,\ |\zeta|<\tfrac12\}.
\]
In the estimate~\eqref{EqEK}, we may take $E$ to have operator wave front set localized near $o$, and with the projection of the support of its Schwartz kernel to either factor contained in $\rho_\cK^{-1}([\frac12\delta,4\delta])$, and hence disjoint from $\cK^+$. Thus~\eqref{EqEK} gives control in a $\frac{54}{50}r_0\sqrt{e}/v$-neighborhood of $r=r_0$ at $\cK^+$, provided we have a priori control on the set $\WFtbh'(E)$; this is a forward propagation estimate.

In the estimate~\eqref{EqEKAdj}, we will instead take $E$ to satisfy
\[
  \WFtbh'(E) \subset \{ |r-r_0|\in[\tfrac{54}{50} r_0\sqrt{e}/v,\tfrac32 r_0\sqrt{e}/v],\ \rho_\cK\leq 3\delta,\ |\zeta|\leq 2 \}.
\]
Thus,~\eqref{EqEKAdj} propagates control from outside a $\frac{54}{50}r_0\sqrt{e}/v$-neighborhood of $r=r_0$ at $\cK^+$ into this neighborhood; this is a backward propagation estimate. See Figure~\ref{FigEK}.

\begin{figure}[!ht]
\centering
\includegraphics{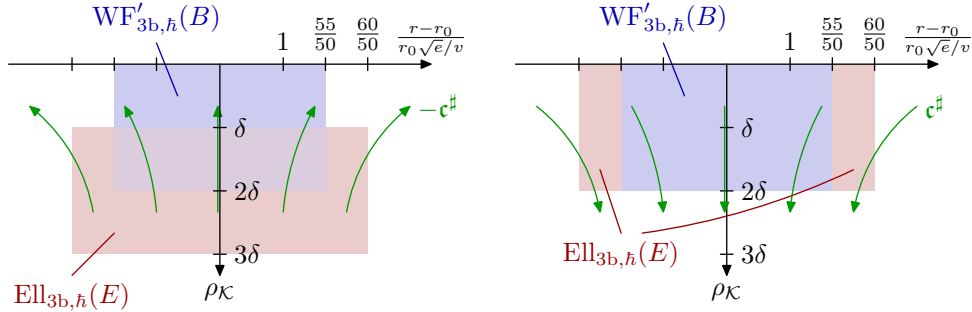}
\caption{\textit{On the left:} illustration of the estimate~\eqref{EqEK}. The arrows indicate the vector field $-\cd^\sharp$ along which, roughly speaking (cf.\ Remark~\ref{RmkCNoe0}), we propagate control near the zero section $o\subset\Ttb^*M$. \textit{On the right:} illustration of the estimate~\eqref{EqEKAdj}. The arrows now indicate the past timelike vector field $\cd^\sharp$ along which, roughly speaking, we now propagate.}
\label{FigEK}
\end{figure}

\begin{proof}[Proof of Proposition~\usref{PropEK}]
  \pfstep{Part~\eqref{ItEKFw}.} Since at finite points of $\Ttb^*M$ but away from the zero section $o$ we have elliptic estimates by Proposition~\ref{PropEE}, it suffices to work in an arbitrarily small neighborhood of $o$. Moreover, the operator $1+i C L_h^*\rho_\sface^{-2}$ is semiclassically elliptic near $o$ (and in fact on all of $\Ttb^*M$ since the principal symbol of $L_h^*$ is self-adjoint), and hence it suffices to prove the estimate~\eqref{EqEK} with $\cL_h$ (from~\eqref{EqEKcL}) in place of $P_h$.
  
  Let $\chi_0,\chi_\cK\in\CIc([0,3))$ be identically $1$ near $[0,1]$, and with $\chi_\cK'\leq 0$ and $(-\chi_\cK\chi_\cK')^{\frac12}\in\CI([0,3))$. Fix furthermore a cutoff
  \[
    \psi=\psi(r)\in\CIc\bigl((r_0-\tfrac{60}{50} r_0\sqrt{e}/v,r_0+\tfrac{60}{50} r_0\sqrt{e}/v)\bigr)
  \]
  with the following properties:
  \begin{enumerate}
  \item $\psi(r)=1$ for $|r-r_0|\leq \tfrac{55}{50}r_0\sqrt{e}/v$;
  \item $\pm\psi'(r)\leq -c<0$ for
    \begin{equation}
    \label{EqEKEll}
      {\pm}(r-r_0)\in\bigl[\tfrac{56}{50}r_0\sqrt{e}/v,\tfrac{59}{50}r_0\sqrt{e}/v\bigr];
    \end{equation}
  \item $0\leq\psi(r)\leq 1$ and $\pm\psi'(r)\leq 0$ for $\pm(r-r_0)\geq 0$;
  \item $(\mp\psi\psi')^{\frac12}$ is smooth in $\pm(r-r_0)\geq 0$.
  \end{enumerate}
  We then consider the commutant
  \begin{equation}
  \label{EqEKComm}
    a = \rho_\cK^{-\alpha_\cK} \psi(r)\chi_\cK(\digamma\rho_\cK)\chi_0(\digamma|\zeta|)
  \end{equation}
  and its quantization $A=A^*\in\rho_\cK^{-\alpha_\cK}\Psitbh$ with $\WFtbh'(A)=\supp a$. (The differential order is arbitrary since we work at finite semiclassical 3b-frequencies. For concreteness, one may take it to be $-\infty$, say.) We then compute the $L^2$-pairing
  \begin{equation}
  \label{EqEKPair}
    h^{-1}\Im \la \cL_h \omega,A^2 \omega\ra =\la\sC \omega,\omega\ra,
  \end{equation}
  where
  \begin{align}
    \sC &= \frac{i}{2 h}[\cL_h,A^2] + h^{-1}(\Im\cL_h)A^2 \nonumber\\
  \label{EqEKC}
      &= \Bigl(\frac{i}{2 h}[-L_h+J,A^2] - h^{-1}(\Im L_h)A^2\Bigr) - h^{-1}A Q A - \frac{1}{2 h}[A,[A,Q]] \in \rho_\cK^{-2\alpha_\cK}\Psitbh.
  \end{align}
  Write $j=\upsigma_{\tbop,\hbar}(J)$; in view of~\eqref{EqEKDecomp}, this is homogeneous of degree $3$ and thus vanishes cubically at the zero section. (Note that $h^2 J_0\in h\Difftbh^1$ is semiclassical subprincipal, unlike $h^3 J_1\in\Difftbh^3$.) The principal symbol of the term in parentheses in~\eqref{EqEKC} is then (using the notation~\eqref{EqENotation}) given by
  \[
    a H_{-\ell+j}a - S a^2;
  \]
  and since $\ell$ is linear in the momentum variables, we have $H_{-\ell}f=-\ell(\dd f)$ when $f$ is the lift of a function on $M$. Therefore, we can write
  \begin{equation}
  \label{EqEKSymbol}
    a H_{-\ell+j}a - S a^2 = -c_0 a^2 - b_0^2 - b_R^2 + e^2 + f,
  \end{equation}
  where, for $c_0>0$ to be chosen momentarily, we define the $\End(\cT^*)$-valued symbols
  \begin{align}
  \label{EqEKb0}
    b_0 &:= \rho_\cK^{-\alpha_\cK}\psi\chi_\cK\chi_0\Bigl(S-c_0-\alpha_\cK\rho_\cK\ell\Bigl(\frac{\dd\rho_\cK}{\rho_\cK^2}\Bigr) - \alpha_\cK\rho_\cK H_j(\rho_\cK^{-1})\Bigr)^{\frac12}, \\
  \label{EqEKbR}
    b_R &:= \rho_\cK^{-\alpha_\cK}\chi_\cK\chi_0\bigl( \psi \psi'(\ell(\dd r)-H_j r)\bigr)^{\frac12}, \\
  \label{EqEKe}
    e &:= \rho_\cK^{-\alpha_\cK+1}\psi\chi_0\Bigl(-\digamma\chi_\cK\chi_\cK'\Bigl[\ell\Bigl(\frac{\dd\rho_\cK}{\rho_\cK^2}\Bigr)+H_j(\rho_\cK^{-1})\Bigr]\Bigr)^{\frac12}, \\
  \label{EqKf}
    f &:= \rho_\cK^{-2\alpha_\cK}\psi^2\chi_\cK^2\chi_0 H_{-\ell+j}\chi_0.
  \end{align}
  We consider each of these four terms separately:
  \begin{itemize}
  \item On $\supp(\psi\chi_\cK)$, the endomorphism $S$ is positive definite by Proposition~\ref{PropEC}\eqref{ItECPosInt}, and hence bounded from below by $2 c_0$ for some $c_0>0$. Since $\rho_\cK^{-1}=t$ and $\frac{\dd\rho_\cK}{\rho_\cK^2}=-\dd t$, the terms $\ell(\frac{\dd\rho_\cK}{\rho_\cK^2})$ and $H_j\rho_\cK^{-1}$ are smooth (in particular bounded) down to $(\cK^+)^\circ$ on $\supp\chi_0(\digamma|\zeta|)$. Therefore, the argument of the square root in~\eqref{EqEKb0} is bounded from below (as a self-adjoint endomorphism of $\cT^*$) by $c_0>0$ provided $\digamma$ is chosen sufficiently large (and hence $\rho_\cK\lesssim\digamma^{-1}$ is small).
  \item Write the term involving $\ell$ in the argument of the square root in~\eqref{EqEKbR} as $\mp\psi\psi'\cdot(-\ell(\pm\dd r))$. Observe that for $r\in\supp(\psi')$ with $\pm(r-r_0)>0$, the endomorphism $-\ell(\pm\dd r)$ is positive definite by Proposition~\ref{PropEC}\eqref{ItECPosEsc}, and hence has a positive lower bound $2 c_1>0$. The contribution of $H_j$ vanishes quadratically at $\zeta=0$ since $|j|=\cO(|\zeta|^3)$, and hence for sufficiently large $\digamma$, we have $|H_j r|<c_1$ on $\supp\chi_0(\digamma|\zeta|)$. This proves that $b_R$ is well-defined.
  \item The term $e$ is supported away from $\rho_\cK=0$ and gives rise to the a priori control term in~\eqref{EqEK}. It is well-defined since $\frac{\dd\rho_\cK}{\rho_\cK^2}=-\dd t$ is future timelike; we must also use that $H_j\rho_\cK^{-1}$ is smooth and vanishes quadratically at $\zeta=0$, and take $\digamma$ sufficiently large.
  \item Finally, $f$ is supported at finite nonzero semiclassical 3b-frequencies where we already have elliptic estimates by Proposition~\ref{PropEE}.
  \end{itemize}
  This finishes the symbolic analysis of the first term in~\eqref{EqEKC}.

  The second term in~\eqref{EqEKC} features the mild complex absorption $Q=h^2 Q_0+h^3 Q_1$, see~\eqref{EqEKDecomp}. Since $Q_0\in\Difftb^2$ (dropping $\sface$-weights), as a \emph{non-semiclassical} 3b-operator, has a positive definite principal symbol, we can express it as
  \[
    Q_0=\Theta^*\Theta+R,\quad \Theta\in\Psitb^1,\ R\in\Psitb^{-\infty}.
  \]
  This allows us to estimate
  \begin{equation}
  \label{EqEKErr0}
    -h^{-1}\la A h^2 Q_0 A \omega,\omega\ra = -h\|\Theta A \omega\|^2 - h\la R A \omega,A \omega\ra \leq C h\|A \omega\|^2
  \end{equation}
  in view of the $L^2$-boundedness of $R$. For the term involving $Q_1\in\Difftb^2$, we note that $h^3 Q_1$ (which is schematically of the form $h(h D_\tbop)^2+h^2(h D_\tbop)+h^3$ where $D_\tbop$ is a 3b-derivative) has the property that the semiclassical 3b-principal symbol $q_{1 0}$ of $h^{-1}h^3 Q_1$ vanishes quadratically at the zero section. Therefore, $|q_{1 0}|\leq C\digamma^{-2}$ on $\supp a$ for some constant $C>0$ which is independent of $\digamma$, and hence G\aa{}rding's inequality gives
  \begin{equation}
  \label{EqEKErr2}
    \bigl|\Re\bigl( h^{-1}\la A h^3 Q_1 A \omega,\omega\ra\bigr)\bigr| \leq \bigl( 2 C\digamma^{-2} + C_\digamma h\bigr)\|A \omega\|^2.
  \end{equation}

  The last term in~\eqref{EqEKC} finally lies in $h\rho_\cK^{-2\alpha_\cK}\Psitbh$ (since $Q\in\Difftbh^2$, and each commutator gains a power of $h$) and has operator wave front set contained in $\WFtbh'(A)$. Its contribution to the right-hand side of~\eqref{EqEKPair} is therefore bounded from above by
  \begin{equation}
  \label{EqEKErr3}
    \Bigl|\Re\Big\la\frac{1}{2 h}[A,[A,Q]]\omega,\omega\Big\ra\Bigr| \leq C\Bigl(h\|\tilde B_0 \omega\|_{\rho_\cK^{\alpha_\cK}L^2}^2 + h^N\|\chi \omega\|^2_{\rho_\cK^{\alpha_\cK}\Htbh^{-N}}\Bigr),
  \end{equation}
  where
  \[
    \tilde B_0\in\Psitbh(M),\quad \Elltbh(\tilde B_0)\supset\WFtbh'(A),
  \]
  and $\WFtbh'(\tilde B_0)$ is contained in an arbitrary but fixed neighborhood of $\supp a$.

  In summary, denoting by\footnote{For notational brevity, we incorporate the $\cK^+$-weight into the operators here, unlike in the estimate~\eqref{EqEK}.} $B_0,B_R,E\in\rho_\cK^{-\alpha_\cK}\Psitbh$ and $F\in\rho_\cK^{-2\alpha_\cK}\Psitbh$ quantizations of the corresponding lower case symbols~\eqref{EqEKb0}--\eqref{EqKf}, we obtain from~\eqref{EqEKPair}, \eqref{EqEKSymbol} and \eqref{EqEKErr0}--\eqref{EqEKErr3} the $L^2$-estimate
  \begin{align*}
    c_0\|A \omega\|^2 + \|B_0 \omega\|^2 + \|B_R \omega\|^2 &\leq c_0^{-1}h^{-2}\|A\cL_h \omega\|^2 + \bigl(\tfrac14 c_0+(C+C_\digamma)h + 2 C\digamma^{-2}\bigr)\|A \omega\|^2 \\
      &\qquad + \|E \omega\|^2 + |\la F \omega,\omega\ra| + C h\|\tilde B_0 \omega\|_{\rho_\cK^{\alpha_\cK}L^2}^2 + C h^N\|\chi\omega\|_{\rho_\cK^{\alpha_\cK}H_{\tbop,h}^{-N}}^2,
  \end{align*}
  where the last two terms also take care of the subprincipal terms not captured by the principal symbol calculation~\eqref{EqEKSymbol}. First fixing $\digamma$ to be sufficiently large and then $h$ to be sufficiently small, we can arrange for the prefactor of $\|A \omega\|^2$ on the right to be less than $c_0$, and hence this term can be absorbed into the first term on the left. Since $\WFtbh'(F)\subset\Elltbh(\cL_h)$, we can further estimate
  \begin{equation}
  \label{EqEKFEst}
    |\la F \omega,\omega\ra| \leq C\Bigl( \|\tilde B_0\cL_h \omega\|_{\rho_\cK^{\alpha_\cK}L^2}^2 + h^N\|\chi \omega\|^2_{\rho_\cK^{\alpha_\cK}\Htbh^{-N}} \Bigr).
  \end{equation}
  Upon simplifying and dropping the term involving $B_R$, we have now proved
  \begin{equation}
  \label{EqEKHalf}
    \|B_0 \omega\| \leq C\Bigl( h^{-1}\|\tilde B_0\cL_h \omega\|_{\rho_\cK^{\alpha_\cK}L^2} + \|E \omega\| + h^{\frac12}\|\tilde B_0 \omega\|_{\rho_\cK^{\alpha_\cK}L^2} + h^N\|\chi \omega\|_{\rho_\cK^{\alpha_\cK}\Htbh^{-N}}\Bigr),
  \end{equation}
  thus improving the control of $\omega$ by one power of $h^{\frac12}$. Arranging that $\tilde B_0$ has operator wave front set contained in a very small neighborhood of $\supp a$, we can now apply the same argument to estimate $\tilde B_0 \omega$ upon enlarging the supports of the cutoffs by an arbitrarily small but fixed amount. After a finite number of steps, each of which improves the error term by one power of $h^{\frac12}$, we obtain the estimate~\eqref{EqEK}; the quantity $\delta$ in the statement can be taken to be $\digamma^{-1}$.

  \pfstep{Part~\eqref{ItEKBw}.} We may replace $P_h^*$ by $\cL_h^*=i P_h^*(1-i C\rho_\sface^{-2}L_h)=-L_h^*+J+i Q$ in the notation of~\eqref{EqEKcL}, since $1-i C\rho_\sface^{-2}L_h$ is elliptic, and hence microlocally invertible, near the zero section. We consider the commutant
  \[
    a = \rho_\cK^{\alpha_\cK}\psi(r)\chi_\cK(\digamma\rho_\cK)\chi_0(\digamma|\zeta|).
  \]
  (We switch the sign in the weight compared to~\eqref{EqEKComm} in order to obtain an estimate on spaces dual to those in~\eqref{EqEK}.) We then consider the $L^2$-pairing
  \begin{equation}
  \label{EqEKAdjPair}
    h^{-1}\Im\la\cL_h^* \omega,A^2 \omega\ra=\la\sC \omega,\omega\ra
  \end{equation}
  analogously to~\eqref{EqEKPair}, where $\sC$ is given by the expression~\eqref{EqEKC} except with $L_h$, $Q$ replaced by $L_h^*$, $-Q$, so
  \begin{equation}
  \label{EqEKAdjC}
    \sC = \Bigl(\frac{i}{2 h}[-L_h^*+J,A^2]+h^{-1}(\Im L_h)A^2\Bigr) + h^{-1}A Q A + \frac{1}{2 h}[A,[A,Q]] \in \rho_\cK^{2\alpha_\cK}\Psitbh.
  \end{equation}
  We can estimate $h^{-1}\la A Q A \omega,\omega\ra\geq -(2 C\digamma^{-2}+(C+C_\digamma)h)\|A \omega\|^2$ using~\eqref{EqEKErr0}--\eqref{EqEKErr2} (but now getting a lower bound), and we still have the estimate~\eqref{EqEKErr3} for the double commutator term. The first term of $\sC$ has semiclassical principal symbol
  \[
    a H_{-\ell+j}a + S a^2 = c_0 a^2 + b_0^2 + b_\cK^2 - e^2 + f,
  \]
  where for small $c_0>0$ we set
  \begin{align*}
    b_0 &:= \rho_\cK^{\alpha_\cK}\psi\chi_\cK\chi_0\Bigl(S-c_0-\alpha_\cK\rho_\cK\ell\Bigl(\frac{\dd\rho_\cK}{\rho_\cK^2}\Bigr) - \alpha_\cK\rho_\cK H_j\rho_\cK^{-1}\Bigr)^{\frac12}, \\
    b_\cK &:= \rho_\cK^{\alpha_\cK+1}\psi\chi_0\Bigl(-\digamma\chi_\cK\chi_\cK'\Bigl[\ell\Bigl(\frac{\dd\rho_\cK}{\rho_\cK^2}\Bigr)+H_j\rho_\cK^{-1}\Bigr]\Bigr)^{\frac12}, \\
    e &:= \rho_\cK^{\alpha_\cK}\chi_\cK\chi_0\bigl( \psi \psi' (\ell(\dd r)-H_j r)\bigr)^{\frac12}, \\
    f &:= \rho_\cK^{2\alpha_\cK}\psi^2\chi_\cK^2\chi_0 H_{-\ell+j}\chi_0.
  \end{align*}
  The terms $b_\cK$ and $e$ take the place of $e$ and $b_R$ in~\eqref{EqEKe} and \eqref{EqEKbR}. Denoting by upper case letters the semiclassical 3b-quantizations of the corresponding lower case symbols, we now obtain from~\eqref{EqEKAdjPair}
  \begin{align*}
    &c_0^{-1}h^{-2}\|A\cL_h^* \omega\|^2 + \tfrac14 c_0\|A \omega\|^2 \\
    &\qquad \geq \bigl(c_0-2 C\digamma^{-2}-(C+C_\digamma)h\bigr)\|A \omega\|^2 + \|B_0 \omega\|^2 + \|B_\cK \omega\|^2 \\
    &\qquad\qquad - \|E \omega\|^2 - |\la F \omega,\omega\ra| - C h \|\tilde B_0 \omega\|_{\rho_\cK^{-\alpha_\cK}L^2}^2 - h^N\|\chi \omega\|_{\Htbh^{-N,(-\alpha_\sface-2,-\alpha_\cK)}}^2,
  \end{align*}
  where $\tilde B_0\in\Psitbh$ is elliptic on $\supp a$ (arising from~\eqref{EqEKErr3} and also controlling subprincipal terms not seen by the above principal symbol computation). Fixing $\digamma>1$ large and then $h$ small such that the second term on the left can be absorbed into the first term on the right, and estimating the term involving $F$ as in~\eqref{EqEKFEst} (with $\cL_h^*$ in place of $\cL_h$), the desired estimate~\eqref{EqEKAdj} follows by using the iterative procedure explained after~\eqref{EqEKHalf}.
\end{proof}

\subsubsection{Zero section propagation near the saddle point at the corner}
\label{SssEB}

The estimates near the other critical points of $-r\cd^\sharp$ are obtained in a similar fashion to Proposition~\ref{PropEK}, and hence we shall be rather brief. We use the local defining functions $\rho_\cK=\frac{r}{t}$ and $\rho_\sface=\frac{1}{r}$, and recall the quantity $\delta_r>0$ from Proposition~\ref{PropEC}.

\begin{prop}[Propagation near $\pa\cK^+$]
\label{PropEB}
  Recall the weights $\alpha_\sface,\alpha_\cK\in\R$ from Theorem~\usref{ThmET}. There exists $\delta\in(0,\delta_r]$ such that the following holds for $K:=\{\rho_\cK\leq 3\delta_r,\ \rho_\sface\leq 3\delta\}$, for any cutoff function $\chi\in\CI(M)$ which is identically $1$ on $K$ and supported in a small neighborhood of $K$, for any operators $B_0,B_1,E\in\Psitbh(M)$ with operator wave front sets which are compact subsets of $\Ttb^*M$ and with Schwartz kernels supported in both factors in $K$, and for all sufficiently small $h>0$.
  \begin{enumerate}
  \item\label{ItEBFw}{\rm (Estimate for $P_h$.)} Suppose that
  \begin{alignat*}{3}
    \WFtbh'(B_0) &\subset \bigl\{ \rho_\cK\leq\tfrac32\delta_r,\ &&\rho_\sface\leq\tfrac32\delta,\ &&|\zeta|<1 \bigr\}, \\
    \Elltbh(E) &\supset \bigl\{ \rho_\cK\leq 2\delta_r,\ &&\tfrac54\delta\leq\rho_\sface\leq 2\delta,\ &&|\zeta|<1 \bigr\}, \\
    \Elltbh(B_1) &\supset \bigl\{ \rho_\cK\leq 2\delta_r,\ &&\rho_\sface\leq 2\delta,\ &&|\zeta|<2 \bigr\}.
  \end{alignat*}
  Then for any $N\in\R$ there exists a constant $C$ such that
  \begin{equation}
  \label{EqEB}
    \|B_0 \omega\|_{\rho_\sface^{\alpha_\sface}\rho_\cK^{\alpha_\cK}L^2} \leq C\Bigl( h^{-1}\|B_1 P_h \omega\|_{\rho_\sface^{\alpha_\sface+2}\rho_\cK^{\alpha_\cK}L^2} + \|E \omega\|_{\rho_\sface^{\alpha_\sface}\rho_\cK^{\alpha_\cK}L^2} + h^N\|\chi \omega\|_{\Htbh^{-N,(\alpha_\sface,\alpha_\cK)}}\Bigr).
  \end{equation}
  \item\label{ItEBBw}{\rm (Estimate for $P_h^*$.)} Suppose that $B_0,B_1$ are as in part~\eqref{ItEBFw}, while now
  \[
    \Elltbh(E) \supset \bigl\{ \tfrac54\delta_r\leq\rho_\cK\leq 2\delta_r,\ \rho_\sface\leq 2\delta,\ |\zeta|<1 \bigr\}, \\
  \]
  Then for any $N\in\R$ there exists a constant $C$ such that
  \begin{equation}
  \label{EqEBAdj}
  \begin{split}
    \|B_0 \omega\|_{\rho_\sface^{-\alpha_\sface-2}\rho_\cK^{-\alpha_\cK}L^2} &\leq C\Bigl( h^{-1}\|B_1 P_h^* \omega\|_{\rho_\sface^{-\alpha_\sface}\rho_\cK^{-\alpha_\cK}L^2} + \|E \omega\|_{\rho_\sface^{-\alpha_\sface-2}\rho_\cK^{-\alpha_\cK}L^2} \\
      &\quad\hspace{12em} + h^N\|\chi \omega\|_{\Htbh^{-N,(-\alpha_\sface-2,-\alpha_\cK)}}\Bigr).
  \end{split}
  \end{equation}
  \end{enumerate}
\end{prop}

Choosing $B_0$ to be elliptic on $\{\rho_\cK\leq\delta_r,\,\rho_\sface\leq\delta,\,|\zeta|<\frac12\}$, Proposition~\ref{PropEB} permits us to propagate a priori control (for solutions of $P_h$) at the zero section from $\cK^+\setminus\pa\cK^+$ into $\pa\cK^+$. The adjoint estimate~\eqref{EqEBAdj} propagates control in the reverse direction, i.e., from $\sface\setminus\pa\cK^+$ into $\pa\cK^+$. See Figure~\ref{FigEB}.

\begin{figure}[!ht]
\centering
\includegraphics{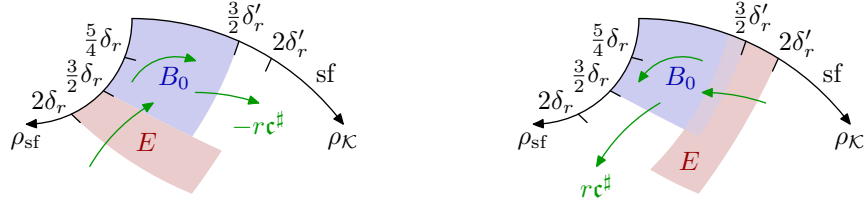}
\caption{\textit{On the left:} setup of the forward propagation estimate~\eqref{EqEB}. Control from the set labeled $E$ propagates into the region labeled $B_0$. The flow of the vector field $-r\cd^\sharp$ is indicated by green arrows. \textit{On the right:} the backward propagation estimate~\eqref{EqEBAdj}.}
\label{FigEB}
\end{figure}

\begin{proof}[Proof of Proposition~\usref{PropEB}]
  \pfstep{Part~\eqref{ItEBFw}.} As in the proof of Proposition~\ref{PropEK}, we may replace $P_h$ by the microlocally elliptic (near $o$) multiple $\cL_h$ defined in~\eqref{EqEKcL}. We fix cutoffs
  \begin{alignat*}{2}
    \chi_\sface&\in\CIc([0,\tfrac32)),&\quad \chi_\sface&=1\ \text{on}\ [0,\tfrac54], \\
    \chi_\cK&\in\CIc([0,\tfrac32\delta_r)),&\quad \chi_\cK&=1\ \text{on}\ [0,\tfrac54\delta_r], \\
    \chi_0&\in\CIc([0,\infty)),&\quad \chi_0&=1\ \text{near}\ 0,
  \end{alignat*}
  and arrange that $(-\chi_\cK\chi_\cK')^{\frac12},(-\chi_\sface\chi_\sface')^{\frac12}\in\CI([0,\infty))$. For a parameter $\digamma>\delta_r^{-1}$ that will be chosen later on, we then consider the commutant
  \begin{equation}
  \label{EqEBa}
    a = w^{-1}\chi_\sface(\digamma\rho_\sface)\chi_\cK(\rho_\cK)\chi_0(\digamma|\zeta|),\quad w:=\rho_\sface^{\alpha_\sface+1}\rho_\cK^{\alpha_\cK}.
  \end{equation}
  Let $A=A^*\in\rho_\sface^{-\alpha_\sface-1}\rho_\cK^{-\alpha_\cK}\Psitbh$ with $\WFtbh'(A)=\supp a$ denote its quantization. We then consider the $L^2$-pairing~\eqref{EqEKPair}, where $\sC$ is given by the expression~\eqref{EqEKC} and now lies in $\rho_\sface^{-2\alpha_\sface}\rho_\cK^{-2\alpha_\cK}\Psitb$. We can write the principal symbol of the first, main, term in~\eqref{EqEKC} as
  \begin{equation}
  \label{EqEBSymbol}
    a H_{-\ell+j}a - S a^2 = -c_0\rho_\sface^2 a^2 - b_0^2 - b_\sface^2 + e^2 + f,
  \end{equation}
  where for small $c_0>0$ we set
  \begin{align*}
    b_0 &:= \rho_\sface^{-\alpha_\sface}\rho_\cK^{-\alpha_\cK}\chi_\sface\chi_\cK\chi_0\Bigl(\rho_\sface^{-2}\Bigl[-\ell\Bigl(\frac{\dd w}{w}\Bigr)+S+w^{-1} H_j w\Bigr] - c_0\Bigr)^{\frac12}  , \\
    b_\sface &:= \rho_\sface^{-\alpha_\sface}\rho_\cK^{-\alpha_\cK}\chi_\sface\chi_0\bigl( -\chi_\cK\chi_\cK'\rho_\sface^{-2}\bigl[-\ell(\dd\rho_\cK)+H_j\rho_\cK\bigr]\bigr)^{\frac12}, \\
    e &:= \rho_\sface^{-\alpha_\sface}\rho_\cK^{-\alpha_\cK} \chi_\cK\chi_0\Bigl(-\digamma\chi_\sface\chi_\sface'\Bigl[\ell\Bigl(\frac{\dd\rho_\sface}{\rho_\sface^2}\Bigr)+H_j(\rho_\sface^{-1})\Bigr]\Bigr)^{\frac12}, \\
    f &:= \rho_\sface^{-2\alpha_\sface}\rho_\cK^{-2\alpha_\cK} \chi_\sface^2\chi_\cK^2 \chi_0\rho_\sface^{-2}H_{-\ell+j}\chi_0.
  \end{align*}
  By Proposition~\ref{PropEC}\eqref{ItECPosCorner} and the observation that $\rho_\sface^{-2}j$ vanishes cubically at $o\subset\Ttb^*M$, the square root in the definition of $b_0$ is indeed well-defined for sufficiently small $c_0>0$ and large $\digamma>1$ (used to localize near the zero section). For $b_\sface$, we note that on $\supp\chi_\cK'$ we have $\rho_\cK\in[\frac54\delta_r,\frac32\delta_r)$, and hence Proposition~\ref{PropEC}\eqref{ItECPosEsc} ensures the existence of the square root for sufficiently large $\digamma>1$ (so that $\rho_\sface$ and $|\zeta|$ are small on $\supp a$). For $e$ finally, note that $\frac{\dd\rho_\sface}{\rho_\sface^2}=-\dd r$, and hence the square root is well-defined by Proposition~\ref{PropEC}\eqref{ItECPosEsc}; here we also use that $H_j(\rho_\sface^{-1})=H_j r$ is smooth on $\supp a$ and vanishes at the zero section.

  The estimate~\eqref{EqEB} is now proved by following the arguments of the proof of Proposition~\ref{PropEK}; the correct weights at $\sface$ are obtained by estimating the left-hand side of~\eqref{EqEKPair} by $c_0^{-1}h^{-2}\|\rho_\sface^{-1}A\cL_h \omega\|^2+\frac14 c_0\|\rho_\sface A \omega\|^2$. The second, a priori control, term on the right-hand side in~\eqref{EqEB} arises due to the presence of $e$, which is supported away from $\sface$. The quantity $\delta$ is taken to be equal to $\digamma^{-1}$ for the (large) choice of $\digamma$ made analogously to the proof of Proposition~\ref{PropEK}.

  \pfstep{Part~\eqref{ItEBBw}.} We argue similarly to the proof of Proposition~\ref{PropEK}\eqref{ItEKBw}. We replace $P_h^*$ by $\cL_h^*$. We now consider the commutant
  \[
    a = w\chi_\sface(\digamma\rho_\sface)\chi_\cK(\rho_\cK)\chi_0(\digamma|\zeta|),\quad
    w = \rho_\sface^{\alpha_\sface+1}\rho_\cK^{\alpha_\cK},
  \]
  where the cutoffs $\chi_\sface,\chi_\cK,\chi_0$ are as before~\eqref{EqEBa}. We use the weight $w$ here rather than $w^{-1}$ in~\eqref{EqEBa} as we wish to prove an estimate on dual spaces. Letting $A=A^*$ denote the quantization of $a$, we consider the $L^2$-pairing~\eqref{EqEKAdjPair}. The main term of~\eqref{EqEKAdjC} has principal symbol
  \[
    a H_{-\ell+j}a + S a^2 = c_0 \rho_\sface^2 a^2 + b_0^2 + b_\cK^2 - e^2 + f,
  \]
  where for small $c_0>0$ we set
  \begin{align*}
    b_0 &:= \rho_\sface^{\alpha_\sface+2}\rho_\cK^{\alpha_\cK}\chi_\sface\chi_\cK\chi_0\Bigl(\rho_\sface^{-2}\Bigl[-\ell\Bigl(\frac{\dd w}{w}\Bigr) + S + w^{-1}H_j w\Bigr] - c_0\Bigr)^{\frac12}  , \\
    b_\cK &:= \rho_\sface^{\alpha_\sface+2}\rho_\cK^{\alpha_\cK} \chi_\cK\chi_0\Bigl(-\digamma\chi_\sface\chi_\sface'\Bigl[\ell\Bigl(\frac{\dd\rho_\sface}{\rho_\sface^2}\Bigr)+H_j(\rho_\sface^{-1})\Bigr]\Bigr)^{\frac12}, \\
    e &:= \rho_\sface^{\alpha_\sface+2}\rho_\cK^{\alpha_\cK}\chi_\sface\chi_0\bigl( -\chi_\cK\chi_\cK'\rho_\sface^{-2}\bigl[-\ell(\dd\rho_\cK)+H_j\rho_\cK\bigr]\bigr)^{\frac12}, \\
    f &:= \rho_\sface^{2(\alpha_\sface+2)}\rho_\cK^{2\alpha_\cK} \chi_\sface^2\chi_\cK^2 \chi_0\rho_\sface^{-2}H_{-\ell+j}\chi_0.
  \end{align*}
  The terms $b_\cK$ and $e$ take the place of (and are given by the same expressions, with $\alpha_\sface$ replaced by $-\alpha_\sface-2$, as) $e$ and $b_\sface$ in~\eqref{EqEBSymbol}. The square root defining $b_0$ is positive definite on $\supp\chi_\sface\chi_\cK\chi_0$ for sufficiently large $\digamma>1$ in view of Proposition~\ref{PropEC}\eqref{ItECPosCorner}. Upon quantization, we obtain control on the elliptic set of $b_0$ (and also on that of $b_\cK$), provided we have a priori control on $\supp e$. This proves~\eqref{EqEBAdj}.
\end{proof}

\begin{rmk}[Necessity of $-r\cd^\sharp$ being outgoing at $\pa\cK^+$]
\label{RmkEBOut}
  We return once more to the question (cf.\ Remark~\ref{RmkM0WrongV}) why we cannot work with a vector field $\cd^\sharp$ that has fewer critical sets by requiring $V=-r\cd^\sharp=\pa_t-v'\pa_r$ with $v'\in(0,1)$ for large $r$, i.e., $V$ is pointing towards the black hole (and $\cd=r^{-1}(\dd t+v'\,\dd r)$---notice the sign difference to~\eqref{EqECExpl2}). This $\cd$ (with $r\cd$ a past timelike scattering 1-form) could be extended to all of $M$ in such a way that the only critical set of $V$ would be $\pa\cK^+$: the flow of $V$ on $\pa M$ would pass from $\sface^\circ$ through the saddle $\pa\cK^+$ into $(\cK^+)^\circ$ and onwards through the event horizon of the black hole. The only condition on the weights at $\sface$ and $\cK^+$ would arise from the zero section propagation through $\pa\cK^+$,
  \[
    \alpha_\cK < \alpha_\sface - \frac12,
  \]
  by mirroring the calculation in~Lemma~\ref{LemmaMIbdy}. \emph{However}, this places a strict upper bound on the decay rate $\alpha_\cK$ when $\alpha_\sface$ is fixed. This implies that one \emph{cannot} exclude the existence of nontrivial stationary elements of $\ker\Box_g^{\cC_h}$ unless they have pointwise decay $r^{-2-\eps}$, $\eps>0$. (And in fact there would exist stationary elements in the kernel with $r^{-2}$ decay.) This is far from our requirement stated in Theorem~\ref{ThmIRough}\eqref{ItIRough0}. Directly related to this is that for such $\cd$, the indicial roots at zero energy, for suitably chosen small $e$ and $h$, lie (according to numerical experiments) outside the strip $\{2<\Re\lambda<C_0\}$ for any desired fixed $C_0>2$, with the same number of roots of each individual type on either side of this strip, and with $\lambda=2$ always a vector type $1$ indicial root.
\end{rmk}

\subsubsection{Zero section propagation near the sink in \texorpdfstring{$\sface^\circ$}{the interior of the side face}}
\label{SssES}

We now use the positivity property stated in Proposition~\ref{PropEC}\eqref{ItECPosSink} to propagate semiclassical estimates into the sink of $-r\cd^\sharp$ inside of $\sface^\circ$. We continue to use $\rho_\cK=\frac{r}{t}$ and $\rho_\sface=r^{-1}$. We shall drop the decay order at $\cK^+$ from the notation; moreover, the 3b-differential order continues to be irrelevant (i.e., arbitrary) since we are working near the zero section of $\Ttb^*M$.

\begin{prop}[Propagation near the sink of $-r\cd^\sharp$ in $\sface^\circ$]
\label{PropES}
  Recall the weight $\alpha_\sface\in\R$ from Theorem~\usref{ThmET}, so $\alpha_\sface<-\frac12$. There exists $\delta\in(0,\delta_r]$ such that the following holds for $K:=\{|\frac{r}{t}/v-1|\leq 3\delta_\sface,\ \rho_\sface\leq 3\delta\}$, any cutoff function $\chi\in\CI(M)$ which is identically $1$ on $K$ and supported in a small neighborhood of $K$, and for all sufficiently small $h>0$.
  \begin{enumerate}
  \item\label{ItESFw}{\rm (Estimate for $P_h$.)} Let $B_0,B_1,E\in\Psitbh(M)$ have operator wave front sets which are compact subsets of $\Ttb^*M$, and suppose that
  \begin{alignat*}{3}
    \WFtbh'(B_0) &\subset \{ |\tfrac{r}{t}/v-1| \leq \tfrac32\delta_\sface,\ &&\rho_\sface\leq\tfrac32\delta,\ &&|\zeta|<1 \}, \\
    \Elltbh(E) &\supset \{ |\tfrac{r}{t}/v-1| \leq 2\delta_\sface,\ &&\rho_\sface\leq 2\delta,\ &&|\zeta|<1 \} \setminus \{ |\tfrac{r}{t}/v-1| \leq \tfrac54\delta_\sface,\ \rho_\sface\leq \tfrac54\delta \}, \\
    \Elltbh(B_1) &\supset \{ |\tfrac{r}{t}/v-1| \leq 2\delta_\sface,\ &&\rho_\sface\leq 2\delta,\ &&|\zeta|<2 \}.
  \end{alignat*}
  Suppose moreover that the Schwartz kernels of $B_0,B_1,E$ are supported in both factors in $K$. Then for any $N\in\R$ there exists a constant $C$ such that for all sufficiently small $h>0$,
  \begin{equation}
  \label{EqES}
    \|B_0 \omega\|_{\rho_\sface^{\alpha_\sface}L^2} \leq C\Bigl( h^{-1}\|B_1 P_h \omega\|_{\rho_\sface^{\alpha_\sface+2}L^2} + \|E \omega\|_{\rho_\sface^{\alpha_\sface}L^2} + h^N\|\chi \omega\|_{\rho_\sface^{\alpha_\sface}\Htbh^{-N}}\Bigr).
  \end{equation}
  \item\label{ItESBw}{\rm (Estimate for $P_h^*$.)} Let $B_0,B_1\in\Psitbh(M)$ be as in part~\eqref{ItESFw}. Then for any $N\in\R$ there exists a constant $C$ such that for all sufficiently small $h>0$,
  \begin{equation}
  \label{EqESAdj}
    \|B_0\omega\|_{\rho_\sface^{-\alpha_\sface-2}L^2} \leq C\Bigl( h^{-1}\|B_1 P_h^*\omega\|_{\rho_\sface^{-\alpha_\sface}L^2} + h^N\|\chi\omega\|_{\rho_\sface^{-\alpha_\sface-2}\Htbh^{-N}}\Bigr).
  \end{equation}
  \end{enumerate}
\end{prop}

See Figure~\ref{FigES}. Thus, for $P_h$ we can propagate control from a punctured neighborhood of the sink into the sink, whereas for $P_h^*$ we get a free estimate at the sink (which is a \emph{source} for backwards propagation).

\begin{figure}[!ht]
\centering
\includegraphics{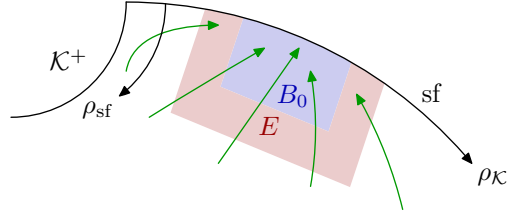}
\caption{Setup of the forward propagation estimate~\eqref{EqES}. Control from the set labeled $E$ propagates into the region labeled $B_0$. The flow of the vector field $-r\cd^\sharp$ is indicated by green arrows. For the backward propagation estimate~\eqref{EqESAdj}, we get free control on the set labeled $B_0$, without the need for any a priori control region.}
\label{FigES}
\end{figure}

\begin{proof}[Proof of Proposition~\usref{PropES}]
  For part~\eqref{ItESFw}, we consider the commutant
  \[
    a = (w')^{-1}\chi_\sface(\digamma\rho_\sface)\chi_\cR\bigl(\tfrac{r}{t}/v-1\bigr)\chi_0(\digamma|\zeta|),\quad
    w'=\rho_\sface^{\alpha_\sface+1},
  \]
  where the cutoff functions are
  \begin{alignat*}{2}
    \chi_\sface&\in\CIc([0,\tfrac32)),&\quad \chi_\sface&=1\ \text{on}\ [0,\tfrac54], \\
    \chi_\cR&\in\CIc((-2\delta_\sface,2\delta_\sface)),&\quad \chi_\cR&=1\ \text{on}\ [-\tfrac54\delta_\sface,\tfrac54\delta_\sface], \\
    \chi_0&\in\CIc([0,\infty)),&\quad \chi_0&=1\ \text{near}\ 0;
  \end{alignat*}
  we further require $(-\chi_\sface\chi_\sface')^{\frac12}\in\CI$, and $(\mp\chi_\cR\chi_\cR')^{\frac12}\in\CI$ on $\pm[0,\infty)$. The parameter $\digamma>\delta_r^{-1}$ will be chosen large. Let $A=A^*\in\rho_\sface^{-\alpha_\sface-1}\Psitbh$ be a quantization of $a$ satisfying $\WFtbh'(A)=\supp a$. We again consider~\eqref{EqEKPair}--\eqref{EqEKC}, now writing
  \begin{equation}
  \label{EqESSymbol}
    a H_{-\ell+j}a - S a^2 = -c_0\rho_\sface^2 a^2 - b_0^2 + e_1^2 + e_2^2 + f,
  \end{equation}
  where for small $c_0>0$ we define the $\End(\cT^*)$-valued symbols
  \begin{align*}
    b_0 &:= \rho_\sface^{-\alpha_\sface}\chi_\sface\chi_\cR\chi_0\Bigl(\rho_\sface^{-2}\Bigl[-\ell\Bigl(\frac{\dd w'}{w'}\Bigr)+S+w'{}^{-1}H_j w'\Bigr]-c_0\Bigr)^{\frac12}, \\
    e_1 &:= \rho_\sface^{-\alpha_\sface}\chi_0\chi_\cR\Bigl(-\digamma\chi_\sface\chi_\sface'\Bigl[\ell\Bigl(\frac{\dd\rho_\sface}{\rho_\sface^2}\Bigr)+H_j(\rho_\sface^{-1})\Bigr]\Bigr)^{\frac12}, \\
    e_2 &:= \rho_\sface^{-\alpha_\sface}\chi_0\chi_\sface\Bigl(-v^{-1}\chi_\cR\chi_\cR'\rho_\sface^{-2}\Bigl[\ell\Bigl(\dd\frac{r}{t}\Bigr)-H_j\frac{r}{t}\Bigr]\Bigr)^{\frac12}, \\
    f &:= \rho_\sface^{-2\alpha_\sface}\chi_\sface^2\chi_\cR^2\chi_0\rho_\sface^{-2}H_{-\ell+j}\chi_0.
  \end{align*}
  By Proposition~\ref{PropEC}\eqref{ItECPosSink}, the square root in the expression for $b_0$ is smooth if we choose $\digamma$ large enough. The symbol $e_1$ is smooth by Proposition~\ref{PropEC}\eqref{ItECPosEsc} since $\frac{\dd\rho_\sface}{\rho_\sface^2}=-\dd r$ and $\rho_\sface^{-1}=r$; the symbol $e_2$ is smooth as well since on the support of $\dd\chi_\cR$, the sign of $-r^2\ell(\dd(\frac{r}{t}))$ matches that of $\chi_\cR'$. The terms $e_1,e_2$ have supports disjoint from $\{|\frac{r}{t}/v-1|\leq\delta_\sface,\ \rho_\sface\leq\digamma^{-1}\}$, and together give rise to the a priori control term $\|E \omega\|^2$ in~\eqref{EqES}. The term $f$ is supported in the elliptic set of $P_h$ as usual. The rest of the proof of part~\eqref{ItESFw} is completely analogous to that of Propositions~\ref{PropEK} and \ref{PropEB}.

  For part~\eqref{ItESBw}, we take $a=w'\chi_\sface(\digamma\rho_\sface)\chi_\cR(\frac{r}{t}/v-1)\chi_0(\digamma|\zeta|)$ with $w'=\rho_\sface^{\alpha_\sface+1}$ as above; then
  \[
    a H_{-\ell+j}a + S a^2 = c_0\rho_\sface^2 a^2 + b_0^2 + b_1^2 + b_2^2 + f,
  \]
  where we set
  \begin{align*}
    b_0 &:= \rho_\sface^{\alpha_\sface+2}\chi_\sface\chi_\cR\chi_0\Bigl(\rho_\sface^{-2}\Bigl[-\ell\Bigl(\frac{\dd w'}{w'}\Bigr)+S+w'{}^{-1}H_j w'\Bigr]-c_0\Bigr)^{\frac12}, \\
    b_1 &:= \rho_\sface^{\alpha_\sface+2}\chi_0\chi_\cR\Bigl(-\digamma\chi_\sface\chi_\sface'\Bigl[\ell\Bigl(\frac{\dd\rho_\sface}{\rho_\sface^2}\Bigr)+H_j(\rho_\sface^{-1})\Bigr]\Bigr)^{\frac12}, \\
    b_2 &:= \rho_\sface^{\alpha_\sface+2}\chi_0\chi_\sface\Bigl(-v^{-1}\chi_\cR\chi_\cR'\rho_\sface^{-2}\Bigl[\ell\Bigl(\dd\frac{r}{t}\Bigr)-H_j\frac{r}{t}\Bigr]\Bigr)^{\frac12}, \\
    f &:= \rho_\sface^{2(\alpha_\sface+2)}\chi_\sface^2\chi_\cR^2\chi_0\rho_\sface^{-2}H_{-\ell+j}\chi_0.
  \end{align*}
  All terms (except for $f$ which is however supported in the elliptic set of $P_h^*$) have the same sign. Therefore, we obtain the estimate~\eqref{EqESAdj} without any a priori control terms.
\end{proof}

\subsubsection{Real principal type zero section propagation in the remaining regions}
\label{SssER}

We next demonstrate how to propagate control from the saddle point near $r=r_0$ inside of $(\cK^+)^\circ$ towards the black hole. Since we shall be working away from $\sface$ and near the zero section of $\Ttb^*M$, we omit the weight at $\sface$ and the 3b-differential order from the notation (as already in~\S\ref{SssEK}). We write $\rho_\cK=t^{-1}$.

\begin{prop}[Propagation between the interior saddle point and the black hole]
\label{PropER}
  Let
  \[
    r_3:=r_0-\tfrac{52}{50}r_0\sqrt{e}/v,\quad r_4:=r_0-\tfrac{51}{50}r_0\sqrt{e}/v.
  \]
  Fix $0<t_1<t_2<t_3$ and $\bhm<r_1<r_2<r_3$. Let $B_0,B_1,E\in\Psitbh(M)$, with operator wave front sets that are compact subsets of $\Ttb^*M$.
  \begin{enumerate}
  \item\label{ItERFw}{\rm (Estimate for $P_h$: propagation towards the black hole.)} Suppose that
    \begin{alignat*}{3}
      \WFtbh'(B_0) &\subset \bigl\{ \rho_\cK\leq t_2^{-1},\ &&r_1<r<r_4,\ &&|\zeta|<1 \bigr\}, \\
      \Elltbh(E) &\supset \bigl\{ t_3^{-1} \leq \rho_\cK \leq t_1^{-1},\ &&r_1\leq r\leq r_4,\ &&|\zeta|<1 \bigr\} \cup \{ \rho_\cK \leq t_1^{-1},\ r_3\leq r\leq r_4,\ |\zeta|<1 \bigr\}, \\
      \Elltbh(B_1) &\supset \{ \rho_\cK\leq t_1^{-1},\ &&r_1\leq r\leq r_4,\ &&|\zeta|<2 \bigr\}.
    \end{alignat*}
    Suppose moreover that the Schwartz kernels of $B_0,B_1,E$ are supported in both factors in a compact subset $K$ of $\{ \bhm<r<r_0,\ \rho_\cK<\infty \}$, and let $\chi\in\CI(M)$ be identically $1$ on $K$. Then for any $N\in\R$ there exists a constant $C$ such that
    \begin{equation}
    \label{EqER}
      \|B_0 \omega\|_{\rho_\cK^{\alpha_\cK}L^2} \leq C\Bigl( h^{-1}\|B_1 P_h \omega\|_{\rho_\cK^{\alpha_\cK}L^2} + \|E \omega\|_{\rho_\cK^{\alpha_\cK}L^2} + h^N\|\chi \omega\|_{\rho_\cK^{\alpha_\cK}\Htbh^{-N}}\Bigr).
    \end{equation}
  \item\label{ItERBw}{\rm (Estimate for $P_h^*$: propagation towards the saddle point.)} Suppose that $B_0,B_1$ are as in part~\eqref{ItERFw}, while
    \[
      \Elltbh(E) \supset \bigl\{ \rho_\cK\leq t_1^{-1},\ r_1\leq r\leq r_2,\ |\zeta|<1 \bigr\}.
    \]
    Suppose that the Schwartz kernels of $B_0,B_1,E$ are supported in both factors in a compact subset $K$ of $\{ \bhm<r<r_0,\ \rho_\cK<\infty \}$, and let $\chi\in\CI(M)$ be identically $1$ on $K$. Then for any $N\in\R$ there exists a constant $C$ such that
    \begin{equation}
    \label{EqERAdj}
      \|B_0 \omega\|_{\rho_\cK^{-\alpha_\cK}L^2} \leq C\Bigl( h^{-1}\|B_1 P_h^* \omega\|_{\rho_\cK^{-\alpha_\cK}L^2} + \|E \omega\|_{\rho_\cK^{-\alpha_\cK}L^2} + h^N\|\chi \omega\|_{\rho_\cK^{-\alpha_\cK}\Htbh^{-N}}\Bigr).
    \end{equation}
  \end{enumerate}
\end{prop}

See Figure~\ref{FigER} for an illustration of part~\eqref{ItERFw}.

\begin{figure}[!ht]
\centering
\includegraphics{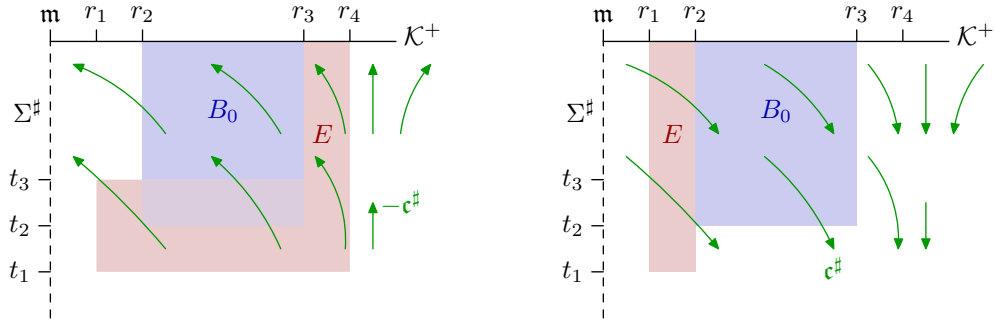}
\caption{\textit{On the left:} illustration of the forward propagation estimate~\eqref{EqER}. Control from the elliptic set of $E$ can be propagated towards $r=\bhm$, giving control on the elliptic set of $B_0$ (here taken to go up to $r=r_3$). The flow of $-\cd^\sharp$ is indicated by green arrows. \textit{On the right:} illustration of the backward propagation estimate~\eqref{EqERAdj}.}
\label{FigER}
\end{figure}

\begin{proof}[Proof of Proposition~\usref{PropER}]
  For part~\eqref{ItERFw}, fix
  \begin{alignat*}{3}
    \chi_1&\in\CI(\R),&\quad \chi_1&=0\ \text{on}\ (-\infty,0], &\quad \chi_1(x)&=e^{-\digamma/x}\ \text{for}\ x>0, \\
    \chi_2&\in\CI(\R),&\quad \chi_2&=1\ \text{near}\ (-\infty,r_3],&\quad \chi_2&=0\ \text{near}\ [r_4,\infty), \\
    \chi_\cK&\in\CIc([0,\infty)),&\quad \chi_\cK&=1\ \text{near}\ [0,t_2^{-1}],&\quad \chi_\cK&=0\ \text{near}\ [t_1^{-1},\infty), \\
    \chi_0&\in\CIc([0,\infty)),&\quad \chi_0&=1\ \text{near}\ 0;
  \end{alignat*}
  here $\digamma>1$ is chosen later, and we require $(-\chi_2\chi_2')^{\frac12}\in\CI$, $(-\chi_\cK\chi_\cK')^{\frac12}\in\CI$. We then set
  \[
    a = \rho_\cK^{-\alpha_\cK} \chi_1(r-r_1)\chi_2(r)\chi_\cK(\rho_\cK)\chi_0(\digamma|\zeta|).
  \]
  We have $\chi_1'(r-r_1)=\digamma (r-r_1)^{-2}\chi_1(r-r_1)$, and for any fixed $C$ we have $\digamma(r-r_1)^{-2}\geq C$ for all $r\in[r_1,r_4]$ provided $\digamma$ is sufficiently large. Thus, we can write
  \begin{equation}
  \label{EqERSymbol}
    a H_{-\ell+j}a - S a^2 = -a^2 - b_0^2 + e_R^2 + e_T^2 + f,
  \end{equation}
  where, in view of Proposition~\ref{PropEC}\eqref{ItECPosEsc} and the future timelike nature of $\frac{\dd\rho_\cK}{\rho_\cK^2}=-\dd t$, the symbols
  \begin{align*}
    b_0 &:= \rho_\cK^{-\alpha_\cK}\chi_1\chi_2\chi_\cK\chi_0 \Bigl( -\digamma(r-r_1)^{-2}(-\ell(\dd r)+H_j r) + S - \alpha_\cK\rho_\cK\Bigl[\ell\Bigl(\frac{\dd\rho_\cK}{\rho_\cK^2}\Bigr)+H_j(\rho_\cK^{-1})\Bigr] - 1 \Bigr)^{\frac12}, \\
    e_R &:= \rho_\cK^{-\alpha_\cK}\chi_1\chi_\cK\chi_0\bigl(-\chi_2\chi_2' (\ell(\dd r)-H_j r)\bigr)^{\frac12}, \\
    e_T &:= \rho_\cK^{-\alpha_\cK+1}\chi_1\chi_2\chi_0\Bigl(-\chi_\cK\chi_\cK'\Bigl[\ell\Bigl(\frac{\dd\rho_\cK}{\rho_\cK^2}\Bigr)+H_j(\rho_\cK^{-1})\Bigr]\Bigr)^{\frac12}, \\
    f &:= \rho_\cK^{-2\alpha_\cK}\chi_1^2\chi_2^2\chi_\cK^2 \chi_0 H_{-\ell+j}\chi_0
  \end{align*}
  are smooth for sufficiently large $\digamma>1$. The rest of the proof proceeds as usual: the terms $e_R$ (with support on $\supp(\chi_2'\chi_\cK)$) and $e_T$ (with support on $\supp(\chi_2\chi_\cK')$) give rise to the a priori control term $\|E \omega\|$ in~\eqref{EqER}, whereas $f$ is supported in the elliptic set of $P_h$.

  For part~\eqref{ItERBw}, we use the commutant
  \[
    a = \rho_\cK^{\alpha_\cK}\chi_1(r_4 - r)\tilde\chi_2(r)\chi_\cK(\rho_\cK)\chi_0(\digamma|\zeta|)
  \]
  where $\chi_1,\chi_\cK,\chi_0$ are as before, and we fix
  \[
    \tilde\chi_2\in\CI(\R),\quad \tilde\chi_2=0\ \text{near}\ (-\infty,r_1],\quad \tilde\chi_2=1\ \text{near}\ [r_2,\infty),
  \]
  with $(\tilde\chi_2\tilde\chi_2')^{\frac12}\in\CI$. Thus,
  \[
    a H_{-\ell+j}a + S a^2 = a^2 + b_0^2 + b_T^2 - e_R^2 + f
  \]
  where
  \begin{align*}
    b_0 &:= \rho_\cK^{\alpha_\cK}\chi_1\tilde\chi_2\chi_\cK\chi_0\Bigl( -\digamma(r_4-r)^{-2}(-\ell(\dd r)+H_j r) + S - \alpha_\cK\rho_\cK\Bigl[\ell\Bigl(\frac{\dd\rho_\cK}{\rho_\cK^2}\Bigr)+H_j(\rho_\cK^{-1})\Bigr] - 1\Bigr)^{\frac12}, \\
    b_T &:= \rho_\cK^{\alpha_\cK+1}\chi_1\tilde\chi_2\chi_0\Bigl(-\chi_\cK\chi_\cK'\Bigl[\ell\Bigl(\frac{\dd\rho_\cK}{\rho_\cK^2}\Bigr)+H_j(\rho_\cK^{-1})\Bigr]\Bigr)^{\frac12}, \\
    e_R &:= \rho_\cK^{\alpha_\cK}\chi_1\chi_\cK\chi_0\bigl(\tilde\chi_2\tilde\chi_2' (\ell(\dd r)-H_j r)\bigr)^{\frac12}, \\
    f &:= \rho_\cK^{2\alpha_\cK}\chi_1^2\tilde\chi_2^2\chi_\cK^2 \chi_0 H_{-\ell+j}\chi_0.
  \end{align*}
  The only a priori control term now arises from $e_R$, which has support in $\supp(\chi_\cK\chi_0\tilde\chi_2')$.
\end{proof}

Completely analogous results, with completely analogous proofs, provide forward propagation estimates near the zero section for propagation from $r=r_0+\frac{51}{50}r_0\sqrt{e}/v$ to any larger finite value of $r$, such as $r=\delta^{-1}$ where $\delta$ is as in Proposition~\ref{PropEB}; this propagation uses $r$ as an escape function. Likewise, one can propagate near $\sface$ from $\rho_\cK=\delta_r$ to $\frac{r}{t}/v-1=-\delta_\sface$, and from $\frac{r}{t}=10$ to $\frac{r}{t}/v-1=\delta_\sface$; for this, one uses $r/t$ as the escape function and the positivity condition arranged by Proposition~\ref{PropEC}\eqref{ItECPosEsc}. We leave the detailed statements of such results (and of the corresponding adjoint estimates) to the reader.

\begin{rmk}[Threshold regularity at the light cone at infinity]
\label{RmkERThres}
  Parallel to Remark~\ref{RmkEEThr}, the reader may have expected an upper bound on the sum of regularity and decay order to be required at the light cone at infinity given by $\{\frac{r}{t}=1,\ r^{-1}=0\}\subset M$ analogously to \cite[Proposition~4.4, \S{3.6}]{BaskinVasyWunschRadMink}. But yet again it is the strong damping provided by $-i L_h$ which now moves this threshold quantity above any fixed number when $h$ is small enough. Recalling the relationship between this threshold quantity and the pointwise decay rate of waves at null infinity, this is consistent with the observation that all eigenvalues of $1+\ubar S^{\cC_h}$ in~\eqref{EqTStruct}, which dictates the decay rates (in powers of $\rho=r^{-1}$) at null infinity, are $\gtrsim h^{-1}\gg 1$ when $h$ is small (as follows from an inspection of~\eqref{EqM00S}).
\end{rmk}

\subsubsection{Energy estimates}
\label{SssEI}

Near the initial and final Cauchy hypersurfaces $X$ and $\Sigma^\sharp$, we need to use energy estimates in order to have estimates which localize sharply to the domain $\Omega$. We follow \cite[\S8.4]{HintzVasyKdSStability} closely (up to various sign switches due to our switched signature convention), but sharpen and generalize the results slightly. Similar results are obtained in \cite[\S{4.5}]{HintzPetersenVasyKdS}. First, we recall an abstract version of \cite[Lemma~8.23]{HintzVasyKdSStability}:

\begin{lemma}[Generalized stress--energy--momentum tensor]
\label{LemmaEIT}
  Let $(T^*,\la\cdot,\cdot\ra_G)$ be a time-oriented Lorentzian vector space, and let $\nu\in T^*$ be past timelike. Denote by $|\cdot|_R$ a norm on $T^*$. Let $E$ be a vector space equipped with a Hermitian inner product $\la\cdot,\cdot\ra_E$ and squared norm $|e|_E^2=\la e,e\ra_E$. Let $\ell\colon T^*\to\End(E)$ be a linear map that takes values in the space of self-adjoint endomorphisms, and suppose that $\ell(\zeta)>0$ for all future causal $\zeta\in T^*\setminus\{0\}$. Define the generalized stress--energy--momentum tensor
  \begin{equation}
  \label{EqEITDef}
    T_{G,\nu,\ell}(\alpha,\beta) := \la\alpha,\nu\ra_G\ell(\beta) + \la\beta,\nu\ra_G\ell(\alpha) - \la\alpha,\beta\ra_G\ell(\nu) \in \End(E),\quad \alpha,\beta\in T^*.
  \end{equation}
  Then there exists $C>0$ such that for all $e,f\in E$ and $\nu'\in\nu^\perp$, we have
  \begin{equation}
  \label{EqEIT}
    \la T_{G,\nu,\ell}(\nu,\nu)e,e\ra_E + 2\Re\la T_{G,\nu,\ell}(\nu,\nu')e,f\ra_E + \la T_{G,\nu,\ell}(\nu',\nu')f,f\ra_E \geq C\bigl( |e|_E^2 + |\nu'|_R^2 |f|_E^2\bigr).
  \end{equation}
\end{lemma}

The connection with the standard stress--energy--momentum tensor is as follows: if one defines~\eqref{EqEITDef} tensorially on a manifold $M$, if $\alpha$ and $\beta$ are both equal to the differential $\dd u$ of a real-valued scalar function on $M$, and if $E=M\times\R$ is the trivial bundle and $\ell$ is a vector field $X$, then $\frac12 T_{G,\nu,X}(\dd u,\dd u)=(\nu^\sharp u)(X u)-\frac12 g(\nu^\sharp,X)\la\dd u,\dd u\ra_G$, $g:=G^{-1}$, is the stress--energy--momentum tensor of $u$ on $(M,g)$ evaluated on the vector fields $\nu^\sharp,X$.

\begin{proof}[Proof of Lemma~\usref{LemmaEIT}]
  We give an invariant version of the proof in the reference. If $\nu'=0$, the estimate follows directly from $T_{G,\nu,\ell}(\nu,\nu)=\la\nu,\nu\ra_G\ell(\nu)$ and the fact that $\ell(\nu)$ is negative definite and $\la\nu,\nu\ra_G<0$. Let us thus assume that $\nu'\neq 0$. Writing $G(v)=\la v,v\ra_G$, the left-hand side of~\eqref{EqEIT} is equal to
  \begin{align}
    &-G(\nu)\la-\ell(\nu)e,e\ra_E + 2 G(\nu)\Re\la\ell(\nu')e,f\ra + G(\nu')\la-\ell(\nu)f,f\ra_E \nonumber\\
  \label{EqEIT2}
  \begin{split}
    &\qquad = -G(\nu)\Bigl| \bigl(-\ell(\nu)\bigr)^{\frac12}e - \bigl(-\ell(\nu)\bigr)^{-\frac12} \ell(\nu')f\Bigr|_E^2 \\
    &\qquad\qquad + G(\nu') \Big\la \Bigl(-\ell(\nu) + \frac{-G(\nu)}{G(\nu')}\ell(\nu')\ell(\nu)^{-1}\ell(\nu')\Bigr)f, f\Big\ra_E.
  \end{split}
  \end{align}
  In the second summand, we have $G(\nu')\geq C |\nu'|_R^2$ for some $C>0$ since $\nu'\in\nu^\perp$ is spacelike. Introducing the normalized vectors $\hat\nu'=\nu'/G(\nu')^{\frac12}$ and $\hat\nu=\nu/(-G(\nu))^{\frac12}$, the second pairing is equal to
  \[
    G(\nu')\bigl(-G(\nu)\bigr)^{\frac12} \big\la S\bigl((-\ell(\hat\nu))^{\frac12}f\bigr),(-\ell(\hat\nu))^{\frac12}f\big\ra_E,\qquad
    S = I - \bigl( (-\ell(\hat\nu))^{-\frac12} \ell(\hat\nu') (-\ell(\hat\nu))^{-\frac12} \bigr)^2.
  \]
  Now, since $-\hat\nu\pm\hat\nu'$ is future lightlike, we have
  \[
    S_\pm = I \pm (-\ell(\hat\nu))^{-\frac12} \ell(\hat\nu') (-\ell(\hat\nu))^{-\frac12} > 0.
  \]
  Moreover, $S_+$ and $S_-$ commute, and therefore $S=S_+S_-=Q^2$ for $Q=(S_+S_-)^{\frac12}>0$. Altogether then,~\eqref{EqEIT2} is bounded from below by
  \begin{equation}
  \label{EqEIT3}
    C\Bigl(\Bigl| \bigl(-\ell(\nu)\bigr)^{\frac12}e - \bigl(-\ell(\nu)\bigr)^{-\frac12} \ell(\hat\nu')\bigl(G(\nu')^{\frac12}f\bigr)\Bigr|_E^2 + \bigl| Q\bigl(-\ell(\nu)\bigr)^{\frac12}\bigl(|\nu'|_R f\bigr) \bigr|_E^2 \Bigr)
  \end{equation}
  for some new constant $C>0$. If we fix the unit vector $\hat\nu'\perp\nu$, then this is a quadratic form in $(e,|\nu'|_R f)$. Now if~\eqref{EqEIT3}---which is a sum of two nonnegative terms---vanishes, then from the second summand we obtain $|\nu'|_R f=0$, and then the first summand gives $e=0$. By homogeneity, \eqref{EqEIT3} thus implies the estimate~\eqref{EqEIT} for some $C>0$.
\end{proof}

This implies the following local estimate (see also \cite[Lemma~8.25]{HintzVasyKdSStability}):

\begin{cor}[Coercivity estimate]
\label{CorEICoercive}
  Denote by $B_r$ the open ball of (Euclidean) radius $r$ around $0$ in $\R_z^{1+n}=\R_\ft\times\R^n_x$. Let $g\in\CI(\ol{B_1};S^2 T^*\R^{1+n})$ denote a Lorentzian metric for which $\dd\ft$ is past timelike, and write $G=g^{-1}$ for the inverse metric. Let $\pi\colon T^*\R^{n+1}\to\R^{n+1}$ denote the projection. Let $E\to\R^{n+1}$ denote a vector bundle with affine connection $\nabla^E$. Let $\la\cdot,\cdot\ra_E$ denote a Hermitian inner product on $E$, and write $|e|_E^2:=\la e,e\ra_E$. Denote by $R$ a Riemannian fiber metric on $T^*\R^{n+1}$, and let $\phi\in\CI(\R^{n+1})$ be a positive function. Suppose finally that $\ell\colon T^*\R^{n+1}\to\End(\pi^*E)$ is a bundle map with the property that $\ell(\zeta)$ is positive definite for all future causal $\zeta\in T^*\R^{n+1}\setminus o$. Then there exist constants $C,C'>0$ such that for all $\omega\in\CIc(B_1;E)$, we have
  \begin{equation}
  \label{EqEICoercive}
    \sum_{\mu,\nu=0}^n \int_{\ft^{-1}(0)} \big\la T_{G,\phi^2\,\dd\ft,\ell}^{\mu\nu} \nabla_\mu^E \omega,\nabla_\nu^E \omega\big\ra_E\,\dd x \geq C\int_{\ft^{-1}(0)} \phi^2|\nabla^E \omega|_{R\otimes E}^2\,\dd x - C'\int_{\ft^{-1}(0)} |\omega|_E^2\,\dd x,
  \end{equation}
  where $T_{G,\phi^2\,\dd\ft,\ell}^{\mu\nu}=T_{G,\phi^2\,\dd\ft,\ell}(\dd z^\mu,\dd z^\nu)$. If the data $g$, $(E,\la\cdot,\cdot\ra_E)$, $\nabla^E$, $\phi$, $\ell$, subject to the above conditions, lie in some fixed compact subset in the $\cC^1$ topology, the constants $C,C'$ can be chosen such that~\eqref{EqEICoercive} holds for all such data, and for all $\omega\in\CIc(B_1;E)$.
\end{cor}
\begin{proof}
  If $E=\R^{n+1}\times\C^k$ is trivial, $\nabla^E=\dd$, and $g,\ell,\phi$ are constant, we in fact have~\eqref{EqEICoercive} with $C'=0$. To see this, change the $x$-coordinates by $x\mapsto x-v\ft$ for an appropriate $v\in\R^n$ such that $\dd\ft$ is orthogonal to $\dd x^1,\ldots,\dd x^n$. Then, denoting the Fourier transform in $x$ by $\hat{\ }$, Plancherel's theorem shows that the estimate~\eqref{EqEICoercive} is equivalent to
  \begin{align*}
    &\int_{\R^n_\xi} \big\la T_{G,\phi^2\,\dd\ft,\ell}(\dd\ft,\dd\ft)\nabla_0^E\hat\omega,\nabla_0^E\hat\omega\big\ra_E + 2 \Re\big\la T_{G,\phi^2\,\dd\ft,\ell}(\dd\ft,\xi)\nabla_0^E\hat\omega,i\hat\omega\big\ra_E + \big\la T_{G,\phi^2\,\dd\ft,\ell}(\xi,\xi)\hat\omega,\hat\omega\big\ra_E\,\dd\xi \\
    &\qquad \geq C\int_{\R^n_\xi} \phi^2\bigl(|\nabla_0^E\hat\omega|_E^2 + |\xi|_R^2|\hat\omega|_E^2\bigr)\,\dd\xi,
  \end{align*}
  where we identify $\xi\in\R^n$ with the 1-form $\xi\cdot\dd x$. This estimate however is an immediate consequence of Lemma~\ref{LemmaEIT} for $\nu=\dd\ft$, $\nu'=\xi\cdot\dd x$ and $e=\nabla_0^E\hat\omega$, $f=i\hat\omega$.

  If $\nabla^E=\dd+A$ for $A\in\CI(\R^{n+1};T^*\R^{n+1}\otimes\End(E))$, the estimate follows for slightly smaller $C$ but now with $C'>0$ by Cauchy--Schwarz. For general $g,\ell,\phi$, one uses a partition of unity on $\ol{B_1}$, localizing to small regions $\cU$ on which $T_{G,\phi^2\,\dd\ft,\ell}$, $\ell$, $g$, and $\phi$ vary by small amounts so that the desired positivity follows from the constant coefficient case, up to an error term---arising from the difference of $T_{G,\phi^2\,\dd\ft,\ell}$ etc.\ and their constant approximations---which can be estimated by a small constant times $\int_\cU \phi^2|\nabla^E \omega|_{R\otimes E}^2\,\dd x$. Note also that terms arising from commuting $\nabla^E$ through an element of the partition of unity localizing to $\cU$ give terms which can be estimated from below by a positive constant times
  \[
    -\int_\cU \phi^2|\omega|_E|\nabla^E \omega|_{R\otimes E}\,\dd x\geq -\eps\int_\cU \phi^2|\nabla^E \omega|_{R\otimes E}^2\,\dd x-C_\eps\int_\cU \phi^2|\omega|_E^2\,\dd x
  \]
  for any $\eps>0$. This completes the proof.
\end{proof}

This formulation allows for an immediate application of the estimate~\eqref{EqEICoercive} to the following setting: $\R^{n+1}$ is replaced by $\tilde\sfM:=I_\ft\times\sfM$ where $I\ni 0$ is an open interval and $\sfM$ is an $n$-dimensional manifold of bounded geometry; the latter means that $\sfM$ has a cover by coordinate charts $\phi_i\colon\cU_i\xra{\cong} B_3\subset\R^n$, with the smaller balls $\phi_i^{-1}(B_1)$ still covering $\sfM$, and such that the transition functions $\phi_i\circ\phi_j^{-1}$ are uniformly bounded in $\CI$. We further require $E\to\tilde\sfM$ to be a vector bundle, with Hermitian inner product and affine connection, of bounded geometry, i.e., it admits trivializations over the preferred charts $\tilde\phi_i\colon I\times\cU_i\to I\times B_3$ on $\tilde\sfM$ for which the transition functions as well as the fiber inner product and the affine connection have uniformly bounded coefficients. Finally, in these induced preferred charts, the data $g$, $(E,\la\cdot,\cdot\ra_E)$, $\nabla^E$, $\ell$, $\phi$, and $\dd\ft$ are required to lie in a compact subset of $\CI(I\times B_3)$ on which the hypotheses of Corollary~\ref{CorEICoercive} are satisfied. In this setup,~\eqref{EqEICoercive} applies to all $\omega\in\CIc(\tilde\sfM;E)$.

This situation arises naturally for metrics which have controlled asymptotic behavior on a suitable compactification of an underlying spacetime, in particular for the conformally rescaled spacetime
\[
  (M,g_\tbop),\quad g_\tbop:=\rho_\sface^2 g\in\CI(M;S^2\,\Ttb^*M),
\]
with $g_\tbop$ a nondegenerate Lorentzian 3b-metric. Namely, near $X$, we can take $\ft=\tau=\frac{t}{r}$, with the preferred charts on $\ft$-level sets being direct products of charts on $\Sph^2$ with dyadic intervals in $\R_r$; thus~\eqref{EqEICoercive} becomes an estimate for b-derivatives of $\omega$ (i.e., derivatives along spherical vector fields, $r\pa_r$, and $\pa_\ft=r\pa_t$), which is precisely what 3b-derivatives are away from $\cK^+$. Near $\Sigma^\sharp$, we can take $\ft=\bhm-r$, and the preferred charts on $r$-level sets are direct products of charts on $\Sph^2$ with unit length intervals in $t$; in this case~\eqref{EqEICoercive} becomes an estimate for derivatives along spherical vector fields, $\pa_t$, and $\pa_\ft=-\pa_r$, which also here means regularity along 3b-vector fields. (In practice, we shall work directly with $g$, in which case we get estimates on 3b-Sobolev spaces with appropriately shifted weights.)

The following result gives energy estimates for both settings:

\begin{prop}[Energy estimates near $X$ and $\Sigma^\sharp$]
\label{PropEI}
\fakephantomsection
  Let $\alpha_\sface,\alpha_\cK\in\R$.
  \begin{enumerate}
  \item\label{ItEISigma}{\rm (Estimate near $X$.)} Recall $\tau=\frac{t}{r}$, and let $T>0$ be such that $\dd\tau$ is past timelike for $g_\tbop$ on $\tau^{-1}([0,T])$. We drop the weight at $\cK^+$ from the notation. Then there exist $h_0>0$ and a constant $C$ such that for all $\omega\in\rho_\sface^{\alpha_\sface}\Htbh^2(M)$ which vanish for $t<0$, the following estimate holds for all $h\in(0,h_0)$:
    \begin{equation}
    \label{EqEISigmaFw}
      \|\omega\|_{\rho_\sface^{\alpha_\sface}\Htbh^1(\tau^{-1}([0,T]))} \leq C h^{-1}\|P_h \omega\|_{\rho_\sface^{\alpha_\sface+2}L^2(\tau^{-1}([0,T]))}.
    \end{equation}
    For $\omega\in\rho_\sface^{-\alpha_\sface-2}\Htbh^2(M)$ without support restrictions, we have
    \begin{equation}
    \label{EqEISigmaBw}
      \|\omega\|_{\rho_\sface^{-\alpha_\sface-2}\Htbh^1(\tau^{-1}([0,T/2]))} \leq C\Bigl( h^{-1}\|P_h^*\omega\|_{\rho_\sface^{-\alpha_\sface}L^2(\tau^{-1}([0,T]))} + \|\omega\|_{\rho_\sface^{-\alpha_\sface-2}\Htbh^1(\tau^{-1}([T/2,T]))} \Bigr).
    \end{equation}
  \item\label{ItEISigmaSharp}{\rm (Estimate near $\Sigma^\sharp$.)} Let $\bhm\leq r_1<r_2<r_3<r_+$. We drop the weight at $\sface$ from the notation. Then there exist $h_0>0$ and $C$ such that for all $\omega\in\rho_\cK^{\alpha_\cK}\Htbh^2(M)$ which vanish for $t<0$, the following estimate holds for all $h\in(0,h_0)$:
    \begin{equation}
    \label{EqEISigmaSharpFw}
      \|\omega\|_{\rho_\cK^{\alpha_\cK}\Htbh^1(r^{-1}([r_1,r_2]))} \leq C\Bigl( h^{-1}\|P_h \omega\|_{\rho_\cK^{\alpha_\cK}L^2(r^{-1}([r_1,r_3]))} + \|\omega\|_{\rho_\cK^{\alpha_\cK}\Htbh^1(r^{-1}([r_2,r_3]))} \Bigr).
    \end{equation}
    For all $\omega\in\rho_\cK^{-\alpha_\cK}\Htbh^2(M)$ which vanish for $r<r_1$ (but not necessarily for $t<0$), we have
    \begin{equation}
    \label{EqEISigmaSharpBw}
      \|\omega\|_{\rho_\cK^{-\alpha_\cK}\Htbh^1(r^{-1}([r_1,r_3])\cap\{\tau\geq 0\})} \leq C h^{-1}\|P_h^*\omega\|_{\rho_\cK^{-\alpha_\cK}L^2(r^{-1}([r_1,r_3])\cap\{\tau\geq 0\})}.
    \end{equation}
  \end{enumerate}
\end{prop}

We have analogous semiclassical energy estimates in compact subsets of $\Omega^\circ$ as well; we shall not state such estimates explicitly here. Proposition~\ref{PropEI} sharpens (and is more natural than) \cite[Propositions~8.26 and 8.27]{HintzVasyKdSStability}, in that the domains in which the solution and the forcing are controlled coincide here; we accomplish this via a slightly different choice of commutants, and by exploiting the strong nature of the non-scalar damping term $-i L_h$ of $P_h$ more effectively.

\begin{proof}[Proof of Proposition~\usref{PropEI}]
  \pfstep{Preliminary computations.} The strategy is to use $a^2 L_h$ as a non-scalar vector field multiplier, where the scalar function $a$ on $M$, chosen later, encodes weights and has timelike differential; the choices relevant for our present purposes are described below. Still working with~\eqref{EqENotation}, we then consider the $L^2(M,|\dd g|;\cT^*)$-pairing
  \begin{equation}
  \label{EqEIComm}
    h^{-1}\Im\la P_h \omega,a^2 L_h \omega\ra = \la\sC \omega,\omega\ra - h^{-1}\|a L_h \omega\|^2,
  \end{equation}
  where we formally integrated by parts (i.e., we dropped boundary terms) and introduced
  \begin{align*}
    \sC &= \frac{i}{2 h}(\Box_h^* a^2 L_h-L_h^*a^2\Box_h) \\
      &= \frac{i}{2 h} \bigl( [\Box_h,a^2]L_h - [L_h,a^2]\Box_h + a^2[\Box_h,L_h] \bigr) + h^{-1}(\Im\Box_h)a^2 L_h-h^{-1}(\Im L_h)a^2\Box_h.
  \end{align*}
  Working in $M^\circ$, we have $\sC\in h^2\Diff^2$ (acting on sections of $T^*M^\circ$) since $\Box_h=:h^2\Box\in h^2\Diff^2$, $L_h=:h L\in h\Diff^1$, and the semiclassical principal symbol of $\sC$ is given by
  \begin{align*}
    T_z^*M^\circ \ni \zeta & \mapsto a(z)(H_G a)|_\zeta \ell(\zeta) - a(z)\ell|_z(\dd a) G(\zeta) - a(z)^2 E(\zeta,\zeta) &\\
      &= a(z) T_{G,\dd a,\ell}(\zeta,\zeta) - a(z)^2 E(\zeta,\zeta)
  \end{align*}
  in the notation of Lemma~\ref{LemmaEIT}, where $E(\cdot,\cdot)$ is the symmetric bilinear form (independent of $a$), valued in the space of self-adjoint endomorphisms of $T^*M^\circ$, which at $(\zeta,\zeta)$ evaluates to $(-1)$ times the semiclassical principal symbol of the second order semiclassical differential operator $\frac{i}{2 h}[\Box_h,L_h]+h^{-1}(\Im\Box_h)L_h-h^{-1}(\Im L_h)\Box_h$. In the passage to the second line, we used the identity $(H_G a)|_\zeta=2 g^{-1}(\zeta,\dd a)$.

  Write now $\nabla$ for any affine connection on the bundle $T^*M^\circ\to M^\circ$. Put $b:=a^2$ and $\sC_b:=\frac{i}{2}(\Box^*b L-L^*b\Box)$. Since this has principal symbol $\frac12 T_{G,\dd b,\ell}(\zeta,\zeta)-b E(\zeta,\zeta)$, we have
  \[
    \tilde\sC_b:=\sC_b - \nabla^*\bigl(\tfrac12 T_{G,\dd b,\ell}-b E\bigr)\nabla \in \Diff^1,
  \]
  which can thus be written as $\sC_{(1)}(b)\nabla+\sC_{(0)}(b)$, with the coefficients $\sC_{(0)},\sC_{(1)}$ depending only on the 2-jet of $b$. In fact, $\sC_{(1)}(b)$ does not contain any second-order derivatives of $b$: in local coordinates, the only first-order term of $\sC_b$ involving $\pa^2 b$ is $-\frac{i}{2}g^{\mu\nu}b_{;\mu\nu}\ell^\lambda D_\lambda=\frac12(\Box_g b)\ell^\lambda\pa_\lambda$, which is equal to that of $\nabla^*\frac12 T_{G,\dd b,\ell}\nabla$, i.e., to that of $-\frac12(T_{G,\dd b,\ell}^{\mu\lambda})_{;\mu}\pa_\lambda$. (The verification is straightforward using $T_{G,\dd b,\ell}^{\mu\lambda}=b^{;\mu}\ell^\lambda+b^{;\lambda}\ell^\mu-g^{\mu\lambda}b_{;\kappa}\ell^\kappa$.) Therefore, we have $\sC_{(1)}(b)=\sC_{(1)}^\sharp(\dd b)+\sC_{(1)}^\flat(b)$ for $b$-independent tensor fields $\sC_{(1)}^\sharp,\sC_{(1)}^\flat$. The only $\pa^2 b$-dependence of the coefficients of $\tilde\sC_b$ is thus in $\sC_{(0)}(b)$; but any $\pa^2 b$-term can be written as $\nabla^*\sC_{(0)}^\sharp(\dd b)+\sC_{(0)}^\flat(\dd b)$ for suitable $b$-independent tensor fields $\sC_{(0)}^\sharp$, $\sC_{(0)}^\flat$. Inserting $b=a^2$ shows that one can write
  \begin{equation}
  \label{EqEIComm2}
  \begin{split}
    h^{-2}\sC = \sC_{a^2} &= -\nabla^*\bigl( a T_{G,-\dd a,\ell}+E a^2\bigr)\nabla \\
      &\qquad + \nabla^*\bigl(a E^\sharp(\dd a)+E^\flat a^2\bigr) + \bigl(a \tilde E^\sharp(\dd a) + \tilde E^\flat a^2\bigr)\nabla + a\tilde E(\dd a) + a^2\tilde E',
  \end{split}
  \end{equation}
  where, upon writing $T^*=T^*M^\circ$, the tensors $E,E^\sharp,E^\flat$, and $\tilde E^\sharp,\tilde E^\flat,\tilde E,\tilde E'$ are independent of $a$, and smooth sections of the bundles $\End(T^*\otimes T^*)$, $\Hom(T^*,\Hom(T^*,T^*\otimes T^*))$, $\Hom(T^*,T^*\otimes T^*)$, and $\Hom(T^*,\Hom(T^*\otimes T^*,T^*))$, $\Hom(T^*\otimes T^*,T^*)$, $\Hom(T^*,\End(T^*))$, $\End(T^*)$, respectively. (This form is wasteful, but symmetric.) Roughly speaking then, the first, main, term in~\eqref{EqEIComm2} will control derivatives of $\omega$ by Corollary~\ref{CorEICoercive}, thus also $\omega$ itself upon integrating; the remaining terms will be much smaller by Cauchy--Schwarz and upon choosing $a$ such that $\dd a$ dominates $a$. Another error term is the $L^2$-norm of $\omega$ on the boundary of the domain of integration, which we shall be able to control from the initial data by making use of the strong (due to the presence of $h^{-1}$) second term in~\eqref{EqEIComm}.

  \pfstep{Estimates for $P_h$.} We implement this strategy to prove~\eqref{EqEISigmaFw}. Let
  \[
    a = \rho_\sface^{-\alpha_\sface-2}\chi_0(\tau)\chi_1(\tau)\chi_2(r),\quad
    \chi_0(\tau)=e^{-\digamma\tau},\ 
    \chi_1(\tau)=H(T-\tau),\ 
    \chi_2(r)=H(r-\bhm),
  \]
  with $\digamma>1$ chosen below; here $H$ is the Heaviside step function. We aim to evaluate~\eqref{EqEIComm} using the formula~\eqref{EqEIComm2} where as the connection on $\cT^*$ we take a 3b-connection
  \[
    \nabla\in\Difftb^1(M;\cT^*,\Ttb^*M\otimes\cT^*\bigr).
  \]
  Concretely, one can take this to be the Levi-Civita connection of $(M,g)$, for which this membership follows from~\eqref{EqK3bNabla} for $\cE=\cT^*$ and the fact that the space of smooth sections of $\Ttb^*M$ is equal to the space of products of $\rho_\sface$ with smooth sections of $\cT^*$ (as follows from Remark~\ref{RmkK3b3sc}). Now,
  \begin{equation}
  \label{EqEIada}
  \begin{split}
    a\,\dd a &= -\rho_\sface^{-2\alpha_\sface-4}\chi_0^2\chi_1^2\chi_2^2\Bigl(\digamma\,\dd\tau+(\alpha_\sface+2)\frac{\dd\rho_\sface}{\rho_\sface}\Bigr) \\
      &\qquad - \tfrac12\rho_\sface^{-2\alpha_\sface-4}\chi_0^2\chi_2^2\delta(T-\tau)\,\dd\tau + \tfrac12\rho_\sface^{-2\alpha_\sface-4}\chi_0^2\chi_1^2\delta(r-\bhm)\,\dd r.
  \end{split}
  \end{equation}
  (In this computation, one avoids products of Heaviside and $\delta$-distributions by computing $\frac12\dd(a^2)$.) As far as weights are concerned, recall that $\rho_\sface^{-2}\ell\in\CI(\Ttb^*M;\pi^*\End(\cT^*))$, $\rho_\sface^{-2}G\in\CI(\Ttb^*M)$, and
  \begin{equation}
  \label{EqEITweight}
    a T_{G,-\dd a,\ell} = T_{G,-a\,\dd a,\ell}=\rho_\sface^4 T_{\rho_\sface^{-2}G,-a\,\dd a,\rho_\sface^{-2}\ell} .
  \end{equation}
  We now integrate two versions of the estimate~\eqref{EqEICoercive} over $\{r\geq\bhm,\ 0\leq\tau\leq T\}$: once for $\dd\tau$ and $\phi^2=\rho_\sface^{-2\alpha_\sface-4}\chi_0^2\chi_2^2(\chi_1^2\digamma+\delta(T-\tau))=\digamma a^2+\rho_\sface^{-2\alpha_\sface-4}\chi_0^2\chi_2^2\delta(T-\tau)$, and once for $-\dd r$ (which, like $\dd\tau$, is past timelike) and $\phi^2=\rho_\sface^{-2\alpha_\sface-4}\chi_0^2\chi_1^2\delta(r-\bhm)$. We furthermore absorb the contribution from the $\frac{\dd\rho_\sface}{\rho_\sface}$-term in~\eqref{EqEIada} by choosing $\digamma>1$ large; this gives
  \begin{equation}
  \label{EqEICommMain}
  \begin{split}
    &{-}\la T_{G,-a\,\dd a,\ell}\nabla \omega,\nabla \omega\ra \\
    &\quad\leq -C\digamma\|\rho_\sface^2 a\nabla \omega\|^2 - C\|\chi_0\chi_2\nabla \omega\|_{\rho_\sface^{\alpha_\sface}L^2(\tau^{-1}(T))}^2 - C\|\chi_0\chi_1\nabla\omega\|_{L^2(r^{-1}(\bhm))}^2 \\
    &\quad\qquad + C'\digamma\|\rho_\sface^2 a \omega\|^2 + C'\|\chi_0\chi_2\omega\|_{\rho_\sface^{\alpha_\sface}L^2(\tau^{-1}(T))}^2 + C'\|\chi_0\chi_1\omega\|_{L^2(r^{-1}(\bhm))}^2.
  \end{split}
  \end{equation}
  Here, we use the volume density $r^3|\dd r\,\dd\slg|$ on level sets of $\tau$, which upon tensoring with $|\dd\tau|$ is a smooth positive multiple of the metric density $|\dd g|$.

  We proceed to estimate the remaining terms in~\eqref{EqEIComm2} when applied to $\omega$ and then paired with $\omega$. We first remark that $E\in\rho_\sface^4\CI(M;\End(\Ttb^*M\otimes\cT^*))$, and similarly for the other terms; that is, the total weight at $\sface$ of all terms as 3b-differential operators matches the weight $\rho_\sface^4$ of $\nabla^*T_{G,-a\,\dd a,\ell}\nabla$ (see~\eqref{EqEITweight}). Note then that the term with $E$ does not involve differentiation of $a$, and hence contributes $|\la E a^2\nabla \omega,\nabla \omega\ra|\leq C\|\rho_\sface^2 a\nabla \omega\|^2$, which can be absorbed into the first term on the right in~\eqref{EqEICommMain} for $\digamma>1$ large. For $E^\sharp$, we use the Peter--Paul inequality to obtain
  \begin{equation}
  \label{EqEIPeterPaul}
  \begin{split}
    &|\la a E^\sharp(\dd a)\omega,\nabla \omega\ra| \\
    &\qquad \leq \bigl(\eta\digamma\|\rho_\sface^2 a\nabla \omega\|^2 + C_\eta\digamma\|\rho_\sface^2 a \omega\|^2\bigr) + \bigl(\eta\|\chi_0\chi_2\nabla\omega\|_{\rho_\sface^{\alpha_\sface}L^2(\tau^{-1}(T))}^2 + C_\eta\|\chi_0\chi_2\omega\|_{\rho_\sface^{\alpha_\sface}L^2(\tau^{-1}(T))}^2\bigr) \\
    &\qquad \quad \quad + \bigl(\eta\|\chi_0\chi_1\nabla\omega\|_{L^2(r^{-1}(\bhm))}^2 + C_\eta\|\chi_0\chi_1\omega\|_{L^2(r^{-1}(\bhm))}^2\bigr)
  \end{split}
  \end{equation}
  for all $\eta>0$, likewise for the term involving $\tilde E^\sharp$, and similarly (but without the factor $\digamma$) for $E^\flat$, $\tilde E^\flat$. The terms from $\tilde E$ and $\tilde E'$ are bounded by $C\bigl(\digamma\|\rho_\sface^2 a \omega\|^2+\|\chi_0\chi_2\omega\|_{\rho_\sface^{\alpha_\sface}L^2(\tau^{-1}(T))}^2+\|\chi_0\chi_1\omega\|_{L^2(r^{-1}(\bhm))}^2\bigr)$ and $C\|\rho_\sface^2 a \omega\|^2$, respectively.
  
  We plug these estimates into the right-hand side of~\eqref{EqEIComm}. The left-hand side of~\eqref{EqEIComm}, on the other hand, we can bound from below by $-\frac12 h^{-1}\|a P_h \omega\|^2-\frac12 h^{-1}\|a L_h \omega\|^2$; the term $\frac12 h^{-1}\|a L_h\omega\|^2$ can be absorbed into the term $h^{-1}\|a L_h\omega\|^2$ in~\eqref{EqEIComm}. We thus obtain from $h^2$ times~\eqref{EqEICommMain} and the subsequent error estimates the $L^2(M,|\dd g|;\cT^*)$-estimate
  \begin{equation}
  \label{EqEIEstD}
  \begin{split}
    \digamma\|\rho_\sface^2 a h\nabla \omega\|^2 + h^{-1}\|a L_h \omega\|^2 &\leq C\Bigl( h^{-1}\|a P_h \omega\|^2 + h^2\digamma\|\rho_\sface^2 a \omega\|^2 \\
      &\quad \qquad + h^2\|\chi_0\chi_2\omega\|_{\rho_\sface^{\alpha_\sface}L^2(\tau^{-1}(T))}^2 + h^2\|\chi_0\chi_1\omega\|_{L^2(r^{-1}(\bhm))}^2\Bigr)
  \end{split}
  \end{equation}
  for all sufficiently large $\digamma$, with $C$ independent of $\digamma$. Here, we absorbed the boundary terms involving $\nabla\omega$ (such as $\eta\|\chi_0\chi_2\nabla\omega\|_{\rho_\sface^{\alpha_\sface}L^2(\tau^{-1}(T))}^2$ in~\eqref{EqEIPeterPaul}) into the boundary terms involving $\nabla\omega$ in~\eqref{EqEICommMain} upon choosing $\eta>0$ sufficiently small.

  Now,~\eqref{EqEIEstD} controls $h\nabla\omega$ by (a small constant times) $\omega$ itself (plus a norm of $P_h\omega$); to close the estimate, we need to control $\omega$ (including its boundary values at $\tau=T$ and $r=\bhm$) in turn by $h\nabla\omega$. We shall do this, roughly speaking, using the fundamental theorem of calculus for $L_h$. More precisely, consider the pairing
  \begin{equation}
  \label{EqEIFTC}
    h^{-1}\Im\la L_h \omega,\rho_\sface^2 a^2 \omega\ra = \big\la \sC' \omega, \omega\big\ra,\quad
    \sC'=\rho_\sface a\,\ell\bigl(\dd(\rho_\sface a)\bigr) + S\rho_\sface^2 a^2.
  \end{equation}
  For sufficiently large $\digamma$, the covector $\rho_\sface a\,\dd(\rho_\sface a)$ is future timelike by~\eqref{EqEIada}. Lemma~\ref{LemmaCTimelike} then implies that the endomorphism $\sC'$ is the sum of three terms: one that is bounded from below by $C\digamma\rho_\sface^4 a^2$ on $r\geq\bhm$ and $0\leq\tau\leq T$, and two $\delta$-distributions: one which is a positive multiple of $\rho_\sface^{-2\alpha_\sface}\chi_0^2\chi_2^2\delta(T-\tau)$, and one which is a positive multiple of $\chi_0^2\chi_1^2\delta(r-\bhm)$. Applying Cauchy--Schwarz to the left-hand side of~\eqref{EqEIFTC} (with one factor of $a$ moved to the first slot) thus gives
  \[
    \digamma\|\rho_\sface^2 a \omega\|^2 + \|\chi_0\chi_2 \omega\|_{\rho_\sface^{\alpha_\sface}L^2(\tau^{-1}(T))}^2 + \|\chi_0\chi_1\omega\|_{L^2(r^{-1}(\bhm))}^2 \leq C h^{-2}\|a L_h \omega\|^2
  \]
  for large $\digamma>1$. Using~\eqref{EqEIEstD}, we obtain
  \begin{align*}
    &\digamma\|\rho_\sface^2 a \omega\|^2 + \|\chi_0\chi_2\omega\|_{\rho_\sface^{\alpha_\sface}L^2(\tau^{-1}(T))}^2 + \|\chi_0\chi_1\omega\|_{L^2(r^{-1}(\bhm))}^2 \\
    &\qquad \leq C\Bigl( h^{-2}\|a P_h \omega\|^2 + h\digamma\|\rho_\sface^2 a \omega\|^2 + h\|\chi_0\chi_2\omega\|_{\rho_\sface^{\alpha_\sface}L^2(\tau^{-1}(T))}^2 + h\|\chi_0\chi_1\omega\|_{L^2(r^{-1}(\bhm))}^2 \Bigr).
  \end{align*}
  Since $\digamma$ is fixed now, we can, for all sufficiently small $h>0$, absorb the last three terms on the right into the left-hand side and obtain
  \begin{equation}
  \label{EqEIFTC2}
    \digamma\|\rho_\sface^2 a \omega\|^2 + \|\chi_0\chi_2\omega\|_{\rho_\sface^{\alpha_\sface}L^2(\tau^{-1}(T))}^2 + \|\chi_0\chi_1\omega\|_{L^2(r^{-1}(\bhm))}^2 \leq C h^{-2}\|a P_h \omega\|^2.
  \end{equation}
  Adding this to~\eqref{EqEIEstD} and absorbing the last three terms on the right in~\eqref{EqEIEstD} into the left-hand side of~\eqref{EqEIFTC2} for small $h>0$, we thus get
  \[
    \|\rho_\sface^2 a \omega\|_{\bar H_{\tbop,h}^1(\tau^{-1}((0,T)))} \leq C h^{-1}\|a P_h \omega\|_{L^2},\quad 0<h<h_0,
  \]
  provided $h_0>0$ is sufficiently small.

  The proof of~\eqref{EqEISigmaSharpFw} is very similar to the previous arguments after a preliminary step. Let $\chi\equiv 1$ for $r\leq r_2$, and $\chi\equiv 0$ for $r\geq r_3$. Put $\omega'=\chi \omega$. Set
  \[
    f' := P_h \omega' = \chi P_h \omega + f,\quad f:=[P_h,\chi]\omega,
  \]
  so $f$, resp.\ $f'$ is supported in $r^{-1}([r_2,r_3])$, resp.\ $\{r\leq r_3\}$. Note that
  \begin{align*}
    h^{-1}\|f'\|_{\rho_\cK^{\alpha_\cK}L^2(r^{-1}([r_1,r_3]))} &\leq h^{-1}\|\chi P_h \omega\|_{\rho_\cK^{\alpha_\cK}L^2(r^{-1}([r_1,r_3]))} + h^{-1}\|f\|_{\rho_\cK^{\alpha_\cK}L^2(r^{-1}([r_1,r_3]))} \\
      &\leq C\Bigl( h^{-1}\|P_h \omega\|_{\rho_\cK^{\alpha_\cK}L^2(r^{-1}([r_1,r_3]))} + \|\omega\|_{\rho_\cK^{\alpha_\cK}\Htbh^1(r^{-1}([r_2,r_3]))}\Bigr).
  \end{align*}
  It thus remains to prove the estimate
  \[
    \|\omega'\|_{\rho_\cK^{\alpha_\cK}\Htbh^1(r^{-1}([r_1,r_3]))} \leq C h^{-1}\|f'\|_{\rho_\cK^{\alpha_\cK}L^2(r^{-1}([r_1,r_3]))}
  \]
  for functions $\omega'\in\rho_\cK^{\alpha_\cK}\Htbh^2(M)$ which vanish when $r>r_3$ or $t<0$. But $-\dd r$ is past timelike in $r^{-1}([r_1,r_3])$, and hence the commutant
  \[
    a = \rho_\cK^{-\alpha_\cK}\chi_0(r)\chi_1(r),\quad
    \chi_0(r)=e^{\digamma r},\ 
    \chi_1(r)=H(r-r_1),\quad \digamma\gg 1,
  \]
  can be used in~\eqref{EqEIComm} (with $\omega'$ in place of $\omega$) to yield the desired estimate~\eqref{EqEISigmaSharpFw}.

  \pfstep{Estimates for $P_h^*$.} Since $P_h^*=\Box_h^*+i L_h^*$, the term $i L_h^*$ is a damping term only for propagation in the \emph{backward} direction. This explains the structure of the estimates~\eqref{EqEISigmaBw} and \eqref{EqEISigmaSharpBw}; their proofs are completely analogous to those of~\eqref{EqEISigmaSharpFw} and \eqref{EqEISigmaFw}.
\end{proof}

\subsubsection{Combination; proof of Theorem~\usref{ThmET}}
\label{SssEP}

We shall now collect the results from the previous sections to prove Theorem~\ref{ThmET}; \emph{we operate under its assumptions for the remainder of this section.} Since it is somewhat more intuitive, we first prove an a priori estimate for the forward problem for $P_h=\Box_{g,h}^{\cC_h}$.

\begin{prop}[Warm-up: a priori estimate]
\label{PropEPWarmup}
  Let $s\geq 1$. Then there exists $h_0>0$ such that for all $h\in(0,h_0)$ and for all $\omega\in\rho_\sface^{\alpha_\sface}\rho_\cK^{\alpha_\cK}\Htbh^{s+1}(\Omega)^{\bullet,-}$ with
  \[
    f:=P_h \omega=\Box_{g,h}^{\cC_h}\omega\in\rho_\sface^{\alpha_\sface+2}\rho_\cK^{\alpha_\cK}\Htbh^{s-1}(\Omega)^{\bullet,-},
  \]
  the estimate
  \begin{equation}
  \label{EqEPApriori}
    \|\omega\|_{\rho_\sface^{\alpha_\sface}\rho_\cK^{\alpha_\cK}\Htbh^s(\Omega)^{\bullet,-}} \leq C h^{-1}\|f\|_{\rho_\sface^{\alpha_\sface+2}\rho_\cK^{\alpha_\cK}\Htbh^{s-1}(\Omega)^{\bullet,-}}
  \end{equation}
  holds.
\end{prop}
\begin{proof}
  We first consider the case $s=1$. We begin with the energy estimate~\eqref{EqEISigmaFw} for $T>0$ as used in~\eqref{EqEECutoffs}, which controls $\omega$ in a neighborhood of $X\subset M$. Recall now the cutoff functions~\eqref{EqEECutoffs}. We can use the microlocal estimates from Proposition~\ref{PropEE} (elliptic regularity) and Proposition~\ref{PropEEInfty} (control at the characteristic set at fiber infinity, $\pa\Sigma$) to control $\chi_{\tau,2}\chi_{r,2}\omega$ in $\rho_\sface^{\alpha_\sface}\rho_\cK^{\alpha_\cK}H_{\tbop,h}^1(M;\cT^*)$ outside the zero section; the second term on the right in~\eqref{EqEEInfty} is already controlled by the energy estimate~\eqref{EqEISigmaFw}.

  Next, near the zero section, we can propagate microlocal $\rho_\sface^{\alpha_\sface}\rho_\cK^{\alpha_\cK}H_{\tbop,h}^1(M;\cT^*)$-estimates along $-r\cd^\sharp$ inside of $\Omega^\circ$ and also in $\sface\cap\{\frac{r}{t}\geq v(1+\delta_\sface)\}$ (see the end of~\S\ref{SssER}). Next, we propagate control into a neighborhood of the saddle point at $r=r_0$ in $(\cK^+)^\circ$ by means of Proposition~\ref{PropEK}; this can then be propagated towards $r=\bhm<r_0$ by means of Proposition~\ref{PropER}, and likewise to $r=\delta^{-1}\geq r_0$. At this point, the term $E\omega$ in the estimate~\eqref{EqEB} of Proposition~\ref{PropEB} is now controlled, and hence we can propagate control to a neighborhood $\{\frac{r}{t}\leq\delta_r,\ \rho_\sface\leq\delta\}$ of the zero section over $\pa\cK^+$. This control is then further propagated along $\sface\cap\{\delta_r\leq \frac{r}{t}\leq v(1-\delta_\sface)\}$. Now we have control at the zero section outside of $\{|\frac{r}{t}/v-1|\leq\delta_\sface,\ \rho_\sface\leq\delta\}$ and can therefore apply Proposition~\ref{PropES} to obtain control near the sink of $-r\cd^\sharp$ in $\sface^\circ$. Altogether, we have now proved the estimate
  \begin{align*}
    &\|\omega\|_{\Htbh^{1,(\alpha_\sface,\alpha_\cK)}(\tau^{-1}([0,T]))} + \| \chi_{\tau,2}\chi_{r,2}\omega \|_{\Htbh^{1,(\alpha_\sface,\alpha_\cK)}} \\
    &\qquad \leq C\Bigl( h^{-1}\| P_h \omega \|_{\Htbh^{0,(\alpha_\sface+2,\alpha_\cK)}(\Omega)^{\bullet,-}} + h^N\|\chi_{\tau,0}\chi_{r,0}\omega \|_{\Htbh^{-N,(\alpha_\sface,\alpha_\cK)}} \Bigr).
  \end{align*}

  Note that $\chi_{r,2}=1$ for $r\geq r_2:=\bhm+\frac{2+\frac34}{10}(r_+-\bhm)$. To close the estimate, it thus remains to estimate $\omega$ on $r^{-1}([\bhm,r_2])$, which we accomplish using the energy estimate~\eqref{EqEISigmaSharpFw} near $\Sigma^\sharp$; there we take $r_1=\bhm$ and $r_3=\frac{r_2+r_+}{2}$. This yields
  \begin{equation}
  \label{EqEPAprioriAlmost}
    \|\omega\|_{\Htbh^{1,(\alpha_\sface,\alpha_\cK)}(\Omega)^{\bullet,-}} \leq C\Bigl( h^{-1}\| P_h \omega \|_{\Htbh^{0,(\alpha_\sface+2,\alpha_\cK)}(\Omega)^{\bullet,-}} + h^N\|\chi_{\tau,0}\chi_{r,0}\omega \|_{\Htbh^{-N,(\alpha_\sface,\alpha_\cK)}(\Omega)^{\bullet,-}} \Bigr).
  \end{equation}
  For sufficiently small $h>0$, we can absorb the second term on the right into the left-hand side, and thereby obtain~\eqref{EqEPApriori} for $s=1$.

  In order to prove~\eqref{EqEPApriori} for general $s\geq 1$, notice that in the previous arguments we only used $s=1$ to apply the energy estimates from Proposition~\ref{PropEI}. A standard extension/restriction argument (see, e.g., \cite[Lemma~2.7]{HintzVasySemilinear} and \cite[Chapter~9.4]{HintzMicro}) combined with the propagation of regularity removes this restriction. We only sketch the argument here. We extend $f$ to a function $\tilde f\in\rho_\sface^{\alpha_\sface+2}\rho_\cK^{\alpha_\cK}\Htbh^{s-1}(\tilde\Omega)^{\bullet,-}$ where $\tilde\Omega={\rm cl}(\{t\geq 0,\ r\geq\bhm-\eps\})\subset M_0$; we choose this extension to have the same norm as $f$. (Here, we use the quotient norm on $\rho_\sface^{\alpha_\sface+2}\rho_\cK^{\alpha_\cK}\Htbh^{s-1}(\Omega)^{\bullet,-}$ induced by the surjective restriction map from $\tilde\Omega$ to $\Omega$, and the existence of $\tilde f$ is then straightforward.) We then solve the forward problem for the wave equation
  \[
    \Box_{g,h}^{\cC_h}\tilde \omega=\tilde f
  \]
  (with $g$ extended as the Kerr metric to this region, and $\cd$ extended as a stationary and past timelike 1-form). This is a wave equation (with $r$ a time function for $\bhm-\eps\leq r<r_+$), and the existence of $\tilde\omega\in\rho_\sface^{\alpha_\sface}\rho_\cK^{\alpha_\cK}\Htbh^0(\tilde\Omega)$ which satisfies the norm bound $\|\tilde \omega\|_{\rho_\sface^{\alpha_\sface}\rho_\cK^{\alpha_\cK}\Htbh^0(\tilde\Omega)^{\bullet,-}}\leq C h^{-1}\|\tilde f\|\leq C' h^{-1}\|f\|$ (where the norms on $\tilde f$ and $f$ can be taken to be the $\rho_\sface^{\alpha_\sface+2}\rho_\cK^{\alpha_\cK}\Htbh^{-1}(\tilde\Omega)^{\bullet,-}$-norms) can be shown using duality and the Hahn--Banach theorem (cf.\ the estimate~\eqref{EqEPDual} below, for $s'=1$, and the subsequent discussion). Moreover, $\tilde\omega=\omega$ on $\Omega$ by finite speed of propagation. The elliptic regularity and propagation results can now be applied to $\tilde\omega$ (including near $\tau=0$ for $r>\bhm-\frac12\eps$, where we propagate from $\tau<0$ where $\tilde\omega=0$ to small $\tau>0$), giving $\rho_\cK^{\alpha_\cK}\Htbh^s$-control of $\tilde\omega$ in $r^{-1}([\bhm-\frac12\eps,r_+])$, and thus by restriction also of $\omega$ in $r^{-1}([\bhm,r_+])$.
\end{proof}

Now, we are actually interested in \emph{solvability} of $P_h \omega=f$ on the function spaces in~\eqref{EqEPApriori}. In order to establish this, we argue by duality using the adjoint versions of the elliptic, complex absorption type, and radial point type estimates used above. Thus, near the zero section of $\Ttb^*M$, Proposition~\ref{PropES}\eqref{ItESBw} gives a free estimate (no a priori control terms) near $\sface\cap\{\frac{r}{t}=v\}$---which is a source for the $r\cd^\sharp$-flow---from where we propagate control through $\pa\cK^+$ towards $r=r_0$ by means of Proposition~\ref{PropEB}\eqref{ItEBBw}, and also into all of $\frac{r}{t}\in(v,10)$. We get another free energy estimate at $\Sigma^\sharp$ on spaces of supported distributions (thus encoding vanishing in $r<\bhm$) from the energy estimate~\eqref{EqEISigmaSharpBw}, from where we can propagate control near the zero section towards $r=r_0$ using Proposition~\ref{PropER}\eqref{ItERBw}. Having thus obtained control near the zero section over $\cK^+\cap\{|r-r_0|\geq r_0\sqrt{e}/v\}$, we can propagate into the zero section over $\cK^+\cap\{|r-r_0|\leq r_0\sqrt{e}/v\}$ using Proposition~\ref{PropEK}\eqref{ItEKBw}. Elliptic estimates and the adjoint estimate~\eqref{EqEEInftyAdj} of Proposition~\ref{PropEEInfty}, together with propagation towards $\tau=0$ by means of the energy estimate~\eqref{EqEISigmaBw} ultimately prove the a priori estimate
\begin{equation}
\label{EqEPDual}
  \|\omega'\|_{\rho_\sface^{-\alpha_\sface-2}\rho_\cK^{-\alpha_\cK}\Htbh^{s'}(\Omega)^{-,\bullet}} \leq C h^{-1}\|P_h^*\omega'\|_{\rho_\sface^{-\alpha_\sface}\rho_\cK^{-\alpha_\cK}\Htbh^{s'-1}(\Omega)^{-,\bullet}}
\end{equation}
for $0<h<h_0$ when $h_0>0$ is sufficiently small, initially for $s'=1$, and then for all $s'\geq 1$ by an extension/restriction and elliptic regularity/propagation argument as in the proof of Proposition~\ref{PropEPWarmup}. Using duality and the Hahn--Banach theorem, this implies the solvability of $P_h \omega=f\in\rho_\sface^{\alpha_\sface+2}\rho_\cK^{\alpha_\cK}\Htbh^{s-1}(\Omega)^{\bullet,-}$ with $\omega\in\rho_\sface^{\alpha_\sface}\rho_\cK^{\alpha_\cK}\Htbh^s(\Omega)^{\bullet,-}$ satisfying the estimate~\eqref{EqEPApriori}, provided $s=-(s'-1)\leq 0$. Higher regularity then follows again by means of an extension/restriction argument and elliptic/propagation estimates. This completes the proof of the first part of Theorem~\ref{ThmET} (and proves the quantitative estimate~\eqref{EqETQuant}).

The solution of the initial value problem~\eqref{EqETIVP} can easily be reduced to that of the forcing problem~\eqref{EqET}. We leave the proof of semiclassical estimates for the solution $\omega$ to the reader and only present the details for \emph{fixed} $h>0$, which is all we need in our application. One argument proceeds by solving the initial value problem up to $\tau=\frac{t}{r}=T$ where $T>0$ is sufficiently small so that $\tau^{-1}(T)$ is spacelike; this is possible in the desired regularity class (note that 3b-regularity and b-regularity are the same) by standard short-time theory for the strictly hyperbolic operator $\Box_g^{\cC_h}$. Regarding the weights of the data~\eqref{EqETIVP} at $\sface$, note that $|\dd g|$ is a smooth positive multiple of $\dd\tau\cdot r^3\,\dd r|\dd\slg|$, and $\rho_\sface^{\alpha_\sface+\frac12}\Hbext^s(X,r^2\,|\dd r\,\dd\slg|)=\rho_\sface^{\alpha_\sface}\Hbext^s(X,r^3\,\dd r|\dd\slg|)$, with weight now matching that of the sought-after solution $\omega$. Given this short time solution $\omega'\in\rho_\sface^{\alpha_\sface}\Hb^1(\tau^{-1}([0,T]))$, fix a smooth cutoff $\chi=\chi(\tau)$ which equals $1$ for $\tau\geq T/2$ and $0$ for $\tau\leq T/4$; then the solution $\omega$ of~\eqref{EqETIVP} satisfies
\[
  \Box_g^{\cC_h}(\chi \omega)=\chi f+[\Box_g^{\cC_h},\chi]\omega' \in \Htb^{0,(\alpha_\sface+2,\alpha_\cK)}(\Omega)^{\bullet,-}.
\]
But from the first part of Theorem~\ref{ThmET} this implies that $\chi \omega\in\Htb^{1,(\alpha_\sface,\alpha_\cK)}(\Omega)^{\bullet,-}$, and hence $\omega=\chi \omega+(1-\chi)\omega'$ has the regularity and decay claimed in Theorem~\ref{ThmET}.

Another argument, which avoids the need to explicitly control initial data altogether, proceeds as follows. It suffices to assume that the initial data and the forcing have infinite b-regularity. (A density argument and energy estimates for the initial value problem then yield the finite-regularity statement.) A solution $\omega$ of~\eqref{EqETIVP}, sharply cut off to non-negative times via $\omega':=H(\tau)\omega$, satisfies
\begin{equation}
\label{EqETFwd}
  \Box_g^{\cC_h}\omega' = f' := H(\tau)f + [\Box_g^{\cC_h},H(\tau)]\omega,
\end{equation}
with the second summand expressible solely in terms of the initial data of $\omega$ and $\delta(\tau)$, $\delta'(\tau)$. Thus, $f'\in\Htb^{s'-1,(\alpha_\sface+2,\alpha_\cK)}(\Omega)^{\bullet,-}$ for $s'-1<-\frac32$ to accommodate the (differentiated) $\delta$-distribution at $\tau=0$. Since~\eqref{EqETFwd} is now a forward problem, we know that $\omega'\in\Htb^{s',(\alpha_\sface,\alpha_\cK)}(\Omega)^{\bullet,-}$. However, the (b-)wave front set of $f'$ near $\tau=0$ is contained in the conormal bundle of $\tau^{-1}(0)$, and hence in the elliptic set of $\Box_g^{\cC_h}$, and therefore the same holds for the wave front set of $\omega'$. Letting $\chi$ be as above, we thus have
\[
  \Box_g^{\cC_h}((1-\chi)\omega)=f'' := (1-\chi)f-[\Box_g^{\cC_h},\chi]\omega' \in \rho_\sface^{\alpha_\sface+2}\Hb^\infty(\tau^{-1}([0,T])),
\]
with $f''$ vanishing for $\tau\geq\frac12 T$. Letting then $\tilde f\in\rho_\sface^{\alpha_\sface+2}\Hb^\infty(\tau^{-1}([-T,T]))$ denote an extension of $f''$, we may solve the forcing problem $\Box_g^{\cC_h}\tilde \omega=\tilde f$ \emph{backwards}, with $\tilde \omega$ vanishing for $\tau\geq\frac12 T$, with the solution satisfying $\tilde \omega\in\rho_\sface^{\alpha_\sface}\Hb^\infty(\tau^{-1}([-T,T]))$. But $\tilde \omega=(1-\chi)\omega$ for $\tau>0$; and hence we conclude by restricting to $\tau\geq 0$ that $\omega\in\rho_\sface^{\alpha_\sface}\rho_\cK^{\alpha_\cK}\bar H_\tbop^1(\Omega)$, the new information being the uniform regularity of $\omega$ down to $\tau=0$ as an extendible distribution. (This argument is inspired by \cite[Exercise~9.10]{HintzMicro}.)

The proof of Theorem~\ref{ThmET} is complete.

\section{Mode stability and zero energy behavior}
\label{ST}

We now show how to deduce a precise version of Theorem~\ref{ThmIRough} from Theorem~\ref{ThmET}.

We continue denoting by $g=g_{\bhm,a}$ a subextremal Kerr metric (see~\eqref{EqKBL}--\eqref{EqKMetStar}) on the spacetime manifold $M^\circ$ where the compactification $M$ is defined in Definition~\ref{DefKMfd}, with compactified Cauchy hypersurface $X$. We use the coordinates $t,r,\theta,\phi$ (or $t,x=r(\sin\theta\cos\phi,\sin\theta\sin\phi,\cos\theta)$) from~\eqref{EqKMetCoord} and~\eqref{EqKMetStar}. Note that $X\subset M$ is diffeomorphic to $[\bhm,\infty]_r\times\Sph^2$, i.e., the closure of $[\bhm,\infty)\times\Sph^2$ inside the radial compactification of $\R^3$. Recall also the scattering bundles $\cT,\cT^*$ from Definition~\ref{DefKTsc}, and $\cT^*_X:=\cT^*|_X$.

\begin{definition}[Outgoing]
\label{DefTOutgoing}
  Define the \emph{tortoise coordinate} by
  \[
    r_*:=r+2\bhm\log(r-2\bhm).
  \]
  Given $0\neq\sigma\in\C$, we say that $\omega(t,x)\in\CI(\R\times X^\circ;T^*\R^4)$ is \emph{outgoing at frequency $\sigma$} if $\omega(t,x)=e^{-i\sigma t}\omega_0(x)$ where $\omega_0\in\CI(X^\circ;T^*_{X^\circ}\R^4)$ is a smooth stationary 1-form which in addition, for $r\geq 4\bhm$, is of the form $\omega_0(x)=e^{i\sigma r_*}\omega_1(x)$ where $\omega_1\in\cA^\gamma(X;\cT^*_X)$ is conormal at $\pa X$ for some $\gamma\in\R$ (see Definition~\ref{DefK3bConormal} for the definition of $\cA^\gamma$).
\end{definition}

The outgoing condition is often phrased only for the scalar wave equation for \emph{fully separated} modes in Boyer--Lindquist coordinates $\ft,r,\theta,\phi$ as used in~\eqref{EqKBL} (and in the region of validity of the latter). We recall from \cite[Definition~1.1]{ShlapentokhRothmanModeStability} that this means that $u(\ft,r,\theta,\phi)=e^{-i\sigma\ft}e^{i m\phi}S(\theta)R(r)$ where $S(\theta)$ is an oblate spheroidal harmonic, and $R$ is outgoing in the sense that $r e^{-i\sigma r_*}R(r)$ is smooth in $r^{-1}$, and near the event horizon $r=r_+$, the function $(r-r_+)^{-\frac{i(a m-2\bhm r_+\sigma)}{2\sqrt{\bhm^2-a^2}}}R(r)$ is smooth down to $r=r_+$. The latter is equivalent to smoothness of $u$ on $M\cap\{r\geq r_+\}$ down to the future event horizon, and thus consistent with Definition~\ref{DefTOutgoing}. The smoothness in $r^{-1}$ for large $r$ on the other hand can be recovered from the merely conormal regularity required in Definition~\ref{DefTOutgoing} using indicial root considerations (see below).

\begin{thm}[Mode stability of the constraint propagation wave operator: precise form]
\label{ThmT}
  Let $v\in(0,1)$ and $C_0>0$. Then for all sufficiently small $e>0$, there exist $h_0\in(0,1)$ and a smooth stationary 1-form $\cd$ on $M^\circ$, which equals $r^{-1}(\dd t-v\,\dd r)$ for large $r$, such that for all $h\in(0,h_0)$, the operator $\Box_g^{\cC_h}$ defined by~\eqref{EqIRoughC} has the following properties.
  \begin{enumerate}
  \item\label{ItTNonzero}{\rm (Mode stability in the punctured upper half plane.)} Let $\sigma\in\C$, $\Im\sigma\geq 0$, $\sigma\neq 0$. Suppose $\omega(t,x)=e^{-i\sigma t}\omega_0(x)$ is an outgoing solution of $\Box_g^{\cC_h}\omega=0$. Then $\omega=0$.
  \item\label{ItTZero}{\rm (Invertibility properties of the zero energy operator.)} Denote the zero energy operator of $\Box_g^{\cC_h}$ by $\wh{\Box_g^{\cC_h}}(0)\in\Diffb^2(X;\cT^*_X)$ (see Definition~\usref{DefK3bSpec}). Then for all $\ell\in[-C_0-\frac32,-\frac12)$ and $s\geq-C_0$, the operator
    \begin{equation}
    \label{EqTZero}
    \begin{split}
      \wh{\Box_g^{\cC_h}}(0) &\colon \bigl\{ \omega\in\Hbext^{s,\ell}(X;\cT^*_X) \colon \wh{\Box_g^{\cC_h}}(0)\omega\in\Hbext^{s-1,\ell+2}(X;\cT^*_X) \bigr\} \\
        &\quad\quad \to \Hbext^{s-1,\ell+2}(X;\cT^*_X)
    \end{split}
    \end{equation}
    is invertible; here we use $r^2\,|\dd r\,\dd\slg|$ as the volume density on $X$. Moreover, the indicial roots $\lambda\in\C$ of $\wh{\Box_g^{\cC_h}}(0)$ (see~\eqref{EqK3bIndRoots}) satisfy $\Re\lambda\geq 1$ or $\Re\lambda<-C_0$ and accumulate only at $\pm\infty$; and the only root with $\Re\lambda=1$ is $\lambda=1$, which is a simple scalar type $0$ and a simple scalar type $1$ root.
  \end{enumerate}
\end{thm}

Theorem~\ref{ThmIRough} is an immediate consequence of Theorem~\ref{ThmT} for $C_0+2$ in place of $C_0$. Indeed, for Theorem~\ref{ThmIRough}\eqref{ItIRough0}, note that if $\omega=\cO(r^{C_0})$, then $\omega\in r^{C_0+\frac32+\eps}L^2(X;\cT^*_X)$ for all $\eps>0$. Since among the indicial roots $\lambda$ of $\wh{\Box_g^{\cC_h}}(0)$, the one with $\Re\lambda\geq-C_0-\frac32-\eps$ and smallest $\Re\lambda$ is $\lambda=1$ (when $\eps$ is sufficiently small), a normal operator/indicial root argument shows that, in fact, $\omega\in r^{\frac12+\eps}L^2(X;\cT^*_X)$, and therefore the injectivity of the map~\eqref{EqTZero} (for $\ell=-\frac12-\eps$ and $s=0$) implies that $\omega=0$.

\begin{proof}[Proof of Theorem~\usref{ThmT}]
  Fix the weights $\alpha_\cK=1$ and $\alpha_\sface=-C_0-2$, which satisfy~\eqref{EqETThr}. Application of Theorem~\ref{ThmET} with $C_0$ increased by $1$ produces $e,h_0\in(0,1)$ and a stationary 1-form $\cd$ such that the forward and initial value problems~\eqref{EqET}--\eqref{EqETIVP} for the operator $\Box_g^{\cC_h}$ can be solved in the function spaces stated in Theorem~\ref{ThmET}, and such that the indicial roots of the zero energy operator $\wh{\Box_g^{\cC_h}}(0)$ are as in Theorem~\ref{ThmM0}; we take $C_0$ there to be equal to $C_0+4$ in present notation.

  \pfstep{Part~\eqref{ItTNonzero}.} Note that $\omega=e^{-i\sigma t}\omega_0$ is a solution of the initial value problem
  \[
    \left\{
    \begin{alignedat}{2}
    \Box_g^{\cC_h}\omega&=0, \\
    (\omega,r\cL_{\pa_t}\omega)|_{t=0}&=(\omega^{(0)},\omega^{(1)}) := (\omega_0,-i\sigma r\omega_0).
    \end{alignedat}
    \right.
  \]
  We claim that $\omega_1:=e^{-i\sigma r_*}\omega_0$, which we are only assuming to be conormal at $r^{-1}=0$ with \emph{some} weight, in fact has better decay. For present purposes, it will be enough to establish some fixed bound; we shall prove
  \begin{equation}
  \label{EqTDecay}
    \omega_1\in\cA^1.
  \end{equation}
  Once this is shown, we can conclude that the initial data of $\omega$ obey the pointwise bounds
  \[
    |\omega^{(0)}| \leq C r^{-1},\quad
    |\omega^{(1)}| \leq C,
  \]
  since $\omega_0=e^{i\sigma r_*}\omega_1$ and $\Im\sigma\geq 0$. Therefore, $\omega^{(j)}\in\rho_\sface^{-\frac32-\eps}\bar H_\bop^{1-j}(X,r^2 |\dd r\,\dd\slg|;\cT^*_X)$ for any $\eps>0$. Since $\alpha_\sface+\frac12<-\frac32$, Theorem~\ref{ThmET} (with $s=1$) implies that
  \begin{equation}
  \label{EqTMembership}
    \omega\in\rho_\sface^{\alpha_\sface}\rho_\cK^{\alpha_\cK}\bar H_\tbop^1(\Omega;\cT^*) \subset \rho_\sface^{-C_0-2}L^2(\Omega,|\dd g|;\cT^*).
  \end{equation}
  On the other hand, if $\omega_0$ were nonzero, then it would not be identically $0$ on some annulus $A=\{ r_1\leq r\leq r_2 \}$ with $\bhm<r_1<r_2<\infty$, which contradicts the square integrability of $|\omega|=e^{(\Im\sigma)t}|\omega_0|$ on $[0,\infty)_t\times A$. Therefore, $\omega_0=0$ and thus $\omega=0$.
  
  It remains to establish~\eqref{EqTDecay}. We first note that, in the coordinates $t_*$ from~\eqref{EqKtstar}, $\rho=r^{-1}$, and spherical coordinates, we have
  \begin{equation}
  \label{EqTStruct}
    \Box_g^{\cC_h} = -2\pa_{t_*}\rho(\rho\pa_\rho-1-\ubar S^{\cC_h}-\rho\tilde S) + \wh{\Box_g^{\cC_h}}(0) + Q\pa_{t_*} + R\pa_{t_*}^2,
  \end{equation}
  where the endomorphism $\ubar S^\cC$ of the bundle $\cT_X^*$ is given by~\eqref{EqM00S}, $\tilde S$ is a smooth section of $\End(\cT_X^*)$, and
  \[
    Q\in\rho^3\Diffb^1(X;\cT^*_X),\quad
    R\in\rho^2\CI(X),\quad
    \wh{\Box_g^{\cC_h}}(0)\in\rho^2\Diffb^2(X;\cT^*_X).
  \]
  Indeed, this follows by replacing $g_{\bhm,a}$ by the Minkowski metric $\ubar g$, with $\ubar g^{-1}=-2\pa_{t_*}\otimes_s\pa_r + \pa_r^2 + r^{-2}\slg^{-1}$, recalling~\eqref{EqM00Box}, and estimating the error terms by means of~\eqref{EqKtstarInner} and by comparison with~\eqref{EqKtstarMet} (which gives the decay order of the principal part of $Q$, while subprincipal $\rho^2\Diffb^0$-contributions to $Q$ can be absorbed into the term $\tilde S$); see \cite[\S3.3]{HaefnerHintzVasyKerr} and \cite[\S{7.1}]{HintzKerrStab} for similar calculations. (Moreover, $\wh{\Box_g^{\cC_h}}(0)$ is given by $\rho^2$ times~\eqref{EqM00} plus an error term of class $\rho^3\Diffb^2$.)

  Write now $\omega=e^{-i\sigma t}e^{i\sigma r_*}\omega_1=e^{-i\sigma t_*}\omega_1'$, where $\omega_1'=f\omega_1$, $f=e^{-i\sigma(t-r_*-t_*)}$; note that $f,f^{-1}\in\cA^0$ near $r^{-1}=0$. Thus, from~\eqref{EqTStruct} we obtain
  \[
    2 i\sigma\rho(\rho\pa_\rho-1-\ubar S^\cC)\omega_1 \in (\rho^2\Diffb^2 + \sigma\rho^3\Diffb^1 + \sigma^2\rho^2\Diffb^0)\omega_1
  \]
  where $\Diffb^k=\Diffb^k(X;\cT^*_X)$. Under the assumption $\omega_1\in\cA^\gamma$, the right-hand side lies in $\cA^{\gamma+2}$. We can then solve this regular-singular ODE (upon division by $\rho$) in $\rho$ and thus conclude $\omega_1\in\cA^{\gamma+1}$ as long as $\gamma+1$ is less than $1+\min\spec\ubar S^\cC$. But the eigenvalues of $\ubar S^\cC$ are all positive, and hence we (a fortiori) have $\omega_1\in\cA^1$, as claimed.

  \pfstep{Part~\eqref{ItTZero}.} When $s\geq 1$, any element $\omega_0$ in the kernel of~\eqref{EqTZero} gives rise to a stationary solution $\omega(t,x)=\omega_0(x)$ of the initial value problem for $\Box_g^{\cC_h}\omega=0$ with initial data $(\omega,r\cL_{\pa_t}\omega)|_{t=0}=(\omega_0,0)$. But $\omega_0\in\Hbext^{s,\ell}(X,r^2\,|\dd r\,\dd\slg|;\cT^*_X)$ with $\ell\geq-C_0-\frac32=\alpha_\sface+\frac12$. Therefore, Theorem~\ref{ThmET} applies, and we again conclude the membership~\eqref{EqTMembership} for $\omega$, which forces $\omega_0=0$. Thus, the map~\eqref{EqTZero} is injective. For general $s$, it is technically more convenient to pick $\chi=\chi(\tau)$ to be equal to $0$ for $\tau\leq T/4$ and $1$ for $\tau\geq T/2$; then $\Box_g^{\cC_h}(\chi\omega_0)=f:=[\Box_g^{\cC_h},\chi]\omega_0\in\rho_\sface^{\alpha_\sface+2}\rho_\cK^{\alpha_\cK}\Htb^{s-1}(\Omega;\cT^*)^{\bullet,-}$, so Theorem~\ref{ThmET} gives $\chi\omega_0\in\rho_\sface^{\alpha_\sface}\rho_\cK^{\alpha_\cK}\Htb^s(\Omega;\cT^*)^{\bullet,-}$, which again suffices to conclude.

  To prove surjectivity, we note that the map~\eqref{EqTZero} is Fredholm. (See \cite[Theorem~4.3]{HaefnerHintzVasyKerr}, \cite[\S6]{VasyLowEnergy}, \cite[Theorem~6.9]{VasyMinicourse} for such statements on spacetimes including Kerr, and also the earlier \cite{MelroseAPS}, \cite[Theorem~2.1]{GuillarmouHassellResI}.) Its index is equal to $0$; this follows from Theorem~\ref{ThmMiInv} and \cite[Corollary~9.43]{HintzNonstat2}. A direct proof proceeds by deforming $g$ to a Schwarzschild metric, and then noting that upon expanding into spherical harmonics, we obtain a collection of radial ordinary differential operators. Theorem~\ref{ThmM0}\eqref{ItM0Number} implies that, acting on suitable weighted spaces corresponding to the range considered in~\eqref{EqTZero}, they are of index $0$; therefore, $\wh{\Box_g^{\cC_h}}(0)$ is of index $0$ and thus invertible since it is injective.

  Alternatively, one can argue directly by proving that the $L^2(X,|\dd g|_X|;\cT^*)$-adjoint $\wh{\Box_g^{\cC_h}}(0)^*=\wh{(\Box_g^{\cC_h})^*}(0)$ (where $|\dd g|_X|$ is the spatial volume density of Kerr, so $|\dd t\,\dd g|_X|=|\dd g|$) has trivial nullspace on $\Hbsupp^{s',\ell'}(X;\cT^*_X)$ for $s'=-s+1$ and $\ell'=-\ell-2>-\frac32$. But we know from (the proof of) Theorem~\ref{ThmET} that $(\Box_g^{\cC_h})^*$ is injective on $\Htb^{-s+1,(-\alpha_\sface-2,-\alpha_\cK)}(\Omega)^{-,\bullet}=\Htb^{s',(C_0,-1)}(\Omega)^{-,\bullet}$. This implies that no nontrivial stationary solution $\omega(t,x)=\omega_0(x)$ of $(\Box_g^{\cC_h})^*\omega=0$ exists with $\omega_0\in\rho_\sface^{C_0+\frac12}\Hbsupp^{s'}(X;\cT^*_X)$. Suppose now that $\omega_0\in\Hbsupp^{s',\ell'}(X;\cT^*_X)$, $\ell'>-\frac32$, satisfies $\wh{(\Box_g^{\cC_h})^*}(0)\omega_0=0$. The indicial roots of $\wh{\Box_g^{\cC_h}}(0)$ all satisfy $\Re\lambda<-C_0-4$ or $\Re\lambda\geq 1$, and hence the indicial roots $\lambda^*=1-\bar\lambda$ of $\wh{\Box_g^{\cC_h}}(0)^*$ satisfy $\Re\lambda^*\leq 0$ or $\Re\lambda^*>C_0+3$. Taking into account the shift in the weight at $\pa X$ by $-\frac32$ when passing from $L^\infty$-based spaces to $L^2$-based spaces on $X$, we thus have $\omega_0\in\cA^{C_0+3}(X)$ near $\pa X$ by a standard normal operator argument. Therefore, $\omega_0\in\rho_\sface^{C_0+\frac12}\Hbsupp^{s'}(X;\cT^*_X)$ globally, so $\omega_0=0$ by what we already proved.
\end{proof}

\begin{rmk}[Invertibility of the zero energy operator]
\label{RmkTInv0}
  One could analyze $\wh{\Box_g^{\cC_h}}(0)$ directly using semiclassical analysis in $\ol{\Tb^*}X$ and prove its invertibility as a map~\eqref{EqTZero} for any fixed $\ell\in[-C_0-\frac32,-\frac12)$, or indeed for $\ell$ lying in any fixed compact subset of $[-C_0-\frac32,-\frac12)$. (This would rely on elliptic and complex absorption type estimates at nonzero semiclassical frequencies, and near the zero section on a source estimate near $r=r_0$, a sink estimate at $\pa X$, and propagation in between, with energy estimates near $r=\bhm$ to close the estimate.) \emph{However}, the required upper bound on the semiclassical parameter $h$ for invertibility to hold might, a priori, degenerate to $0$ as one increases the size of this compact subset. Only a calculation of the indicial roots, as performed in Theorem~\ref{ThmM0}, allows one to get invertibility (which is crucial for our application in \cite{HintzKerrStab}) for some fixed positive $h$ in the full stated range of weights.
\end{rmk}

\begin{rmk}[Direct analysis of the spectral family]
\label{RmkTInvCx}
  It appears difficult to prove the invertibility of $\wh{\Box_g^{\cC_h}}(\sigma)$ for all small $h>0$, uniformly for all $\sigma\in\C$ with $\Im\sigma\geq 0$, on appropriate Sobolev spaces (b-Sobolev spaces for $\sigma=0$, scattering Sobolev spaces for $\sigma\neq 0$), directly. For example, for large frequencies $\sigma=h^{-1}\varsigma$ where $\varsigma\in e^{i(0,\pi)}$, the semiclassical principal symbol of $h^2\wh{\Box_g^{\cC_h}}(h^{-1}\varsigma)$ is
  \[
    G(-\varsigma\,\dd t+\zeta) - i\ell_{g,\cd,e}(-\varsigma\,\dd t+\zeta),\quad \zeta\in T^*X^\circ.
  \]
  Since the covector $-\varsigma\,\dd t+\zeta$ is now \emph{complex}, the results in~\S\ref{SC} do not suffice anymore for the analysis of this symbol, and indeed its analysis appears to be a very challenging linear algebra problem. The operator $\wh{\Box_g^{\cC_h}}(h^{-1}(s+i\tau))$, $\tau>0$, is the spectral family at the \emph{real} frequency $h^{-1}s$ of the conjugation $e^{-\tau t/h}\Box_g^{\cC_h}e^{\tau t/h}$, and this operator describes $\Box_g^{\cC_h}$ acting on \emph{exponentially growing} ($\sim e^{\tau t/h}$) inputs. From this perspective, the fact that we can prove the solvability of $\Box_g^{\cC_h}$ on polynomially weighted spaces directly, via \emph{semiclassical} analysis on spacetime, offers an enormous simplification which is not available for the study of general wave operators (i.e., operators without suitable large parameters).
\end{rmk}

\appendix

\section{Proof of Theorem~\usref{ThmM0}}
\label{SM0}

We need the following identities for $\scal\in\scalspace_l$ and $\vect:=\slstar\sld\scal$:
\begin{equation}
\label{EqM0Spherical}
  \slDelta\scal = l(l+1)\scal,\quad
  \slDelta\vect = (l(l+1)-1)\vect,\quad
  \sldelta\vect = 0.
\end{equation}

\subsection{Vector type indicial roots}
\label{SsM0v}

Let $l\geq 1$. Acting on the 1-dimensional space spanned by $(0,0,\vect)$ where $\vect=\slstar\sld\scal$, $0\neq\scal\in\scalspace_l$, the operator $h N_{v,e,h}(\lambda)$ (using~\eqref{EqM0NormOp} and Corollary~\ref{CorM00}) is multiplication by the quadratic (in $\lambda$) polynomial
\begin{align*}
  p_\rmv(h,l,\lambda) &= h\bigl(-\lambda^2+\lambda+(l(l+1)-1) + 1 + h^{-1}q_{\slash\slash}(\lambda)\bigr) \\
    &= -h\lambda^2+(h-2 v)\lambda+ \bigl(l(l+1)h + 4 v\bigr),
\end{align*}
where $q_{\slash\slash}(\lambda)$ is obtained from $q_{\slash\slash}$ by replacing $\rho\pa_\rho$ by $\lambda$. This does not depend on the parameter $e$. We will show that for sufficiently small $h>0$, there is one root of this polynomial on either side of $\{-C_0\leq\Re\lambda\leq 1\}$. Now, the quadratic equation $p_\rmv(h,l,\lambda)=0$ can of course be solved explicitly for $\lambda$, and the desired conclusion can then be obtained by direct estimates. For example, for $l=1$, the two roots are $\lambda_-(h,1)=-\frac{2 v}{h}-1$ and $\lambda_+(h,1)=2$. As a point of comparison for later on, we remark that for fixed $l\geq 2$, the Taylor expansions of its roots around $h=0$ are
\begin{equation}
\label{EqM0vTaylor}
  \lambda_-(h,l) = -\frac{2 v}{h} + \cO_l(1),\quad
  \lambda_+(h,l) = 2 + \cO_l(h).
\end{equation}
Carefully note, however, that the $\cO(1)$ and $\cO(h)$ error terms depend on $l$; therefore, these expansions alone are not sufficient to conclude that one sufficiently small value of $h>0$ guarantees that $\lambda_-(h,l)<-C_0$ and $\lambda_+(h,l)>1$ for \emph{all} $l$.

We present an alternative, more conceptual (albeit for the present case of course unnecessarily elaborate) analysis, the main ideas of which generalize to the rather involved scalar type $l\geq 1$ considerations in~\S\ref{SsMls} below. Thus, we study the roots of $p_\rmv(h,l,\cdot)$ in four asymptotic regimes:
\begin{enumerate}
\item $l\lesssim 1$ is fixed (or bounded) and $h\to 0$;
\item $1\ll l\ll h^{-1}$ is large but much smaller than $h^{-1}$;
\item $l\sim h^{-1}$ is comparable to $h^{-1}$.
\item $l\gg h^{-1}$ is much larger than $h^{-1}$.
\end{enumerate}
Roughly speaking, in the first regime, the Taylor expansion~\eqref{EqM0vTaylor} suffices (and we will derive it by a simple computation); in the fourth regime, there are two roots roughly at $\pm l$. The transition between the two behaviors takes place in the second and third regimes.

The details are as follows.

\pfstep{{\rm (1)} Fixed, or bounded, $l\in\N$.} Fix $l\in\N$. The polynomial $p_\rmv(0,l,\lambda)=-2 v(\lambda-2)$ vanishes simply at $2$, and therefore, by the implicit function theorem, has a root $\lambda_+(h,l)$ also for $h>0$ that depends continuously (in fact, analytically) on $h$, with $\lambda_+(0,l)=2$. To find the other root (which disappeared in this naive limit of $p_\rmv(h,l,\lambda)$ as $h\to 0$), we shall instead work near $\lambda\sim h^{-1}$. To effect this, we ``blow up $h=0,\lambda=\infty$,'' i.e., we introduce
\begin{equation}
\label{EqM0tilde}
  \tilde h=h,\ 
  \tilde l=l,\ 
  \tilde\lambda=(h\lambda)^{-1}\quad\Longleftrightarrow\quad
  h=\tilde h,\ 
  l=\tilde l,\ 
  \lambda=(\tilde h\tilde\lambda)^{-1},
\end{equation}
and set
\[
  \tilde p_\rmv(\tilde h,\tilde l,\tilde\lambda) := \tilde h\tilde\lambda^2 p_\rmv(\tilde h,\tilde l,(\tilde h\tilde\lambda)^{-1}) = -1+(\tilde h-2 v)\tilde\lambda+\tilde l(\tilde l+1)\tilde h^2\tilde\lambda^2 + 4\tilde h v\tilde\lambda^2.
\]
The prefactor $\tilde h\tilde\lambda^2$ is chosen such that this has smooth coefficients down to $\tilde h=0$. The restriction of $\tilde p_\rmv$ to $\tilde h=0$,
\[
  \tilde p_\rmv(0,\tilde l,\tilde\lambda)=-1-2 v\tilde\lambda,
\]
vanishes simply for $\tilde\lambda_-(0)=-(2 v)^{-1}$, which is the value at $\tilde h=0$ of a smooth family of roots $\tilde\lambda_-(h,\tilde l)$ of $\tilde p_\rmv(\tilde h,\tilde l,\tilde\lambda)$. In summary, for fixed $l$, or indeed for $l\leq l_0$ for any fixed $l_0\in\N$, and for sufficiently small $h$ (depending on $l_0$), the $\rmv l$ indicial roots are
\[
  \lambda_+(h,l)=2+o(1),\quad
  \lambda_-(h,l)=(h\tilde\lambda_-(h,l))^{-1}=-(2 v+o(1))h^{-1},\quad l\leq l_0,\ h\to 0.
\]
(This recovers the leading order terms of~\eqref{EqM0vTaylor}.) See Figure~\ref{FigM0vSmall}. Therefore, for any $l_0$, there exists $h_0>0$ such that
\begin{equation}
\label{EqM0vSmall}
  \lambda_+(h,l)\geq\tfrac32,\quad
  \lambda_-(h,l)\leq -v h^{-1},\quad l\leq l_0,\ h\leq h_0.
\end{equation}

\begin{figure}[!ht]
\centering
\includegraphics{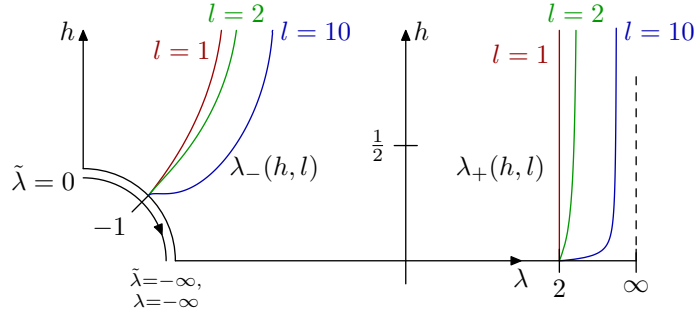}
\caption{The vector type $l$ roots $\lambda_-(h,l)$ and $\lambda_+(h,l)$ for $v=\frac12$ and $l=1,2,10$. The root $\lambda_+(h,l)$ converges to $2$ as $h\searrow 0$. We already point out here that for fixed and fairly small $h$ the root corresponding to $l=10$ is already fairly far from $2$; this is indicative of the existence of the transitional regimes~(2)--(3). The root $\lambda_-(h,l)$ is smooth only on the resolution of $[0,1)_h\times\ol{\R_\lambda}$ at $\{0\}\times\{-\infty\}$, and limits to the point $\tilde\lambda=-(2 v)^{-1}$ (here $=-1$) as $h\to 0$; here $\tilde\lambda$ is an affine coordinate on the front face of this resolution (blow-up).}
\label{FigM0vSmall}
\end{figure}

\pfstep{{\rm (2)} Small intermediate $l$, i.e., $1\ll l\ll h^{-1}$.} We can no longer regard $l$ as fixed, and instead need to regard it as a coordinate on the space $(0,1)_h\times\N_l\times\R_\lambda$ on which the two roots (in $\lambda$) of $p_\rmv(h,l,\lambda)=0$ describe a subset that we are in the process of describing. With an eye towards the $l\to\infty$ behavior of the eigenvalues discussed later (they will behave like $\pm l$), we introduce $\hat\lambda:=\lambda l^{-1}$ and work in $(0,1)_h\times\N_l\times\R_{\hat\lambda}$. In order to study the regime $1\ll l\ll h^{-1}$, we introduce
\begin{equation}
\label{EqM0hat}
  \hat h=h l,\ 
  \hat l=l^{-1},\ 
  \hat\lambda=\lambda l^{-1} \quad\Longleftrightarrow\quad
  h=\hat h\hat l,\ 
  l=\hat l^{-1},\ 
  \lambda=\hat\lambda\hat l^{-1}.
\end{equation}
Thus, $\hat h,\hat l,\hat\lambda$ are local coordinates on the blow-up of $[0,1)_h\times[1,\infty]_l\times\R_{\hat\lambda}$ at $\{0\}\times\{\infty\}\times\R_{\hat\lambda}$; see Figure~\ref{FigM0vAll}.

\begin{figure}[!ht]
\centering
\includegraphics{FigM0vAll}
\caption{Illustration of the coordinates $\hat h$, $\hat l$, $\hat\lambda$ from~\eqref{EqM0hat}, the coordinates $\hat h',\hat l',\hat\lambda'$ from~\eqref{EqM0hatp} (which can be regarded as arising from the blow-up of $[0,1)_h\times[1,\infty]_l\times\ol{\R_{\hat\lambda}}$ at $h=0,l=\infty$ and then at $\hat h=0$, $\hat\lambda=\infty$), and $\check h,\check l,\check\lambda$ from~\eqref{EqM0check}. Shown in red is the positive root $\lambda_+(h,l)$ of $p_\rmv(h,l,\lambda)$ (computed in the course of our arguments in the various coordinate systems), and shown in green is the negative root $\lambda_-(h,l)$. (The color coding is different from that of Figure~\ref{FigM0vSmall}.) More precisely, we show the limits of the roots as $h$, $\hat h$, $\hat l$, or $\check l$ tend to $0$.}
\label{FigM0vAll}
\end{figure}

In the regime $1\ll l\ll h^{-1}$ of current interest, we have $0\leq\hat h,\hat l\ll 1$. Put
\[
  \hat p_\rmv(\hat h,\hat l,\hat\lambda):=\hat l p_\rmv(\hat h\hat l,\hat l^{-1},\hat\lambda\hat l^{-1}) = -\hat h\hat\lambda^2 - (2 v-\hat h\hat l)\hat\lambda + \bigl(4 v\hat l+\hat h(1+\hat l)\bigr),
\]
which is thus smooth and restricts to $\hat l=0$ as
\[
  \hat p_\rmv(\hat h,0,\hat\lambda)=-\hat h\hat\lambda^2-2 v\hat\lambda+\hat h.
\]
The roots of this are $\hat\lambda_\pm(\hat h,0)=\hat h^{-1}(-v\pm(v^2+\hat h^2)^{\frac12})$. The root $\hat\lambda_+(\hat h,0)=\hat h/(v+(v^2+\hat h^2)^{\frac12})$ is finite and simple for all $0\leq\hat h\lesssim 1$, and hence extends smoothly into $\hat l>0$ as a root $\hat\lambda_+(\hat h,\hat l)$ of $\hat p_\rmv(\hat h,\hat l,\hat\lambda)$. Furthermore, $\hat\lambda_+(0,\hat l)=2\hat l=2 l^{-1}=\lambda_+(0,l)l^{-1}$ is the rescaling of the root $\lambda_+(0,l)=2$ of $p_\rmv(0,l,\lambda)$ discussed previously. Moreover, $\pa_{\hat h}\hat\lambda_+(0,0)=(2 v)^{-1}$ is positive; therefore, there exists $\hat l_0>0$ such that $\pa_{\hat h}\hat\lambda_+(0,\hat l)>\frac12(2 v)^{-1}$ for $0\leq\hat l\leq\hat l_0$, and therefore there exists, moreover, a constant $\hat h_0>0$ such that
\[
  \hat\lambda_+(\hat h,\hat l) \geq 2\hat l + \tfrac14(2 v)^{-1}\hat h,\quad 0\leq\hat l\leq \hat l_0,\ \ 0\leq\hat h\leq\hat h_0.
\]
This gives the following behavior of one root of $p_\rmv(h,l,\lambda)$:
\begin{equation}
\label{EqM0lambdap1}
  \lambda_+(h,l) \geq 2 + \tfrac14(2 v)^{-1}h l^2,\quad h l\leq\hat h_0,\ \ l=\hat l^{-1}\geq\hat l_0^{-1}.
\end{equation}

In order to track the behavior of the root $\hat\lambda_-(\hat h)$ (which diverges to $-\infty$ as $\hat h\searrow 0$), we mimic the ``blow-up'' \eqref{EqM0tilde} done for regime~(1) (see also Figure~\ref{FigM0vSmall}) by introducing
\begin{equation}
\label{EqM0hatp}
  \hat h' = \hat h,\ 
  \hat l' = \hat l,\ 
  \hat\lambda' = (\hat h\hat\lambda)^{-1}\quad\Longleftrightarrow\quad
  \hat h=\hat h',\ 
  \hat l=\hat l',\ 
  \hat\lambda=(\hat h'\hat\lambda')^{-1}.
\end{equation}
We then consider the rescaling
\[
  \hat p'_\rmv(\hat h',\hat l',\hat\lambda') = \hat h'\hat\lambda'{}^2\hat p_\rmv(\hat h',\hat l',(\hat h'\hat\lambda')^{-1}) = -1 - (2 v-\hat h'\hat l')\hat\lambda' + \bigl(4 v\hat l'+\hat h'(1+\hat l')\bigr)\hat h'\hat\lambda'{}^2.
\]
At $\hat h'=0$, this vanishes at the unique and simple root $\hat\lambda'_-(0,\hat l')=-(2 v)^{-1}$, and hence there exists a root nearby for small $\hat h'$, which we denote by
\[
  \hat\lambda'_-(\hat h',\hat l')\in \bigl(-\tfrac32(2 v)^{-1},-\tfrac12(2 v)^{-1}\bigr),\quad \hat h'\leq\hat h'_0,
\]
for sufficiently small $\hat h'_0>0$. Since $\lambda=h^{-1}(\hat\lambda')^{-1}$, this translates into a root
\begin{equation}
\label{EqM0lambdahatp}
  \lambda_-(h,l) = h^{-1}\hat\lambda'_-(h l,l^{-1})^{-1} \leq -\tfrac23(2 v)h^{-1},\quad h l\leq\hat h'_0.
\end{equation}

In summary, for $l\gg 1$ and $h\ll l^{-1}$, we have found both roots of $p_\rmv(h,l,\lambda)$, and they are given by~\eqref{EqM0lambdap1} (which is larger than $1$) and~\eqref{EqM0lambdahatp} (which is very negative).

\pfstep{{\rm (3)} Intermediate $l$, i.e., $l\sim h^{-1}$.} We continue using the notation from the previous regime. For $\hat h_0\leq\hat h\leq C$ (for any fixed constants $\hat h_0>0$ and $C<\infty$), we have $\hat\lambda_+(\hat h,0)\geq c'>0$, and hence $\hat\lambda_+(\hat h,\hat l)\geq\frac12 c'$ for $\hat h_0\leq\hat h\leq C$ when $\hat l$ is sufficiently small; upon shrinking $\hat l_0$ if necessary, this means
\begin{equation}
\label{EqM0lambdap2}
  \lambda_+(h,l) \geq \tfrac12 c' l>1,\quad \hat h_0\leq\hat h=h l\leq C,\ \ l=\hat l^{-1}\geq\hat l_0^{-1}.
\end{equation}
Similarly, $\hat p_\rmv(\hat h,0,\hat\lambda)$ has the simple root $\hat\lambda_-(\hat h,0)\leq -c'<0$ for $\hat h_0\leq\hat h\leq C$, and thus (upon shrinking $\hat l_0$ further if necessary)
\begin{equation}
\label{EqM0lambdam2}
  \lambda_-(h,l) \leq -\tfrac12 c' l<-C_0,\quad \hat h_0\leq\hat h=h l\leq C,\ \ l=\hat l^{-1}\geq\hat l_0^{-1}.
\end{equation}

\pfstep{{\rm (4)} Large $l$, i.e., $l\gtrsim h^{-1}$.} Finally, we need to consider the region $l\gtrsim h^{-1}$. There, we use
\begin{equation}
\label{EqM0check}
  \check h=h,\ 
  \check l=(h l)^{-1},\ 
  \check\lambda=\lambda l^{-1}\quad\Longleftrightarrow\quad
  h=\check h,\ 
  l=(\check h\check l)^{-1},\ 
  \lambda=\check\lambda(\check h\check l)^{-1}.
\end{equation}
These can be thought of as coordinates near the lift of $l=\infty$ in the blow-up of $[0,1)_h\times[1,\infty]_l\times\R_{\hat\lambda}$, $\hat\lambda=\lambda l^{-1}$, at $h=0,l=\infty$; see again Figure~\ref{FigM0vAll}. The rescaling
\[
  \check p_\rmv(\check h,\check l,\check\lambda):=\check h\check l^2 p_\rmv\bigl(\check h,(\check h\check l)^{-1},\check\lambda(\check h\check l)^{-1}\bigr) = 1 - \check\lambda^2 - 2 v\check l\check\lambda + \check h\check l(1+\check\lambda+4 v\check l)
\]
is a smooth coefficient polynomial in $\check\lambda$. Its restriction to $\check l=0$, which is given by $\check p_\rmv(\check h,0,\check\lambda)=1-\check\lambda^2$, has simple roots $\check\lambda_\pm(\check h,0)=\pm 1$; these thus extend to two continuous families of roots $\check\lambda_\pm(\check h,\check l)$ for small $\check l$. For any fixed $\check h_0>0$, this gives as roots for $p_\rmv(h,l,\lambda)$:
\begin{equation}
\label{EqM0checkRoots}
  \pm l\check\lambda_\pm(h,(h l)^{-1}) \geq \tfrac12 l,\quad
  \check h=h\leq\check h_0,\ \ h l=\check l^{-1}\geq\check l_0^{-1},
\end{equation}
where $\check l_0>0$ is small (depending on $\check h_0$). Note that $l\geq\check l_0^{-1}h^{-1}$ is large when $h$ is small.

\pfstep{Conclusion.} This finishes our estimates of the vector type roots for sufficiently small $h$. More precisely, let $\hat h_0,\hat l_0$ be so small that~\eqref{EqM0lambdap1} and \eqref{EqM0lambdahatp} hold (giving control of the roots for $1\ll l\ll h^{-1}$), and let $\check l_0$ be such that~\eqref{EqM0checkRoots} holds for $\check h_0=\frac12$, say (giving control of the roots for $l\gg h^{-1}$). Upon shrinking $\hat l_0$ further if necessary, we then also have~\eqref{EqM0lambdap2}--\eqref{EqM0lambdam2} for $C:=\check l_0^{-1}$ (giving control of the roots for $l\sim h^{-1}$). Finally, we fix $h_0$ such that~\eqref{EqM0vSmall} holds (giving control for the roots for the remaining finitely many $l\ll 1$).

In summary, for all $h\leq h_0$, one of the two roots of $p_\rmv(h,l,\lambda)$ is $>1$ and the other is $<-C_0$ for all $l\in\N$; and the roots diverge to $\pm\infty$ as $l\to\infty$.

\subsection{Scalar type \texorpdfstring{$0$}{0} indicial roots}
\label{SsM0s}

We now let $N_{v,e,h}(\lambda)$ in~\eqref{EqM0NormOp} act on the 2-dimensional space of sections of the form $(a,b,0)$ where $a,b\in\C$; the action is given by the matrix
\[
  N_{\rms 0}(\lambda) = -\lambda^2+\lambda+\begin{pmatrix} 1 & -1 \\ -1 & 1 \end{pmatrix} + h^{-1} \begin{pmatrix} q_{0 0}(\lambda) & q_{0 1}(\lambda) \\ q_{1 0}(\lambda) & q_{1 1}(\lambda) \end{pmatrix},
\]
where $q_{i j}(\lambda)$ is obtained from $q_{i j}$ by replacing $\rho\pa_\rho$ by $\lambda$. One easily verifies that $N_{\rms 0}(1)$ is singular for any choice of $e,v,h$, so $\lambda=1$ is always a scalar type $0$ indicial root. We shall thus study the roots of
\begin{align*}
  p_{\rms 0}(e,h,\lambda) &= h^2(\lambda-1)^{-1}\det N_{\rms 0}(\lambda) \\
    &= h^2\lambda^3 + \bigl(2(3-e)h v-h^2\bigr)\lambda^2 \\
    &\qquad + (4(2-e)v^2-4 e-2(e+3)h v-2 h^2)\lambda - 4((2+e)v^2+e+h v).
\end{align*}
Now, $\frac14 p_{\rms 0}(e,0,\lambda)=((2-e)v^2-e)\lambda-((2+e)v^2+e)$ has a simple root
\[
  \lambda(e,0) = \frac{(2+e)v^2+e}{(2-e)v^2-e},
\]
which for fixed $v>0$ satisfies $\lambda(e,0)>1$ for all sufficiently small $e>0$. One can continue this root to a real-analytic function $h\mapsto\lambda(e,h)$ for small $h$ such that $p_{\rms 0}(e,h,\lambda(e,h))=0$. Thus, $\lambda(e,h)>1$ still for all sufficiently small $h>0$ (depending on $e$).

We show that the remaining two roots of $p_{\rms 0}(e,h,\lambda)$ are very negative by passing to the coordinates $\tilde h=h$, $\tilde\lambda=(h\lambda)^{-1}$ as in~\eqref{EqM0tilde} and considering the rescaling
\[
  \tilde p_{\rms 0}(e,\tilde h,\tilde\lambda) = \tilde h\tilde\lambda^3 p_{\rms 0}(e,\tilde h,(\tilde h\tilde\lambda)^{-1});
\]
this is a polynomial in $\tilde\lambda$ depending smoothly on $e$ and $\tilde h$, and we have
\[
  \tilde p_{\rms 0}(0,0,\tilde\lambda) = 8 v^2\tilde\lambda^2 + 6 v\tilde\lambda + 1.
\]
The roots of $\tilde p_{\rms 0}(0,0,\tilde\lambda)$ are $\tilde\lambda_{\ll,1}(0,0)=-(2 v)^{-1}$ and $\tilde\lambda_{\ll,2}(0,0)=-(4 v)^{-1}$. Hence, for small $e$, there are two negative roots $\tilde\lambda_{\ll,j}(e,0)$ which extend to negative roots $\tilde\lambda_{\ll,j}(e,\tilde h)$ for small $\tilde h$, corresponding to roots
\[
  \lambda_{\ll,j}(e,h) = h^{-1}\tilde\lambda_{\ll,j}(e,h)^{-1} \leq -v h^{-1},\quad j=1,2,
\]
of $p_{\rms 0}(e,h,\lambda)$.

Altogether, we have shown that there exists $e_0>0$ (depending on $v$) such that for all $e\in(0,e_0]$ there exists $h_0>0$ such that for all $0<h\leq h_0$ there are exactly two scalar type $0$ indicial roots $1$ and $\lambda(e,h)>1$ and two further roots $\lambda_{\ll,j}(e,h)\leq -v h^{-1}<-C_0$, $j=1,2$.

\subsection{Scalar type \texorpdfstring{$l\geq 1$ indicial roots}{indicial roots with l at least 1}}
\label{SsMls}

Fix $0\neq\scal\in\scalspace_l$. Then $N_{v,e,h}(\lambda)$ acts on the 3-dimensional space of sections of the form $(a\scal,b\scal,c\,\sld\scal)$, $a,b,c\in\C$, via the matrix
\begin{align*}
  N_{\rms l}(\lambda) &= -\lambda^2+\lambda+\begin{pmatrix} l(l+1) & 0 & 0 \\ 0 & l(l+1) & 0 \\ 0 & 0 & l(l+1)-1 \end{pmatrix} + \begin{pmatrix} 1 & -1 & -l(l+1) \\ -1 & 1 & l(l+1) \\ -2 & 2 & 1 \end{pmatrix} \\
    &\qquad + h^{-1}\begin{pmatrix} q_{0 0}(\lambda,l) & q_{0 1}(\lambda,l) & q_{0\slash}(\lambda,l) \\ q_{1 0}(\lambda,l) & q_{1 1}(\lambda,l) & q_{1\slash}(\lambda,l) \\ q_{\slash 0}(\lambda,l) & q_{\slash 1}(\lambda,l) & q_{\slash\slash}(\lambda,l) \end{pmatrix},
\end{align*}
where $q_{\bullet\bullet}(\lambda,l)$ is obtained from $q_{\bullet\bullet}$ in the notation of Corollary~\ref{CorM00} by replacing $\rho\pa_\rho$ by $\lambda$, further $\sld$ by $1$, and $\sldelta$ by $l(l+1)$, and we use~\eqref{EqM0Spherical} for the spherical Laplacian. We shall study the roots of the determinant
\begin{equation}
\label{EqMsPoly}
\begin{split}
  p(e,h,l,\lambda) &:= -h^3\det N_{\rms l}(\lambda) \\
    &= h^3\lambda^6 + (2(4-e)v-3 h)h^2\lambda^5 \\
    &\quad + \bigl(4(5-2 e)v^2-4 e - 2(13-e)h v-(3 l(l+1)-1)h^2\bigr)h\lambda^4 \\
    &\quad + \bigl( 8((2-e)v^2-e)v - 4((18-3 e)v^2-e)h \\
    &\quad\hspace{4em} + 2((10+e)v+2(e-4)v l(l+1))h^2 + 3(1+2 l(l+1))h^3 \bigr)\lambda^3 \\
    &\quad + \Bigl( -16 v((4-e)v^2-e) + 4\bigl((19+2 e)v^2-((5-e)v^2-2 e)l(l+1)+e\bigr)h \\
    &\quad\hspace{4em} + 2(5-e+2(8-e)l(l+1))v h^2 + (3(l^2+l-1)l(l+1)-2)h^3\Bigr)\lambda^2 \\
    &\quad + \Bigl( 8\bigl((10+e)v^2+e(1+(1-v^2)l(l+1))\bigr)v \\
    &\quad\hspace{4em} + 4\bigl((11 l(l+1)-2)v^2-e(1+3 v^2+(1+2 v^2)l(l+1))\bigr)h \\
    &\quad\hspace{4em} - 2\bigl(6+8 l(l+1)-(4-e)l^2(l+1)^2\bigr)v h^2 - 3 l^2(l+1)^2 h^3\Bigr)\lambda \\
    &\quad + \Bigl( -8 v\bigl(4 v^2+e(2(1+v^2)+(1-v^2)l(l+1))\bigr) \\
    &\quad\hspace{4em} - 4\bigl(2(2+3 l(l+1))v^2+e(1-v^2)l(l+1)(l^2+l-1)\bigr)h \\
    &\quad\hspace{4em} -2(2+(3-e)l(l+1))l(l+1)v h^2 - (l-1)l^2(l+1)^2(l+2)h^3\Bigr)
\end{split}
\end{equation}
using a strategy similar to the one in~\S\ref{SsM0v}, except the presence of an additional small parameter $e$ leads to a considerably more delicate separation of regimes. Recall that our aim is to control all roots of $p$ when $e>0$ is some small fixed number, and $h>0$ is chosen arbitrarily small (with the required smallness depending on the choice of $e$), for \emph{all} $l\in\N$. As in the vector type case, half (i.e., $3$) of the roots will be very negative, and the other half will be almost constant for small $e,h$ until $l$ gets rather large. For these two classes of roots, we consider two different classes of regimes.

For the identification and control of the three large negative roots, $3$ regimes suffice:
\begin{enumerate}
\item[(N1)] \textit{Fixed, or bounded, $l\in\N$.} The roots are $\sim -h^{-1}$ (with explicit $l$-independent prefactors in the limit $h\to 0$) when $e=0$, and remain of this size also for small $e>0$.
\item[(N2)] \textit{Intermediate $l$, i.e., $1\ll l\lesssim h^{-1}$.} The roots are still $\sim -h^{-1}$, but now the prefactors are functions of $h l$.
\item[(N3)] \textit{Large $l$, i.e., $l\gg h^{-1}$.} The roots are now $\sim -l$.
\end{enumerate}
Note that the parameter $e$ neither enters in the delineation of the asymptotic regimes nor in the scaling of the roots. We shall control these roots much as in the vector type case.

Identifying and controlling the three remaining roots, which we shall refer to (slightly imprecisely) as the \emph{positive roots}, is much more subtle. In particular, it is the precise relationship between $e,h,l$ that determines the behavior of the roots\footnote{For example, we shall see that there is a regime $\frac{1}{\sqrt{e}}\lesssim l\lesssim\frac{\sqrt{e}}{h}$ where two roots become complex. Also, for a certain range of $l\lesssim\frac{1}{\sqrt{e}}$, one of these roots is \emph{decreasing} as $l$ increases, cf.\ the panels~(P1) and~(P2) in Figure~\ref{FigMsNumeric}.} as well as their scaling. We need to consider \textit{six} regimes:
\begin{enumerate}
\item[(P1)] \textit{Fixed, or bounded, $l\in\N$.} Choosing $e$ small and then $h$ small, there are three roots $\approx 1,1,2$ (and for $l\geq 2$ strictly larger than $1$).
\item[(P2)] \textit{First intermediate regime: $1\ll l\lesssim\frac{1}{\sqrt{e}}$.} One root is close to $1$ and two further roots have real parts in $(1,2)$. The latter two roots become complex roughly when $l\sqrt{e}\geq\sqrt{\frac{v^2}{2(1-v^2)}}$, and as $l\sqrt{e}$ increases, their real parts are $\approx\frac32$ and their imaginary parts are $\sim\pm l\sqrt{e}$.
\item[(P3)] \textit{Second intermediate regime: $\frac{1}{\sqrt{e}}\lesssim l\ll\frac{\sqrt{e}}{h}$.} The roots behave as in the first intermediate regime. (We separate this regime, with $\frac{1}{l\sqrt{e}}\geq 0$, from the first intermediate regime, with $l\sqrt{e}\geq 0$, since the parameters that the roots depend continuously on in a \emph{uniform} fashion are different in the two.)
\item[(P4)] \textit{Third intermediate regime: $\frac{\sqrt{e}}{h}\lesssim l\ll h^{-1}$.} We identify three roots with positive real parts which now scale like $h l^2$. Two of these roots, which were complex in the previous regime, transition back to being real, and are real roughly for $\frac{h l}{\sqrt{e}}\geq 4\sqrt{2(1-v^2)}$. Throughout, their real parts are bounded below by positive constants times $h l^2\gtrsim\frac{e}{h}\gg 1$.
\item[(P5)] \textit{Fourth intermediate regime: $l\sim h^{-1}$.} Now there are three roots $\sim l$, with the quantitative prefactors hidden behind the ``$\sim$'' sign being functions of $h l$.
\item[(P6)] \textit{Large $l$, i.e., $l\gg h^{-1}$.} The three positive roots are now close to $l$.
\end{enumerate}

A concrete numerical example of the positive roots is plotted in Figure~\ref{FigMsNumeric}.

\begin{figure}[!ht]
\centering
  \subcaptionbox*{Regime~(P1): $1\leq l\leq 30$}{\includegraphics[width=0.45\textwidth]{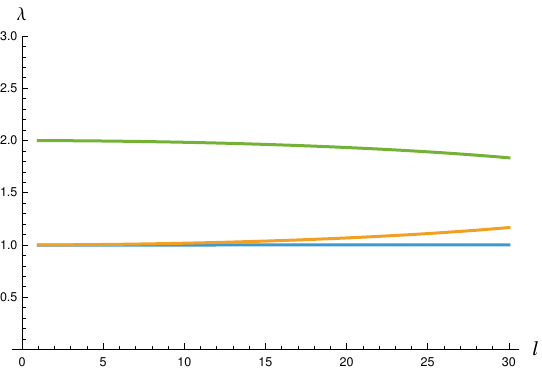}}\hspace{0.05\textwidth}
  \subcaptionbox*{Regime~(P2): $30\leq l\leq\frac{1.5}{\sqrt{e}}=150$}{\includegraphics[width=0.45\textwidth]{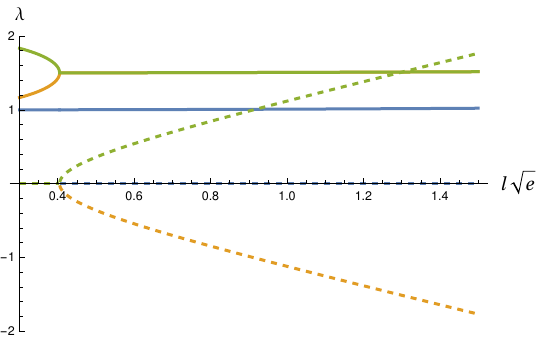}}

  \subcaptionbox*{Regime~(P3): $\frac{1.5}{\sqrt{e}}=150\leq l\leq\frac{\sqrt{e}}{2 h}=5000$}{\includegraphics[width=0.45\textwidth]{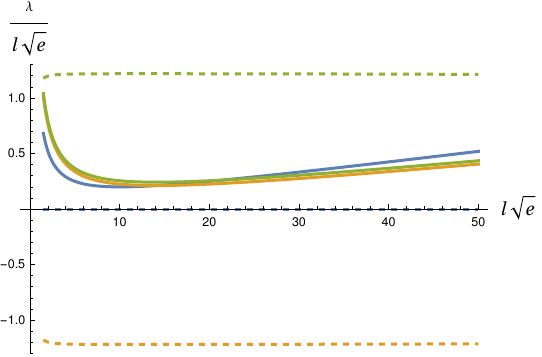}}\hspace{0.05\textwidth}
  \subcaptionbox*{Regime~(P4): $\frac{\sqrt{e}}{2 h}=5000\leq l\leq\frac{0.1}{h}=10^5$}{\includegraphics[width=0.45\textwidth]{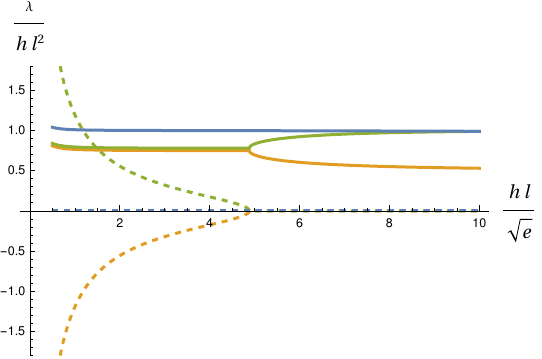}}

  \subcaptionbox*{Regime~(P5): $\frac{0.1}{h}=10^5\leq l\leq\frac{5}{h}=5\cdot 10^6$}{\includegraphics[width=0.45\textwidth]{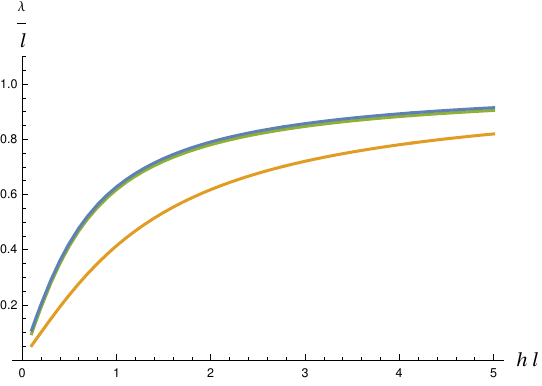}}\hspace{0.05\textwidth}
  \subcaptionbox*{Regime~(P6): $\frac{5}{h}=5\cdot 10^6\leq l<\infty$}{\includegraphics[width=0.45\textwidth]{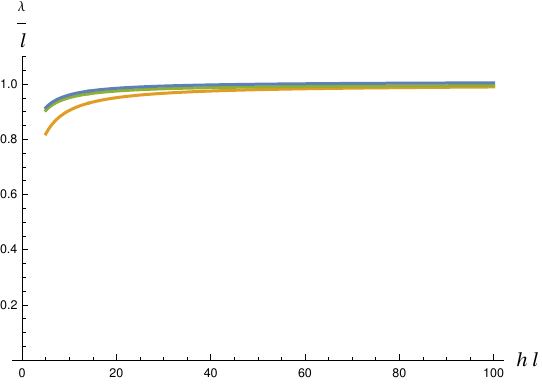}}
\caption{The three scalar type $l$ roots with positive real parts for $v=\frac12$, $e=10^{-4}$, $h=10^{-6}$ in the six regimes~(P1)--(P6). (While $l$ of course only takes on integer values, we plot the roots as functions of $l\in\R$ for aesthetic reasons.) The solid, resp.\ dashed lines are the real, resp.\ imaginary parts of the roots.}
\label{FigMsNumeric}
\end{figure}

\begin{rmk}[Identification of regimes and scaling of roots]
\label{RmkMsIdent}
  We originally arrived at the above regimes for the positive (real part) roots by numerical experimentation, namely, by computing the roots for various values of $e=10^{-N}$ and $h=10^{-M}$, $M\gg N$, and identifying the values of $l$ at which the behavior of roots changes qualitatively by a hands-on scaling analysis (i.e., determining $\alpha,\beta,\alpha',\beta'$ such that a transition happens for $l\sim e^{-\alpha}h^{-\beta}$ and such that the roots behave like $\lambda\sim e^{-\alpha'}h^{-\beta'}$). Of course, we will verify the claimed behavior of the roots, and thus prove Theorem~\ref{ThmM0}, rigorously.
\end{rmk}

We first determine the very negative roots.

\pfstep{{\rm (N1).} Negative roots for fixed, or bounded, $l\in\N$.} Let $l_0\in\N$. Write $\lambda=h^{-1}\breve\lambda$, then the rescaling
\[
  \breve p(e,h,l,\breve\lambda) := h^3 p(e,h,l,h^{-1}\breve\lambda)
\]
has smooth coefficients. Its restriction $\breve p(e,0,l,\breve\lambda)$ to $h=0$ has a triple root $\breve\lambda=0$ (which is the rescaling by $h=0$ of the $3$ positive roots---bounded for $l\leq l_0$ as $h\to 0$---that we will investigate later), and thus
\[
  \breve p'(e,l,\breve\lambda) := \breve\lambda^{-3}\breve p(e,0,l,\breve\lambda)
\]
is a cubic polynomial in $\breve\lambda$ with smooth coefficients. At $e=0$, one finds
\[
  \breve p'(0,l,\breve\lambda)=\breve\lambda^3 + 8 v\breve\lambda^2 + 20 v^2\breve\lambda + 16 v^3,
\]
the roots of which are $-4 v$ (single root) and $-2 v$ (double root). Therefore, $\breve p(e,h,l,\breve\lambda)$ has one root, resp.\ two roots close to $-4 v$, resp.\ $-2 v$ for all $1\leq l\leq l_0$ when $e$ and $h$ are small (depending on $l_0$); in particular, these roots are then still $\leq -v$. In summary,
\begin{equation}
\label{EqMsN1}
\begin{split}
  \forall\,l_0\in\N\ \exists\,e_0,h_0\in(0,1)\ \text{s.t.}\ p(e,h,l,\cdot)\ &\text{has three roots with}\ \Re\lambda_{\ll,j}(e,h,l)\leq -v h^{-1}, \\
    &\ j=1,2,3,\ 0<e\leq e_0,\ 0<h\leq h_0,\ 1\leq l\leq l_0.
\end{split}
\end{equation}

\pfstep{{\rm (N2).} Negative roots for intermediate $l$, i.e., $1\ll l\lesssim h^{-1}$.} We now write $\hat h=h l$ and $\hat l=l^{-1}$ and define the smooth coefficient polynomial
\[
  \hat p(e,\hat h,\hat l,\hat\lambda) := \hat h^3\hat l^3 p\Bigl(e,\hat h\hat l,\hat l^{-1},\frac{\hat\lambda}{\hat h\hat l}\Bigr).
\]
(Note that $\frac{\hat\lambda}{\hat h\hat l}=h^{-1}\hat\lambda$, i.e., we are seeking roots of size $\cO(h^{-1})$ here, as in regime~(N1).) At $\hat l=0$ and for $e=0$, this factors as
\[
  \hat p(0,\hat h,0,\hat\lambda) = (\hat\lambda^2+2 v\hat\lambda-\hat h^2)^2(\hat\lambda^2+4 v\hat\lambda-\hat h^2).
\]
The six roots of this polynomial are thus
\[
  -v \pm \sqrt{v^2+\hat h^2}\ \ \text{(double roots)},\quad
  -2 v\pm \sqrt{4 v^2+\hat h^2}\ \ \text{(single roots)}.
\]
Out of these, the three roots $-v-\sqrt{v^2+\hat h^2}$ (twice) and $-2 v-\sqrt{4 v^2+\hat h^2}$ (once) are $\leq -2 v<0$, and therefore also $\hat p(e,\hat h,\hat l,\hat\lambda)$ has three roots $\leq -v$ when $e,\hat l$ are small. In summary,
\begin{equation}
\label{EqMsN2}
\begin{split}
  \forall\,\hat h_0>0\ \exists\,\hat e_0,\hat l_0>0\ \text{s.t.}\ p(e,h,l,\cdot)\ &\text{has three roots with}\ \Re\lambda_{\ll,j}(e,h,l)\leq -v h^{-1}, \\
    &\ j=1,2,3,\ 0<e\leq\hat e_0,\ \hat l_0^{-1}\leq l\leq\hat h_0 h^{-1}.
\end{split}
\end{equation}

\pfstep{{\rm (N3).} Negative roots for large $l$, i.e., $l\gg h^{-1}$.} We put $\check h=h$, $\check l=(h l)^{-1}$ and seek roots of size $\sim -l$ by considering the smooth coefficient polynomial
\[
  \check p(e,\check h,\check l,\check\lambda) := \check h^3\check l^6 p\Bigl(e,\check h,\frac{1}{\check h\check l},\frac{\check\lambda}{\check h\check l}\Bigr).
\]
At $\check l=0$, this factors as
\[
  \check p(e,\check h,0,\check\lambda) = (\check\lambda-1)^3 (\check\lambda+1)^3
\]
and thus has a triple root at $\check\lambda=-1$. For small $\check l\geq 0$, there are thus three roots nearby. That is,
\begin{equation}
\label{EqMsN3}
\begin{split}
  \forall\,\check h_0,\check e_0>0\ \exists\,\check l_0>0\ \text{s.t.}\ p(e,h,l,\cdot)\ &\text{has three roots with}\ \Re\lambda_{\ll,j}(e,h,l)\leq -\tfrac12 l, \\
    &\ j=1,2,3,\ 0<e\leq \check e_0,\ 0<h\leq\check h_0,\ l\geq h^{-1}\check l_0^{-1}.
\end{split}
\end{equation}

\pfstep{Combination of {\rm (N1)--(N3)}: large negative roots.} Apply~\eqref{EqMsN3} with $\check h_0=\check e_0=1$ to obtain a value $\check l_0>0$. Then apply~\eqref{EqMsN2} with\footnote{Thus, $\hat h_0$ is \emph{large}.} $\hat h_0=\check l_0^{-1}$ (so that the upper bound $\hat h_0 h^{-1}=h^{-1}\check l_0^{-1}$ matches the lower bound on $l$ in~\eqref{EqMsN3}); this produces values $\hat e_0,\hat l_0>0$. Finally, apply~\eqref{EqMsN1} with $l_0=\hat l_0^{-1}$ (so that the upper bound $l_0$ for $l$ in~\eqref{EqMsN1} matches the lower bound $\hat l_0^{-1}$ in~\eqref{EqMsN2}). This produces values $e_0,h_0>0$ such that for all $0<e\leq e_0$ and $0<h\leq h_0$, the polynomial $p(e,h,l,\cdot)$ has three roots $\leq -v h^{-1}$ for $1\leq l\leq\hat h_0 h^{-1}$ and $\leq -\frac12 l$ for $l\geq \hat h_0 h^{-1}$. Reducing $h_0>0$ if necessary, these three roots are thus always $<-C_0$ for any fixed constant $C_0$. We have shown:
\begin{equation}
\label{EqMsN123}
\begin{split}
  \exists\,e_0,h_0>0\ \text{s.t.}\ p(e,h,l,\cdot)\ &\text{has three roots with}\ \Re\lambda_{\ll,j}(e,h,l)<-C_0, \\
    &\ j=1,2,3,\ 0<e\leq e_0,\ 0<h\leq h_0,\ l\in\N.
\end{split}
\end{equation}
Moreover, for any fixed $e,h$, we have $\Re\lambda_{\ll,j}(e,h,l)\to -\infty$ as $l\to\infty$.

\bigskip

We next turn to the positive roots (or, more precisely, the roots with positive real parts) by working our way through the six regimes (P1)--(P6). Since the parameter space of $e,h,l$ is $3$-dimensional, we can visualize these regimes on a resolved space as shown in Figure~\ref{FigMs}. (We are unable to include the variable $\lambda\in\C$, as this would require making a 5-dimensional figure.)

\begin{figure}[!ht]
\centering
\includegraphics{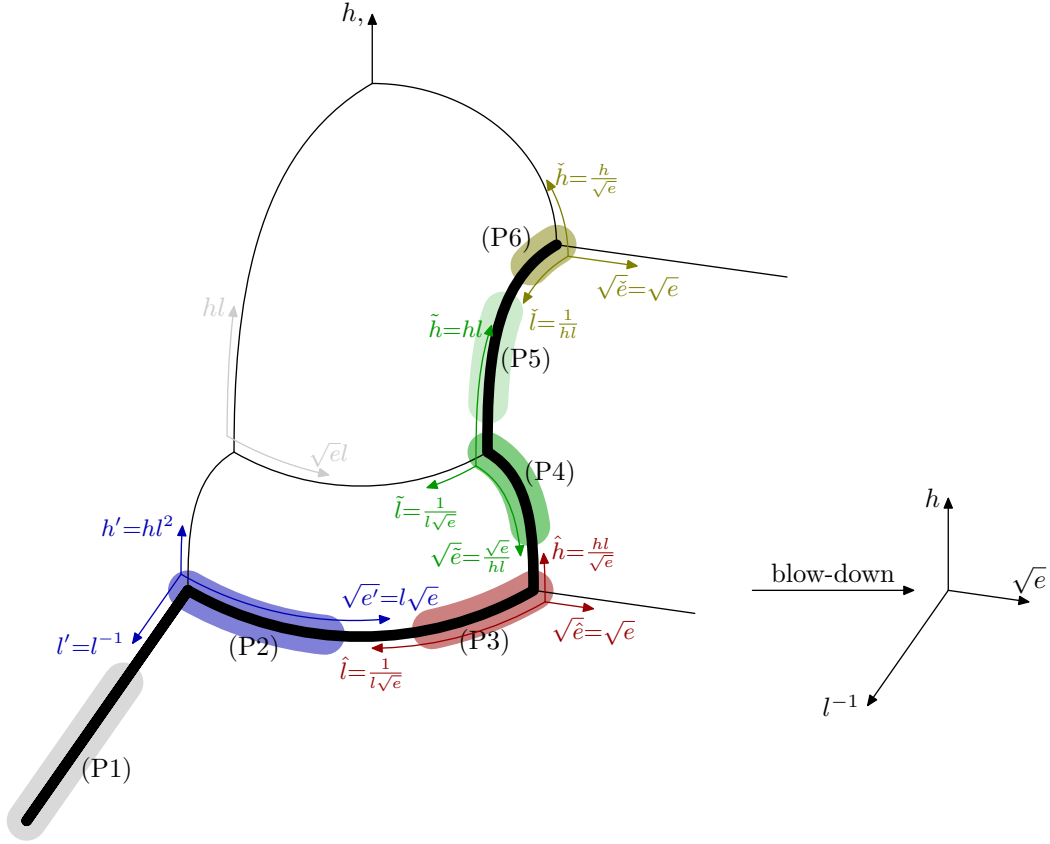}
\caption{Resolution of the total parameter space $\halfopen_{\sqrt{e}}\times\halfopen_h\times[1,\infty]_l$ (obtained by first blowing up $(0,0,\infty)$, then the lift of the $\sqrt{e}$-axis, and finally the first front face intersected with the lift of $h=0$), together with local coordinate systems and an indication of what areas of parameter space the six regimes (P1)--(P6) correspond to. (The regimes do overlap to cover the entire thick curve, but we draw them disjoint here for better readability.) Essentially, we will give fairly precise bounds (sufficient for the proof of Theorem~\ref{ThmM0}) of the positive roots in a neighborhood of the thick piecewise smooth curve which is the double limit $\lim_{e\searrow 0}\lim_{h\searrow 0}\bigl(\{e\}\times\{h\}\times[1,\infty]_l\bigr)$. Thus, if one fixes a sufficiently small value of $e>0$, then for all sufficiently small $h$, all scalar type $l\geq 1$ roots are controlled. (Only near the axis labeled $l^{-1}$ do we not control roots in a full neighborhood of this thick curve, cf.\ \eqref{EqMsP1} and \eqref{EqMsP2Small}.)}
\label{FigMs}
\end{figure}

\pfstep{{\rm (P1)}. Positive roots for fixed, or bounded, $l\in\N$.} Let $l_0\in\N$; we are interested in the roots of $p(e,h,l,\lambda)$ for $1\leq l\leq l_0$. The restriction $p(e,0,l,\lambda)$ to $h=0$ is a cubic polynomial, and $\lambda=\lambda_0(e,0,l):=1$ is always a root. For $l=1$, we in fact have $p(e,h,1,1)=0$ for all $e,h$.

To study the remaining two roots, we introduce
\begin{equation}
\label{EqMsP1p1}
\begin{split}
  p_1(e,l,\lambda) &:= \bigl(8 v(\lambda-1)\bigr)^{-1}p(e,0,l,\lambda) \\
    &= ((2-e)v^2-e)\lambda^2 - ((6-e)v^2-e)\lambda + \bigl(4 v^2 + e\bigl((1-v^2)l(l+1) + 2(1+v^2)\bigr)\bigr).
\end{split}
\end{equation}
To avoid having to study the explicit roots of this quadratic polynomial directly, let us consider its restriction to $e=0$,
\[
  p_1(0,l,\lambda) = 2 v^2(\lambda-2)(\lambda-1),
\]
which thus has two simple roots $\lambda_{1,1}(0,l)=1$ and $\lambda_{1,2}(0,l)=2$; these can be continued to roots $\lambda_{1,j}(e,l)$, $j=1,2$, for $0<e\leq e_0$ where $e_0>0$ is small (depending on $l_0$). In the context of Theorem~\ref{ThmM0}, the root $\lambda=1$ warrants further study: we compute its $e$-derivative to be
\[
  \pa_e\lambda_{1,1}(e,l)|_{e=0} = -\frac{\pa_e p_1(e,l,\lambda)}{\pa_\lambda p_1(e,l,\lambda)}\Big|_{e=0,\lambda=1} = \frac{1}{2 v^2}\bigl( (1-v^2)l(l+1) + 2(1+v^2) \bigr) \geq \frac{1}{v^2} > 0.
\]
Therefore, for sufficiently small $e_0>0$ (depending on $l_0$), we have
\[
  \lambda_{1,1}(e,l) \geq 1 + \frac{e}{2 v^2},\quad 0<e\leq e_0,\ 1\leq l\leq l_0.
\]
The other root satisfies $\lambda_{1,2}(e,l)\geq 2-\cO(e)\geq\frac74$, $1\leq l\leq l_0$, for all sufficiently small $e\geq 0$.

In particular, for $0<e\leq e_0$ (with $e_0$ depending only on $l_0$), the three roots $\lambda_0(e,0,l)=1$, $\lambda_1(e,0,l):=\lambda_{1,1}(e,l)\geq 1+\frac{e}{2 v^2}$, and $\lambda_2(e,0,l):=\lambda_{1,2}(e,l)\geq\frac32$ of $p(e,0,l,\cdot)$ are \emph{simple} and can be continued to roots $\lambda_j(e,h,l)$, $j=0,1,2$, for small $h$. For $0<e\leq e_0$ then, the root $\lambda_0(e,0,l)=1$ warrants further attention. For small $h_0>0$ (depending on $e,l_0$), it continues to a root $\lambda_0(e,h,l)=1+\cO(h)$ of $p(e,h,l,\cdot)$, with its $h$-derivative
\[
  \pa_h\lambda_0(e,h,l)|_{h=0} = -\frac{\pa_h p(e,h,l,\lambda)}{\pa_\lambda p(e,h,l,\lambda)}\Big|_{h=0,\lambda=1} = \frac{(l-1)l(l+1)(l+2)(1-v^2)}{2 v\bigl( (1-v^2)l(l+1) + 2(1+v^2) \bigr)}
\]
being independent of $e$ and positive for $l\geq 2$, an explicit lower bound being $\frac{3(1-v^2)}{v(2-v^2)}$.\footnote{The numerator and denominator of $-\frac{\pa_h p}{\pa_\lambda p}$ at $h=0$, $\lambda=1$ are both proportional to $e$, with the factor of $e$ canceling. The $h$-derivative of the double root $1$ of $p(0,0,l,\lambda)$ cannot be computed by working only at $e=0$; instead we first move to small $e>0$ here to make $1$ a simple root and then compute its $h$-derivative, which \emph{happens} to be $e$-independent.}

In summary, we have shown that for a constant $c>0$ (where we can take $c=\frac{2(1-v^2)}{v(2-v^2)}$),
\begin{equation}
\label{EqMsP1}
\begin{split}
  &\forall\,l_0\in\N\ \exists\,e_0>0\ \forall\,0<e\leq e_0\,\exists\,h_0>0\ \text{s.t.}\ p(e,h,l,\cdot)\ \text{has three roots satisfying} \\
  &\quad \lambda_0(e,h,l) \begin{cases} =1,\!\! & l=1, \\ \geq 1+c h,\!\! & l\geq 2, \end{cases}\ \ \ \lambda_1(e,h,l)\geq 1+\frac{e}{3 v^2},\ \ \lambda_2(e,h,l)\geq\frac32\quad\forall\,1\leq l\leq l_0,\ 0<h\leq h_0.
\end{split}
\end{equation}

\pfstep{{\rm (P2)}. Positive roots in the first intermediate regime: $1\ll l\lesssim\frac{1}{\sqrt{e}}$.} In this regime,\footnote{The roots of~\eqref{EqMsP1p1} are of the form $C_1(e,v)\pm\sqrt{C_2(e,v)-C_3(e,v)e l(l+1)}$ for explicit $C_j(e,v)$, $j=1,2,3$; they are thus functions of $e l(l+1)\approx e l^2$. This motivates the introduction of the quantity $l\sqrt{e}$ here.} we introduce the quantities
\[
  e'=e l^2,\ h'=h l^2,\ l'=l^{-1} \iff e=e'l'{}^2,\ h=h'l'{}^2,\ l=l'{}^{-1}.
\]
The expression for $p$ in terms of these quantities is
\[
  p'(e',h',l',\lambda) := p(e' l'{}^2,h' l'{}^2, l'{}^{-1}, \lambda),
\]
and this is a smooth coefficient polynomial of degree $6$ in $\lambda$.

We first connect to the regime~(P1). (In terms of Figure~\ref{FigMs}, the following discussion concerns a neighborhood of the axis labeled $l^{-1}$.) The restriction of $p'(e',h',l',\lambda)$ to $h'=0$, given by $p(e'l'{}^2,0,l'^{-1},\lambda)$, is cubic in $\lambda$ (as already noted in~(P1)) and always vanishes for $\lambda=1$. The quotient $p'_1(e',l',\lambda):=(8v(\lambda-1))^{-1}p'(e',0,l',\lambda)$, restricted to $e'=0$, has the simple roots $1,2$; the $e'$-derivative of the root $1$ at $e'=0$ is
\[
  -\frac{\pa_{e'}p'_1(e',l',\lambda)}{\pa_\lambda p'_1(e',l',\lambda)}\Big|_{e'=0,\lambda=1} = \frac{(1+l')(1-v^2)+2(1+v^2)l'{}^2}{2 v^2}>0.
\]
For all small $e'>0$ (depending on an upper bound for $l'$), $p'(e',0,l',\cdot)$ thus has three simple roots including $1$, which can thus be extended to roots of $p'(e',h',l',\cdot)$ depending smoothly also on small $h'$ (depending on the choice of $e'$), with the $h'$-derivative of the root $1$ at $h'=0$ being
\begin{equation}
\label{EqMsP2Der}
  -\frac{\pa_{h'}p'(e',h',l',\lambda)}{\pa_\lambda p'(e',h',l',\lambda)}\Big|_{h'=0,\lambda=1} = \frac{(1-v^2)(1-l')(1+l')(1+2 l')}{2 v\bigl( (1-v^2)(1+l')+2(1+v^2)l'{}^2\bigr)} > 0.
\end{equation}

We can, in fact, get control on the size of $h'$ for which the roots near $1$ are controlled down to $e'=0$. To this end, we write $\lambda=1+e'\Lambda$, $h'e'{}^{-1}=H$, and study $P(e',H,l',\Lambda):=e'{}^{-2}p'(e',H e',l',1+e'\Lambda)$. This is a polynomial with
\[
  P(0,0,l',\Lambda)=-16 v^3\Lambda\Bigl(\Lambda-\frac{(1-v^2)(1+l')+2(1+v^2)l'{}^2}{2 v^2}\Bigr);
\]
this has simple roots $\Lambda_0(0,0,l')=0$ and $\Lambda_1(0,0,l')=\frac{(1-v^2)(1+l')+2(1+v^2)l'{}^2}{2 v^2}>0$, which extend to simple roots $\Lambda_j(e',H,l')$, $j=0,1$. Thus, $\Lambda_1>0$ for small $e',H,l'$, which corresponds to a root of $p$ given by $1+e'\Lambda_1\geq 1+c e l^2$ (for some $c>0$). For $\Lambda_0$, we note that $\Lambda_0(e',0,l')=0$ is a root of $P(e',0,l',\Lambda)$ also for nonzero $e'$, and its $H$-derivative at $H=0$,
\[
  \pa_H\Lambda_0(e',0,l') = -\frac{\pa_H P(e',H,l',\Lambda)}{\pa_\Lambda P(e',H,l',\Lambda)}\Big|_{H=0,\Lambda=0} > 0,
\]
is given by the expression~\eqref{EqMsP2Der}. Thus, for some $c>0$, we have $\Lambda_0(e',H,l')\geq c H>0$ for all small $e',H,l'$. This corresponds to a root $1+e'\Lambda_0\geq 1+c e'H=1+c h'=1+c h l^2$ of $p$.

Thus far, we have established that there exists a constant $c>0$ such that the following holds:
\begin{equation}
\label{EqMsP2Small}
\begin{split}
  &\forall\,l'_0\in(0,\tfrac{1}{2})\ \exists\,e'_\ll>0\ \forall\,0<e'\leq e'_\ll\ \exists\,h'_0>0\ \text{s.t.}\ p(e,h,l,\cdot)\ \text{has three roots satisfying} \\
  &\quad \lambda_0(e,h,l)\geq 1+c h l^2,\ \ \lambda_{1/2}(e,h,l)\geq 1+c e l^2\quad\forall\,l^{\prime-1}_0\leq l\leq\sqrt{\frac{e'}{e}},\ 0<h\leq\frac{h'_0 e}{e'}.
\end{split}
\end{equation}
(The upper bound on $h$ guarantees that $h\leq\frac{h'_0}{l^2}$.)

\bigskip

In order to control the roots when $l\sqrt{e}$ is not necessarily small, we must first study the roots of $p'$ restricted to $l'=0$. Now, the restriction to $l'=0$ is a cubic polynomial in $\lambda$, explicitly given by
\begin{equation}
\label{EqMsP2Poly}
  p'(e',h',0,\lambda) = (2 v\lambda-2 v-h')\bigl( 8 v^2\lambda^2 - 6 v(4 v+h')\lambda + (4(4-e')v^2 + 4 h' v+h'{}^2+4 e')\bigr).
\end{equation}
Its restriction to $h'=0$ is still cubic,
\[
  p'(e',0,0,\lambda) = 16 v^3(\lambda-1)\Bigl[ \lambda^2 - 3\lambda + \Bigl(2+\frac{e'(1-v^2)}{2 v^2}\Bigr)\Bigr],
\]
and its roots are given by
\begin{equation}
\label{EqMsP2Roots}
  \lambda_0(e',0,0)=1,\quad
  \lambda_{1/2}(e',0,0) = \frac32 \pm \frac12\sqrt{1 - 2 e'\frac{1-v^2}{v^2}}.
\end{equation}
See Figure~\ref{FigMsP2}. (For $e'=0$, these match the roots $1,1,2$ discussed previously.) Note that the roots $\lambda_{1/2}(e',0)$ become complex for $e'=e l^2>\frac{v^2}{2(1-v^2)}$, at which point their real parts are $\frac32$, while their imaginary parts grow like $\sqrt{e'}$ as $e'$ increases.

\begin{figure}[!ht]
\centering
\includegraphics{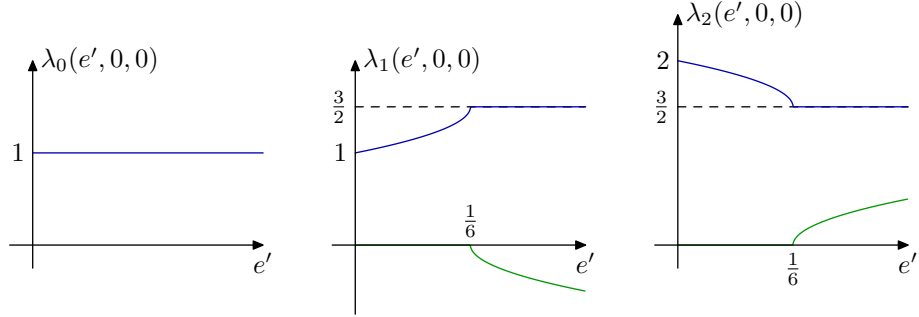}
\caption{The real parts (in blue) and imaginary parts (in green) of the roots $\lambda_j(e',0,0)$ as functions of $e'=e l^2$; here $v=\frac12$ and thus $\frac{v^2}{2(1-v^2)}=\frac16$. Compare this with panel~(P2) in Figure~\ref{FigMsNumeric}.}
\label{FigMsP2}
\end{figure}

For $e'>0$, the root $1$ is simple, while $\lambda_{1/2}$ are not simple for $e'=\frac{v^2}{2(1-v^2)}$. We can thus continue these roots as continuous functions $\lambda_j(e',h',l')$, $j=0,1,2$, for $0<h'\leq h'_0$ and $0<l'\leq l'_0$ when $h'_0,l'_0>0$ are small (depending on $e'$), with $\lambda_0$ smooth; and the lower bound $\Re\lambda_{1/2}(e',h',l')>1$ (which follows for $h'=0$, $l'=0$ and $e'>0$ from the explicit expression~\eqref{EqMsP2Roots}) continues to hold. As for the root $\lambda_0(e',h',l')$, we again note that $\lambda_0(e',0,l')=1$ is always a root of $p'(e',0,l',\cdot)$. For $e'>0$ and small $l'\geq 0$, we have already seen that it is simple; and its $h'$-derivative at $h'=0$ was already computed above to be
\[
  \pa_{h'}\lambda_0(e',h',l')|_{h'=0} = -\frac{\pa_{h'}p'(e',h',l',\lambda)}{\pa_\lambda p'(e',h',l',\lambda)}\Big|_{h'=0,\lambda=1} = \frac{(1-v^2)(1-l')(1+l')(1+2 l')}{2 v\bigl( (1-v^2)(1+l') + 2(1+v^2)l'{}^2 \bigr)} > 0.
\]
Let now $e'_\ll$ be given by~\eqref{EqMsP2Small} for any fixed choice of $l'_0$, say $l'_0=\frac{1}{100}$. Then we have shown that for some $c>0$,
\begin{equation}
\label{EqMsP2}
\begin{split}
  &\forall\,e'_0<\infty\ \exists\,h'_0>0\ \text{s.t.\ for $0<e\leq e'_\ll(l'_0)^2$,}\ p(e,h,l,\cdot)\ \text{has three roots satisfying} \\
  &\quad \lambda_0(e,h,l) \geq 1+c h l^2,\ \ \Re\lambda_{1/2}(e,h,l)>1,\quad \sqrt{\frac{e'_\ll}{e}}\leq l\leq\sqrt{\frac{e'_0}{e}},\ 0<h\leq\frac{h'_0 e}{e'_0}.
\end{split}
\end{equation}
(The upper bound on $e$ ensures that $l\geq\sqrt{e'_\ll/e}$ implies $l^{-1}\leq l'_0$. The upper bound on $h$ guarantees that $h\leq\frac{h'_0}{l^2}$. It is satisfied for fixed $e>0$ for all sufficiently small $h$.) In combination,~\eqref{EqMsP2Small} and~\eqref{EqMsP2} control the roots of $p(e,h,l,\cdot)$ for $100\leq l\leq\sqrt{\frac{e'_0}{e}}$ (and small $h$).

\pfstep{{\rm (P3)}. Positive roots in the second intermediate regime: $\frac{1}{\sqrt{e}}\lesssim l\ll\frac{\sqrt{e}}{h}$.} We now use the coordinates
\begin{equation}
\label{EqMsP3Coord}
  \hat e=e,\ \hat h=\frac{h l}{\sqrt e},\ \hat l=\frac{1}{l\sqrt e} \iff e=\hat e,\ h=\hat e\hat h\hat l,\ l=\frac{1}{\hat l\sqrt{\hat e}},
\end{equation}
as indicated in Figure~\ref{FigMs}, and we shall seek the roots as multiples of $l\sqrt e=\hat l^{-1}$, i.e.,
\[
  \lambda = l\sqrt{e}\hat\lambda = \frac{\hat\lambda}{\hat l}.
\]
The function
\[
  \hat p(\hat e,\hat h,\hat l,\hat\lambda) := \hat l^3 p\Bigl(\hat e,\hat e\hat h\hat l,\frac{1}{\hat l\sqrt{\hat e}},\frac{\hat\lambda}{\hat l}\Bigr)
\]
is then a smooth coefficient (in $\sqrt{\hat e},\hat h,\hat l$) polynomial in $\hat\lambda$. To connect to the regime~(P2), we restrict $\hat p$ to $\hat e=0$ and obtain a cubic polynomial in $\hat\lambda$,
\[
  \hat p(0,\hat h,\hat l,\hat\lambda) = (2 v\hat\lambda - 2\hat l v - \hat h) \bigl( 8 v^2\hat\lambda^2 - 6 v(4 v\hat l+\hat h)\hat\lambda + (4(4\hat l^2-1)v^2 + 4\hat h\hat l v+\hat h^2+4 ) \bigr);
\]
this is equal to~\eqref{EqMsP2Poly} upon changing coordinates appropriately. Its restriction to $\hat h=0$ is, correspondingly, still cubic, with roots given by
\begin{equation}
\label{EqMsP3lambdahat}
  \hat\lambda_0(0,0,\hat l)=\hat l,\quad
  \hat\lambda_{1/2}(0,0,\hat l) = \frac{3}{2}\hat l \pm \frac12\sqrt{\hat l^2-2\frac{1-v^2}{v^2}};
\end{equation}
these are thus $\hat l$ times the roots~\eqref{EqMsP2Roots} (but now as functions of $\hat l=(l\sqrt{e})^{-1}=e'{}^{-\frac12}$). See Figure~\ref{FigMsP3}. They extend to small $\hat e,\hat h\geq 0$ (depending on a choice of upper bound for $\hat l\geq 0$) as roots $\hat\lambda_0(\hat e,\hat h,\hat l)$ and $\hat\lambda_{1/2}(\hat e,\hat h,\hat l)$. We compute the $\hat h$-derivative of the root $\hat\lambda_0(\hat e,0,\hat l)=\hat l$ to be
\begin{equation}
\label{EqMsP3pahath}
  \pa_{\hat h}\hat\lambda_0(\hat e,\hat h,\hat l)|_{\hat h=0} = -\frac{\pa_{\hat h}\hat p(\hat e,\hat h,\hat l,\hat\lambda)}{\pa_{\hat\lambda}\hat p(\hat e,\hat h,\hat l,\hat\lambda)}\Big|_{\hat h=0,\hat\lambda=\hat l} = \frac{(1-v^2)(1-\hat l\sqrt{\hat e})(1+\hat l\sqrt{\hat e})(1+2\hat l\sqrt{\hat e})}{2 v\bigl( (1-v^2)(1+\hat l\sqrt{\hat e}) + 2(1+v^2)\hat l^2\hat e\bigr)} > 0,
\end{equation}
matching earlier computations of the same quantity in different coordinates. (The important new point here is that the present computations are uniform as $\hat l\searrow 0$ in an open neighborhood of $\hat e,\hat h=0$.) For the roots $\hat\lambda_{1/2}(\hat e,\hat h,\hat l)$, we compute
\begin{equation}
\label{EqMsP3pahath2}
  \pa_{\hat h}\hat\lambda_{1/2}(\hat e,\hat h,\hat l)|_{\hat e=0,\hat h=0} = -\frac{\pa_{\hat h}\hat p(\hat e,\hat h,\hat l,\hat\lambda)}{\pa_{\hat\lambda}\hat p(\hat e,\hat h,\hat l,\hat\lambda)}\Big|_{\hat e=0,\hat h=0,\hat\lambda=\hat\lambda_{1/2}(0,0,\hat l)} = \frac{3}{8 v} \mp i\frac{5\hat l}{8\sqrt{2(1-v^2)-\hat l^2 v^2}}.
\end{equation}
If we restrict $\hat l$ to a compact subinterval of $\sqrt{2\frac{1-v^2}{v^2}}$, then $\hat\lambda_{1/2}(0,0,\hat l)$ are a complex conjugate pair, and thus so are $\hat\lambda_{1/2}(\hat e,0,\hat l)$ for all small $\hat e\geq 0$; we can compute their real parts using \eqref{EqMsP1p1} to be
\[
  \Re\hat\lambda_{1/2}(\hat e,0,\hat l) = \frac{(6-\hat e)v^2-\hat e}{2((2-\hat e)v^2-\hat e)}\hat l = \Bigl(\frac32 + \frac{\hat e(1+v^2)}{(2-\hat e)v^2-\hat e}\Bigr)\hat l\geq\frac32\hat l
\]
provided $0\leq\hat e<(2-\hat e)v^2$ (i.e., $\hat e$ is small enough). Therefore, $\Re\hat\lambda_{1/2}(\hat e,\hat h,\hat l)\geq\frac32\hat l+\frac{1}{4 v}\hat h$ for all small enough $\hat e,\hat h$ (depending on an arbitrary but fixed upper bound on $\hat l$).

\begin{figure}[!ht]
\centering
\includegraphics{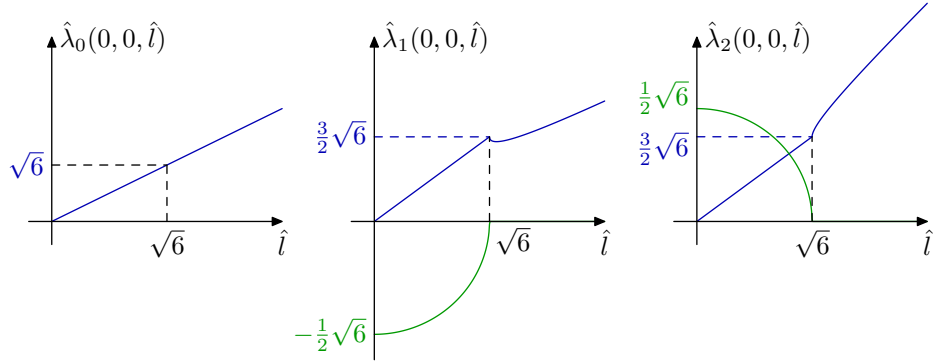}
\caption{The real parts (in blue) and imaginary parts (in green) of the roots $\hat\lambda_j(0,0,\hat l)$ as functions of $\hat l=\frac{1}{l\sqrt{e}}$; here $v=\frac12$. The roots corresponding to $j=1,2$ become complex for $\hat l<\sqrt{2\frac{1-v^2}{v^2}}=\sqrt{6}$. The limiting imaginary parts at $\hat l=0$ are $\pm\frac12\sqrt{2\frac{1-v^2}{v^2}}=\pm\frac12\sqrt{6}$, and the real parts at the transition point are $\frac32\sqrt{2\frac{1-v^2}{v^2}}=\frac32\sqrt{6}$.}
\label{FigMsP3}
\end{figure}

Altogether, noting that $\lambda_j(e,h,l)=\hat l^{-1}\hat\lambda_j(e,\hat h,\hat l)$ and $\hat l^{-1}\hat h=h l^2$, we have shown that for some $c>0$,
\begin{equation}
\label{EqMsP3}
\begin{split}
  &\forall\,\hat l_0<\sqrt{2\frac{1-v^2}{v^2}}\ \exists\,\hat e_0,\hat h_0>0\ \text{s.t.}\ p(e,h,l,\cdot)\ \text{has three roots satisfying} \\
  &\quad \lambda_0(e,h,l)\geq 1+c h l^2,\ \ \Re\lambda_{1/2}(e,h,l)\geq\frac32+c h l^2,\quad 0<e\leq\hat e_0,\ \frac{1}{\hat l_0\sqrt{e}} \leq l \leq \frac{\hat h_0\sqrt{e}}{h}.
\end{split}
\end{equation}

\pfstep{{\rm (P4)}. Positive roots in the third intermediate regime: $\frac{\sqrt{e}}{h}\lesssim l\ll h^{-1}$.} We now work with the coordinates
\begin{equation}
\label{EqMsP4Coord}
  \tilde e=\frac{e}{h^2 l^2},\ 
  \tilde h=h l,\ 
  \tilde l=\frac{1}{l\sqrt{e}} \iff
  e=\tilde e\tilde h^2,\ 
  h=\tilde h^2\tilde l\sqrt{\tilde e},\ 
  l=\frac{1}{\tilde h\tilde l\sqrt{\tilde e}},
\end{equation}
as indicated in Figure~\ref{FigMs}; and we seek roots that are multiples of $h l^2$ by setting
\[
  \lambda = h l^2\tilde\lambda = \frac{\tilde\lambda}{\tilde l\sqrt{\tilde e}}.
\]
Suitably rescaling the expression for $p$ in these coordinates yields a smooth coefficient polynomial (in $\sqrt{\tilde e},\tilde h,\tilde l,\tilde\lambda$),
\[
  \tilde p(\tilde e,\tilde h,\tilde l,\tilde\lambda) := \tilde e^{\frac32}\tilde l^3 p\Bigl(\tilde e\tilde h^2,\tilde h^2\tilde l\sqrt{\tilde e},\frac{1}{\tilde h\tilde l\sqrt{\tilde e}},\frac{\tilde\lambda}{\tilde l\sqrt{\tilde e}}\Bigr).
\]
Its restriction to $\tilde h=0$ is a cubic polynomial, and further restricting to $\tilde l=0$ produces the cubic polynomial
\[
  \tilde p(\tilde e,0,0,\tilde\lambda) = (2 v\tilde\lambda-1)\bigl( 8 v^2\tilde\lambda^2 - 6 v\tilde\lambda + (4(1-v^2)\tilde e + 1) \bigr)
\]
whose roots are
\begin{equation}
\label{EqMsP4lambda}
  \tilde\lambda_0(\tilde e,0,0) = \frac{1}{2 v}, \quad
  \tilde\lambda_{1/2}(\tilde e,0,0) = \frac{3}{8 v} \pm \frac{1}{8 v}\sqrt{1-32(1-v^2)\tilde e}.
\end{equation}
(Here, $\frac{1}{2 v}$ is equal to~\eqref{EqMsP3pahath} for $\hat l=0$. Similarly, $\frac{3}{8 v}$ matches the same number in~\eqref{EqMsP3pahath2}. Lastly, the imaginary part of~\eqref{EqMsP3lambdahat} is $\pm\frac{1}{2 v}\sqrt{2(1-v^2)}+\cO(\hat l^2)$ as $\hat l\to 0$, so upon passing to $\tilde\lambda=\frac{\sqrt e}{h l}\hat\lambda=\sqrt{\tilde e}\hat\lambda$, this gives $\Im\tilde\lambda_{1/2}\sim\pm\frac{1}{2 v}\sqrt{2(1-v^2)}\sqrt{\tilde e}=\pm\frac{1}{8 v}\sqrt{32(1-v^2)\tilde e}$, matching~\eqref{EqMsP4lambda} as $\tilde e\to+\infty$.) Note that $\Re\tilde\lambda_j(\tilde e,0,0)\geq\frac{1}{4 v}$ for $j=0,1,2$. The roots $\tilde\lambda_{1/2}$ describe the transition of the formerly complex roots (cf.\ $\hat\lambda_{1/2}$ in~\eqref{EqMsP3lambdahat} and $\lambda_{1/2}$ in~\eqref{EqMsP3}) back to real roots as $\tilde e$ crosses $\frac{1}{32(1-v^2)}$ in the decreasing direction; see Figure~\ref{FigMsP4}.

\begin{figure}[!ht]
\centering
\includegraphics{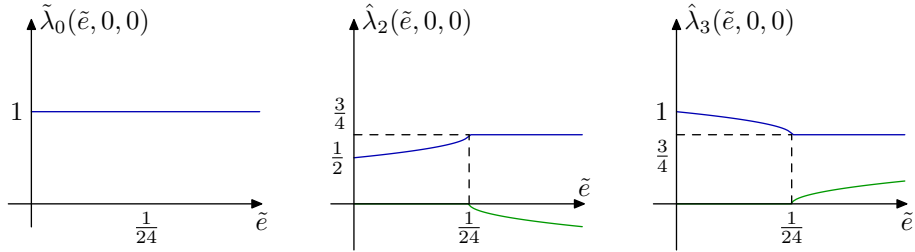}
\caption{The real parts (in blue) and imaginary parts (in green) of the roots $\tilde\lambda_j(\tilde e,0,0)$ as functions of $\tilde e=\frac{e}{h^2 l^2}=(\frac{h l}{\sqrt{e}})^{-2}$; here $v=\frac12$. The roots corresponding to $j=1,2$ become real for $\tilde e\leq\frac{1}{32(1-v^2)}=\frac{1}{24}$, with the real parts at the transition point being $\frac{3}{8 v}=\frac{3}{4}$. As $\tilde e\searrow 0$, the roots tend to $\frac{1}{2 v}=1$ (for $j=0,1$) and $\frac{1}{4 v}=\frac12$ (for $j=2$).}
\label{FigMsP4}
\end{figure}

We can continuously extend the roots $\tilde\lambda_j(\tilde e,\tilde h,\tilde l)$ to small values of $\tilde l$ and $\tilde h$, and their real parts will remain $\geq\frac{1}{8 v}$. We have shown:
\begin{equation}
\label{EqMsP4}
\begin{split}
  &\forall\,\tilde e_0<\infty\ \exists\,\tilde h_0,\tilde l_0>0\ \text{s.t.\ for $0<e\leq\tilde e_0\tilde h_0^2$,}\ p(e,h,l,\cdot)\ \text{has three roots satisfying} \\
  &\quad \lambda_0(e,h,l)\geq\frac{1}{8 v}h l^2,\ \ \Re\lambda_{1/2}(e,h,l)\geq\frac{1}{8 v}h l^2,\quad \frac{1}{\sqrt{\tilde e_0}}\frac{\sqrt{e}}{h}\leq l\leq\frac{\tilde h_0}{h},\ 0<h\leq\frac{e\tilde l_0}{\sqrt{\tilde e_0}}.
\end{split}
\end{equation}
(The lower and upper bounds on $l$ are equivalent to $\tilde e\leq\tilde e_0$ and $\tilde h\leq\tilde h_0$, respectively, while the upper bound on $h$ ensures that this lower bound on $l$ implies also that $\tilde l=\frac{1}{l\sqrt{e}}\leq\tilde l_0$.)

\pfstep{{\rm (P5)}. Positive roots in the fourth intermediate regime: $l\sim h^{-1}$.} At this point, we could use the coordinates from regime~(N2). In order to stay with coordinates suggested by Figure~\ref{FigMs}, we instead use here again the coordinates~\eqref{EqMsP4Coord} and work in the regime $\tilde h\sim 1$. The roots scale like $l=\frac{1}{\tilde h\tilde l\sqrt{\tilde e}}$ now, so we consider
\[
  \tilde p'(\tilde e,\tilde h,\tilde l,\tilde\lambda') := \tilde e^{\frac32}\tilde h^3\tilde l^3 p\Bigl(\tilde e\tilde h^2,\tilde h^2\tilde l\sqrt{\tilde e},\frac{1}{\tilde h\tilde l\sqrt{\tilde e}},\frac{\tilde\lambda'}{\tilde h\tilde l\sqrt{\tilde e}}\Bigr),
\]
which has smooth coefficients (in $\sqrt{\tilde e},\tilde h,\tilde l,\tilde\lambda'$). Its restriction to $\tilde e=0$, $\tilde l=0$ is a sextic polynomial that factors into quadratics,
\[
  \tilde p'(0,\tilde h,0,\tilde\lambda') = (\tilde h\tilde\lambda'{}^2+2 v\tilde\lambda'-\tilde h)^2 ( \tilde h\tilde\lambda'{}^2 + 4 v\tilde\lambda'-\tilde h),
\]
and thus has positive roots
\begin{alignat*}{2}
  \tilde\lambda'_{1/2}(0,\tilde h,0) &= \tilde h^{-1}\biggl( -v+\sqrt{v^2+\tilde h^2} \,\biggr)&&\ \ \text{(double root)}, \\
  \tilde\lambda'_3(0,\tilde h,0) &= \tilde h^{-1}\biggl( -2 v+\sqrt{4 v^2+\tilde h^2} \,\biggr)&&\ \ \text{(single root)}.
\end{alignat*}
Extending these continuously to small $\tilde e,\tilde l\geq 0$, we obtain:
\begin{equation}
\label{EqMsP5}
\begin{split}
  &\forall\,0<\tilde h_0<\tilde h_1<\infty\ \exists\,c,\tilde e_0,\tilde l_0>0\ \text{s.t.}\ p(e,h,l,\cdot)\ \text{has three roots satisfying} \\
  &\quad \Re\lambda_j(e,h,l) \geq c l,\quad j=1,2,3,\ 0<e\leq\tilde e_0\tilde h_0^2,\ \frac{\tilde h_0}{h}\leq l\leq\frac{\tilde h_1}{h},\ 0<h\leq\tilde h_0\tilde l_0\sqrt{e}.
\end{split}
\end{equation}
(The bounds on $l$ are equivalent to $\tilde h_0\leq\tilde h\leq\tilde h_1$. The upper bound on $h$ ensures that the stated lower bound on $l$ implies also that $\tilde l=\frac{1}{l\sqrt{e}}\leq\tilde l_0$. The upper bound on $e$ finally ensures that $\tilde e=\frac{e}{h^2 l^2}\leq\frac{e}{\tilde h_0^2}$ is bounded from above by $\tilde e_0$.)

\pfstep{{\rm (P6)}. Positive roots for large $l$, i.e., $l\gg h^{-1}$.} We could argue as in regime~(N3). In keeping with Figure~\ref{FigMs}, we instead use
\[
  \check e=e,\ \check h=\frac{h}{\sqrt{e}},\ \check l=\frac{1}{h l} \iff e=\check e,\ h=\check h\sqrt{\check e},\ l=\frac{1}{\check h\check l\sqrt{\check e}},
\]
and implement the scaling $\lambda\sim l$ of the roots by considering
\[
  \check p(\check e,\check h,\check l,\check\lambda) := \check e^{\frac32}\check h^3\check l^6 p\Bigl(\check e,\check h\sqrt{\check e},\frac{1}{\check h\check l\sqrt{\check e}},\frac{\check\lambda}{\check h\check l\sqrt{\check e}}\Bigr).
\]
The restriction $\check p(\check e,\check h,0,\check\lambda)$ to $\check l=0$ factors as $(\check\lambda-1)^3(\check\lambda+1)^3$, and hence it has roots $\check\lambda_j(\check e,\check h,0)=1$ which extend to nearby roots $\check\lambda_j(\check e,\check h,\check l)$ for small $\check l\geq 0$. Thus,
\begin{equation}
\label{EqMsP6}
\begin{split}
  &\forall\,\check h_0,\check e_0>0\ \exists\,\check l_0>0\ \text{s.t.}\ p(e,h,l,\cdot)\ \text{has three roots satisfying} \\
  &\quad \lambda_j(e,h,l)\geq\tfrac12 l,\quad j=1,2,3,\ 0<e\leq\check e_0,\ 0<h\leq\check h_0\sqrt{e},\ l\geq\frac{1}{h\check l_0}.
\end{split}
\end{equation}

\pfstep{Combination of {\rm (P1)--(P6)}.} We say below that the (three) positive roots of $p(e,h,l,\cdot)$ are \emph{controlled} if all of them have real parts $>1$, except when $l=1$ and one root is equal to $1$.

We first build our way up from $l\gg 1$ to $l\ll h^{-1}$:
\begin{itemize}
\item (P2), first part. We use~\eqref{EqMsP2Small} for $l'_0=\frac{1}{100}$ to obtain $0<e'_\ll\ll 1$; and then we use $e'=e'_\ll$ there to obtain $h'_0>0$ such that the positive roots are controlled for $100\leq l\leq\sqrt{\frac{e'_\ll}{e}}$ and $0<h\leq\frac{h'_0 e}{e'_\ll}$.
\item (P2), second part. We combine this with~\eqref{EqMsP2} for $e'_0=4\frac{v^2}{2(1-v^2)}$ to deduce, for a possibly smaller value of $h'_0>0$, that the positive roots are in fact controlled for all $100\leq l\leq\sqrt{e'_0/e}$ and $0<h\leq h'_0 e$.
\item (P3). Next, we use~\eqref{EqMsP3} with $\hat l_0=(e'_0)^{-\frac12}=\frac12\sqrt{2\frac{1-v^2}{v^2}}$ to obtain $\hat e_0,\hat h_0>0$ such that we enlarge the interval of $l$ for which the positive roots are controlled to $100\leq l\leq\frac{\hat h_0\sqrt{e}}{h}$, under the assumptions $0<e\leq\hat e_0$ and $0<h\leq h'_0 e$.
\item (P4). We use~\eqref{EqMsP4} with $\tilde e_0=\hat h_0^{-2}$ to obtain $\tilde h_0>0$ such that control of the positive roots extends to $100\leq l\leq\frac{\tilde h_0}{h}$ provided $0<e\leq\hat e_0$ and $0<h\leq h'_0 e$ (where we shrink $h'_0>0$ further so that it is smaller than $\frac{\tilde l_0}{\sqrt{\tilde e_0}}$ in the notation of~\eqref{EqMsP4}).
\end{itemize}

We now build our way back from $l\gg h^{-1}$:
\begin{itemize}
\item (P6). We apply~\eqref{EqMsP6} with $\check h_0=\check e_0=1$ to obtain a value $\check l_0>0$ such that the positive roots are controlled for $l\geq\frac{1}{h\check l_0}$.
\item (P5). We use~\eqref{EqMsP5} with $\tilde h_0$ from before and $\tilde h_1=\check l_0^{-1}$. For $0<e\leq 1$, this produces a value $h_0$ (equal to $\tilde h_0\tilde l_0 e$ in the notation~\eqref{EqMsP5}) such that the positive roots are controlled for all $l\geq\frac{\tilde h_0}{h}$ provided $0<h\leq h_0$.
\end{itemize}

In combination, we get constants $e_0>0$ and $h'_0>0$ such that the positive roots are controlled for all $l\geq 100$ provided $0<e\leq e_0$ and $0<h\leq h'_0 e$. It remains to control the first $100$ values of $l$:

\begin{itemize}
\item (P1). We use~\eqref{EqMsP1} with $l_0=100$ to deduce that for all $0<e\leq e_0$ (for a possibly smaller value of $e_0$) there exists $h_0>0$ such that the positive roots are controlled for all $l\in\N$ provided $0<h\leq h_0$.
\end{itemize}

This completes the proof of Theorem~\ref{ThmM0}.

\section{A dilation-invariant model problem on Minkowski space}
\label{SMi}

We prove a technical result which is needed for the application of the present work to the analysis of the low-energy resolvent of a linearized gauge-fixed Einstein operator (see \cite[\S{8.3}]{HintzKerrStab}, and also \cite[\S\S{9.5}--{9.6}]{HintzNonstat2} for general context for the result proved here). Some context and functional analytic background are also presented in~\cite[Definition~3.13, \S{4.1.4}]{HintzNonstat}. Roughly speaking, the behavior of the spectral family of $\Box_{g_{\bhm,a}}^{\cC_h}$ at the transition from zero to nonzero frequencies $\sigma$ (with $\Im\sigma\geq 0$ and fixed $\arg\sigma\in[0,\pi]$) is captured by a model operator obtained by rescaling $(\sigma,r)\mapsto(\lambda\sigma,r/\lambda)$ and passing to the limit $\lambda\to 0$ (upon factoring out $\lambda^2$); this yields the spectral family of a dilation-invariant operator on Minkowski space (the model spacetime at $r=\infty$) at a frequency $\in e^{i[0,\pi]}$. (This operator is called the \emph{transition face normal operator} in \cite{HintzNonstat}.) Rather than present this procedure in detail, we shall simply present this operator here and establish some of its properties.

\subsection{Setup and statement of the main result}

We work on Minkowski space
\[
  \R_t\times\R^3_x,\quad g=-\dd t^2 + \dd r^2+r^2\,\slg,
\]
which is homogeneous of degree $2$ with respect to spacetime dilations $(t,r)\mapsto(\lambda t,\lambda r)$. We define the dilation-invariant 1-form
\begin{equation}
\label{EqMi1form}
  \cd=r^{-1}(\dd t-v\,\dd r),\quad v\in(0,1),
\end{equation}
and shall study the operator
\begin{equation}
\label{EqMiOp}
  \Box_g^{\cC_h} = 2 \delta_g\sfG_g\delta_{g,\cC_h}^*,\quad
  \delta_{g,\cC_h}^*=\delta_g^*+\cC_h,\quad
  \cC_h=h^{-1}\bigl(2\cd\otimes_s(\cdot)-(1-e)g\,\iota_{\cd^\sharp}\bigr).
\end{equation}
(Cf.\ \eqref{EqIRoughC}.) As in the bulk of the paper, we focus on the regime where $e>0$ is small but fixed, and $h\in(0,h_0)$ where $h_0>0$ is small. (In practice, we are only interested in values of $e$ allowed by Proposition~\ref{PropEC}, for some fixed choice of subextremal Kerr parameters, upon fixing any $C_0>0$.) Note that $\Box_g^{\cC_h}$ is homogeneous of degree $-2$ with respect to dilations, and it has singular coefficients at $r=0$. We shall thus consider it as an operator on the resolved compactification\footnote{The meaning of $M_0,M$ is different here than in the main part of the paper since we neither excise a neighborhood of $r=0$ nor restrict to the future of some spacelike hypersurface (yet).}
\begin{equation}
\label{EqMiMfd}
  M := [M_0; \ol{\{r=0\}}],\quad
  M_0 := \ol{\R^4},
\end{equation}
acting on sections of the bundle
\begin{equation}
\label{EqMiBundle}
  \cT^* := \upbeta^*\Tsc^*\ol{\R^4},
\end{equation}
where $\upbeta\colon M\to\ol{\R^4}$ denotes the blow-down map. (A smooth global frame of $\cT^*$ is given by the spacetime 1-forms $\dd t,\dd x^1,\dd x^2,\dd x^3$.)

Our main interest lies in the spectral family of $\Box_g^{\cC_h}$ at frequencies $\sigma\in\C\setminus\{0\}$, $\Im\sigma\geq 0$, with respect to the null coordinate
\[
  t_* = t-r.
\]
We denote this spectral family by
\begin{equation}
\label{EqMiSpecFam}
  \wh{\Box_g^{\cC_h}}(\sigma) = 2\wh{\delta_g}(\sigma)\sfG_g\wh{\delta_{g,\cC_h}^*}(\sigma) = 2\wh{\delta_g}(\sigma)\sfG_g\bigl(\wh{\delta_g^*}(\sigma)+\cC_h\bigr).
\end{equation}
This is the operator $e^{i\sigma t_*}\Box_g^{\cC_h}e^{-i\sigma t_*}$ acting on stationary 1-forms. We shall consider $\wh{\Box_g^{\cC_h}}(\sigma)$ as an operator on the manifold $X^\circ$ where
\[
  X := [\ol{\R^3};\{0\}] = [0,\infty]_\rho \times \Sph^2,\quad \rho=r^{-1},
\]
is regarded as the level set $\ol{\{t=0\}}\subset M$; thus, $\wh{\Box_g^{\cC_h}}(\sigma)$ acts on sections of the vector bundle
\[
  \cT^*_X := (\upbeta|_X)^*(\Tsc^*\ol{\R^4}),
\]
which has a smooth global frame $\dd t,\dd x^1,\dd x^2,\dd x^3$. In order to write down differential operators acting on sections of $\cT_X^*$, we shall use the bundle splitting~\eqref{EqM0Split} which identifies
\begin{equation}
\label{EqMiSplit}
  \cT^*_X\cong\ubar\R\oplus\ubar\R\oplus T^*\Sph^2.
\end{equation}
In the notation of Corollary~\ref{CorM00} then, we have
\begin{align}
  \wh{\Box_g^{\cC_h}}(\sigma) &= 2 i\sigma\rho(\rho\pa_\rho-1-h^{-1}\ubar S) + \rho^2\left(-(\rho\pa_\rho)^2+\rho\pa_\rho+\slDelta + T(\slnabla) + h^{-1}Q(\rho\pa_\rho,\slnabla)\right) \nonumber\\
\label{EqMiOpB}
\begin{split}
    &= -2 i \sigma r^{-1}(r\pa_r + 1 + h^{-1}\ubar S) \\
    &\qquad + r^{-2}\bigl(-(r\pa_r)^2 - r\pa_r + \slDelta + T(\slnabla) + h^{-1} Q(-r\pa_r,\slnabla)\bigr)
\end{split} \\
\label{EqMiOpSc}
\begin{split}
    &= 2 i \sigma(\rho^2\pa_\rho - \rho-\rho h^{-1}\ubar S) \\
    &\qquad + \bigl( -(\rho^2\pa_\rho)^2 + 2 \rho\cdot\rho^2\pa_\rho + \rho^2\slDelta + \rho^2 T(\slnabla) + h^{-1}\rho^2 Q(\rho\pa_\rho,\slnabla)\bigr)
\end{split}
\end{align}
where (recalling~\eqref{EqM00S}) we write
\begin{equation}
\label{EqMiOpS}
  \ubar S=\begin{pmatrix} 2(1-v) & 0 & 0 \\ (1-e)(1+v) & (1-e)(1-v) & 0 \\ 0 & 0 & 1-v \end{pmatrix},\quad
  T(\slnabla)=\begin{pmatrix} 1 & -1 & -\sldelta \\ -1 & 1 & \sldelta \\ -2\sld & 2\sld & 1 \end{pmatrix},
\end{equation}
and $Q(\rho\pa_\rho,\slnabla)$ is given by the final matrix in~\eqref{EqM00} (without the prefactor $h^{-1}$). Thus,
\begin{equation}
\label{EqMiOpPiecesb}
  T,Q\in\Diffb^1(X;\cT^*_X),\quad r^2\wh{\Box_g^{\cC_h}}(0)\in \Diffb^2(X;\cT^*_X)
\end{equation}
are dilation-invariant operators if we define dilations to act trivially on the three components of a section of~\eqref{EqMiSplit}. Formally setting $h^{-1}=0$, one obtains the spectral family of the 1-form wave operator on Minkowski space (albeit relative to the spacetime splitting $\R_{t_*}\times\R^3_x$); for general $h>0$, we will thus see the usual radial point structure over $r^{-1}=0$ from \cite[\S{4.1.2}]{HintzNonstat} and \cite[\S{9.2.2}]{HintzNonstat2} (which are inspired by \cite{VasyLowEnergyLag,MelroseEuclideanSpectralTheory}). Denote by
\begin{equation}
\label{EqMiBdfs}
  \rho_\scop := \frac{\rho}{\rho+1},\quad
  \rho_\bop := \frac{r}{r+1},
\end{equation}
global defining functions of the two boundaries
\[
  \pa_\scop X:=\rho^{-1}(0)\subset X,\quad
  \pa_\bop X:=r^{-1}(0)\subset X.
\]
The notation reflects that near $\rho=0$, $\wh{\Box_g^{\cC_h}}(\sigma)$ is an unweighted scattering operator, i.e., it is constructed from $\rho^2\pa_\rho,\rho\slnabla$, as follows from~\eqref{EqMiOpSc}; and near $r=0$, $\wh{\Box_g^{\cC_h}}(\sigma)$ is a weighted (with weight $r^{-2}$) b-operator, i.e., it is constructed from $r\pa_r$, $\slnabla$, as follows from~\eqref{EqMiOpB}. We combine both observations concisely as
\begin{equation}
\label{EqMiOpscb}
  \wh{\Box_g^{\cC_h}}(\sigma) \in \rho_\bop^{-2}\Diff_{\scop,\bop}^2(X;\cT^*_X).
\end{equation}
We shall study this operator as a map acting between weighted scattering-b Sobolev spaces
\[
  H_{\scop,\bop}^{s,(\ell,\beta)}(X;\cT^*_X):=\rho_\scop^\ell\rho_\bop^\beta H_{\scop,\bop}^s(X;\cT^*_X).
\]
The underlying $L^2$-space $L^2=H_{\scop,\bop}^{0,(0,0)}$ is defined with respect to the density $r^2\,|\dd r\,\dd\slg|=r^3\,|\frac{\dd r}{r}\,\dd\slg|=\rho^{-3}\frac{\dd\rho}{\rho}|\dd\slg|$. We will analyze~\eqref{EqMiOpscb} by microlocal means in the scattering-b-cotangent bundle ${}^{\scop,\bop}T^*X$, which over $r<2$ is equal to $\Tb^*([0,2)_r\times\Sph^2)$ and over $\rho<2$ (i.e., $r>\frac12$) equal to $\Tsc^*([0,2)_\rho\times\Sph^2)$. The main result of this appendix is:

\begin{thm}[Invertibility of $\wh{\Box_g^{\cC_h}}(\sigma)$]
\label{ThmMiInv}
  Fix $v\in(0,1)$ and define $\cd$ by~\eqref{EqMi1form}. Fix $C_0>0$. Then there exists $e_0>0$ such that for all $e\in(0,e_0)$ there exists $h_0>0$ such that for all $h\in(0,h_0)$ and $s,\ell,\beta\in\R$ with
  \begin{equation}
  \label{EqMiInvRange}
    -\tfrac12-h^{-1}(1-e)(1+v)<\ell<-\tfrac12+h^{-1}(1-e)(1-v),\quad
    \beta\in(\tfrac12,\tfrac32+C_0],
  \end{equation}
  the operator
  \begin{equation}
  \label{EqMiInv}
    \wh{\Box_g^{\cC_h}}(\sigma) \colon \bigl\{ \omega \in H_{\scop,\bop}^{s,(\ell,\beta)}(X;\cT^*_X) \colon \wh{\Box_g^{\cC_h}}(\sigma)\omega \in H_{\scop,\bop}^{s-2,(\ell+1,\beta-2)}(X;\cT^*_X) \bigr\} \to H_{\scop,\bop}^{s-2,(\ell+1,\beta-2)}(X;\cT^*_X)
  \end{equation}
  is invertible for all $\sigma\in e^{i[0,\pi]}$, with uniformly (in $\sigma$) bounded inverse. More generally, this holds for all variable order functions $s\in\CI({}^{\scop,\bop}S^*X)$, $\ell\in\CI(\ol{{}^{\scop,\bop}T_{\pa_\scop X}^*}X)$ satisfying the following conditions: $\ell<-\frac12+h^{-1}(1-e)(1-v)$ at $\cR_{\rm out}$ (defined in~\eqref{EqMiFredRout} below); and if $\sigma\in\{-1,+1\}$, then $\ell>-\frac12-h^{-1}(1-e)(1+v)$ at $\cR_{\rm in}$ (defined in~\eqref{EqMiFredRin} below), and $\ell$ is constant near $\cR_{\rm in},\cR_{\rm out}$ and non-increasing along $\sigma\sfH_{G_\sigma}$.
\end{thm}

We shall prove Theorem~\ref{ThmMiInv} in two steps.
\begin{enumerate}
\item We first prove that~\eqref{EqMiInv} is a Fredholm operator (see~\S\ref{SsMiFred}), with nullspace and cokernel independent of $s,\ell,\beta$; this utilizes (non-semiclassical) microlocal analysis of $\wh{\Box_g^{\cC_h}}(\sigma)$.
\item We then show that~\eqref{EqMiInv} is invertible for a fixed choice of $s,\ell,\beta$ when $h$ is sufficiently small (depending on $e$). To this end, we shall prove an analogue of Theorem~\ref{ThmET} for the wave operator $\Box_g^{\cC_h}$ on spacetime (see~\S\ref{SsMiBox}) and then argue as in~\S\ref{ST} to conclude the proof of Theorem~\ref{ThmMiInv} (see~\S\ref{SsMiPf}). Any attempt to prove the invertibility of~\eqref{EqMiInv} directly would be met with considerable linear algebraic difficulties related to the properties of the semiclassical principal symbol of $h^2\wh{\Box_g^{\cC_h}}(h^{-1}\sigma)$ (which, for fixed $h$, is related to $\wh{\Box_g^{\cC_h}}(\sigma)$ via pullback along $r\mapsto h r$), cf.\ Remark~\ref{RmkTInvCx}.
\end{enumerate}

\subsection{Fredholm analysis of \texorpdfstring{$\wh{\Box_g^{\cC_h}}(\sigma)$}{the spectral family}}
\label{SsMiFred}

We record the b-principal symbol of~\eqref{EqMiOpscb} in $r<2$ to be
\[
  \sigmab\bigl(r^2\wh{\Box_g^{\cC_h}}(\sigma)\bigr)\Bigl(\xi_\bop\frac{\dd r}{r}+\eta_\bop\Bigr) = \xi_\bop^2+|\eta_\bop|_{\slg^{-1}}^2,\quad \eta_\bop\in T^*\Sph^2,
\]
which is thus elliptic; and its scattering principal symbol in $\rho<2$ is given by
\[
  G_\sigma(\rho,\omega;\xi,\eta) := \sigmasc\bigl(\wh{\Box_g^{\cC_h}}(\sigma)\bigr)\Bigl(\xi\frac{\dd\rho}{\rho^2}+\frac{\eta}{\rho}\Bigr) = -2\sigma\xi + \xi^2 + |\eta|_{\slg^{-1}}^2,\quad \eta\in T^*\Sph^2.
\]
This is elliptic for $\rho>0$, and over $\rho=0$ it is always characteristic at the zero section
\begin{equation}
\label{EqMiFredRout}
  \cR_{\rm out} := \{ \rho=0,\ \xi=\eta=0 \} = {}^{\scop,\bop}o_{\pa_\scop X}.
\end{equation}
For $\sigma\in\{\pm 1\}$, the characteristic set is a 2-sphere over each point of $\pa_\scop X$,
\[
  \Char_\sigma = \{ (\xi-\sigma)^2 + |\eta|_{\slg^{-1}}^2 = \sigma^2 \} \subset {}^{\scop,\bop}T^*_{\pa_\scop X},\quad \sigma\in\{\pm 1\},
\]
and the rescaled Hamiltonian vector field over $\rho=0$,
\begin{equation}
\label{EqMiFredHam}
  \sfH_{G_\sigma} := \rho^{-1}H_{G_\sigma} = 2(\xi-\sigma)(\rho\pa_\rho+\eta\pa_\eta) - 2|\eta|_{\slg^{-1}}^2\pa_\xi + (\pa_\eta G_\sigma)\pa_\omega - (\pa_\omega G_\sigma)\pa_\eta,
\end{equation}
scaled by $\sigma$, flows from the source
\begin{equation}
\label{EqMiFredRin}
  \cR_{\rm in} := \{ \rho=0,\ \xi=2\sigma,\ \eta=0 \}
\end{equation}
to the sink $\cR_{\rm out}$. On the other hand, for complex $\sigma$ we have
\begin{equation}
\label{EqMiFredChar}
  \Char_\sigma = \cR_{\rm out},\quad \Im\sigma>0.
\end{equation}
Indeed, $0=\Im G_\sigma=-2(\Im\sigma)\xi=0$ implies $\xi=0$, and then $\Re G_\sigma=0$ gives $\eta=0$.

\begin{prop}[Fredholm property]
\label{PropMiFred}
  Under the assumptions of Theorem~\usref{ThmMiInv}, the map~\eqref{EqMiInv} is Fredholm for all $\sigma\in e^{i[0,\pi]}$ and of index $0$. Moreover, for fixed $h>0$ and $\sigma\in e^{i[0,\pi]}$, its nullspace is independent of the choice of $s,\ell,\beta$ subject to~\eqref{EqMiInvRange}.
\end{prop}
\begin{proof}
  This is a variation on a well-known theme; see, e.g., \cite{MelroseEuclideanSpectralTheory,VasyLowEnergyLag}, and also \cite[Lemma~3.13]{HintzConicWave} for a version that is closely related to the present discussion. One can also obtain this result from \cite[Theorem~9.38]{HintzNonstat2}. We shall thus be rather brief.

  \pfstep{Real $\sigma$.} We only consider the case $\sigma=1$; the case $\sigma=-1$ is completely analogous (upon switching the direction of propagation). We shall prove the estimate
  \begin{equation}
  \label{EqMiFredDir}
    \|\omega\|_{H_{\scop,\bop}^{s,(\ell,\beta)}} \leq C\Bigl( \| \wh{\Box_g^{\cC_h}}(\sigma)\omega \|_{H_{\scop,\bop}^{s-2,(\ell+1,\beta-2)}} + \|\omega\|_{H_{\scop,\bop}^{s_0,(\ell_0,\beta_0)}} \Bigr)
  \end{equation}
  for any fixed $s_0<s$, $\ell_0<\ell$, $\beta_0<\beta$ subject to the stated conditions, and in the strong form that if the right-hand side is finite, then so is the left-hand side. As a first step, we prove
  \begin{equation}
  \label{EqMiFredNob}
    \|\omega\|_{H_{\scop,\bop}^{s,(\ell,\beta)}} \leq C\Bigl( \| \wh{\Box_g^{\cC_h}}(\sigma)\omega \|_{H_{\scop,\bop}^{s-2,(\ell+1,\beta-2)}} + \|\omega\|_{H_{\scop,\bop}^{s_0,(\ell_0,\beta)}} \Bigr),
  \end{equation}
  which differs from~\eqref{EqMiFredDir} only in that the b-decay order of the final, error, term on the right is $\beta$, not $\beta_0<\beta$. The proof of~\eqref{EqMiFredNob} combines microlocal elliptic estimates on ${}^{\scop,\bop}T^*X\setminus\Char_\sigma$ with a radial source estimate (in the scattering calculus) at $\cR_{\rm in}$ and a radial sink estimate at $\cR_{\rm out}$, with real principal type propagation on $\Char_\sigma$ in between these two; the non-elliptic part of the argument is essentially the same as the proof of \cite[Proposition~5.28]{VasyMinicourse}, except we need to determine the relevant threshold quantities.

  Consider the radial sink estimate near $\cR_{\rm out}$. We use a commutant
  \begin{equation}
  \label{EqMiFredComm}
    a=\rho^{-\ell-\frac12}\chi_0(\rho)\chi_1(|\eta|_{\slg^{-1}}^2)\chi_2(\xi)
  \end{equation}
  where $\chi_0,\chi_1,\chi_2\in\CIc((-1,1))$ are cutoffs, equal to $1$ near $0$ and satisfying $(-\chi_j\chi_j')^{\frac12}\in\CI([0,\infty))$. Quantizing $a$ to an operator $A=A^*\in\Psisc^{-\infty,\ell+\frac12}$ on $\rho^{-1}([0,2))$ with Schwartz kernel supported in $\rho^{-1}([0,1])\times\rho^{-1}([0,1])$, we consider
  \begin{equation}
  \label{EqMiFredCommCalc}
    \Im \big\la\wh{\Box_g^{\cC_h}}(\sigma)\omega,A^2\omega\big\ra = \la \sC\omega,\omega\ra,\quad
    \sC = \frac{i}{2}\bigl[\wh{\Box_g^{\cC_h}}(\sigma),A^2\bigr] + \frac{1}{2 i}\Bigl( \wh{\Box_g^{\cC_h}}(\sigma) - \wh{\Box_g^{\cC_h}}(\sigma)^*\Bigr)A^2.
  \end{equation}
  Using~\eqref{EqMiFredHam}, the principal symbol of the first summand of $\sC$ is
  \begin{equation}
  \label{EqMiFredComp}
  \begin{split}
    a H_{G_\sigma}a &= \rho^{-2\ell} \Bigl( (2\ell+1)(\sigma-\xi)\chi_0^2\chi_1^2\chi_2^2 - 2(\sigma-\xi)\rho\chi_0\chi_0'\chi_1^2\chi_2^2 \\
      &\quad\hspace{5em} - 4(\sigma-\xi)|\eta|^2\chi_1\chi_1' \chi_0^2\chi_2^2 - 2|\eta|^2\chi_0^2\chi_1^2\chi_2\chi_2'\Bigr).
  \end{split}
  \end{equation}
  The first term is a negative square when $\ell<-\frac12$. The second term is the square of a smooth function as long as $\xi<\sigma$ on $\supp a$ (and in any case is supported away from $\rho=0$), likewise for the third term. The support of the final term has two pieces; the one in $\xi>0$ is a square, the one in $\xi<0$ is supported in the elliptic set of $\wh{\Box_g^{\cC_h}}(\sigma)$. The second term of $\sC$ adds a shift to the first term; at the zero section, the shift of $(2\ell+1)\sigma$ is given by
  \begin{equation}
  \label{EqMiFredZero}
    \sigmasc^1\Bigl(\frac{1}{2 i\rho}\bigl( \wh{\Box_g^{\cC_h}}(\sigma) - \wh{\Box_g^{\cC_h}}(\sigma)^*\bigr)\Bigr)\Big|_{\rho=0,\,\xi=\eta=0},
  \end{equation}
  which we proceed to evaluate using~\eqref{EqMiOpSc} and~\eqref{EqMiOpPiecesb}. Note first that~\eqref{EqMiOpPiecesb} and the principally scalar nature of $\wh{\Box_g^{\cC_h}}(0)$ imply $\wh{\Box_g^{\cC_h}}(0)-\wh{\Box_g^{\cC_h}}(0)^*\in\rho^2\Diffb^1(X;\cT^*_X)$; dividing this on the left by $\rho$ yields an element of $\rho\Diffb^1\subset\Diffsc^1$ whose scattering principal symbol is linear in $(\xi,\eta)$ and thus vanishes at $(\xi,\eta)=0$. Next, if we regard $2 i\sigma(\rho^2\pa_\rho-\rho)$ as a scalar operator, it is symmetric on $L^2$ with volume density $\rho^{-3}\frac{\dd\rho}{\rho}$, and therefore its imaginary part, as an operator on sections of $\cT^*_X$ and for any smooth choice of fiber inner product, is of class $\rho^2\CI$; division by $\rho$ and restriction to $\rho=0$ yields $0$. The only nontrivial contribution to~\eqref{EqMiFredZero} thus comes from the term $-2 i\sigma\rho h^{-1}\ubar S$. Now, $\ubar S$ is diagonalizable; let us thus choose as the fiber inner product on $\cT^*_X$ a positive definite one for which a basis of eigenvectors of $\ubar S$ is orthonormal. Then~\eqref{EqMiFredZero} equals $-2 h^{-1}\ubar S$, and thus
  \begin{equation}
  \label{EqMiFredSymb}
  \begin{split}
    \sigmasc^{-\infty,2\ell}(\sC) &= \rho^{-2\ell} \Bigl( \bigl((2\ell+1)(\sigma-\xi) - 2 h^{-1}\ubar S + \cO(|\xi|+|\eta|)\bigr)\chi_0^2\chi_1^2\chi_2^2 - 2(\sigma-\xi)\rho\chi_0\chi_0'\chi_1^2\chi_2^2 \\
      &\quad\hspace{5em} - 4(\sigma-\xi)|\eta|^2\chi_1\chi_1' \chi_0^2\chi_2^2 - 2|\eta|^2\chi_0^2\chi_1^2\chi_2\chi_2'\Bigr).
  \end{split}
  \end{equation}
  For $\sigma=1$, the first term is negative (and thus has the opposite sign of the other terms) at $(\xi,\eta)=(0,0)$, and thus in a small neighborhood, if
  \begin{equation}
  \label{EqMiFredThres}
    \ell<-\tfrac12+h^{-1}\min\spec\ubar S = -\tfrac12 + h^{-1}(1-e)(1-v).
  \end{equation}

  Turning to the radial source estimate at $\cR_{\rm in}$ (where $\xi=2\sigma$), we still have~\eqref{EqMiFredComp}, with the second to fourth terms now being negative squares on $\Char_\sigma$ if we take $\chi_2\in\CIc((\sigma,3\sigma))$ to be equal to $1$ near $2\sigma$ and to satisfy $(\chi_2\chi_2')^{\frac12}\in\CI((\sigma,2\sigma])$. We compute the shift of the numerical coefficient $-(2\ell+1)\sigma$, $\sigma=1$, of the first (main) term in~\eqref{EqMiFredComp} arising from~\eqref{EqMiFredZero} at $\rho=0$, $\xi=2\sigma$, $\eta=0$, or equivalently,
  \begin{equation}
  \label{EqMiFredZero2}
    \sigmasc^1\Bigl(\frac{1}{2 i\rho}\Bigl[e^{2 i\sigma r}\wh{\Box_g^{\cC_h}}(\sigma) e^{-2 i\sigma r} - \bigl(e^{2 i\sigma r}\wh{\Box_g^{\cC_h}}(\sigma) e^{-2 i\sigma r}\bigr)^* \Bigr]\Bigr)\Big|_{\rho=0,\,\xi=\eta=0},
  \end{equation}
  using symmetry considerations. Let us write $\Box_{\cd_v,h^{-1}}:=\Box_g^{\cC_h}$ where $\cd_v=r^{-1}(\dd t-v\,\dd r)$ is the 1-form~\eqref{EqMi1form}. Under the isometric involution $\phi\colon t\mapsto -t$ of $(M,g)$, the 1-form $\cd_v$ pulls back to $-\cd_{-v}$. Thus,
  \[
    \phi^* \Box_{\cd_v,h^{-1}} = \Box_{-\cd_{-v},h^{-1}} = \Box_{\cd_{-v},-h^{-1}}.
  \]
  Therefore, acting on stationary sections of $\cT^*$, we have
  \[
    e^{2 i\sigma r}\wh{\Box_g^{\cC_h}}(\sigma)e^{-2 i\sigma r} = e^{i\sigma(t+r)}\Box_{\cd_v,h^{-1}}e^{-i\sigma(t+r)} = \phi^*\bigl( e^{i(-\sigma)t_*}\Box_{\cd_{-v},-h^{-1}}e^{-i(-\sigma)t_*}\bigr).
  \]
  For a fiber inner product of $\cT^*_X$ for which an eigenbasis of $\ubar S_{-v}$ (given by $\ubar S$ in~\eqref{EqMiOpS} but with $v$ replaced by $-v$) is orthonormal, we thus conclude that~\eqref{EqMiFredZero2} is equal to $(-1)\cdot(-2(-h^{-1})\ubar S_{-v})$ (the factor $-1$ arising from the minus sign in front of $\sigma$). The threshold condition for the radial source estimate is thus
  \[
    -(2\ell+1) - 2 h^{-1}(1-e)(1+v) < 0,
  \]
  i.e., $\ell>-\frac12-h^{-1}(1-e)(1+v)$.

  A standard regularization argument in the radial point and real principal type propagation estimates gives~\eqref{EqMiFredNob} in the strong form. Having established~\eqref{EqMiFredNob}, with $s_0=s-3$, say, one obtains~\eqref{EqMiFredDir} in the usual fashion of b-analysis by estimating the localization of the error term in~\eqref{EqMiFredNob} to $r=0$ by means of the (inverse) Mellin transform. We only discuss the range of $\beta$ here. Making the volume density explicit in the notation, we have
  \[
    \Hb^{s,\beta}([0,1)_r;r^2\,|\dd r|) = \Hb^{s,-\frac32+\beta}\Bigl([0,1)_r;\Bigl|\frac{\dd r}{r}\Bigr|\Bigr).
  \]
  Therefore, as long as no indicial root of $\wh{\Box_g^{\cC_h}}(0)$ at $r=0$ has real part $-\frac32+\beta$, one can conclude~\eqref{EqMiFredDir}; and the strong form holds whenever the interval $[-\frac32+\beta_0,-\frac32+\beta]$ does not contain the real part of any indicial root. Since $\rho=r^{-1}$, the indicial roots $\lambda\in\C$ of $\wh{\Box_g^{\cC_h}}(0)$ are given by $\lambda=-\mu$ where $\mu$ ranges over all indicial roots described in Theorem~\ref{ThmM0}. Therefore, either $\Re\lambda\leq -1$ or $\Re\lambda>C_0$, giving the range $\beta\in(-1-(-\frac32),C_0-(-\frac32)]$ as in~\eqref{EqMiInvRange}.

  The proof of the estimate dual to~\eqref{EqMiFredDir},
  \begin{equation}
  \label{EqMiFredAdj}
    \|\omega\|_{H_{\scop,\bop}^{-s+2,(-\ell-1,-\beta+2)}} \leq C\Bigl(\|\wh{\Box_g^{\cC_h}}(\sigma)^*\omega\|_{H_{\scop,\bop}^{-s,(-\ell,-\beta)}} + \|\omega\|_{H_{\scop,\bop}^{s_1,(\ell_1,\beta_1)}}\Bigr)
  \end{equation}
  where $s_1<-s+2$, $\ell_1<-\ell-1$, $\beta_1<-\beta+2$, is analogous (via propagation in the opposite direction) and thus omitted.

  \pfstep{Complex $\sigma$.} For $\sigma\in e^{i(0,\pi)}$, we only need to prove a radial point type estimate at the characteristic set $\Char_\sigma$, cf.\ \eqref{EqMiFredChar}. Everywhere else the operator~\eqref{EqMiOpscb} is microlocally elliptic, and the b-normal operator argument at $r=0$ is unchanged; thus we again have the estimate~\eqref{EqMiFredDir}. To prove the radial point type estimate, we follow the discussion prior to \cite[Lemma~4.7]{VasyLowEnergyLag} and consider the scalar multiple
  \[
    \sigma^{-1}\wh{\Box_g^{\cC_h}}(\sigma) = 2 i(\rho^2\pa_\rho - \rho - \rho h^{-1}\ubar S) + (\Re\sigma)\wh{\Box_g^{\cC_h}}(0) - i (\Im\sigma)\wh{\Box_g^{\cC_h}}(0).
  \]
  We use here that $\sigma^{-1}=\bar\sigma=\Re\sigma-i\Im\sigma$. We use the same fiber inner product as in~\eqref{EqMiFredSymb} and write this as
  \begin{equation}
  \label{EqMiInvPQ}
    \sigma^{-1}\wh{\Box_g^{\cC_h}}(\sigma) = P - i Q,\quad P=\sum_{j=1}^3 P_j,\ Q=\sum_{j=0}^3 Q_j,
  \end{equation}
  where
  \begin{alignat*}{4}
    &&&& Q_0&:=2\rho h^{-1}\ubar S, \\
    P_1&:=\Re\bigl( 2 i(\rho^2\pa_\rho-\rho)\bigr), &&&\quad
    Q_1&:= -\Im\bigl(2 i(\rho^2\pa_\rho-\rho)\bigr) &&\in \rho^2\CI, \\
    P_2&:=(\Re\sigma)\Re\wh{\Box_g^{\cC_h}}(0)&&\in\rho^2\Diffb^2, &\quad
    Q_2&:=-(\Re\sigma)\Im\wh{\Box_g^{\cC_h}}(0) &&\in \rho^2\Diffb^1, \\
    P_3&:= (\Im\sigma)\Im\wh{\Box_g^{\cC_h}}(0) &&\in \rho^2\Diffb^1, &\quad
    Q_3&:=(\Im\sigma)\Re\wh{\Box_g^{\cC_h}}(0) && \in \rho^2\Diffb^2.
  \end{alignat*}
  Note that $Q_3\in\rho^2\Diffb^2\subset\Diffsc^2$ is a principal term (with nonnegative principal symbol); it acts as complex absorption. For the commutant $A\in\Psisc^{-\infty,\ell+\frac12}$ given by quantizing the symbol $a=\rho^{-\ell-\frac12}\chi_0\chi_1\chi_2$ from~\eqref{EqMiFredComm}, we then consider, analogously to~\eqref{EqMiFredCommCalc}, the operator
  \begin{align*}
    \sC &= \frac{i}{2}\bigl[\sigma^{-1}\wh{\Box_g^{\cC_h}}(\sigma),A^2\bigr] + \frac{1}{2 i}\Bigl( \sigma^{-1}\wh{\Box_g^{\cC_h}}(\sigma) - \Bigl(\sigma^{-1}\wh{\Box_g^{\cC_h}}(\sigma)\Bigr)^*\Bigr)A^2 \\
      &= \frac{i}{2}[P,A^2] - A Q A - \frac12 [A,[A,Q]].
  \end{align*}
  The contributions to $\sC$ from all terms in~\eqref{EqMiInvPQ} except for $Q_3$ lie in $\Psisc^{-\infty,2\ell}$. Consider first the term $\frac{i}{2}[P,A^2]$. Since the scattering symbol of $P_2$ vanishes quadratically at the zero section, and since $P_3\in\rho\Diffsc^1$ is subprincipal (relative to $\Diffsc^2$), only the contribution from $P_1$ is elliptic at the zero section; in view of $\sigmasc^1(P_1)=-2\xi$, the principal symbol of $\sC$ there is $\rho^{-\ell-\frac12}\cdot(-2\rho^2\pa_\rho(\rho^{-\ell-\frac12}))=(2\ell+1)\rho^{-2\ell}$. Turning to $-A Q A$, the term $Q_0$ contributes $-2\rho^{-2\ell}h^{-1}\ubar S$ (for a total of $(2\ell+1-2 h^{-1}\ubar S)\rho^{-2\ell}$ at the zero section, cf.\ the first term in~\eqref{EqMiFredSymb} and the resulting threshold condition~\eqref{EqMiFredThres} under which this main term is negative), whereas the terms arising from $Q_1,Q_2$ are subprincipal. The double commutator $[A,[A,Q]]$ lies in $\Psisc^{-\infty,2\ell-1}$ and is thus also subprincipal. The only remaining term is $-A Q_3 A$. Since the b-principal symbol of the operator $\wh{\Box_g^{\cC_h}}(0)$ is $\xi_\bop^2+|\eta_\bop|_{\slg^{-1}}^2$ by inspection of~\eqref{EqMiOpB}, we can express $\wh{\Box_g^{\cC_h}}(0)$ as a finite sum
  \[
    \wh{\Box_g^{\cC_h}}(0) = \sum_j T_j^*T_j + R,\quad T_j\in\rho\Diffb^1,\ R\in\rho^2\Diffb^1.
  \]
  But then the $L^2$-pairing $\la -A T_j^*T_j A\omega,\omega\ra=-\|T_j A\omega\|^2$ is negative (and thus has the same sign as the main term), while the scattering principal symbol of $-A R A\in\Psisc^{-\infty,2\ell}$ vanishes at the zero section and thus does not affect the threshold condition. Lastly, when the Hamiltonian vector field of $P$ falls on the cutoffs $\chi_0,\chi_1,\chi_2$ in $a$, one obtains terms supported away from the zero section and thus in the elliptic set of $\wh{\Box_g^{\cC_h}}(\sigma)$. Equipped with this information about $\sC$, one can now prove microlocal $H_\scop^{s,\ell}$-estimates for $\omega$ in terms of $\wh{\Box_g^{\cC_h}}(\sigma)\omega\in H_\scop^{s-2,\ell+1}$.

  As mentioned before, one thus obtains~\eqref{EqMiFredDir}; and the adjoint estimate~\eqref{EqMiFredAdj} is obtained in a completely analogous fashion.

  \pfstep{Uniformity in $\sigma$.} An inspection of these arguments shows that for $s,\ell,\beta$ that are admissible for $\sigma=\pm 1$, the estimate~\eqref{EqMiFredDir} holds uniformly for $\sigma\in e^{i[0,\pi]}$ close to $\pm 1$ (or indeed for $\sigma\in e^{i[0,\pi-\delta]}$, resp.\ $\sigma\in e^{i[\delta,\pi]}$ for any fixed $\delta>0$). This is an instance of the limiting absorption principle. The proof, as discussed in \cite[\S{9.2.2}]{HintzNonstat2} (following \cite[\S{5}]{VasyMinicourse} and \cite[\S{4}]{VasyLAPLag}), requires computing the symbolic contribution arising from the full Hamiltonian vector field, including that of $P_2$, acting on the cutoff functions of the commutant~\eqref{EqMiFredComm}; but this has the same sign structure as discussed after~\eqref{EqMiFredComp} for $\sigma=1$.

  \pfstep{Index $0$.} We will argue using a relative index formula and a deformation argument. Upon separating into spherical harmonics, the operator $\wh{\Box_g^{\cC_h}}(\sigma)$ decomposes into a direct sum of ordinary differential operators (acting on sections of $\cT^*_X$, i.e., $\C^4$-valued functions) on $(0,\infty)_r$ of b- (i.e., regular singular) type at $r=0$. Consider the operator describing the action on, say, scalar type $l$ 1-forms where $l\geq 1$ is fixed, and write $I$ for its Fredholm index on $H_{\scop,\bop}^{s,(\ell,\beta)}$ where we take $\ell$ to be a variable order function which is $>-\frac12$ at $\cR_{\rm in}$ and $<-\frac12$ at $\cR_{\rm out}$. If we shift the weight $\beta$ from its original value $\beta_0\in(\frac12,\frac32+C_0]$ to a large negative number $\beta_1$ (at this point, any number smaller than $1-\Re\lambda$ works, where $\lambda$ ranges over all scalar type $l$ indicial roots), the index of this operator increases by $3$ as a consequence of Theorem~\ref{ThmM0}\eqref{ItM0Number}. (This is a special case of Melrose's relative index theorem, see~\cite[Theorem~6.5]{MelroseAPS}.) We now shift $h^{-1}$ from its initial value to $0$ while keeping $\beta_1$ below the scalar type $l$ indicial roots of $\wh{\Box_g^{\cC_h}}(\sigma)$ (which depend continuously on $h^{-1}\in[0,\infty)$, as follows immediately from the expression~\eqref{EqMiOpB}); then the index of $\wh{\Box_g^{\cC_\infty}}(\sigma)=\wh{\Box_g}(\sigma)$, restricted to scalar type $l$ inputs, is still equal to $I+3$; here $\Box_g$ is the tensor wave operator on 1-forms. But the scalar type $l$ indicial roots of this operator at $r=0$ are given by $\lambda=-l-2,-l-1,-l$ and $l-1,l,l+1$, as follows from a computation using the form of $\Box_{\ubar g}$ in Lemma~\ref{LemmaM0Ops}. Shifting $\beta_1$ up to a value $\beta_2$ with $\beta_2-\frac32\in(-l,l-1)$, the index of $\wh{\Box_g}(\sigma)$ on the scalar type $l$ subspace of $H_{\scop,\bop}^{s,(\ell,\beta)}$ goes down by $3$ and is thus given by $(I+3)-3=I$ when $\beta-\frac32\in(-l,l-1)$. But $\wh{\Box_g}(\sigma)$ is invertible as a map~\eqref{EqMiInv} (with $\Box_g$ in place of $\Box_g^{\cC_h}$); see, e.g., \cite[Proposition~9.40]{HintzNonstat2} and \cite[Lemma~4.1]{HintzConicWave}. Therefore, $I=0$.
\end{proof}

\subsection{Solvability of the constraint propagation wave operator on spacetime}
\label{SsMiBox}

Acting on the bundle $\cT^*$ (see~\eqref{EqMiBundle}), the operator $P_h=h^2\Box_g^{\cC_h}$ is schematically of the form $h^2 D^2+h r^{-1}D+h r^{-2}$ where $D=D_{t,x}$. Due to the singularity at $r=0$, even local existence and uniqueness are delicate; we will adapt arguments from \cite{HintzConicWave} to our setting to study $\Box_g^{\cC_h}$. We can desingularize the coefficients of $P_h$ at $r=0$ by passing to $r^2 P_h$, which schematically reads $(h r D)^2+h r D+h$. In $r\gtrsim|t|$, this is a semiclassical 3b-operator on the compactified Kerr spacetime manifold used in the bulk of the paper. Now we need to consider this as an operator on $M$ defined in~\eqref{EqMiMfd}.

\begin{definition}[edge-b-vector fields]
\label{DefMiBoxeb}
  We write $\Veb(M)$ for the space of smooth vector fields on $M$ which, over $\CI(M)$, are spanned by $r\pa_t$, $r\pa_x$.
\end{definition}

The terminology arises as follows. Denote the boundary hypersurfaces of $M$ by
\[
  \sface := \upbeta^*\pa\ol{\R^4},\quad
  \cface := \upbeta^*\ol{\{r=0\}}.
\]
See Figure~\ref{FigMiMfd}. We shall write $\rho_\sface,\rho_\cface\in\CI(M)$ for defining functions of $\sface,\cface$. Then $\Veb(M)$ consists of all smooth vector fields on $M$ which are tangent to $\sface$ (hence ``b'') and tangent to the fibers of $\Sph^2-\cface\to\ol{\{r=0\}}$ (hence ``edge,'' following Mazzeo \cite{MazzeoEdge}). Let us check this in local coordinates
\begin{equation}
\label{EqMiBoxRhoCoord}
  \rho_\sface=\frac{1}{t},\quad
  \rho_\cface=\frac{r}{t},
\end{equation}
and $\omega\in\Sph^2$ near the future corner $\sface\cap\cface\cap t^{-1}(\infty)$: there $r\pa_t=-\rho_\cface(\rho_\cface\pa_{\rho_\cface}+\rho_\sface\pa_{\rho_\sface})$ and $r\pa_r=\rho_\cface\pa_{\rho_\cface}$, whose $\CI([0,1)_{\rho_\sface}\times[0,1)_{\rho_\cface})$-span is equal to that of $\rho_\cface\pa_{\rho_\cface}$ and $\rho_\cface\rho_\sface\pa_{\rho_\sface}$. (The close relationship between the 3b- and edge-b-algebras is discussed around \cite[Equation~(1.13), \S{3.3}]{Hintz3b}.) Analogously to Definition~\ref{DefK3bConormal} and~\S\ref{SssK3bBdd}, we can realize the space $\cA_\ebop\Diffeb^m(M)$ (finite linear combinations of edge-b-differential operators with edge-b-regular coefficients) as bounded geometry operators on $M^\circ=\R^4\setminus\{r=0\}$, now for the family of unit cells
\begin{align*}
  U_{j,k,\beta} &:= (2^j(k-2),2^j(k+2))_t \times (2^{j-2},2^{j+2})_r \times \slU_\beta, \\
  \phi_{j,k,\beta}(t,r,\omega) &= \bigl( 2^{-j}t-k, \log_2(r)-j,\ \slvarphi_\beta(\omega)\bigr),
\end{align*}
where $\slU_\beta$ is a finite cover of $\Sph^2$ by charts $\slvarphi_\beta\colon\slU_\beta\to(-2,2)^2$. With these unit cells, the parameter space $(0,1)_h$, and the scaling $h$ on each chart, we obtain a parameterized scaled bounded geometry structure on $M^\circ$. The corresponding operators are semiclassical edge-b-(pseudo)differential operators with uniformly edge-b-regular coefficients.

\begin{figure}[!ht]
\centering
\includegraphics{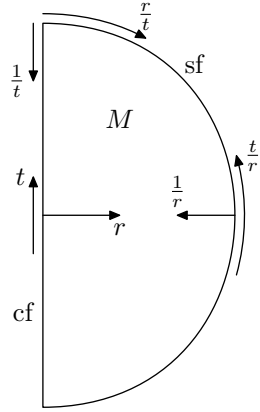}
\caption{The manifold $M=[\ol{\R^4};\ol{\{r=0\}}]$, shown here after quotienting out by $\Sph^2$, and some local coordinates.}
\label{FigMiMfd}
\end{figure}

The operator
\[
  P_h = h^2\Box_g^{\cC_h} = \Box_h - i L_h \in r^{-2}\Diff_{\ebop,\semi}^2(M;\cT^*),\quad \Box_h=h^2\Box_g,\ L_h=h L_{g,\cd,e},
\]
where $L_{g,\cd,e}=L_{g,\cd,e,1}$ in the notation of~\eqref{EqCOp}, belongs to this class, and after factoring out $r^{-2}$ in fact has smooth coefficients on $[0,1)\times M$ as a semiclassical edge-b-operator, i.e., when expressed in terms of $r\pa_t,r\pa_x$; its principal symbol is equal to
\[
  \sigmaebh^2(P_h) = G - i\ell,\quad \ell:=\ell_{g,\cd,e},
\]
by Lemma~\ref{LemmaCSymb}.

The ability to work with edge-b-regular coefficients will be important for our analysis: following \cite[\S{3.3}]{HintzConicWave}, we will use functions $\chi(t/r)$, $\chi\in\CI(\R)$, which are edge-b-regular but not smooth on $M$, as cutoffs near the initial hypersurface
\[
  X := \ol{\{t=0\}} \subset M
\]
of the domain
\[
  \Omega := \ol{\{t\geq 0\}} \subset M
\]
on which we will solve forcing problems for $\Box_g^{\cC_h}$. In the analysis of $P_h$, we shall use weighted semiclassical edge-b-Sobolev spaces
\[
  H_{\ebop,h}^{s,(\alpha_\sface,\alpha_\cface)}(M) = \rho_\sface^{\alpha_\sface}\rho_\cface^{\alpha_\cface}H_{\ebop,h}^s(M),\quad H_{\ebop,h}^{s,(\alpha_\sface,\alpha_\cface)}(\Omega)^{\bullet,-},
\]
with the underlying $L^2$-space defined using the density $|\dd g|=r^3|\dd t\,\frac{\dd r}{r}\,\dd\slg|$. Non-semiclassical edge-b-Sobolev spaces are denoted $H_\ebop^{s,(\alpha_\sface,\alpha_\cface)}(M)$.

\begin{thm}[Solvability and uniqueness]
\label{ThmMiBox}
  Fix $v\in(0,1)$, $C_0>0$. Let $\alpha_\sface,\alpha_\cface\in\R$ be such that
  \[
    \alpha_\sface < -\frac12,\quad
    \alpha_\cface > \frac12.
  \]
  Let $\cd=r^{-1}(\dd t-v\,\dd r)$. Then there exists $e_0>0$ such that for all $e\in(0,e_0)$, there exists $h_0>0$ such that the following holds for all $h\in(0,h_0)$.
  \begin{enumerate}
  \item\label{ItMiBoxFwd}{\rm (Forward problem for $\Box_g^{\cC_h}$.)} For all $s\in[-C_0,C_0]$, the forcing problem\footnote{We use the standard convention from \cite[Appendix~B]{HormanderAnalysisPDE3} that the dot in $\dot H_\ebop$ indicates that elements of this space are distributions with support in $\Omega$.}
    \begin{equation}
    \label{EqMiBoxFwd}
      \Box_g^{\cC_h}\omega = f \in \dot H_\ebop^{s-1,(\alpha_\sface+2,\alpha_\cface-2)}(\Omega;\cT^*)
    \end{equation}
    admits a solution $\omega\in\dot H_\ebop^{s,(\alpha_\sface,\alpha_\cface)}(\Omega;\cT^*)$. This solution is unique in the following sense: if $\omega'$ solves $\Box_g^{\cC_h}\omega'=f$ and satisfies
      \begin{equation}
      \label{EqMiBoxFwdp}
        \omega'|_{\Omega_{\leq t_0}} \in H_\ebop^{-N,(-N,\alpha_\cface)}(\Omega_{\leq t_0})^{\bullet,-},\quad \Omega_{\leq t_0}:=\Omega\cap\ol{\{t\leq t_0\}},
      \end{equation}
      for all $t_0$ and some $N\leq C_0$, then $\omega'=\omega$.
  \item\label{ItMiBoxIVP}{\rm (Initial value problem.)} Making the volume densities explicit in the notation, the initial value problem
    \begin{equation}
    \label{EqMiBoxIVP}
      \left\{
      \begin{alignedat}{2}
        \Box_g^{\cC_h}\omega&=f\in\bar H_\ebop^{0,(\alpha_\sface+2,\alpha_\cface-2)}(\Omega,|\dd g|;\cT^*), \\
        (\omega,r\cL_{\pa_t}\omega)|_X &\in \rho_\sface^{\alpha_\sface+\frac12}\rho_\cface^{\alpha_\cface-\frac12}\Hbext^1(X,r^2\,|\dd r\,\dd\slg|;\cT^*_X) \oplus \rho_\sface^{\alpha_\sface+\frac12}\rho_\cface^{\alpha_\cface-\frac12}\Hbext^0(X,r^2\,|\dd r\,\dd\slg|;\cT^*_X)
      \end{alignedat}
      \right.
    \end{equation}
    admits a solution $\omega\in\bar H_\ebop^{1,(\alpha_\sface,\alpha_\cface)}(\Omega;\cT^*)$. This solution is unique among all $\omega'$ solving~\eqref{EqMiBoxIVP} whose restrictions to $\Omega_{\leq t_0}$ lie in $\bar H_\ebop^{1,(-N,\alpha_\cface)}(\Omega_{\leq t_0})$ for all $t_0$ and some $N$.
  \end{enumerate}
\end{thm}
\begin{proof}
  The strategy of the proof, and most of the relevant (symbolic) computations, are very similar to those used in the proof of Theorem~\ref{ThmET}. Thus, we shall be rather brief.

  \pfstep{Energy estimate near $X$.} Set
  \[
    \tau=\frac{t}{r};
  \]
  then $\dd\tau=r^{-1}(\dd t-\tau\,\dd r)$ is past timelike for $|\tau|<1$. Moreover, the metric
  \[
    g = r^2\Bigl( -\dd\tau^2 - 2\tau\,\dd\tau\,\frac{\dd r}{r} + (1-\tau^2)\frac{\dd r^2}{r^2} + \slg\Bigr)
  \]
  is a weighted (by $r^2$) b-metric on $(-1,1)_\tau\times[0,\infty]_r\times\Sph^2$. Geometrically, working with $\tau$ instead of $t$ amounts to working on the blow-up
  \[
    M_1 := [M; \{t=r=0\}],\quad \upbeta_1\colon M_1\to M,
  \]
  with $\tau$ being an affine coordinate along the front face $\ol{\R_\tau}\times\Sph^2$ (cf.\ \cite[Remark~3.8]{HintzConicWave}). See Figure~\ref{FigMiM1}.

  \begin{figure}[!ht]
  \centering
  \includegraphics{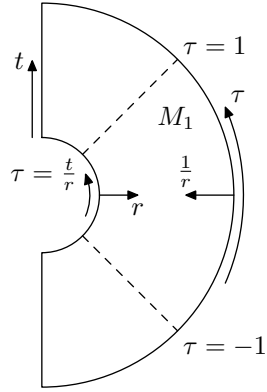}
  \caption{The manifold $M_1$ (quotiented by $\Sph^2$), which is convenient for the purpose of localizing in $\tau=\frac{t}{r}$, and some local coordinates.}
  \label{FigMiM1}
  \end{figure}

  Moreover, b-Sobolev norms on $I_\tau\times[0,\infty]_r\times\Sph^2$ for intervals $I\subset\R$ are the same as edge-b-Sobolev norms on $\{(t,r,\omega)\colon\frac{t}{r}\in I\}$. A purely notational modification of the proof of Proposition~\ref{PropEI} thus yields, for any fixed $T\in(0,1)$, the (forward) energy estimate
  \begin{equation}
  \label{EqMiBoxEnFwd}
    \|\omega\|_{H_{\ebop,h}^{1,(\alpha_\sface,\alpha_\cface)}(\tau^{-1}([0,T]))} \leq C h^{-1}\|P_h\omega\|_{H_{\ebop,h}^{0,(\alpha_\sface+2,\alpha_\cface-2)}(\tau^{-1}([0,T]))}
  \end{equation}
  for all $\omega$ vanishing in $\tau<0$. (In this estimate, we can take the local defining functions $\rho_\sface=\frac{1}{1+r}$ and $\rho_\cface=\frac{r}{1+r}$ to define the weights here.) The backward energy estimate for the adjoint $P_h^*$ takes the form
  \begin{equation}
  \label{EqMiBoxEnBwd}
    \|\omega\|_{H_{\ebop,h}^{1,(\alpha_\sface,\alpha_\cface)}(\tau^{-1}([0,T/2]))} \leq C\Bigl( h^{-1}\|P_h^*\omega\|_{H_{\ebop,h}^{0,(\alpha_\sface+2,\alpha_\cface-2)}(\tau^{-1}([0,T]))} + \|\omega\|_{H_{\ebop,h}^{1,(\alpha_\sface,\alpha_\cface)}(\tau^{-1}([T/2,T]))}\Bigr)
  \end{equation}
  and is valid for all $\omega$. We shall henceforth fix
  \[
    T := \frac12.
  \]

  \pfstep{Estimates at nonzero and infinite semiclassical frequencies.} Denote by $\chi_j=\chi_j(\tau)$, $0\leq j\leq 4$, cutoff functions such that
  \begin{equation}
  \label{EqMiBoxCutoffs}
    \chi_j=0\ \text{for}\ \tau\leq\frac{j+\frac14}{10}T,\quad
    \chi_j=1\ \text{for}\ \tau\geq\frac{j+\frac34}{10}T,
  \end{equation}
  and such that $\chi_j'\geq 0$ and $(\chi_j\chi_j')^{\frac12}\in\CI$. Let $B_0,B_1\in\Psi_{\ebop,\semi}^0(M)$ (i.e., ps.d.o.s associated with the parameterized scaled bounded geometry structure giving rise to semiclassical edge-b-ps.d.o.\ with uniformly edge-b-regular coefficients) be elliptic at the intersection of $\upbeta_1^*(\ol{\Teb^*}M)|_{\supp\chi_0}$ with the characteristic set of $P_h$ at fiber infinity, and such that $\WF_{\ebop,\semi}'(B_0)\cap o=\emptyset$ and $\WF_{\ebop,\semi}'(B_0)\subset\Ell_{\ebop,\semi}(B_1)$. (We use here that the radial compactification $\upbeta_1^*(\ol{\Teb^*}M)$ of $\upbeta_1^*(\Teb^*M)\to M_1$ is an admissible compactification for the edge-b bounded geometry setting in the sense of \cite[Definition~3.62]{HintzScaledBddGeo}, and its product with $[0,1)_h$ is an admissible compactification for the semiclassical edge-b parameterized scaled bounded geometry setting.) We then have the following analogue of Proposition~\ref{PropEEInfty}: for all $s,N\in\R$ there exist $h_0>0$ and $C<\infty$ such that
  \begin{align}
  \label{EqMiBoxInfty}
  \begin{split}
    &\|B_0\chi_2\omega\|_{H_{\ebop,h}^{s,(\alpha_\sface,\alpha_\cface)}} \leq C\Bigl( \|B_1\chi_1 P_h\omega\|_{H_{\ebop,h}^{s-1,(\alpha_\sface+2,\alpha_\cface-2)}} + h^{\frac12}\|B_1\chi_0(1-\chi_2)\omega\|_{H_{\ebop,h}^{s,(\alpha_\sface,\alpha_\cface)}} \\
      &\quad\hspace{11em} + h^N\|\chi_0\omega\|_{H_{\ebop,h}^{-N,(\alpha_\sface,\alpha_\cface)}} \Bigr),
    \end{split} \\
    &\|B_0\chi_2\omega\|_{H_{\ebop,h}^{-s+1,(-\alpha_\sface-2,-\alpha_\cface+2)}} \leq C\Bigl( \|B_1\chi_1 P_h^*\omega\|_{H_{\ebop,h}^{-s,(-\alpha_\sface,-\alpha_\cface)}} + h^N\|\chi_0\omega\|_{H_{\ebop,h}^{-N,(-\alpha_\sface-2,-\alpha_\cface+2)}}\Bigr) \nonumber
  \end{align}
  for all $h\in(0,h_0)$. The proof is completely analogous to that of Proposition~\ref{PropEEInfty}. With $\hat\rho\in\CI(\ol{\Teb^*}M)$ being a defining function of fiber infinity, one uses the commutant
  \[
    a := \rho_\sface^{-\alpha_\sface-1}\rho_\cface^{-\alpha_\cface+1}\hat\rho^{-s+\frac12}\chi^\pm\chi(\digamma\hat\rho)\chi(\digamma\hat G)\chi_1(\tau),\quad \hat G := \hat\rho^2 r^2 G \in \CI(\ol{\Teb^*}M),
  \]
  where $\chi\in\CIc(\R)$ equals $1$ near $0$, $\chi^\pm\in\CI(\ol{\Teb^*}M)$ equals $1$, resp.\ $0$ in a neighborhood of $\pa\Sigma^\pm$ (the future and past component of the characteristic set of $P_h$ at fiber infinity), resp.\ $\pa\Sigma^\mp$, and $\digamma>1$ is large. Upon quantizing $a$ to an operator $A=A^*\in\rho_\sface^{-\alpha_\sface-1}\rho_\cface^{-\alpha_\cface+1}\Psi_{\ebop,\semi}^{s-\frac12}(M)$, one considers the $L^2$-pairing~\eqref{EqEEInftyPair}. The arguments following~\eqref{EqEEInftyPair} then apply \emph{mutatis mutandis}. The main change is that instead of~\eqref{EqEEInftyComm2}, we now write
  \[
    \mp\frac{i}{2}[P_h,A^2] = -h E_\tau^*E_\tau + h A\tilde B_2 A + h A Q P_h + h F'
  \]
  (i.e., there is no term corresponding to $E_r$ in~\eqref{EqEEInftyComm2}), where
  \begin{alignat*}{2}
    E_\tau &\in \rho_\sface^{-\alpha_\sface}\rho_\cface^{-\alpha_\cface}\Psi_{\ebop,\semi}^s, &\quad
    \tilde B_2 &\in \rho_\sface^2\rho_\cface^{-2}\Psi_{\ebop,\semi}^1, \\
    Q &\in \rho_\sface^{-\alpha_\sface-1}\rho_\cface^{-\alpha_\cface+1}\Psi_{\ebop,\semi}^{s-\frac32}, &\quad
    F' &\in \rho_\sface^{-2\alpha_\sface}\rho_\cface^{-2\alpha_\cface}\Psi_{\ebop,\semi}^{2 s-1}.
  \end{alignat*}
  This can be arranged by symbolic considerations much as in the penultimate paragraph of the proof of Proposition~\ref{PropEEInfty}, except now there is no cutoff in $r$. Also in the present setting, one uses that $\mp\chi_1 H_G\chi_1=\mp\chi_1\chi_1' H_G\tau$ is a negative square on $\supp a$ (which gives rise to the term $-h E_\tau^*E_\tau$) since $\pm H_G\tau>0$ at future (``$+$''), resp.\ past (``$-$'') lightlike covectors.

  We point out that the term $B_1\chi_0(1-\chi_2)\omega$ in~\eqref{EqMiBoxInfty} can be controlled by the energy estimate~\eqref{EqMiBoxEnFwd}.

  \pfstep{Zero section propagation near the source over $\cface$.} This is the only place which has no direct analogue in the Kerr setting. Note that $\cface$ is a critical set (indeed, a normal source) for $-r^2\cd^\sharp=r\pa_t+v r\pa_r$. We shall thus prove a radial source estimate there. Analogously to~\eqref{EqEKcL}, it is convenient to study the operator
  \[
    \cL_h := -i(1+i C L_h^*r^2)P_h = -L_h + J - i Q \in \rho_\cface^{-2}\rho_\sface^2\Diff_{\ebop,\semi}^3,
  \]
  where $J=J^*$ and $Q=Q^*$ are given by~\eqref{EqEKDecomp} with $r^{-1}$ in place of $\rho_\sface$, except memberships in $\rho_\sface^2\Difftb^k$ are replaced by memberships in $\rho_\sface^2\rho_\cface^{-2}\Diffeb^k$. Working in $t\geq-\frac{r}{2}$, we fix the local boundary defining functions
  \[
    \rho_\sface := \frac{1}{t+r+1},\quad
    \rho_\cface := \frac{r}{t+r+1}.
  \]
  (In particular, $r=\rho_\cface\rho_\sface^{-1}$.) Keeping in mind the need for sharp localization in $\frac{t}{r}$, we thus consider the commutant
  \[
    a = w^{-1}\chi_\cface(\digamma\rho_\cface)\chi_0(\digamma|\zeta|)\chi_1(\tfrac{t}{r}),\quad w:=\rho_\sface^{\alpha_\sface+1}\rho_\cface^{\alpha_\cface-1},
  \]
  where we write $\zeta\in\Teb^*M$ for the argument of $a$; here $\chi_\cface,\chi_0\in\CIc([0,3))$ are equal to $1$ near $[0,1]$ and satisfy $\chi_\cface',\chi_0'\leq 0$ and $(-\chi_\cface\chi_\cface')^{\frac12}\in\CI([0,3))$; and $\chi_1$ was fixed in~\eqref{EqMiBoxCutoffs}. Upon quantizing $a$ to an operator $A=A^*\in\rho_\sface^{-\alpha_\sface-1}\rho_\cface^{-\alpha_\cface+1}\Psi_{\ebop,\semi}$, one then computes the $L^2$-pairing~\eqref{EqEKPair} (with the fiber inner product on $\cT^*$ given by Lemma~\ref{LemmaCInner}). We only discuss the analogue of the symbolic computation~\eqref{EqEKSymbol} (or~\eqref{EqEBSymbol} which is more similar to the present setting), where now $\ell,j\in r^{-2}\CI(\Teb^*M)$: with $S\in r^{-2}\CI(\Teb^*M)$ denoting the skew-adjoint part of $h^{-1}L_h$ (see Lemma~\ref{LemmaCSubpr}), we can write
  \begin{equation}
  \label{EqMiBoxCommcf}
    a H_{-\ell+j}a - S a^2 = -c_0 r^{-2} a^2 - b_0^2 - b_\cface^2 + e_\tau^2 + f,
  \end{equation}
  where, for sufficiently small $c_0>0$, we define the $\End(\cT^*)$-valued symbols
  \begin{align*}
    b_0 &:= \rho_\sface^{-\alpha_\sface}\rho_\cface^{-\alpha_\cface}\chi_\cface\chi_0\chi_1\Bigl( r^2\Bigl[-\ell\Bigl(\frac{\dd w}{w}\Bigr) + S + w^{-1}H_j w\Bigr] - c_0\Bigr)^{\frac12}, \\
    b_\cface &:= \rho_\sface^{-\alpha_\sface}\rho_\cface^{-\alpha_\cface}\chi_0\chi_1\bigl(-\digamma\chi_\cface\chi_\cface' r^2\bigl[-\ell(\dd\rho_\cface)+H_j\rho_\cface\bigr]\bigr)^{\frac12}, \\
    e_\tau &:= \rho_\sface^{-\alpha_\sface}\rho_\cface^{-\alpha_\cface}\chi_\cface\chi_0\Bigl( \chi_1\chi_1'r^2\Bigl[ -\ell\Bigl(\dd\frac{t}{r}\Bigr) + H_j\frac{t}{r}\Bigr]\Bigr)^{\frac12}, \\
    f &:= \rho_\sface^{-2\alpha_\sface}\rho_\cface^{-2\alpha_\cface}\chi_\cface^2\chi_1^2\chi_0 r^2 H_{-\ell+j}\chi_0.
  \end{align*}
  In order to analyze these symbols, we compute
  \begin{align*}
    \frac{\dd\rho_\cface}{\rho_\cface} &= \frac{\dd r}{r} - \rho_\cface\Bigl(\frac{\dd t}{r}+\frac{\dd r}{r}\Bigr) \equiv \frac{\dd r}{r} \bmod \rho_\cface\CI(M;\Teb^*M), \\
    \frac{\dd w}{w} &= (\alpha_\cface-1)\frac{\dd r}{r} - \rho_\cface(\alpha_\sface+\alpha_\cface)\Bigl(\frac{\dd t}{r}+\frac{\dd r}{r}\Bigr) \equiv (\alpha_\cface-1)\frac{\dd r}{r}\bmod\rho_\cface\CI(M;\Teb^*M).
  \end{align*}
  Then:
  \begin{itemize}
  \item $r^2(-\ell(\frac{\dd w}{w})+S)$ is positive definite at $\rho_\cface=0$ when $e=0$ provided $\alpha_\cface-1>-\frac12$, i.e., $\alpha_\cface>\frac12$, as follows from~\eqref{EqMIbdyEll} (with $\beta_\cK-\beta_\sface$ replaced by $\alpha_\cface-1$),~\eqref{EqMIbdyS}, and~\eqref{EqMIbdyEval}. The same is then true for all sufficiently small $e$ and on $\supp a$ when $\digamma$ is sufficiently large (thus $\supp a$ is localized in a sufficiently small neighborhood of $\cface$). Choosing $\digamma$ even larger so that the contribution of $w^{-1}H_j w$ (which vanishes at the zero section) to $b_0$ is dominated by $r^2(-\ell(\frac{\dd w}{w})+S)$, the symbol $b_0$ is well-defined and smooth.
  \item Again using the computation~\eqref{EqMIbdyEll},
    \[
      -\frac12 r^2\ell\Bigl(\frac{\dd r}{r}\Bigr) = \begin{pmatrix} v & -1 & 0 \\ -e & (2-e)v & 0 \\ 0 & 0 & v \end{pmatrix},
    \]
    which has positive eigenvalues for $v>0$ and small $e\geq 0$ and is thus positive definite (with respect to the fiber inner product from Lemma~\ref{LemmaCInner}). (Cf.\ the  positivity of $-\ell_{g,\cd,e}(\dd r)$ stated in Proposition~\ref{PropEC}\eqref{ItECPosEsc}.) Therefore, $b_\cface$ is smooth.
  \item Since $\dd\frac{t}{r}$ is past timelike on $\supp\chi_1'$, Lemma~\ref{LemmaCTimelike} implies that $r^2\ell(\dd\frac{t}{r})$ is negative definite, and thus $e_\tau$ is smooth.
  \item The term $f$ is supported in the elliptic set of $P_h$.
  \end{itemize}
  Quantizing~\eqref{EqMiBoxCommcf} and proceeding as in the proof of Proposition~\ref{PropEK}, one obtains the estimate
  \begin{equation}
  \label{EqMiBoxcfFwd}
  \begin{split}
    &\| B_0\chi_2\omega \|_{\rho_\sface^{\alpha_\sface}\rho_\cface^{\alpha_\cface}L^2} \\
    &\qquad \leq C\Bigl( h^{-1}\|B_1\chi_1 P_h\omega\|_{\rho_\sface^{\alpha_\sface+2}\rho_\cface^{\alpha_\cface-2}L^2} + \|E\chi_0(1-\chi_2)\omega\|_{\rho_\sface^{\alpha_\sface}\rho_\cface^{\alpha_\cface}L^2} + h^N\|\chi_0\omega\|_{\rho_\sface^{\alpha_\sface}\rho_\cface^{\alpha_\cface}L^2}\Bigr);
  \end{split}
  \end{equation}
  here $B_0,B_1\in\Psi_{\ebop,\semi}(M)$ are elliptic at the zero section over $\cface$ and have operator wave front sets in small neighborhoods thereof, while $E$ is elliptic near the zero section and has $\WFebh'(E)$ near the zero section. See the left panel of Figure~\ref{FigMicf}. Thus, the a priori control term $E\chi_0(1-\chi_2)$ controls the term arising from $e_\tau$ in~\eqref{EqMiBoxCommcf}. It can itself be controlled by the energy estimate~\eqref{EqMiBoxEnFwd}.

  \begin{figure}[!ht]
  \centering
  \includegraphics{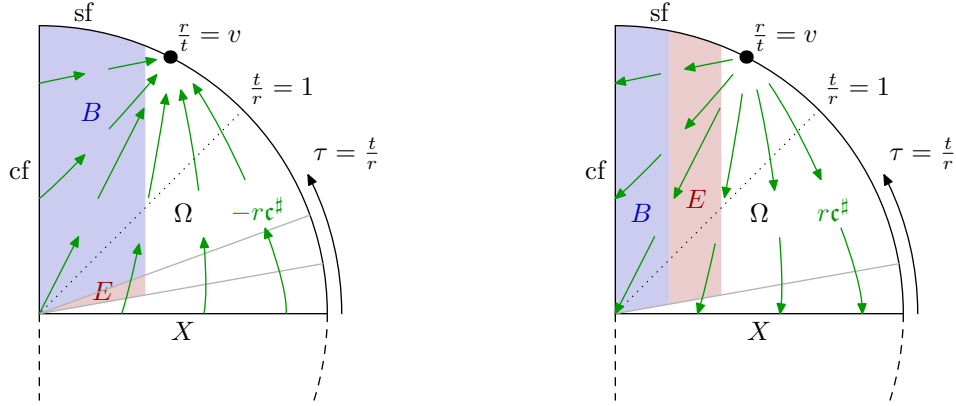}
  \caption{\textit{On the left:} illustration of the estimate~\eqref{EqMiBoxcfFwd}, which propagates control at the zero section from the operator wave front set of $E$ (red) to the elliptic set of $B_0$ (blue). The flow of the vector field $-r\cd^\sharp$ is indicated by green arrows. \textit{On the right:} illustration of the estimate~\eqref{EqMiBoxcfBwd}. The flow of the vector field $r\cd^\sharp$ is indicated by green arrows.}
  \label{FigMicf}
  \end{figure}

  The zero section propagation \emph{into} $\cface$ for the operator $P_h^*$ is proved similarly; the resulting estimate reads
  \begin{equation}
  \label{EqMiBoxcfBwd}
  \begin{split}
    &\| B_0\chi_4(\tau)\chi_{\cface,2}\omega \|_{\rho_\sface^{-\alpha_\sface-2}\rho_\cface^{-\alpha_\cface+2}L^2} \\
    &\qquad \leq C\Bigl( h^{-1}\|B_1\chi_3(\tau)\chi_{\cface,1}P_h^*\omega\|_{\rho_\sface^{-\alpha_\sface}\rho_\cface^{-\alpha_\cface}L^2} + \|E\chi_2(\tau)\chi_{\cface,0}(1-\chi_{\cface,2})\omega\|_{\rho_\sface^{-\alpha_\sface-2}L^2} \\
    &\qquad \hspace{8em} + h^N\|\chi_2(\tau)\chi_{\cface,0}\omega\|_{\rho_\sface^{-\alpha_\sface-2}\rho_\cface^{-\alpha_\cface+2}L^2}\Bigr).
  \end{split}
  \end{equation}
  Here, $\chi_{\cface,j}=\chi_{\cface,j}(\rho_\cface)\in\CIc([0,\digamma^{-1}))$ is equal to $1$ near $0$ and on $\supp\chi_{\cface,j+1}$, and $\digamma$ is sufficiently large. Thus,~\eqref{EqMiBoxcfBwd} propagates control from a punctured neighborhood of $\cface$ into $\cface$ itself, with additional localization in $\tau=\frac{t}{r}$. See the right panel of Figure~\ref{FigMicf}.

  \pfstep{Zero section propagation in $\sface^\circ$.} Away from $\cface$, semiclassical edge-b- and 3b-spaces are the same, and thus we can simply quote Proposition~\ref{PropES} for propagation into the radial sink at $\sface\cap\{\frac{r}{t}=v\}$ under the assumption $\alpha_\sface<-\frac12$, and the discussion in~\S\ref{SssER} for the real principal type propagation along $\sface^\circ$ (in the direction of increasing, resp.\ decreasing $\frac{r}{t}$ for $\frac{r}{t}<v$, resp.\ $\frac{r}{t}>v$).

  \pfstep{Combination: solvability.} Combining the energy estimate~\eqref{EqMiBoxEnFwd} with the estimate~\eqref{EqMiBoxInfty} at nonzero semiclassical edge-b-frequencies and the estimates~\eqref{EqMiBoxcfFwd} (radial source estimate at $\cface$), real principal type propagation, and Proposition~\ref{PropES} (radial sink estimate at $\sface\cap\{\frac{r}{t}=v\}$) gives
  \[
    \|\omega\|_{\dot H_{\ebop,h}^{s,(\alpha_\sface,\alpha_\cface)}(\Omega)} \leq C\Bigl(h^{-1}\|P_h\omega\|_{\dot H_{\ebop,h}^{s-1,(\alpha_\sface+2,\alpha_\cface-2)}(\Omega)} + h^N\|\omega\|_{\dot H_{\ebop,h}^{s,(\alpha_\sface,\alpha_\cface)}(\Omega)}\Bigr).
  \]
  For sufficiently small $h>0$, the error term on the right-hand side can be absorbed into the left-hand side, and we conclude that
  \begin{equation}
  \label{EqMiBoxFwdEst}
    \|\omega\|_{\dot H_{\ebop,h}^{s,(\alpha_\sface,\alpha_\cface)}(\Omega)} \leq C h^{-1}\|P_h\omega\|_{\dot H_{\ebop,h}^{s-1,(\alpha_\sface+2,\alpha_\cface-2)}(\Omega)},
  \end{equation}
  initially for $s=1$ and then for arbitrary but fixed $s\geq 1$ using an extension/restriction argument.

  For $P_h^*$, we propagate in the reverse direction (starting with the radial source estimate at $\sface\cap\{\frac{r}{t}=v\}$, followed by real principal type propagation, which yields control of the term involving $E$ in~\eqref{EqMiBoxcfBwd}, and ultimately applying the energy estimate~\eqref{EqMiBoxEnBwd}) and obtain
  \[
    \|\omega\|_{\bar H_{\ebop,h}^{-s+1,(-\alpha_\sface-2,-\alpha_\cface+2)}(\Omega)} \leq C h^{-1}\|P_h^*\omega\|_{\bar H_{\ebop,h}^{-s,(-\alpha_\sface,-\alpha_\cface)}(\Omega)}
  \]
  for $-s+1\geq 1$. By duality, this gives solvability of the equation~\eqref{EqMiBoxFwd}, i.e., $P_h\omega=h^2 f$, together with the quantitative estimate~\eqref{EqMiBoxFwdEst} for $s\leq 0$, and then for all $s\in\R$ by a propagation of regularity argument. The initial value problem~\eqref{EqMiBoxIVP} can be solved by a straightforward adaptation of the arguments at the end of~\S\ref{SssEP}.

  \pfstep{Strong uniqueness.} The uniqueness results in \cite{HintzConicWave} require, \emph{as an input}, the invertibility of certain model operators (see \cite[Definition~3.12]{HintzConicWave}) which in the present setting would be $\wh{\Box_g^{\cC_h}}(\sigma)$, $|\sigma|=1$, $\Im\sigma\geq 0$---the invertibility of which we are trying to establish. We must thus instead again rely on semiclassical tools. To prove the uniqueness claim of Theorem~\ref{ThmMiBox}\eqref{ItMiBoxFwd}, it suffices to show that if $\Box_g^{\cC_h}\omega'=0$ on $\Omega_{\leq 1}=\ol{\{0\leq t\leq 1\}}$, where $\omega'$ satisfies~\eqref{EqMiBoxFwdp} for $t_0=1$, then $\omega'=0$ on $\Omega_{\leq\frac12}$. (Iterating this $2\lceil t_0\rceil$ times for the pullback of $\omega'$ along translations in $t$ gives $\omega'=0$ on $\Omega_{\leq t_0}$, and taking $t_0\to\infty$ gives $\omega'=0$.) By finite speed of propagation for smooth coefficient wave-type equations, we have $\omega'=0$ for $t\in[0,1]$ and $r>t$. It thus suffices to show the following local uniqueness statement: let
  \[
    \Omega_\smallrighttriangle := \cl\{ 0\leq t\leq 1,\ r\leq 1-t\} \subset M,
  \]
  and suppose that $\omega'\in\rho_\cface^{\alpha_\cface}\Heb^{-N}(\Omega_\smallrighttriangle)^{\bullet,-}$ (supported character at $t=0$, extendible character at $r=1-t$) solves $\Box_g^{\cC_h}\omega'=0$. Then $\omega'=0$ on $\Omega_\smallrighttriangle$. (Note that the union of $\{0\leq t\leq 1,\ r\geq t\}$ and $\{0\leq t\leq 1,\ r\leq 1-t\}$ contains $\{0\leq t\leq\frac12\}$.) But this follows from the semiclassical a priori estimate
  \[
    \|\omega\|_{\rho_\cface^{\alpha_\cface}H_{\ebop,h}^s(\Omega_\smallrighttriangle)^{\bullet,-}} \leq C h^{-1}\|P_h\omega\|_{\rho_\cface^{\alpha_\cface-2}H_{\ebop,h}^{s-1}(\Omega_\smallrighttriangle)^{\bullet,-}},
  \]
  valid for any fixed $s\in\R$; this estimate in turn follows from a combination of energy estimates (now including one near the final hypersurface $r=1-t$ of $\Omega_\smallrighttriangle$, utilizing localization in $\frac{t-1}{r}$) and microlocal estimates as above.

  Uniqueness of solutions to the initial value problem~\eqref{EqMiBoxIVP} follows from an application of these considerations to the difference of two solutions $\omega$ and $\omega'$.
\end{proof}

\subsection{Conclusion of the proof of Theorem~\usref{ThmMiInv}}
\label{SsMiPf}

Without loss, we may assume $C_0\geq 2$. In the notation of Theorem~\ref{ThmMiInv}, we consider the fixed orders
\begin{equation}
\label{EqMiPfEllBeta}
  \ell = -\frac12,\quad
  \beta = \frac72.
\end{equation}
In view of Proposition~\ref{PropMiFred}, it suffices to show that
\begin{equation}
\label{EqMiPf0}
  \omega_0 \in H_{\scop,\bop}^{1,(\ell,\beta)}(X;\cT^*_X),\quad \wh{\Box_g^{\cC_h}}(\sigma)\omega_0 = 0\ \implies\ \omega_0=0.
\end{equation}

To this end, we use Theorem~\ref{ThmMiBox} for
\[
  \alpha_\sface = -\frac52,\quad
  \alpha_\cface = \frac52.
\]
(As the value of $e_0>0$ in Theorem~\ref{ThmMiInv}, we take the smaller one of the ones produced by Theorems~\ref{ThmMiBox} and \ref{ThmM0}.) Thus, $\omega(t,x) := e^{-i\sigma(t-r)}\omega_0$ solves $\Box_g^{\cC_h}\omega=0$ with initial data
\begin{equation}
\label{EqMiPfID}
  (\omega,\,r\cL_{\pa_t}\omega)|_{t=0} = (e^{i\sigma r}\omega_0,\,-i\sigma r e^{i\sigma r}\omega_0) =: (\omega_{(0)},\omega_{(1)}).
\end{equation}
Since $\Im\sigma\geq 0$, we have
\[
  \omega_{(0)} \in \rho_\scop^\ell\rho_\bop^\beta H_{\scop,\bop}^1 \subset \rho_\scop^{\ell-1}\rho_\bop^\beta H_\bop^1,\quad
  \omega_{(1)} \in \rho_\scop^{\ell-1}\rho_\bop^\beta L^2,
\]
and thus $\omega_{(j)}\in\rho_\sface^{\ell-1}\rho_\cface^{\beta}H_\bop^{1-j}(X,r^2|\dd r\,\dd\omega|;\cT^*_X)$; note here that on $X$, we can take $\rho_\sface$ and $\rho_\cface$ to be $\rho_\scop$ and $\rho_\bop$, respectively. For the choices~\eqref{EqMiPfEllBeta}, the space here is contained in $\rho_\sface^{\alpha_\sface+\frac12}\rho_\cface^{\alpha_\cface-\frac12}L^2$ (with room to spare). Theorem~\ref{ThmMiBox} now implies that the initial value problem for $\Box_g^{\cC_h}\omega'=0$ with initial data~\eqref{EqMiPfID} has a solution
\begin{equation}
\label{EqMiPfOmegap}
  \omega' \in \rho_\sface^{\alpha_\sface}\rho_\cface^{\alpha_\cface}\bar H^1_\ebop(\Omega;\cT^*).
\end{equation}
But since the restriction of $\omega$ to $\Omega_{\leq t_0}=\Omega\cap\{t\leq t_0\}$ is easily checked to lie in $\rho_\sface^{\alpha_\sface}\rho_\cface^{\alpha_\cface}\bar H_\ebop^1$ for all $t_0$, the uniqueness part of Theorem~\ref{ThmMiBox}\eqref{ItMiBoxIVP} implies $\omega'=\omega$.

Suppose now that $\omega_0\neq 0$, then $|\omega_0|\geq c>0$ on some non-empty open set $U\subset\R^3_x\setminus\{0\}$ with compact closure $\bar U\subset\R^3_x\setminus\{0\}$. But then, using the boundary defining functions $\rho_\cface=\frac{r}{t}$ and $\rho_\sface=\frac{1}{t}$ in $t\geq 1$, $r\leq t$, we have, for $T:=\max(1,\sup_U r)$,
\[
  \int_U \int_T^\infty \rho_\sface^{-2\alpha_\sface}\rho_\cface^{-2\alpha_\cface}|\omega|^2\,\dd t\,\dd x \geq c^2\int_U \biggl(\int_T^\infty t^{2(\alpha_\sface+\alpha_\cface)}\,\dd t\biggr) r^{-2\alpha_\cface}\,\dd x = \infty
\]
since $2(\alpha_\sface+\alpha_\cface)=0\geq-1$, contradicting the membership~\eqref{EqMiPfOmegap} of $\omega'=\omega$.

\bibliographystyle{alphaurl}


\begin{thebibliography}{GMGCH05}

\bibitem[AAG18]{AngelopoulosAretakisGajicLate}
Yannis Angelopoulos, Stefanos Aretakis, and Dejan Gajic.
\newblock Late-time asymptotics for the wave equation on spherically symmetric,
  stationary spacetimes.
\newblock {\em Adv. Math.}, 323:529--621, 2018.
\newblock \href {https://doi.org/10.1016/j.aim.2017.10.027}
  {\path{doi:10.1016/j.aim.2017.10.027}}.

\bibitem[AAG23]{AngelopoulosAretakisGajicKerr}
Yannis Angelopoulos, Stefanos Aretakis, and Dejan Gajic.
\newblock Late-time tails and mode coupling of linear waves on {K}err
  spacetimes.
\newblock {\em Advances in Mathematics}, 417:108939, 2023.
\newblock \href {https://doi.org/https://doi.org/10.1016/j.aim.2023.108939}
  {\path{doi:https://doi.org/10.1016/j.aim.2023.108939}}.

\bibitem[AHW26]{AnderssonHaefnerWhitingMode}
Lars Andersson, Dietrich H\"afner, and Bernard~F. Whiting.
\newblock Mode analysis for the linearized {E}instein equations on the {K}err
  metric: the large {${a}$} case.
\newblock {\em J. Eur. Math. Soc. (JEMS)}, 28(9):3919--3981, 2026.
\newblock \href {https://doi.org/10.4171/jems/1544}
  {\path{doi:10.4171/jems/1544}}.

\bibitem[AMPW17]{AnderssonMaPaganiniWhitingModeStab}
Lars Andersson, Siyuan Ma, Claudio Paganini, and Bernard~F. Whiting.
\newblock Mode stability on the real axis.
\newblock {\em J. Math. Phys.}, 58(7):072501, 19, 2017.
\newblock \href {https://doi.org/10.1063/1.4991656}
  {\path{doi:10.1063/1.4991656}}.

\bibitem[Bes20]{BessetRNdSDecay}
Nicolas Besset.
\newblock Decay of the local energy for the charged {K}lein--{G}ordon equation
  in the exterior {D}e {S}itter--{R}eissner--{N}ordstr\"om spacetime.
\newblock {\em Ann. Henri Poincar\'e}, 21(8):2433--2484, 2020.
\newblock \href {https://doi.org/10.1007/s00023-020-00919-z}
  {\path{doi:10.1007/s00023-020-00919-z}}.

\bibitem[BFHR99]{BrodbeckFrittelliHubnerReulaSCP}
Othmar Brodbeck, Simonetta Frittelli, Peter H\"ubner, and Oscar~A. Reula.
\newblock Einstein's equations with asymptotically stable constraint
  propagation.
\newblock {\em J. Math. Phys.}, 40(2):909--923, 1999.
\newblock \href {https://doi.org/10.1063/1.532694}
  {\path{doi:10.1063/1.532694}}.

\bibitem[BH08]{BonyHaefnerDecay}
Jean-Fran{\c{c}}ois Bony and Dietrich H{\"a}fner.
\newblock Decay and non-decay of the local energy for the wave equation on the
  de {S}itter--{S}chwarzschild metric.
\newblock {\em Communications in Mathematical Physics}, 282(3):697--719, 2008.
\newblock \href {https://doi.org/10.1007/s00220-008-0553-y}
  {\path{doi:10.1007/s00220-008-0553-y}}.

\bibitem[BL67]{BoyerLindquistKerr}
Robert~H. Boyer and Richard~W. Lindquist.
\newblock Maximal analytic extension of the {K}err metric.
\newblock {\em J. Mathematical Phys.}, 8:265--281, 1967.
\newblock \href {https://doi.org/10.1063/1.1705193}
  {\path{doi:10.1063/1.1705193}}.

\bibitem[BVW15]{BaskinVasyWunschRadMink}
Dean Baskin, Andr\'as Vasy, and Jared Wunsch.
\newblock Asymptotics of radiation fields in asymptotically {M}inkowski space.
\newblock {\em Amer. J. Math.}, 137(5):1293--1364, 2015.
\newblock \href {https://doi.org/10.1353/ajm.2015.0033}
  {\path{doi:10.1353/ajm.2015.0033}}.

\bibitem[CK93]{ChristodoulouKlainermanStability}
Demetrios Christodoulou and Sergiu Klainerman.
\newblock {\em The global nonlinear stability of the {M}inkowski space},
  volume~41 of {\em Princeton Mathematical Series}.
\newblock Princeton University Press, Princeton, NJ, 1993.

\bibitem[DH72]{DuistermaatHormanderFIO2}
Johannes~J. Duistermaat and Lars H{\"o}rmander.
\newblock Fourier integral operators. {II}.
\newblock {\em Acta Math.}, 128(3-4):183--269, 1972.
\newblock \href {https://doi.org/10.1007/BF02392165}
  {\path{doi:10.1007/BF02392165}}.

\bibitem[DHRT21]{DafermosHolzegelRodnianskiTaylorSchwarzschild}
Mihalis Dafermos, Gustav Holzegel, Igor Rodnianski, and Martin Taylor.
\newblock The nonlinear stability of the {S}chwarzschild solution to
  gravitational perturbations.
\newblock {\em Preprint, arXiv:2104.08222}, 2021.

\bibitem[Dya11a]{DyatlovQNMExtended}
Semyon Dyatlov.
\newblock Exponential energy decay for {K}err--de {S}itter black holes beyond
  event horizons.
\newblock {\em Math. Res. Lett.}, 18(5):1023--1035, 2011.
\newblock \href {https://doi.org/10.4310/MRL.2011.v18.n5.a19}
  {\path{doi:10.4310/MRL.2011.v18.n5.a19}}.

\bibitem[Dya11b]{DyatlovQNM}
Semyon Dyatlov.
\newblock Quasi-normal modes and exponential energy decay for the {K}err-de
  {S}itter black hole.
\newblock {\em Comm. Math. Phys.}, 306(1):119--163, 2011.
\newblock \href {https://doi.org/10.1007/s00220-011-1286-x}
  {\path{doi:10.1007/s00220-011-1286-x}}.

\bibitem[Dya12]{DyatlovAsymptoticDistribution}
Semyon Dyatlov.
\newblock Asymptotic distribution of quasi-normal modes for {K}err--de {S}itter
  black holes.
\newblock {\em Annales Henri Poincar{\'e}}, 13(5):1101--1166, 2012.
\newblock \href {https://doi.org/10.1007/s00023-012-0159-y}
  {\path{doi:10.1007/s00023-012-0159-y}}.

\bibitem[Fan26]{FangKdS}
Allen~Juntao Fang.
\newblock Nonlinear stability of the slowly-rotating {K}err--de {S}itter
  family.
\newblock {\em Ann. PDE}, 12(1):Paper No. 5, 79, 2026.
\newblock \href {https://doi.org/10.1007/s40818-025-00227-x}
  {\path{doi:10.1007/s40818-025-00227-x}}.

\bibitem[Fri86]{FriedrichStability}
Helmut Friedrich.
\newblock On the existence of {$n$}-geodesically complete or future complete
  solutions of {E}instein's field equations with smooth asymptotic structure.
\newblock {\em Comm. Math. Phys.}, 107(4):587--609, 1986.
\newblock \href {https://doi.org/10.1007/BF01205488}
  {\path{doi:10.1007/BF01205488}}.

\bibitem[Gan14]{GannotSAdS}
Oran Gannot.
\newblock Quasinormal modes for {S}chwarzschild--{A}d{S} black holes:
  exponential convergence to the real axis.
\newblock {\em Comm. Math. Phys.}, 330(2):771--799, 2014.
\newblock \href {https://doi.org/10.1007/s00220-014-2002-4}
  {\path{doi:10.1007/s00220-014-2002-4}}.

\bibitem[Gan17]{GannotKerrAdS}
Oran Gannot.
\newblock Existence of quasinormal modes for {K}err--{A}d{S} black holes.
\newblock {\em Ann. Henri Poincar\'e}, 18(8):2757--2788, 2017.
\newblock \href {https://doi.org/10.1007/s00023-017-0568-z}
  {\path{doi:10.1007/s00023-017-0568-z}}.

\bibitem[GH08]{GuillarmouHassellResI}
Colin Guillarmou and Andrew Hassell.
\newblock Resolvent at low energy and {R}iesz transform for {S}chr\"{o}dinger
  operators on asymptotically conic manifolds. {I}.
\newblock {\em Math. Ann.}, 341(4):859--896, 2008.
\newblock \href {https://doi.org/10.1007/s00208-008-0216-5}
  {\path{doi:10.1007/s00208-008-0216-5}}.

\bibitem[GKS24]{GiorgiKlainermanSzeftelStability}
Elena Giorgi, Sergiu Klainerman, and J\'er\'emie Szeftel.
\newblock Wave equations estimates and the nonlinear stability of slowly
  rotating {K}err black holes.
\newblock {\em Pure Appl. Math. Q.}, 20(7):2865--3849, 2024.
\newblock \href {https://doi.org/10.4310/pamq.241128023033}
  {\path{doi:10.4310/pamq.241128023033}}.

\bibitem[GMGCH05]{GundlachCalabreseHinderMartinConstraintDamping}
Carsten Gundlach, Jose~M. Martin-Garcia, Gioel Calabrese, and Ian Hinder.
\newblock {Constraint damping in the Z4 formulation and harmonic gauge}.
\newblock {\em Class. Quant. Grav.}, 22:3767--3774, 2005.
\newblock \href {http://arxiv.org/abs/gr-qc/0504114}
  {\path{arXiv:gr-qc/0504114}}, \href
  {https://doi.org/10.1088/0264-9381/22/17/025}
  {\path{doi:10.1088/0264-9381/22/17/025}}.

\bibitem[HHV21]{HaefnerHintzVasyKerr}
Dietrich H{\"a}fner, Peter Hintz, and Andr{\'a}s Vasy.
\newblock Linear stability of slowly rotating {K}err black holes.
\newblock {\em Inventiones mathematicae}, 223:1227--1406, 2021.
\newblock \href {https://doi.org/10.1007/s00222-020-01002-4}
  {\path{doi:10.1007/s00222-020-01002-4}}.

\bibitem[HHV25]{HaefnerHintzVasyKerrLarge}
Dietrich H\"afner, Peter Hintz, and Andr\'as Vasy.
\newblock Linear stability of {K}err black holes in the full subextremal range.
\newblock {\em Preprint, arXiv:2506.21183}, 2025.

\bibitem[Hin17]{HintzPsdoInner}
Peter Hintz.
\newblock Resonance expansions for tensor-valued waves on asymptotically
  {K}err--de {S}itter spaces.
\newblock {\em J. Spectr. Theory}, 7:519--557, 2017.
\newblock \href {https://doi.org/10.4171/JST/171} {\path{doi:10.4171/JST/171}}.

\bibitem[Hin18]{HintzKNdSStability}
Peter Hintz.
\newblock {N}on-linear {S}tability of the {K}err--{N}ewman--de {S}itter
  {F}amily of {C}harged {B}lack {H}oles.
\newblock {\em Annals of PDE}, 4(1):11, Apr 2018.
\newblock \href {https://doi.org/10.1007/s40818-018-0047-y}
  {\path{doi:10.1007/s40818-018-0047-y}}.

\bibitem[Hin22]{HintzPrice}
Peter Hintz.
\newblock {A} sharp version of {P}rice's law for wave decay on asymptotically
  flat spacetimes.
\newblock {\em Communications in Mathematical Physics}, 389:491--542, 2022.
\newblock \href {https://doi.org/10.1007/s00220-021-04276-8}
  {\path{doi:10.1007/s00220-021-04276-8}}.

\bibitem[Hin23a]{HintzMink4Gauge}
Peter Hintz.
\newblock Exterior stability of {M}inkowski space in generalized harmonic
  gauge.
\newblock {\em Archive for Rational Mechanics and Analysis}, 247(99), 2023.
\newblock \href {https://doi.org/10.1007/s00205-023-01931-3}
  {\path{doi:10.1007/s00205-023-01931-3}}.

\bibitem[Hin23b]{HintzGlueLocI}
Peter Hintz.
\newblock Gluing small black holes along timelike geodesics {I}: formal
  solution.
\newblock {\em Preprint, arXiv:2306.07409}, 2023.

\bibitem[Hin23c]{HintzNonstat}
Peter Hintz.
\newblock Linear waves on non-stationary asymptotically flat spacetimes. {I}.
\newblock {\em Preprint, arXiv:2302.14647}, 2023.

\bibitem[Hin23d]{Hintz3b}
Peter Hintz.
\newblock Microlocal analysis of operators with asymptotic translation- and
  dilation-invariances.
\newblock {\em Preprint, arXiv:2302.13803}, 2023.

\bibitem[Hin24]{HintzGlueLocIII}
Peter Hintz.
\newblock Gluing small black holes along timelike geodesics {III}: construction
  of true solutions and extreme mass ratio mergers.
\newblock {\em Preprint, arXiv:2408.06715}, 2024.

\bibitem[Hin25a]{HintzMicro}
Peter Hintz.
\newblock {\em {A}n {I}ntroduction to {M}icrolocal {A}nalysis}.
\newblock Graduate Texts in Mathematics, No. 304. Springer-Verlag, Cham, 2025.
\newblock \href {https://doi.org/10.1007/978-3-031-90706-7}
  {\path{doi:10.1007/978-3-031-90706-7}}.

\bibitem[Hin25b]{HintzKdSMS}
Peter Hintz.
\newblock Mode stability and shallow quasinormal modes of {K}err--de {S}itter
  black holes away from extremality.
\newblock {\em J. Eur. Math. Soc. (JEMS)}, 27(12):4891--4996, 2025.
\newblock \href {https://doi.org/10.4171/jems/1463}
  {\path{doi:10.4171/jems/1463}}.

\bibitem[Hin26a]{HintzConicWave}
Peter Hintz.
\newblock Local theory of wave equations with timelike curves of conic
  singularities.
\newblock {\em Advances in Mathematics}, 494:110925, 2026.
\newblock URL:
  \url{https://www.sciencedirect.com/science/article/pii/S0001870826001477},
  \href {https://doi.org/https://doi.org/10.1016/j.aim.2026.110925}
  {\path{doi:https://doi.org/10.1016/j.aim.2026.110925}}.

\bibitem[Hin26b]{HintzNonstat2}
Peter Hintz.
\newblock ({N}on-)linear waves on asymptotically flat spacetimes. {II}:
  trapping, bound states, nonlinear applications.
\newblock {\em Preprint, arXiv:2606.28008v1}, 2026.

\bibitem[Hin26c]{HintzKerrStab}
Peter Hintz.
\newblock Nonlinear stability of subextremal {K}err black holes.
\newblock {\em Preprint, arXiv:2606.28253}, 2026.

\bibitem[Hin26d]{HintzScaledBddGeo}
Peter Hintz.
\newblock Pseudodifferential operators on manifolds with scaled bounded
  geometry.
\newblock {\em Commun. Am. Math. Soc.}, 6:153--243, 2026.
\newblock \href {https://doi.org/10.1090/cams/58} {\path{doi:10.1090/cams/58}}.

\bibitem[H{\"o}r71]{HormanderEnseignement}
Lars H{\"o}rmander.
\newblock On the existence and the regularity of solutions of linear
  pseudodifferential equations.
\newblock {\em Enseignement Math.}, 2(17):99--163, 1971.

\bibitem[H{\"o}r07]{HormanderAnalysisPDE3}
Lars H{\"o}rmander.
\newblock {\em The analysis of linear partial differential operators. {III}}.
\newblock Classics in Mathematics. Springer, Berlin, 2007.
\newblock Pseudo-differential operators, Reprint of the 1994 edition.
\newblock \href {https://doi.org/10.1007/978-3-540-49938-1}
  {\path{doi:10.1007/978-3-540-49938-1}}.

\bibitem[HPV25]{HintzPetersenVasyKdS}
Peter Hintz, Oliver Petersen, and Andr\'as Vasy.
\newblock {C}onditional non-linear stability of {K}err--de~{S}itter spacetimes:
  the full subextremal range.
\newblock {\em Preprint, arXiv:2508.06620}, 2025.

\bibitem[HV15]{HintzVasySemilinear}
Peter Hintz and Andr{\'a}s Vasy.
\newblock Semilinear wave equations on asymptotically de {S}itter, {K}err--de
  {S}itter and {M}inkowski spacetimes.
\newblock {\em Anal. PDE}, 8(8):1807--1890, 2015.
\newblock \href {https://doi.org/10.2140/apde.2015.8.1807}
  {\path{doi:10.2140/apde.2015.8.1807}}.

\bibitem[HV18]{HintzVasyKdSStability}
Peter Hintz and Andr{\'a}s Vasy.
\newblock {T}he global non-linear stability of the {K}err--de {S}itter family
  of black holes.
\newblock {\em Acta mathematica}, 220:1--206, 2018.
\newblock \href {https://doi.org/10.4310/acta.2018.v220.n1.a1}
  {\path{doi:10.4310/acta.2018.v220.n1.a1}}.

\bibitem[HV20]{HintzVasyMink4}
Peter Hintz and Andr{\'a}s Vasy.
\newblock {S}tability of {M}inkowski space and polyhomogeneity of the metric.
\newblock {\em Annals of PDE}, 6(2), 2020.
\newblock \href {https://doi.org/10.1007/s40818-020-0077-0}
  {\path{doi:10.1007/s40818-020-0077-0}}.

\bibitem[HV24]{HintzVasyKdSCosm}
Peter Hintz and Andr{\'a}s Vasy.
\newblock Stability of the expanding region of {K}err--de~{S}itter spacetimes
  and smoothness at the conformal boundary.
\newblock {\em Preprint, arXiv:2409.15460}, 2024.

\bibitem[HV26]{HintzVasyScrieb}
Peter Hintz and Andr\'as Vasy.
\newblock Microlocal analysis near null infinity in asymptotically flat
  spacetimes.
\newblock {\em Anal. PDE}, 19(1):1--106, 2026.
\newblock \href {https://doi.org/10.2140/apde.2026.19.1}
  {\path{doi:10.2140/apde.2026.19.1}}.

\bibitem[HX22]{HintzXieSdS}
Peter Hintz and YuQing Xie.
\newblock Quasinormal modes of small {S}chwarzschild--de {S}itter black holes.
\newblock {\em Journal of Mathematical Physics}, 63(1):011509, 2022.
\newblock \href {https://doi.org/10.1063/5.0062985}
  {\path{doi:10.1063/5.0062985}}.

\bibitem[HZ25]{HitrikZworskiQNM}
Michael Hitrik and Maciej Zworski.
\newblock {O}verdamped {QNM} for {S}chwarzschild {B}lack {H}oles.
\newblock {\em Annals of PDE}, 11(2):28, Oct 2025.
\newblock \href {https://doi.org/10.1007/s40818-025-00222-2}
  {\path{doi:10.1007/s40818-025-00222-2}}.

\bibitem[Ian17]{IantchenkoRNdSDirac}
Alexei Iantchenko.
\newblock {Quasi-normal modes for de Sitter-Reissner-Nordstr\"om Black Holes}.
\newblock {\em Math. Res. Lett.}, 24:83--117, 2017.
\newblock \href {https://doi.org/10.4310/MRL.2017.v24.n1.a5}
  {\path{doi:10.4310/MRL.2017.v24.n1.a5}}.

\bibitem[Ian18]{IantchenkoKNdSDirac}
Alexei Iantchenko.
\newblock {Quasi-normal modes for Dirac fields in the Kerr-Newman-de Sitter
  black holes.}
\newblock {\em Anal. Appl. , Singap.}, 16(4):449--524, 2018.
\newblock \href {https://doi.org/10.1142/S0219530518500057}
  {\path{doi:10.1142/S0219530518500057}}.

\bibitem[Kei18]{KeirWeak}
Joseph Keir.
\newblock {T}he weak null condition and global existence using the p-weighted
  energy method.
\newblock {\em Preprint, arXiv:1808.09982}, 2018.

\bibitem[Ker63]{KerrKerr}
Roy~P. Kerr.
\newblock Gravitational field of a spinning mass as an example of algebraically
  special metrics.
\newblock {\em Phys. Rev. Lett.}, 11:237--238, 1963.
\newblock \href {https://doi.org/10.1103/PhysRevLett.11.237}
  {\path{doi:10.1103/PhysRevLett.11.237}}.

\bibitem[KI03]{KodamaIshibashiMaster}
Hideo Kodama and Akihiro Ishibashi.
\newblock A master equation for gravitational perturbations of maximally
  symmetric black holes in higher dimensions.
\newblock {\em Progr. Theoret. Phys.}, 110(4):701--722, 2003.
\newblock \href {https://doi.org/10.1143/PTP.110.701}
  {\path{doi:10.1143/PTP.110.701}}.

\bibitem[KS21]{KlainermanSzeftelPolarized}
Sergiu Klainerman and J{\'e}r{\'e}mie Szeftel.
\newblock {\em {G}lobal {N}onlinear {S}tability of {S}chwarzschild {S}pacetime
  under {P}olarized {P}erturbations}, volume 210 of {\em Annals of Mathematics
  Studies}.
\newblock Princeton University Press, Princeton, NJ, 2021.
\newblock \href {https://doi.org/10.1515/9780691218526}
  {\path{doi:10.1515/9780691218526}}.

\bibitem[KS23]{KlainermanSzeftelKerr}
Sergiu Klainerman and J\'er\'emie Szeftel.
\newblock Kerr stability for small angular momentum.
\newblock {\em Pure Appl. Math. Q.}, 19(3):791--1678, 2023.
\newblock \href {https://doi.org/10.4310/pamq.2023.v19.n3.a1}
  {\path{doi:10.4310/pamq.2023.v19.n3.a1}}.

\bibitem[LO24]{LukOhTwoTails}
Jonathan Luk and Sung-Jin Oh.
\newblock Late time tail of waves on dynamic asymptotically flat spacetimes of
  odd space dimensions.
\newblock {\em Preprint, arXiv:2404.02220}, 2024.

\bibitem[LR10]{LindbladRodnianskiGlobalStability}
Hans Lindblad and Igor Rodnianski.
\newblock The global stability of {M}inkowski space-time in harmonic gauge.
\newblock {\em Ann. of Math. (2)}, 171(3):1401--1477, 2010.
\newblock \href {https://doi.org/10.4007/annals.2010.171.1401}
  {\path{doi:10.4007/annals.2010.171.1401}}.

\bibitem[Maz91]{MazzeoEdge}
Rafe~R. Mazzeo.
\newblock {Elliptic theory of differential edge operators I}.
\newblock {\em Communications in Partial Differential Equations},
  16(10):1615--1664, 1991.
\newblock \href {https://doi.org/10.1080/03605309108820815}
  {\path{doi:10.1080/03605309108820815}}.

\bibitem[Mel93]{MelroseAPS}
Richard~B. Melrose.
\newblock {\em The {A}tiyah-{P}atodi-{S}inger index theorem}, volume~4 of {\em
  Research Notes in Mathematics}.
\newblock A K Peters, Ltd., Wellesley, MA, 1993.
\newblock \href {https://doi.org/10.1016/0377-0257(93)80040-i}
  {\path{doi:10.1016/0377-0257(93)80040-i}}.

\bibitem[Mel94]{MelroseEuclideanSpectralTheory}
Richard~B. Melrose.
\newblock Spectral and scattering theory for the {L}aplacian on asymptotically
  {E}uclidian spaces.
\newblock In {\em Spectral and scattering theory ({S}anda, 1992)}, volume 161
  of {\em Lecture Notes in Pure and Appl. Math.}, pages 85--130. Dekker, New
  York, 1994.

\bibitem[MM83]{MelroseMendozaB}
Richard~B. Melrose and Gerardo Mendoza.
\newblock {\em Elliptic operators of totally characteristic type}.
\newblock Mathematical Sciences Research Institute, 1983.

\bibitem[MSBV14a]{MelroseSaBarretoVasyResolvent}
Richard Melrose, Ant\^onio S\'a{}~Barreto, and Andr\'as Vasy.
\newblock Analytic continuation and semiclassical resolvent estimates on
  asymptotically hyperbolic spaces.
\newblock {\em Comm. Partial Differential Equations}, 39(3):452--511, 2014.
\newblock \href {https://doi.org/10.1080/03605302.2013.866957}
  {\path{doi:10.1080/03605302.2013.866957}}.

\bibitem[MSBV14b]{MelroseSaBarretoVasySdS}
Richard Melrose, Ant\^onio S\'a{}~Barreto, and Andr\'as Vasy.
\newblock Asymptotics of solutions of the wave equation on de
  {S}itter-{S}chwarzschild space.
\newblock {\em Comm. Partial Differential Equations}, 39(3):512--529, 2014.
\newblock \href {https://doi.org/10.1080/03605302.2013.866958}
  {\path{doi:10.1080/03605302.2013.866958}}.

\bibitem[Pre05]{PretoriusBinaryBlackHole}
Frans Pretorius.
\newblock Evolution of binary black-hole spacetimes.
\newblock {\em Phys. Rev. Lett.}, 95(12):121101, 2005.
\newblock \href {https://doi.org/10.1103/PhysRevLett.95.121101}
  {\path{doi:10.1103/PhysRevLett.95.121101}}.

\bibitem[PV25]{PetersenVasySubextremal}
Oliver Petersen and Andr\'as Vasy.
\newblock Wave equations in the {K}err--de {S}itter spacetime: the full
  subextremal range.
\newblock {\em J. Eur. Math. Soc. (JEMS)}, 27(8):3497--3526, 2025.
\newblock \href {https://doi.org/10.4171/jems/1448}
  {\path{doi:10.4171/jems/1448}}.

\bibitem[Rin08]{RingstromEinsteinScalarStability}
Hans Ringstr{\"o}m.
\newblock Future stability of the {E}instein--non-linear scalar field system.
\newblock {\em Inventiones mathematicae}, 173(1):123--208, 2008.
\newblock \href {https://doi.org/10.1007/s00222-008-0117-y}
  {\path{doi:10.1007/s00222-008-0117-y}}.

\bibitem[RW57]{ReggeWheelerSchwarzschild}
Tullio Regge and John~A. Wheeler.
\newblock Stability of a {S}chwarzschild singularity.
\newblock {\em Phys. Rev.}, 108(4):1063--1069, 1957.
\newblock \href {https://doi.org/10.1103/PhysRev.108.1063}
  {\path{doi:10.1103/PhysRev.108.1063}}.

\bibitem[SBZ97]{SaBarretoZworskiResonances}
Ant{\^o}nio S{\'a}~Barreto and Maciej Zworski.
\newblock Distribution of resonances for spherical black holes.
\newblock {\em Mathematical Research Letters}, 4:103--122, 1997.
\newblock URL: \url{https://dx.doi.org/10.4310/MRL.1997.v4.n1.a10}.

\bibitem[Shu92]{ShubinBounded}
M.~A. Shubin.
\newblock Spectral theory of elliptic operators on noncompact manifolds.
\newblock {\em Ast\'erisque}, (207):5, 35--108, 1992.
\newblock M\'ethodes semi-classiques, Vol.\ 1 (Nantes, 1991).

\bibitem[SR15]{ShlapentokhRothmanModeStability}
Yakov Shlapentokh-Rothman.
\newblock Quantitative mode stability for the wave equation on the {K}err
  spacetime.
\newblock {\em Ann. Henri Poincar\'e}, 16(1):289--345, 2015.
\newblock \href {https://doi.org/10.1007/s00023-014-0315-7}
  {\path{doi:10.1007/s00023-014-0315-7}}.

\bibitem[Sus24]{SussmanResolventPhg}
Ethan Sussman.
\newblock Complete asymptotic analysis of low energy scattering for
  {S}chr\"odinger operators with a short-range potential.
\newblock {\em Preprint, arXiv:2411.04220}, 2024.

\bibitem[Teu72]{TeukolskySeparation}
Saul~A. Teukolsky.
\newblock Rotating black holes: separable wave equations for gravitational and
  electromagnetic perturbations.
\newblock {\em Phys. Rev. Lett.}, 29(16):1114--1118, 1972.
\newblock \href {https://doi.org/10.1103/PhysRevLett.29.1114}
  {\path{doi:10.1103/PhysRevLett.29.1114}}.

\bibitem[Vas00]{VasyThreeBody}
Andr\'as Vasy.
\newblock Propagation of singularities in three-body scattering.
\newblock {\em Ast\'erisque}, (262):vi+151, 2000.

\bibitem[Vas13]{VasyMicroKerrdS}
Andr{\'a}s Vasy.
\newblock Microlocal analysis of asymptotically hyperbolic and {K}err--de
  {S}itter spaces (with an appendix by {S}emyon {D}yatlov).
\newblock {\em Invent. Math.}, 194(2):381--513, 2013.
\newblock \href {https://doi.org/10.1007/s00222-012-0446-8}
  {\path{doi:10.1007/s00222-012-0446-8}}.

\bibitem[Vas18]{VasyMinicourse}
Andr\'as Vasy.
\newblock A minicourse on microlocal analysis for wave propagation.
\newblock In {\em Asymptotic analysis in general relativity}, volume 443 of
  {\em London Math. Soc. Lecture Note Ser.}, pages 219--374. Cambridge Univ.
  Press, Cambridge, 2018.

\bibitem[Vas21a]{VasyLAPLag}
Andr\'as Vasy.
\newblock Limiting absorption principle on {R}iemannian scattering
  (asymptotically conic) spaces, a {L}agrangian approach.
\newblock {\em Communications in Partial Differential Equations},
  46(5):780--822, 2021.
\newblock \href {https://doi.org/10.1080/03605302.2020.1857400}
  {\path{doi:10.1080/03605302.2020.1857400}}.

\bibitem[Vas21b]{VasyLowEnergy}
Andr{\'a}s Vasy.
\newblock {R}esolvent near zero energy on {R}iemannian scattering
  (asymptotically conic) spaces.
\newblock {\em Pure and Applied Analysis}, 3(1):1--74, 2021.
\newblock \href {https://doi.org/10.2140/paa.2021.3.1}
  {\path{doi:10.2140/paa.2021.3.1}}.

\bibitem[Vas21c]{VasyLowEnergyLag}
Andr{\'a}s Vasy.
\newblock {R}esolvent near zero energy on {R}iemannian scattering
  (asymptotically conic) spaces, a {L}agrangian approach.
\newblock {\em Communications in Partial Differential Equations},
  46(5):823--863, 2021.
\newblock \href {https://doi.org/10.1080/03605302.2020.1857401}
  {\path{doi:10.1080/03605302.2020.1857401}}.

\bibitem[Vis70]{VishveshwaraSchwarzschild}
C.~V. Vishveshwara.
\newblock Stability of the {S}chwarzschild metric.
\newblock {\em Phys. Rev. D}, 1(10):2870--2879, 1970.
\newblock \href {https://doi.org/10.1103/PhysRevD.1.2870}
  {\path{doi:10.1103/PhysRevD.1.2870}}.

\bibitem[War15]{WarnickQNMs}
Claude~M. Warnick.
\newblock {O}n quasinormal modes of asymptotically anti-de {S}itter black
  holes.
\newblock {\em Communications in Mathematical Physics}, 333(2):959--1035, 2015.
\newblock \href {https://doi.org/10.1007/s00220-014-2171-1}
  {\path{doi:10.1007/s00220-014-2171-1}}.

\bibitem[Whi89]{WhitingKerrModeStability}
Bernard~F. Whiting.
\newblock Mode stability of the {K}err black hole.
\newblock {\em J. Math. Phys.}, 30(6):1301--1305, 1989.
\newblock \href {https://doi.org/10.1063/1.528308}
  {\path{doi:10.1063/1.528308}}.

\bibitem[Zer70]{ZerilliPotential}
Frank~J. Zerilli.
\newblock Effective potential for even-parity {R}egge--{W}heeler gravitational
  perturbation equations.
\newblock {\em Phys. Rev. Lett.}, 24(13):737--738, 1970.
\newblock \href {https://doi.org/10.1103/PhysRevLett.24.737}
  {\path{doi:10.1103/PhysRevLett.24.737}}.

\end{thebibliography}

\end{document}